\pdfoutput=1
\documentclass[10pt]{amsart}
\usepackage{lscape}
\usepackage{tikz, stmaryrd}
\usetikzlibrary{matrix,arrows,decorations.pathmorphing}

\usepackage{color}
\usepackage{amsmath,amssymb, amsfonts, amscd}
\usepackage[all]{xy}
\usepackage{verbatim}
\usepackage{enumerate}
\usepackage{tikz-cd}
\usepackage{yhmath}
\usepackage{xcolor}
\usepackage{hyperref}

\usepackage{graphicx}
\usepackage{lipsum}

\newcommand{\adj}[4]{#1\negmedspace: #2\rightleftarrows #3:\negmedspace #4}
\newtheorem{thm}{Theorem}[section]
\newtheorem{cor}[thm]{Corollary}
\newtheorem{lem}[thm]{Lemma}

\newtheorem{ass}[thm]{Assumption}
\newtheorem{prop}[thm]{Proposition}

\newtheorem{defn}[thm]{Definition}

\newtheorem{example}[thm]{Example}

\newtheorem{rem}[thm]{Remark}
\numberwithin{equation}{section}

\DeclareFontFamily{U}{rsf}{} \DeclareFontShape{U}{rsf}{m}{n}{
  <5> <6> rsfs5 <7> <8> <9> rsfs7 <10->  rsfs10}{}
\DeclareMathAlphabet{\mathscr}{U}{rsf}{m}{n}

\newcommand*{\defeq}{\mathrel{\vcenter{\baselineskip0.5ex \lineskiplimit0pt
                     \hbox{\scriptsize.}\hbox{\scriptsize.}}}%
                     =}

\usepackage{mathrsfs}
\DeclareMathAlphabet{\mathpzc}{OT1}{pzc}{m}{it}

\newcommand{\Q}{{\mathbb Q}}

\renewcommand{\imath}{\sqrt{-1}}

\def\Z{{\mathbb Z}}

\usepackage{scalerel,stackengine}
\stackMath
\newcommand\reallywidehat[1]{%
\savestack{\tmpbox}{\stretchto{%
  \scaleto{%
    \scalerel*[\widthof{\ensuremath{#1}}]{\kern-.6pt\bigwedge\kern-.6pt}%
    {\rule[-\textheight/2]{1ex}{\textheight}}%WIDTH-LIMITED BIG WEDGE
  }{\textheight}% 
}{0.5ex}}%
\stackon[1pt]{#1}{\tmpbox}%
}

\DeclareMathOperator{\colim}{colim}

\def\sup{{\rm sup\,}}

\numberwithin{equation}{section}

\begin{document}
\title[]{Koszul Duality in Exact Categories}
\author{Jack Kelly}
\address{Jack Kelly, Lincoln College, Oxford.}
\hfill \break
\bigskip
\email{kellyj29@tcd.ie}
\thanks{This research was conducted during the author's graduate studies supported by the EPSRC studentship BK/13/007, while employed at the University of Oxford on the EPSRC grant `Symmetries and Correspondences', while supported by the Simons Foundation at Trinity College Dublin under the
program `Targeted Grants to Institutes', and during his time at Lincoln College, Oxford.}

\begin{abstract}
In this paper we establish Koszul duality-type results in the setting of chain complexes in exact categories. In particular, we prove generalisations of Vallette's Koszul duality theorem \cite{vallette2014homotopy} and Koszul duality along the lines of Lurie \cite{lurie2011derived} and Pridham \cite{MR2628795}. We also prove a connective version. We conclude with some applications. These include our main example, the category of complete bornological spaces over a Banach field, as well as a sketch of deformation theory in formal geometry. 
\end{abstract}
\maketitle
\newpage
\newpage
\tableofcontents
\newpage
\section*{Introduction}
%while their derived versions are smooth

\subsection*{Background and Motivation}
Derived algebraic geometry has provided significant new insights into the theory of algebraic moduli spaces. Classical moduli spaces are often pathological and singular, but can be realised as truncations of smooth derived spaces. One tool that has proven fundamental for analysing moduli spaces in derived geometry is Koszul duality. Koszul duality, at least for Lie algebras and augmented commutative algebras, says that, over a field of characteristic zero, there is an adjunction of $(\infty,1)$-categories 
$$\adj{\mathbf{D}_{\kappa}}{(\textbf{cdga}^{aug})^{op}}{\textbf{dgla}}{\hat{\mathbf{C}}_{\kappa}}$$
between the category of augmented commutative differentially graded algebras and the category of differentially graded Lie algebras. Moreover, after forgetting the Lie algebra structure, the underlying complex of $\mathbf{D}_{\kappa}(A)$ can be identified with a shift of the tangent complex of $A$ at the canonical point of $\mathrm{Spec}(A)$ corresponding to the augmentation of $A$. One can also give conditions under which the unit $\eta_{\mathfrak{g}}:\mathfrak{g}\rightarrow \mathbf{D}_{\kappa}\circ \hat{\mathbf{C}}_{\kappa}(\mathfrak{g})$ is an equivalence. The idea behind the application of Koszul duality to moduli theory is as follows. If $m$ is a point of a moduli problem $M$, one considers the shifted tangent complex at this point $\mathbb{T}_{M,m}[-1]$. This object is equipped with a canonical Lie algebra structure (it is essentially the tangent Lie algebra to the based loop group) that typically satisfies the necessary assumptions for the unit to be an equivalence. The spectrum of the commutative algebra that is Koszul dual to this Lie algebra can be regarded as a formal neighbourhood of the point $m$. This idea is most conveniently formulated by Lurie using the language of $(\infty,1)$-categories and formal moduli problems in \cite{lurie2011derived}, but as pointed out in loc. cit., these ideas appeared previously in work by, for example, Hinich (\cite{MR1843805}, \cite{MR2101417}), Kontsevich-Soibelman (\cite{kontsevich2002deformation}), Manetti (\cite{MR1891232}), and Pridham (\cite{MR2628795}).

The adjunction above actually factors through an \textit{equivalence} of $(\infty,1)$-categories
$$\adj{\Omega_{\kappa}}{\textbf{coComm}_{\kappa}^{coaug}}{\textbf{dgla}}{\textbf{B}_{\kappa}}$$
between the category of differentially graded Lie algebras and a certain $(\infty,1)$-category of cocommutative coalgebras. Note that this category of coalgebras is not simply the localisation of coalgebras in chain complexes at quasi-isomorphisms, but something more refined. The functor $\hat{\mathbf{C}}_{\kappa}$ is obtained by composing $\textbf{B}_{\kappa}$ with the dualising functor $(-)^{\vee}:\textbf{coComm}^{coaug}_{\kappa}\rightarrow\textbf{cdga}^{op}$. Hinich gives a geometric interpretation of this equivalence by interpreting coalgebras as representing \textit{formal stacks}.

Many moduli spaces, such as the moduli space of instantons or of Galois representations, are naturally analytic or smooth rather than algebraic in nature. In its nascence, the goal of this paper was to develop an analogue of Koszul duality that could be applied to the study of analytic and smooth moduli problems. The idea is as follows. Derived algebraic geometry over a ring $R$ may be viewed as derived geometry, in the sense of To\"{e}n and Vezzosi \cite{toen2004homotopical}, relative to the monoidal model category $Ch({}_{R}\mathpzc{Mod})$ (or $Ch_{\ge0}({}_{R}\mathpzc{Mod})$). The category of affine spaces is opposite to the category of cdgas, and a suitable Grothendieck topology allows one to define and study derived stacks. Following work of Bambozzi, Ben-Bassat, and Kremnizer \cite{koren}, \cite{orenbambozzi}, \cite{bambozzi} on non-derived analytic geometry, the novel perspective of the author, Ben-Bassat, and Kremnizer suggests that derived analytic geometry over a Banach field $R$ of characteristic $0$ may be viewed as derived geometry relative to the monoidal model category $Ch(Ind(Ban_{R}))$ or $Ch_{\ge0}(Ind(Ban_{R}))$. This has since been developed in \cite{dang}. Here $Ind(Ban_{R})$ is the formal completion of the category of Banach $R$-modules by inductive limits. One can also consider instead the full subcategory $CBorn_{R}$ of complete bornological $R$-modules. Our analytic analogue of Koszul duality relates the categories of commutative algebras, Lie algebras, and cocommutative coalgebras internal to the category $Ch(Ind(Ban_{R}))$. The present work is a crucial ingredient that will allow us to utilise the general setup of formal moduli problems from \cite{MR4509772} in our context. Applications to deformation theory in derived analytic and smooth geometry will be the subject of future joint work with Rhiannon Savage, though we sketch an example in the context of formal geometry.

Before precisely outlining our results, let us first review some of the extensive history of Koszul duality. Vallette's work on duality \cite{vallette2014homotopy} generalises previous work of Hinich \cite{MR1843805}, which interprets duality between Lie algebras and cocommutative coalgebras in terms of formal stacks. This in turn generalises results from the seminal work of Quillen \cite{quillen1969rational} on rational homotopy theory. Getzler and Jones \cite{getzler1994operads} have also done crucial work on Koszul duality. Motivated by studying differential forms on iterated loop spaces, they in particular study duality for $E_{n}$-algebras. In the process of establishing chiral Koszul duality, Francis and Gaitsgory \cite{francis2012chiral} prove a general Koszul duality result in the context of pro-nilpotent $(\infty,1)$-categories. This has since been developed in the spectral case by Ching and Harper (\cite{ching2015derived}), and fully solved by Heuts \cite{heuts2024koszul}. Note that the Ching-Harper, Francis-Gaitsgory, and Heuts story of `derived Koszul duality' is related to, but somewhat different from, the flavour of Koszul duality in the present paper. We will remark on the relationship in this work. A version of Koszul duality for operads has been established by Ginzburg and Kapranov \cite{ginzburg1994koszul}. There is also a version for curved operads due to Hirsh and Mill\`{e}s \cite{hirsh2012curved}. The relationship between Koszul duality and deformation theory has also been extensively studied by Kontsevich and Soibelman in \cite{kontsevich2000deformations} and \cite{kontsevich2003deformation}, as well as by Lurie \cite{lurie2011derived} and Hennion \cite{hennion2015tangent}. In \cite{MR4398644}, following previous work of the same author in \cite{MR2723021}, Positselski considers non-homogeneous Koszul duality over a general non-commutative base ring. Recent work of Roca i Lucio \cite{https://doi.org/10.48550/arxiv.2209.09833} discusses Koszul duality for \textit{absolute algebras} (in some sense adically completed, or `pro-nilpotent' algebras) in the case of vector spaces over a field of characteristic zero. We expect that the results therein also hold in more general exact categories. Le Grignou and Roca i Lucio have also considered Koszul duality for fields of positive characteristic (\cite{grignou2023homotopicaloperadiccalculuspositive},   \cite{grignou2023symmetric}, \cite{i2023higher}, \cite{grignou2023new}).

%In it's nascence the goal of this paper was to prove a generalisation of Koszul duality for Lie algebras and (co)commutative (co)algebras internal to the category $Ch(Ind(Ban_{R}))$. 

\subsection*{Outline of Results}

Although motivated by Koszul duality for the category $CBorn_{R}$ of bornological (or ind-Banach) modules over a Banach ring $R$, we will in fact prove Koszul duality for a large class of monoidal model categories, that includes the categories of vector spaces over a field, locally convex topological algebras, adically complete Banach algebras, and sheaves. The model structure on $Ch(CBorn_{R})$ arises from a Quillen exact (in fact quasi-abelian) structure on $CBorn_{R}$.  We introduce the notion of a Koszul category $\mathpzc{M}$. This is a monoidal model category of the form ${}_{R}\mathpzc{Mod}(Ch(\mathpzc{E}))$, where $\mathpzc{E}$ is an exact category, in which the techniques of Koszul duality work.

In fact, we shall prove something much more general, which also incorporates, for example, generalisations of $E_{n}$ self-duality \cite{getzler1994operads}. Namely, we generalise the notion of a Koszul twisting morphism $\alpha:\mathfrak{C}\rightarrow\mathfrak{P}$ from a cooperad to an operad. We then identify an $(\infty,1)$-category of coalgebras over $\mathfrak{C}$, $\textbf{coAlg}^{conil}_{\mathfrak{C}}$, which is equivalent to the $(\infty,1)$-category of algebras over $\mathfrak{P}$. Our main theorem is the following (Theorem \ref{coopKoszuldual}).

% has proven a fundamental concept in the study and development of algebraic geometry, particularly formal geometry, moduli theory, and deformation theory. 

\begin{thm}
Let $\mathpzc{M}$ be a Koszul category, and $\alpha:\mathfrak{C}\rightarrow\mathfrak{P}$ a Koszul morphism in $\mathpzc{M}$. The bar-cobar adjunction induces an adjoint equivalence of $(\infty,1)$-categories
$$\adj{\Omega_{\alpha}}{\textbf{coAlg}^{conil}_{\mathfrak{C}}}{\textbf{Alg}_{\mathfrak{P}}}{\textbf{B}_{\alpha}}.$$
\end{thm}

(Note that later in the statement of the theorem there is an extra decoration $\mathfrak{C}_{top}$ on the coalgebra side. This is because $\mathfrak{C}$ will be equipped with a filtration, and $\mathfrak{C}_{top}$ denotes the underlying object of this filtered object. For clarity of presentation in the introduction, we omit this here.)

The $\mathfrak{P}$-algebra side is the $(\infty,1)$-category presented by the transferred model structure along the free-forgetful adjunction
$$\adj{\mathrm{Free}_{\mathfrak{P}}(-)}{\mathpzc{M}}{\mathpzc{Alg}_{\mathfrak{P}}(\mathpzc{M})}{|-|_{\mathfrak{P}}}.$$
However, the coalgebra side is the hammock localisation of the relative categories in which the weak equivalences are those morphisms $f:C\rightarrow D$ such that $\Omega_{\alpha}(f)$ is a weak equivalence in the transferred model structure on $\mathpzc{Alg}_{\mathfrak{P}}$.

The statement of the theorem, and much of its proof, is a generalisation of the proof of \cite{vallette2014homotopy} Theorem 2.1 (1) and (2). Given a Koszul twisting morphism $\alpha:\mathfrak{C}\rightarrow\mathfrak{P}$ satisfying some additional assumptions, we consider the $(\infty,1)$-category of algebras over the dual operad of $\mathfrak{C}$, $\mathfrak{C}^{\vee}$. In the case of the twisting morphism $\mathfrak{S}^{c}\otimes_{H}\mathfrak{coComm}^{ncu}\rightarrow\mathfrak{Lie}$ this is of course just the $(\infty,1)$-category of commutative algebras. We consider the functor
$$\hat{\mathbf{C}}_{\alpha}\defeq(-)^{\vee}\circ\textbf{B}_{\alpha}:\textbf{Alg}_{\mathfrak{P}}\rightarrow\textbf{Alg}_{\mathfrak{C}^{\vee}}^{op}.$$
Specialising to what we call $\alpha$-\textit{operadic Koszul categories}, our first theorem on operadic Koszul duality is the following (Theorem \ref{alphadjoint}).
\begin{thm}
Let $\alpha$ be a Koszul twisting morphism in a Koszul category $\mathpzc{M}$. Suppose that $\mathpzc{M}$ is $\alpha$-operadic. Then the functor $\hat{\mathbf{C}}_{\alpha}:\textbf{Alg}_{\mathfrak{P}}\rightarrow(\textbf{Alg}_{\mathfrak{C}^{\vee}})^{op}$ admits a right adjoint $\mathbf{D}_{\alpha}$.
\end{thm}

Then we will specialise to the Koszul twisting morphism $\kappa:\mathfrak{S}^{c}\otimes_{H}\mathfrak{coComm}^{ncu}\rightarrow\mathfrak{Lie}$. As in \cite{lurie2011derived}, for an augmented commutative algebra $A$, we relate the underlying object of the Lie algebra $\mathbf{D}_{\kappa}(A)$ with the shifted tangent complex $\mathbb{T}_{0}(A)[-1]$ of $A$ at its canonical point $A\rightarrow R$ (Proposition \ref{shiftedtanget}). Denote by $\mathrm{L^{H}}(\mathpzc{M})$ the $(\infty,1)$-categorical localisation of $\mathpzc{M}$ at its weak equivalences.

\begin{thm}
Let $|-|:\textbf{Alg}_{\mathfrak{Lie}}(\mathrm{L^{H}}(\mathpzc{M}))\rightarrow\mathrm{L^{H}}(\mathpzc{M})$ be the forgetful functor. There is a natural equivalence of functors $|-|\circ\mathbf{D}_{\kappa}\cong\mathbb{T}_{0}[-1]$. 
\end{thm}

Finally, we prove a general theorem identifying when the unit of the adjunction is an equivalence. As a consequence, we obtain an interesting result in the setting of complete bornological vector spaces. Recall that a Banach field $k$ is said to be \textit{spherically complete} if any nested sequence of open balls has a non-trivial intersection. For us the important consequence of this assumption is the Hahn-Banach Theorem. This implies that $k$ is injective as an object of the quasi-abelian category of Banach spaces. We obtain the following result (Theorem \ref{analyticopkoszul}). 

\begin{thm}
Let $k$ be a spherically complete field, and let $\mathfrak{g}$ be a bornological Lie algebra over $k$ concentrated in negative degrees and such that each $|\mathfrak{g}|_{n}$ is a dual nuclear bornological space, where $|\mathfrak{g}|$ denotes the underlying object of $\mathfrak{g}$. If the differentials in $|\mathfrak{g}|$ are Fredholm operators, then $\eta_{\mathfrak{g}}:\mathfrak{g}\rightarrow\mathbf{D}_{\kappa}\circ\hat{\mathbf{C}}_{\kappa}(\mathfrak{g})$ is an equivalence.
\end{thm} 

We further sketch an example from formal geometry, involving Koszul duality for $E_{\infty}$-algebras and lifting problems for points in (affine) formal stacks.
%
%We then set about proving a vast generalisation of Vallete's formally planar cooperadic Koszul duality  for monoidal elementary quasi-abelian categories. Before outlining our results we first review some of the vast history of Koszul duality and its various manifestiations. 

\subsection*{Structure of the Paper}
The paper is laid out as follows. After establishing the conventions of the paper, in Sections \ref{appcatalg} and \ref{algcoalghom} we introduce the two main ingredients of our version of Koszul duality, namely operad theory in additive categories, and homotopy theory in exact categories. Section \ref{appcatalg} mainly consists of recalling foundational results from \cite{loday2012algebraic} reformulated to work for monoidal additive categories admitting certain colimits. The primary goal of this section is to arrive at the \textit{bar-cobar construction} for a twisting morphism between a formally planar cooperad and an operad. After recalling some results about monoidal model structures on categories of chain complexes in exact categories from \cite{kelly2016homotopy}, and the homotopy theory of algebras over operads in Section \ref{algcoalghom}, in Section \ref{homotopexact} our work begins in earnest. Namely, we introduce the notion of a \textit{Koszul category}. Koszul categories are essentially categories of chain complexes in exact categories equipped with (almost) monoidal model structures such that Koszul duality for Koszul twisting morphisms works. We then give a method of constructing such categories from \textit{basic Koszul categories} that provides a rich class of examples. We also discuss connective Koszul categories as well as categories of filtered objects in Koszul categories, and their associated $(\infty,1)$-categories. We then work out the theory of purely filtered operads in Koszul categories.

Section \ref{seccoopkosz} is the main event. After defining the categories of coalgebras we consider, and studying their homotopy theory, we prove various generalisations of the Koszul duality of \cite{vallette2014homotopy}. Then in Section \ref{secopkosz} we prove a generalisation of Lurie's \cite{lurie2011derived} version of Koszul duality. Finally, in Section \ref{secexamples} we give various examples of our version of Koszul duality, including generalisations of previously known results for categories of modules over a ring, as well as totally new results for categories of complete bornological spaces over Banach fields.

%We conclude by suggesting some further directions in which this work could be continued, such as analytic and smooth versions of chiral Koszul duality. 

%and co-algebras over formally planar cooperads in Koszul categories. 
%Typically the proofs from \cite{LodyaVallete}, written in the context of vector spaces over a field, work mutatis mutandis for such categories and so we omit them. 
\subsection*{Notation and Conventions}\label{notation}

Throughout this work we will use the following notation.
\begin{itemize}
\item 
The symbol $\mathbb{N}_{0}$ denotes the set of natural numbers including $0$, and $\mathbb{N}$ denotes the natural numbers not including $0$. 
\item
We will denote $1$-categories using the mathpzc font $\mathpzc{C},\mathpzc{D},\mathpzc{E}$, etc. In particular, we denote by $\mathpzc{Ab}$ the category of abelian groups and ${}_{\Q}\mathpzc{Vect}$ the category of $\Q$-vector spaces.
\item 
For categories $\mathpzc{I}$ and $\mathpzc{M}$, we denote by
$\mathpzc{Fun}(\mathpzc{I},\mathpzc{M})^{cocont}$
the category of cocontinuous (i.e., colimit-preserving) functors.
\item
 If $\mathpzc{M}$ is a model category, or a category with weak equivalences, its associated $(\infty,1)$-category will be denoted by $\mathrm{L^{H}}(\mathpzc{M})$ or, when there is no confusion, by boldface $\textbf{M}$. For concreteness, we fix quasi-categories as our model for $(\infty,1)$-categories.
\item 
Operads will be denoted using capital fraktur letters $\mathfrak{C},\mathfrak{P}$, etc. The category of algebras over an operad will be denoted $\mathpzc{Alg}_{\mathfrak{P}}$.
%Algebras over an operad will generally be denoted using small fractal letters $\mathfrak{g},\mathfrak{h}$, etc.
 \item
We denote the operads for unital associative algebras, unital commutative algebras, non-unital commutative algebras, and Lie algebras by $\mathfrak{Ass},\mathfrak{Comm},\mathfrak{Comm}^{nu}$, and $\mathfrak{Lie}$, respectively. Similarly, when enriched over $\mathbb{Q}$, we denote the cooperads for counital cocommutative coalgebras and non-counital cocommutative coalgebras by $\mathfrak{coComm}$ and $\mathfrak{coComm}^{ncu}$, respectively.
\item
%If $\mathfrak{P}$ is a (co)operad in category $\mathpzc{E}$ and $V$ is an objet of $\mathpzc{E}$, then typically we will denote the (co)free (co)algebra on $V$ by $\mathfrak{P}(V)$.
 For the operads $\mathfrak{Ass},\mathfrak{Comm},\mathfrak{Lie}$ we will denote the corresponding free algebras by $T(V)$, $S(V)$, and $\mathfrak{Lie}(V)$, respectively. We also denote by $\hat{S}(V)$ the commutative algebra of formal power series on an object $V$.
 \item
Unless stated otherwise, the unit in a monoidal category will be denoted by $k$, the tensor functor by $\otimes$, and, for a closed monoidal category, the internal hom functor will be denoted by $\underline{\mathrm{Hom}}$. 
\item
For symmetric monoidal categories, the symmetric braiding will be denoted $\sigma$.
\item
We will assume that tensor functors commute with colimits in each variable.
\item
Filtered colimits will be denoted by $\colim$. Projective limits will be denoted $\lim_{\leftarrow}$.
%\item
%We will say a category is complete and cocomplete if it has all limits and colimits.
\end{itemize}
%\subsection{Chain Complexes}\label{chaincomplexes}
\subsubsection*{Chain Complexes}
Let us now introduce some conventions for chain complexes. Let $\underline{Gr}_{\mathbb{Z}}(\mathpzc{E})$ denote the category whose objects are $\mathbb{Z}$-indexed collections of objects of $\mathpzc{E}$, which we usually write as $\bigoplus_{n\in\mathbb{Z}}E_{n}$. If $\bigoplus_{n\in\mathbb{Z}}E_{n}$ and $\bigoplus_{n\in\mathbb{Z}}F_{n}$ are two objects in this category, then we set 
$$Hom_{\underline{Gr}_{\mathbb{Z}}(\mathpzc{E})}(\bigoplus_{m\in\mathbb{Z}}E_{m},\bigoplus_{n\in\mathbb{Z}}F_{n})\defeq\bigoplus_{n\in\mathbb{Z}}\prod_{i\in\mathbb{Z}}Hom_{\mathpzc{E}}(E_{i},F_{i+n})$$
If $f\in \prod_{i}Hom_{\mathpzc{E}}(E_{i},F_{i+n})$, then $f$ is said to have degree $n$. We denote the subgroup $\prod_{i}Hom_{\mathpzc{E}}(E_{i},F_{i+n})$ consisting of degree $n$ morphisms by $Hom_{n}(E,F)$. The category $\underline{Gr}_{\mathbb{Z}}(\mathpzc{E})$ can be made into a symmetric monoidal category, in which
$$(X_{\bullet}\otimes Y_{\bullet})_{n}=\bigoplus_{i+j=n}X_{i}\otimes Y_{j}.$$

We shall adopt the Koszul sign rule for categories enriched in $\underline{Gr}_{\mathbb{Z}}(\mathpzc{Ab})$. Namely, if $f:V\rightarrow V'$ is a degree $p$ morphism and $g:W\rightarrow W'$ is a degree $q$ morphism, then one defines

$$f\otimes g|_{V_{m}\otimes W_{n}}\defeq(-1)^{qm}f_{m}\otimes g_{n}.$$
The monoidal unit is the graded object $G_{0}(k)$, where $(G_{0}(k))_{0}=k$ and $(G_{0}(k))_{n}=0$ for $n\neq0$. 
We denote by $\mathpzc{Gr}_{\mathbb{Z}}(\mathpzc{E})$ the wide subcategory of $\underline{Gr}_{\mathbb{Z}}$ in which $Hom_{\mathpzc{Gr}_{\mathbb{Z}}}(E,F)\defeq Hom_{0}(E,F)$, and by $\mathpzc{Gr}_{\mathbb{N}_{0}}(\mathpzc{E})$ the full subcategory of $\mathpzc{Gr}_{\mathbb{Z}}(\mathpzc{E})$ on objects $\bigoplus_{n\in\Z}E_{n}$ such that $E_{n}=0$ for $n<0$.

\begin{defn}
A \textbf{pre-differential graded module} in $\mathpzc{E}$ is a pair $(X_{\bullet},d_{X})$, where $X\in\underline{Gr}_{\mathbb{Z}}(\mathpzc{E})$ and $d_{X}\in Hom_{-1}(X,X)$. If $(X_{\bullet},d_{X})$ and $(Y_{\bullet},d_{Y})$ are pre-differential graded modules, a morphism from $(X_{\bullet},d_{X})$ to $(Y_{\bullet},d_{Y})$ is a morphism $f\in Hom_{0}(X,Y)$ that commutes with $d_{X}$ and $d_{Y}$. The category of pre-differential graded modules is denoted $\widetilde{Ch}(\mathpzc{E})$.
\end{defn}

We will frequently use the following special pre-differential graded module.

\begin{defn}
If $E$ is an object of an additive category $\mathpzc{E}$, we let $S^{n}(E)\in \widetilde{Ch}(\mathpzc{E})$ be the pre-differential graded module whose $n$th entry is $E$, with all other entries being $0$. We also denote by $D^{n}(E)\in \widetilde{Ch}(\mathpzc{E})$ the pre-differential graded module whose $n$th and $(n-1)$st entries are $E$, with all other entries being $0$, and the differential $d_{n}$ being the identity. Finally, we write $j^{n}_{E}:S^{n}(E)\rightarrow D^{n+1}(E)$ for the natural inclusion. 
\end{defn}

We also define $\widetilde{Ch}_{\ge0}(\mathpzc{E})$ to be the full subcategory of $\widetilde{Ch}(\mathpzc{E})$ on objects $A_{\bullet}$ such that $A_{n}=0$ for $n<0$, $\widetilde{Ch}_{\le0}(\mathpzc{E})$ to be the full subcategory of $\widetilde{Ch}(\mathpzc{E})$ on objects $A_{\bullet}$ such that $A_{n}=0$ for $n>0$, $\widetilde{Ch}_{+}(\mathpzc{E})$ to be the full subcategory of objects $A_{\bullet}$ such that $A_{n}=0$ for $n\ll0$, $\widetilde{Ch}_{-}(\mathpzc{E})$ to be the full subcategory of objects $A_{\bullet}$ such that $A_{n}=0$ for $n\gg0$, and $\widetilde{Ch}_{b}(\mathpzc{E})$ to be the full subcategory of $\widetilde{Ch}(\mathpzc{E})$ on objects $A_{\bullet}$ such that $A_{n}\neq 0$ for only finitely many $n$.

The category $\widetilde{Ch}(\mathpzc{E})$ also comes equipped with a shift functor. It is given on objects by $(A_{\bullet}[1])_{i}=A_{i+1}$ with differential $d_{i}^{A[1]}=-d^{A}_{i+1}$. The shift of a morphism $f_{\bullet}$ is given by $(f_{\bullet}[1])_{i}=f_{i+1}$. The functor $[1]$ is an auto-equivalence with inverse denoted $[-1]$. We set $[0]=Id$ and $[n]=[1]^{n}$ for any integer $n$.\newline
\\
 Finally, we define the mapping cone as follows.

\begin{defn}
Let $X_{\bullet}$ and $Y_{\bullet}$ be pre-differential graded modules in an additive category $\mathpzc{E}$ and $f_{\bullet}:X_{\bullet}\rightarrow Y_{\bullet}$ be a morphism between them. The \textbf{mapping cone of} $f_{\bullet}$, denoted $ cone(f_{\bullet})$, is the pre-differential graded module whose components are
$$ cone(f_{\bullet})_{n}=X_{n-1}\oplus Y_{n},$$
and whose differentials are
 \begin{displaymath}
d^{ cone(f)}_{n}  = \left(
     \begin{array}{lr}
       -d_{n-1}^{X} &0\\
       -f_{n-1} & d^{Y}_{n}

   \end{array}
            \right).
\end{displaymath} 
\end{defn}
There is a natural morphism $\tau:Y_{\bullet}\rightarrow  cone(f)$ induced by the injections $Y_{i}\rightarrow X_{i-1}\oplus Y_{i}$, and a natural morphism $\pi: cone(f)\rightarrow X_{\bullet}[-1]$ induced by the projections $X_{i-1}\oplus Y_{i}\rightarrow X_{i-1}$. The sequence
$$Y_{\bullet}\rightarrow cone(f)\rightarrow X_{\bullet}[-1]$$
is split exact in each degree.

If $\mathpzc{E}$ is a symmetric monoidal additive category, we may regard $\widetilde{Ch}(\mathpzc{E})$ as a symmetric monoidal additive category itself by defining
$$(X,d_{X})\otimes (Y,d_{Y})\defeq(X\otimes Y,d_{X}\otimes Id_{Y}+Id_{X}\otimes d_{Y}),$$
where we are using the Koszul sign rule, so that we have
$$(d_{X}\otimes Id_{Y}+Id_{X}\otimes d_{Y})|_{X_{i}\otimes Y_{j}}=d_{X_{i}}\otimes Id_{Y_{j}}+(-1)^{i}Id_{X_{i}}\otimes d_{Y_{j}}.$$

\begin{defn}
The \textbf{category of chain complexes in }$\mathpzc{E}$, denoted $Ch(\mathpzc{E})$, is the full subcategory of $\widetilde{Ch}(\mathpzc{E})$ consisting of pre-differential graded modules $(X,d_{X})$ such that $d_{X}^{2}=0$. 
\end{defn}

All the constructions above restrict to $Ch(\mathpzc{E})$. In particular, if $\mathpzc{E}$ is a symmetric monoidal category then so is $Ch(\mathpzc{E})$. If $\mathpzc{E}$ is closed symmetric monoidal, then $Ch(\mathpzc{E})$ is also closed symmetric monoidal with internal hom
$$\underline{\mathrm{Hom}}(X_{\bullet},Y_{\bullet})_{n}=\prod_{i\in\Z}\underline{\mathrm{Hom}}_{\mathpzc{E}}(X_{i},Y_{i+n}),$$
and differential $d_{n}$ defined on $\mathrm{Hom}_{\mathpzc{E}}(X_{i},Y_{i+n})$ by

$$d=\underline{\mathrm{Hom}}(Id,d_{i+n}^{Y_{\bullet}})-(-1)^{n}\underline{\mathrm{Hom}}(d_{i+1}^{X_{\bullet}},Id).$$

For an arbitrary additive category $\mathpzc{E}$ (i.e., not necessarily closed symmetric monoidal) we write
$$\mathbf{Hom}(X_{\bullet},Y_{\bullet})$$
for the enrichment in $Ch(\mathpzc{Ab})$, where
$$\mathbf{Hom}_{n}(X_{\bullet},Y_{\bullet})=\prod_{i\in\mathbb{Z}}\mathrm{Hom}_{\mathpzc{E}}(X_{i},Y_{i+n})$$
and the differential is defined similarly to the one on the internal hom above.

\subsection*{Acknowledgements}
 The author would like to thank Kobi Kremnizer and the anonymous referees for many useful comments and conversations throughout the course of this work.
\subsection*{Declaration of AI use}
I sometimes use a non-pro version of ChatGPT 5.6 provided by my university to check for grammatical errors and for typos. I do not use it to generate proofs or ideas.

 \section{Operadic Algebra in Additive Categories}\label{appcatalg}

Before developing the homotopy theory of algebras in exact categories, let us first discuss some of the basics of operadic and formally planar cooperadic algebra in additive categories. For that part of this section concerning standard (and not new!) results on operads, cooperads and twisting morphisms, the main reference is \cite{loday2012algebraic}, where, unless stated otherwise, anything left unproved in this section can be found. While \cite{loday2012algebraic} works in the context of vector spaces, most of the proofs work mutatis mutandis for monoidal additive categories with some mild assumptions. At the outset, let us fix exactly what we mean by a monoidal additive category for the purposes of this paper. A \textbf{monoidal additive category} is an additive category $\mathpzc{E}$ together with a unital associative bifunctor $\otimes:\mathpzc{E}\times\mathpzc{E}\rightarrow\mathpzc{E}$ that commutes with colimits in each variable. We shall assume also that $\mathpzc{E}$ has countable colimits and kernels. The unit will be denoted $k$.

%Namely, we shall require that our additive category $\mathpzc{E}$ has kernels, cokernels, and countable products.
\subsection{$\Sigma$-Modules}
In the next two subsections we mostly follow \cite{loday2012algebraic} Chapter 5.
\subsubsection{Discrete Groups in Monoidal Categories}

Let $(\mathpzc{E},\otimes, k)$ be a monoidal category with all small coproducts, such that $\otimes$ commutes with coproducts in each variable. We denote by $k[-]:\mathpzc{Set}\rightarrow\mathpzc{E}$ the functor that sends a set $S$ to the object $k[S]=\coprod_{s\in S}k$. If $f:S\rightarrow T$ is a morphism of sets, then $k[f]:k[S]\rightarrow k[T]$ is the morphism that sends the copy of $k$ indexed by $s\in S$ to the copy indexed by $f(s)\in T$. Objects and morphisms in the essential image of the functor $k[-]:\mathpzc{Set}\rightarrow\mathpzc{E}$ will be called \textbf{discrete}.

\begin{prop}
Let $(\mathpzc{E},\otimes,k)$ be a monoidal category. Endow $\mathpzc{Set}$ with its Cartesian monoidal structure. Then the functor $k[-]:\mathpzc{Set}\rightarrow\mathpzc{E}$ is strong monoidal.
\end{prop}

\begin{proof}
Let $S$ and $T$ be sets. Then
$$
\Bigl(\coprod_{S}k\Bigr)\otimes\Bigl(\coprod_{T}k\Bigr)
\cong\coprod_{S}\coprod_{T} k\otimes k
\cong\coprod_{S}\coprod_{T} k
\cong\coprod_{S\times T}k.
$$
\end{proof}

In particular, $k[-]:\mathpzc{Set}\rightarrow\mathpzc{E}$ sends groups to Hopf monoids. If $G$ is a group, we call $k[G]$ the group monoid of $G$ in $\mathpzc{E}$. Note that, since $k[G]$ is a Hopf monoid, whenever $M$ and $N$ are $k[G]$-modules, there is a canonical $k[G]$-module structure on $M\otimes N$, such that ${}_{k[G]}\mathpzc{Mod}$ becomes a symmetric monoidal category. 

If $G$ and $H$ are groups and $H\rightarrow G$ is a morphism, then we get a morphism of group monoids $k[H]\rightarrow k[G]$. We denote by 
$Ind_{H}^{G}:{}_{k[H]}\mathpzc{Mod}\rightarrow {}_{k[G]}\mathpzc{Mod}$ the functor $k[G]\otimes_{k[H]}(-)$.

\begin{rem}\label{rem:ind}
    If $H\rightarrow G$ is the inclusion of a subgroup, then on underlying objects of $\mathpzc{E}$, $Ind_{H}^{G}(M)\cong(\bigoplus_{[g]\in G\big\slash H}k)\otimes M$.
\end{rem}

\subsubsection{Graded Objects and $\Sigma$-Modules}
%For a category $\mathpzc{E}$ we denote by  $\mathpzc{Gr}_{\mathbb{N}_{0}}(\mathpzc{E})$ the category of $\mathbb{N}_{0}$-\textbf{graded objects} in $\mathpzc{E}$. An object of this category is an $\mathbb{N}_{0}$-indexed collection of objects of $\mathpzc{E}$, which we usually write as $\bigoplus_{n\in\mathbb{N}_{0}}E_{n}$. A morphism $f:\bigoplus_{n\in\mathbb{N}_{0}}E_{n}\rightarrow \bigoplus_{n\in\mathbb{N}_{0}}F_{n}$ in $\mathpzc{Gr}_{\mathbb{N}_{0}}(\mathpzc{E})$ is just a collection of morphisms $f_{n}:E_{n}\rightarrow F_{n}$ for each $n\in\mathbb{N}_{0}$. If $\mathpzc{E}$ is (symmetric) monoidal then $\mathpzc{Gr}_{\mathbb{N}_{0}}(\mathpzc{E})$ can naturally be made into a (symmetric) monoidal category by
%$$\bigoplus_{n\in\mathbb{N}_{0}}E_{n}\otimes \bigoplus_{m\in\mathbb{N}_{0}}F_{n}\defeq\bigoplus_{n\in\mathbb{N}_{0}}(\coprod_{i+j=n}E_{i}\otimes F_{j})$$

We denote by $k[\Sigma]$ the monoid in $\mathpzc{Gr}_{\mathbb{N}_{0}}(\mathpzc{E})$ defined as follows. In degree $n$, it is given by the monoid
$k[\Sigma_{n}]$, the free monoid on the symmetric group in $n$ letters.

\begin{defn}
The \textbf{category of } $\Sigma$-\textbf{modules} in $\mathpzc{E}$ is the category of right $k[\Sigma]$-modules in $\mathpzc{Gr}_{\mathbb{N}_{0}}(\mathpzc{E})$. It is denoted $\mathpzc{Mod}_{\Sigma}(\mathpzc{E})$.
\end{defn}

\begin{defn}
If $M$ is either an $\mathbb{N}_{0}$-graded object or a $\Sigma$-module, we say that $M$ is \textbf{reduced} if $M(0)\cong0$.
\end{defn}

%\begin{defn}
%Let $M$ be a $\Sigma$-module. Its associated \textbf{Schur functor} is the endofunctor $\tilde{M}:\mathpzc{E}\rightarrow\mathpzc{E}$ defined by
%$$\tilde{M}(V)=\bigoplus_{n\ge0}M(n)\otimes_{\Sigma_{n}} V^{\otimes n}.$$
%The assignment $M\mapsto\tilde{M}$ is functorial in a natural way. We denote the functor $\mathpzc{Mod}_{\Sigma}\rightarrow\mathpzc{End}(\mathpzc{E},\mathpzc{E})$ by $\textrm{Sch}$.
%\end{defn}

\subsubsection{Formally Planar $\Sigma$-Modules}

Since we are not always working in characteristic $0$, invariants and coinvariants of finite group actions need not coincide. We will therefore often have to restrict what kinds of symmetric sequences we consider.

\begin{defn}\label{defn:stfp}
\begin{enumerate}
\item
Let $G$ be a finite group. A $G$-module $M$ is said to be \textbf{formally planar} if the norm morphism
$$M_{G}\rightarrow M^{G}$$
from coinvariants to invariants is an isomorphism.
\item
Let $\mathpzc{E}$ be a symmetric monoidal additive category. A $G$-module $M$ is said to be \textbf{stably formally planar} if for any $G$-module $N$, $M\otimes N$ is formally planar. 
\item
An object $M\in\mathpzc{Mod}_{\Sigma}(\mathpzc{E})$ is said to be \textbf{formally planar} (resp. \textbf{stably formally planar}) if each $M(n)$ is formally planar (resp. stably formally planar) as a $\Sigma_{n}$-module. The full subcategory of $\mathpzc{Mod}_{\Sigma}(\mathpzc{E})$ consisting of formally planar (resp. stably formally planar) modules is denoted $\mathpzc{Mod}^{fp}_{\Sigma}(\mathpzc{E})$ (resp. $\mathpzc{Mod}^{stfp}_{\Sigma}(\mathpzc{E})$).
\end{enumerate}
\end{defn}

\begin{rem}
    The terminology `formally planar' here is in reference to the quasi-planar cooperads of \cite{grignou2023homotopicaloperadiccalculuspositive}. In particular, quasi-planar cooperads are formally planar. 
\end{rem}

\begin{example}\label{example:formplanconds}
\begin{enumerate}
\item
Let $\mathpzc{E}$ be a symmetric monoidal additive category with unit $k$, and let $G$ be a finite group. Then $k[G]$ is stably formally planar. 
\item
Any retract of a stably formally planar $G$-module is stably formally planar.
\item
If filtered colimits commute with finite limits in $\mathpzc{E}$, then any filtered colimit of stably formally planar $G$-modules is stably formally planar. 
\end{enumerate}
\end{example}

\begin{lem}

    Let $R$ be a unital commutative ring, and let $G$ be a finite group. Let $M$ be a right $G$-module in ${}_{R}\mathpzc{Mod}(\mathpzc{Ab})$, i.e., a right $R[G]$-module. Then the following are equivalent.
    \begin{enumerate}
        \item 
        The module $M$ is flat in ${}_{R[G]}\mathpzc{Mod}$.
        \item 
        The module $M$ is a filtered colimit of finitely generated free $R[G]$-modules.
        \item 
        The module $M$ is stably formally planar and $M$ is flat as an object of ${}_{R}\mathpzc{Mod}$.
    \end{enumerate}
\end{lem}

\begin{proof}
The first two statements are equivalent by the Govorov-Lazard theorem (see \cite{MR252439} Proposition 3 for a good reference) applied to $R[G]$.

%Hence a colimit of finitely generated free $R[G]$-modules is also flat as an $R$-module. Hence $(1)\Rightarrow (2)$. 
That $(2)\Rightarrow (3)$ follows from Example \ref{example:formplanconds} and the fact that, as an $R$-module, $R[G]$ is free and hence flat. 

Now, for $(3)\Rightarrow(2)$, suppose that $M$ is stably formally planar and the $R$-module underlying $M$ is flat.
Let 
$$0\rightarrow K\rightarrow L\rightarrow N\rightarrow 0$$
be a short exact sequence of left $R[G]$-modules. Since $M$ is flat as an $R$-module, the sequence
$$0\rightarrow M\otimes K\rightarrow M\otimes L\rightarrow M\otimes N\rightarrow 0$$
is short exact. Now, taking coinvariants is a right exact functor, and taking invariants is a left exact functor. Since each term in the exact sequence is formally planar, coinvariants are isomorphic to invariants, and so we get an exact sequence
$$0\rightarrow M\otimes_{R[G]}K\rightarrow M\otimes_{R[G]}L\rightarrow M\otimes_{R[G]}N\rightarrow 0$$
as required. 
\end{proof}

\subsubsection{The Tensor Product of $\Sigma$-Modules}

We are going to define a monoidal structure on the category of $\Sigma$-Modules.

\begin{defn}
Let $M$ and $N$ be two $\Sigma$-modules. The \textbf{tensor product} of $M$ and $N$, denoted $M\otimes N$, is the $\Sigma$-module defined by
$$(M\otimes N)(n)=\bigoplus_{i+j=n}Ind_{\Sigma_{i}\times\Sigma_{j}}^{\Sigma_{n}}(M(i)\otimes N(j)).$$
\end{defn}

\begin{defn}
The $\Sigma$-module $I$ is defined by $I(i)=0$ for $i\neq 1$ and $I(1)=k$.
\end{defn}

%The following can be proven as in \cite{loday2012algebraic} Section 5.1.

%\begin{prop}
%$(\mathpzc{Mod}_{\Sigma},\otimes, I)$ is a symmetric monoidal category. Moreover, if we endow $[\mathpzc{E},\mathpzc{E}]$ with its object wise monoidal structure, $\textrm{Sch}$ is a strong monoidal functor.
%\end{prop}

\subsubsection{The Composite Products of $\Sigma$-Modules}
Operads are defined as monoids with respect to a different monoidal structure on $\mathpzc{Mod}_{\Sigma}$ that we define here. There are, in fact, numerous important bifunctors on $\mathpzc{Mod}_{\Sigma}$. Some of these form part of monoidal structures, and some are not.

\begin{defn}
Let $M$ and $N$ be two $\Sigma$-modules. 
\begin{enumerate}
\item
The \textbf{composite product} of $M$ and $N$, denoted $M\circ N$, is defined by 
$$M\circ N(n)=\bigoplus_{k\ge0}M(k)\otimes_{\Sigma_{k}}N^{\otimes k}(n).$$
\item
The \textbf{infinite composite product} of $M$ and $N$, denoted $M\hat{\circ} N$, is defined by 
$$M\hat{\circ} N(n)=\prod_{k\ge0}M(k)\otimes_{\Sigma_{k}}N^{\otimes k}(n).$$
\item
The \textbf{invariant composite product} of $M$ and $N$, denoted $M\overline{\circ} N$, is defined by 
$$M\overline{\circ} N(n)=\bigoplus_{k\ge0}(M(k)\otimes N^{\otimes k}(n))^{\Sigma_{k}}.$$
\item
If infinite products exist, the \textbf{infinite invariant composite product} of $M$ and $N$, denoted $M\hat{\overline{\circ}} N$, is defined by 
$$M\hat{\overline{\circ}} N(n)=\prod_{k\ge0}(M(k)\otimes N^{\otimes k}(n))^{\Sigma_{k}}.$$
\end{enumerate}
\end{defn}

We will mainly be interested in the composite product, though when discussing cooperads below, we will address the invariant and infinite invariant composite products.

 \begin{prop}[\cite{loday2012algebraic} Proposition 5.1.5] The tuple 
 $(\mathpzc{Mod}_{\Sigma},\circ,I)$ is a monoidal category.
 \end{prop}
Note that when finite limits commute with countable coproducts, $\overline{\circ}$ also defines a symmetric monoidal structure (see \cite{loday2012algebraic} 5.1.15).
 
Another useful notion, particularly in the context of differentials on operads, is the infinitesimal composite product of $\Sigma$-modules.
\begin{defn}
The functor $-\circ(-;-):\mathpzc{Mod}_{\Sigma}^{3}\rightarrow\mathpzc{Mod}_{\Sigma}$ is the subfunctor of $-\circ(-\oplus-):\mathpzc{Mod}_{\Sigma}^{3}\rightarrow\mathpzc{Mod}_{\Sigma}$ that is linear in the last variable.
\end{defn}

\begin{defn}[\cite{loday2012algebraic} Section 6.1]
The \textbf{infinitesimal composite product} functor $-\circ_{(1)}-:\mathpzc{Mod}_{\Sigma}^{2}\rightarrow\mathpzc{Mod}_{\Sigma}$ is the functor 
$$-\circ_{(1)}-\defeq-\circ(I;-).$$
\end{defn}

The following is from \cite{loday2012algebraic} Section 6.1.3.

\begin{defn}
Let $f:M_{1}\rightarrow M_{2}$, $g:N_{1}\rightarrow N_{2}$ be morphisms of $\Sigma$-modules. The \textbf{infinitesimal composite} of $f$ and $g$ is the morphism
$$f\circ' g:M_{1}\circ N_{1}\rightarrow M_{2}\circ (N_{1}; N_{2})$$
given by the formula
$$f\circ'g\defeq f\circ (Id_{N_{1}};g)=\sum_{i}f\otimes (Id_{N_{1}}\otimes\ldots\otimes Id_{N_{1}}\otimes g\otimes Id_{N_{1}}\otimes\ldots\otimes Id_{N_{1}}),$$
where $g$ is in the $i$th position.
\end{defn}
%
%Note that if $f$ is of bidegree $(m,n)$ and $g$ of bidegree $(k,l)$, then $\circ'$ is of bidegree $(m+k,n+l)$.

\subsubsection{Operads and Formally Planar Cooperads}

We can now define operads and cooperads as in \cite{loday2012algebraic} Chapter 5.

\begin{defn}\label{defn:coop}

\begin{enumerate}
\item
The \textbf{category of operads}, denoted $\mathpzc{Op}(\mathpzc{E})$, is the category of associative monoids $(\mathfrak{P},\gamma,\eta)$ in the monoidal category $(\mathpzc{Mod}_{\Sigma},I,
\circ)$. 
\item
The \textbf{category of conilpotent cooperads}, denoted $\mathpzc{coOp}^{conil}(\mathpzc{E})$, is the category of coassociative comonoids $(\mathfrak{C},\Delta,\epsilon)$ in the monoidal category $(\mathpzc{Mod}_{\Sigma},I,
\overline{\circ})$ (when this monoidal structure is well-defined).
\item
The \textbf{category of conilpotent coinvariant cooperads}, denoted $\mathpzc{coOp}^{coinv}(\mathpzc{E})$, is the category of coassociative comonoids $(\mathfrak{C},\Delta,\epsilon)$ in the monoidal category $(\mathpzc{Mod}_{\Sigma},I,
\circ)$. 
%\item
%The  \textbf{category of formally planar cooperads} denoted $\textit{coOp}(\mathpzc{E})$ is the category of coassociative comonoids  in the monoidal category $(\mathpzc{Mod}_{\Sigma},I,
%\circ)$.
\item
An operad $(\mathfrak{P},\gamma,\eta)$ is said to be \textbf{augmented} if there is a morphism $u:\mathfrak{P}\rightarrow I$ such that $u\circ \eta$ is the identity. One similarly defines \textbf{coaugmented cooperads}.
\end{enumerate}
\end{defn}

\begin{rem}
From now on we shall assume, for the sake of simplicity, that \textbf{all operads are augmented and all cooperads are coaugmented}.
\end{rem}

\begin{defn}
The category of \textbf{stably formally planar coinvariant (conilpotent) cooperads}, denoted $\mathpzc{coOp}^{stfp}(\mathpzc{E})$, is the full subcategory of $\mathpzc{coOp}^{coinv}(\mathpzc{E})$ consisting of cooperads whose underlying $\Sigma$-module is stably formally planar. 
\end{defn}

In particular, these are also naturally objects of $\mathpzc{coOp}(\mathpzc{E})$, provided that $\overline{\circ}$ is a monoidal structure. Given that we are almost always concerned with stably formally planar cooperads, we may therefore just work with coinvariant cooperads. 

\begin{rem}
\begin{enumerate}
\item
There is an important subtlety to keep in mind here. Even in the $\mathbb{Q}$-enriched setting, the respective definitions of an operad and a cooperad are not dual \textit{relative to }$\mathpzc{E}$. Namely, there is a fully faithful functor
$$\mathpzc{coOp}^{conil}(\mathpzc{E})\rightarrow\mathpzc{Op}(\mathpzc{E}^{op})^{op},$$
which is not an equivalence in general. This is because the construction of the composite product is not self-dual (in the opposite category one would get an infinite product rather than an infinite sum in the definition of the composite product).
\item
We really want to regard a cooperad $(\mathfrak{C},\Delta,\epsilon)$ as a coassociative comonoid in $(\mathpzc{Mod}_{\Sigma},I,
\hat{\overline{\circ}})$ such that the morphism
$$\mathfrak{C}\rightarrow\mathfrak{C}\hat{\overline{\circ}}\mathfrak{C}$$
factors through
$$\mathfrak{C}\rightarrow\mathfrak{C}\overline{\circ}\mathfrak{C}.$$
One must be careful, as $\hat{\overline{\circ}}$ will not in general be a monoidal structure.
\item
We may think of a stably formally planar cooperad $(\mathfrak{C},\Delta,\epsilon)$ as a coassociative comonoid in $(\mathpzc{Mod}_{\Sigma},I,
\hat{\circ})$ such that the morphism
$$\mathfrak{C}\rightarrow\mathfrak{C}\hat{\circ}\mathfrak{C}$$
factors through
$$\mathfrak{C}\rightarrow\mathfrak{C}\circ\mathfrak{C}.$$
Again one must be slightly careful, as $\hat{\circ}$ will not in general be a monoidal structure.
\end{enumerate}
\end{rem}

\begin{prop}[\cite{loday2012algebraic} Section 5.5]
The forgetful functor $|-|:\mathpzc{Op}(\mathpzc{E})\rightarrow\mathpzc{Mod}_{\Sigma}(\mathpzc{E})$ has a left adjoint $\mathfrak{T}(-)$, called the \textbf{free operad functor}.
\end{prop}

%For cooperads the story is again more subtle. The category of cooperads has a full subcategory $\mathpzc{coOp}^{con}(\mathpzc{E})$ called the category of \textbf{conilpotent cooperads}. We will not go into details here but a cooperad $\mathfrak{C}$ is conilpotent if it is coaugmented and a certain filtration on $\mathfrak{C}$, called the \textbf{coradical filtration} is exhaustive (see \cite{loday2012algebraic} Section 5.8 for details). 

\begin{prop}\label{prop:cofreecooperad}

The forgetful functor $|-|:\mathpzc{coOp}^{coinv}(\mathpzc{E})\rightarrow\mathpzc{Mod}_{\Sigma}(\mathpzc{E})$ has a right adjoint $\mathfrak{T}^{c}(-)$, called the \textbf{cofree (conilpotent) coinvariant cooperad functor}.

%If, in $\mathpzc{E}$, finite limits commute with countable sums, then  the functor $|-|:\mathpzc{coOp}^{stfp}(\mathpzc{E})\rightarrow\mathpzc{Mod}^{stfp}_{\Sigma}(\mathpzc{E})$) 
%admits a right adjoint, $\mathfrak{T}^{c}(-)$, called the \textbf{cofree formally planar (conilpotent) cooperad functor}.
\end{prop}

\begin{defn}
     
    A $\Sigma$-module $M$ is said to be \textbf{strongly stably formally planar} if $\mathfrak{T}^{c}(M)$ is stably formally planar.
\end{defn}

\begin{rem}
            
       From the explicit construction of $\mathfrak{T}^{c}(M)$, it can be shown that any free $\Sigma$-module is strongly stably formally planar, and that the class of stably formally planar $\Sigma$-modules is closed under retracts. Finally, if filtered colimits commute with finite limits, then the class of strongly stably formally planar $\Sigma$-modules is closed under filtered colimits. If $\mathpzc{Mod}{}_{\Sigma}^{sstfp}$ is the full subcategory of $\mathpzc{Mod}{}_{\Sigma}$ consisting of strongly stably formally planar objects, and $\mathpzc{coOp}^{sstfp}$ denotes the full subcategory of conilpotent coinvariant cooperads whose underlying $\Sigma$-module is strongly stably formally planar, then the forgetful functor $|-|:\mathpzc{coOp}^{sstfp}\rightarrow\mathpzc{Mod}_{\Sigma}^{sstfp}$ admits a right adjoint, also given by $\mathfrak{T}^{c}(-)$.
\end{rem}

From now on, all cooperads will be assumed to be conilpotent, so we suppress this part of the terminology. We finish this subsection with some more important terminology.

\begin{defn}[\cite{loday2012algebraic} Section 6.1.2]
Let $(\mathfrak{P},\gamma,\eta)$ be an operad. The \textbf{infinitesimal composition morphism} $\gamma_{(1)}:\mathfrak{P}\circ_{(1)}\mathfrak{P}\rightarrow\mathfrak{P}$ is given by the composition
\begin{displaymath}
\xymatrix{
\mathfrak{P}\circ_{(1)}\mathfrak{P}\ar@{=}[r] & \mathfrak{P}\circ(I;\mathfrak{P})\ar[r] & \mathfrak{P}\circ (I\oplus\mathfrak{P})\ar[rr]^{Id_{\mathfrak{P}}\circ(\eta+Id_{\mathfrak{P}})} & & \mathfrak{P}\circ\mathfrak{P}\ar[r]^{\gamma} & \mathfrak{P}.
}
\end{displaymath}
\end{defn}

\begin{defn}[\cite{loday2012algebraic} Section 6.1.4]
Let $(\mathfrak{C},\Delta,\epsilon)$ be a stably formally planar cooperad. The \textbf{infinitesimal decomposition morphism} $\Delta_{(1)}:\mathfrak{C}\rightarrow\mathfrak{C}\circ_{(1)}\mathfrak{C}$ is given by the composition
\begin{displaymath}
\xymatrix{
\mathfrak{C}\ar[r]^{\Delta} & \mathfrak{C}\circ\mathfrak{C}\ar[rr]^{Id_{\mathfrak{C}}\circ'Id_{\mathfrak{C}}} & & \mathfrak{C}\circ(\mathfrak{C};\mathfrak{C})\ar[rr]^{Id_{\mathfrak{C}}\circ(\epsilon;Id_{\mathfrak{C}})} & & \mathfrak{C}\circ(I;\mathfrak{C})\ar@{=}[r] & \mathfrak{C}\circ_{(1)}\mathfrak{C}.
}
\end{displaymath}
\end{defn}

There is another useful tensor product on the category of $\Sigma$-modules called the Hadamard tensor product.

\begin{defn}[\cite{loday2012algebraic} Section 5.1.12]
If $\mathfrak{O}$ and $\mathfrak{P}$ are $\Sigma$-modules, then we define their \textbf{Hadamard tensor product} by $(\mathfrak{O}\otimes_{H}\mathfrak{P})(n)\defeq\mathfrak{O}(n)\otimes\mathfrak{P}(n)$ equipped with the diagonal action of $\Sigma_{n}$. 
\end{defn}

If $\mathfrak{O}$ and $\mathfrak{P}$ are operads, then their Hadamard tensor product has a natural operad structure. The Hadamard tensor product of a stably formally planar cooperad with any cooperad is also a stably formally planar cooperad.

\subsection{(Co)modules and (Co)algebras Over (Co)operads}

Let us now discuss categories of algebras over operads, following \cite{loday2012algebraic} Section 5.2.

%There are obvious notions of morphisms of algebras and coalgebras, giving categories $\mathpzc{Alg}_{\mathfrak{P}}(\mathpzc{E})$ and $\mathpzc{coAlg}_{\mathfrak{P}}(\mathpzc{E})$.
%
%For  $\Sigma$-modules $\mathfrak{C}$, denote by $\mathfrak{C}\hat{\circ} C\defeq\prod_{n\ge 0}\mathfrak{C}(n)\otimes_{\Sigma_{n}}C^{\otimes n}$. If $\mathfrak{C}$ is a formally planar cooperad then a $\mathfrak{C}$-comodule is an object $C$ of $\mathpzc{E}$ equipped with a morphism $\Delta:C\rightarrow\mathfrak{C}\hat{\circ}C$ satisfying axioms similar to the duals of Definition \ref{defoperad}.
%
%\begin{defn}
%Let $\mathfrak{C}$ be a formally planar cooperad. The \textbf{category of} $\mathfrak{C}$-\textbf{comodules} is $\mathpzc{coMod}_{\mathfrak{C}}\defeq ({}_{\mathfrak{C}}\mathpzc{Mod}(\mathpzc{E}^{op}))^{op}$ 
%\end{defn}
%
%The full subcategory of $\mathfrak{C}$-comodules consisting of conilpotent $\mathfrak{C}$-comodules. Let $(C,\Delta_{C})$ be a $\mathfrak{C}$-formally planar cooperad. Again this definition is equivalent to the so-called coradical filtration filtration on $V$ being exhaustive. The full subcategory of $\mathpzc{coMod}_{\mathfrak{C}}^{conil}$ consisting of comodules concentrated in arity $1$ is denoted $\mathpzc{coAlg}_{\mathfrak{C}}^{conil}$

\begin{defn}\label{defn:conilpotent}
\begin{enumerate}
\item
Let $(\mathfrak{P},\gamma,\eta)$ be an operad in $\mathpzc{E}$. A left (resp. right) $\mathfrak{P}$-module is a left (resp. right) $\mathfrak{P}$-module internal to $(\mathpzc{Mod}_{\Sigma},\circ,I)$. The category of left (resp. right) $\mathfrak{P}$-modules is denoted ${}_{\mathfrak{P}}\mathpzc{Mod}(\mathpzc{E})$ (resp. $\mathpzc{Mod}_{\mathfrak{P}}(\mathpzc{E})$).
\item
A $\mathfrak{P}$-\textbf{algebra} is a left $\mathfrak{P}$-module that is concentrated in arity $1$. The full subcategory of $\mathfrak{P}$-modules consisting of $\mathfrak{P}$-algebras is denoted $\mathpzc{Alg}_{\mathfrak{P}}(\mathpzc{E})$. 
%\item
%Let $\mathfrak{C}$ be a cooperad. A left/ right $\mathfrak{C}$-\textbf{comodule} is a left/ right $\mathfrak{C}$-comodule internal to the category $(\mathpzc{Mod}_{\Sigma},\hat{\overline{\circ}}, I)$. A $\mathfrak{C}$-\textbf{coalgebra} is a  $\mathfrak{C}$-comodule concentrated in arity $1$. The category of left/ right $\mathfrak{C}$-comodules is denoted ${}_{\mathfrak{C}}\mathpzc{coMod}(\mathpzc{E})/\mathpzc{coMod}_{\mathfrak{C}}(\mathpzc{E})$.
\item
Let $\mathfrak{C}$ be a cooperad. A \textbf{conilpotent invariant} $\mathfrak{C}$-\textbf{comodule} is a $\mathfrak{C}$-comodule in $(\mathpzc{Mod}_{\Sigma},\overline{\circ}, I)$ (when this monoidal structure is well-defined). The category of conilpotent $\mathfrak{C}$-comodules is denoted $\mathpzc{coMod}_{\mathfrak{C}}^{inv-conil}(\mathpzc{E})$. 
\item 
Let $\mathfrak{C}$ be a cooperad. A \textbf{conilpotent coinvariant} $\mathfrak{C}$-\textbf{comodule} is a $\mathfrak{C}$-comodule in $(\mathpzc{Mod}_{\Sigma},\circ, I)$. The category of conilpotent coinvariant $\mathfrak{C}$-comodules is denoted $\mathpzc{coMod}_{\mathfrak{C}}^{coinv-conil}(\mathpzc{E})$.
\item A \textbf{conilpotent invariant} $\mathfrak{C}$-\textbf{coalgebra} (resp. a \textbf{conilpotent coinvariant} $\mathfrak{C}$-\textbf{coalgebra}) is a conilpotent invariant $\mathfrak{C}$-comodule (resp. a conilpotent coinvariant $\mathfrak{C}$-comodule) concentrated in arity $1$. The category of conilpotent coinvariant $\mathfrak{C}$-coalgebras is denoted $\mathpzc{coAlg}_{\mathfrak{C}}^{coinv-conil}(\mathpzc{E})$.
\end{enumerate}
\end{defn}

\begin{rem}
For us, $\mathfrak{C}$ will always be a stably formally planar cooperad. In this case, a conilpotent invariant $\mathfrak{C}$-comodule is the same as a conilpotent coinvariant $\mathfrak{C}$-comodule. We then just write $\mathpzc{coMod}_{\mathfrak{C}}^{conil}(\mathpzc{E})$ (resp. $\mathpzc{coAlg}^{conil}_{\mathfrak{C}}(\mathpzc{E})$).
\end{rem}

A $\mathfrak{P}$-\textbf{module} can be described as a pair $(N,\lambda_{N})$, where $N$ is a $\Sigma$-module and $\lambda_{N}:\mathfrak{P}\circ N\rightarrow N$ is a morphism of $\Sigma$-modules such that the following diagrams commute
\begin{displaymath}
\xymatrix{
(\mathfrak{P}\circ\mathfrak{P})\circ N\ar@{=}[r]\ar[d]^{\gamma\circ N} & \mathfrak{P}\circ(\mathfrak{P}\circ N)\ar[rr]^{\mathfrak{P}(\lambda_{N})} & & \mathfrak{P}\circ N\ar[d]^{\lambda_{N}}\\
\mathfrak{P}\circ N\ar[rrr]^{\lambda_{N}} & & & N
}
\end{displaymath}
\begin{displaymath}
\xymatrix{
I\circ N\ar@{=}[dr]\ar[r]^{\eta\circ N} & \mathfrak{P}\circ N\ar[d]^{\lambda_{N}}\\
& N.
}
\end{displaymath}

There is an obvious forgetful functor $|-|:{}_{\mathfrak{P}}\mathpzc{Mod}(\mathpzc{E})\rightarrow\mathpzc{Mod}_{\Sigma}(\mathpzc{E})$. For a $\Sigma$-module $N$ there is a $\mathfrak{P}$-module structure on $\mathfrak{P}\circ N$ given by
\begin{displaymath}
\xymatrix{
\mathfrak{P}\circ(\mathfrak{P}\circ N)\ar@{=}[r] & (\mathfrak{P}\circ\mathfrak{P})\circ N\ar[r]^{\;\;\gamma\circ Id_{N}} & \mathfrak{P}\circ N.
}
\end{displaymath}
This construction is functorial and gives the left adjoint to the forgetful functor.

There is a similar story for $\mathfrak{C}$-comodules. In this case, the forgetful functor is a left adjoint.
%\subsubsection{Conilpotent Coalgebras}

%$\Sigma$-module $V$ together with a morphism $\Delta_{V}: V\rightarrow\mathfrak{C}\circ V$ such that the composite morphism $V\rightarrow\mathfrak{C}\circ V\rightarrow\mathfrak{C}\hat{\circ}V$ endows $V$ with the structure of a $\mathfrak{C}$-comodule. 

%Denote by $\overline{\Delta}_{C}$ the morphism $\Delta_{C}-\eta_{C}:C\rightarrow\mathfrak{C}\circ C$. 

%\subsubsection{Modules Over Algebras and Ideals}
%
%\begin{defn}
%Let $(\mathfrak{P},\gamma,\eta)$ be an operad and $(A,\gamma_{A})$ an algebra over $\mathfrak{P}$. A \textbf{module over }$A$ is an object $M$ of $\mathpzc{E}$ together with morphisms $\gamma_{M}:\mathfrak{P}\circ(A;M)\rightarrow M$ and $\eta_{M}:M\rightarrow\mathfrak{P}\circ(A;M)$ such that the following diagrams commute
%\begin{displaymath}
%\xymatrix{
%(\mathfrak{P}\circ\mathfrak{P})\circ(A;M)\ar[r]^{\sim}\ar[d]^{\gamma\circ(Id;Id)} & \mathfrak{P}\circ(\mathfrak{P}(A);\mathfrak{P}\circ (A;M))\ar[r]^{Id\circ(\gamma_{A};\gamma_{M})}&\mathfrak{P}\circ(A;M)\ar[d]^{\gamma_{M}}\\
%\mathfrak{P}\circ(A;M)\ar[rr]^{\gamma_{M}} & & M
%}
%\end{displaymath}
%\begin{displaymath}
%\xymatrix{
%M\ar@{=}[dr]\ar[r]^{\eta_{M}} & \mathfrak{P}\circ(A;M)\ar[d]^{\gamma_{M}}\\
% & M
%}
%\end{displaymath}
%\end{defn}
%
%Note that $A$ is canonically a module over itself. An \textbf{ideal} of $A$ is a module $I$ over $A$ which is a submodule of $A$ regarded as a module over itself.
\subsection{Derivations}\label{secder}
Typically, the operads and algebras we consider will be equipped with differentials making them into chain complexes. However, we will want to `twist' these differentials, a process that will often break the equation $d^{2}=0$. In this section, we present a rather general discussion of derivations on operads and algebras that do not necessarily square to zero, i.e., we consider categories of pre-differential graded modules. Let $\mathpzc{E}$ be a closed symmetric monoidal additive category. Consider the category $\widetilde{Ch}(\mathpzc{E})$ with its symmetric monoidal structure.

\begin{defn}
A \textbf{unital commutative ring with derivation} is an object
$$(R,d_{R})\in\mathpzc{Alg}_{\mathfrak{Comm}}(\widetilde{Ch}(\mathpzc{E})).$$
An $(R,d_{R})$-\textbf{module} is an object of 
$${}_{(R,d_{R})}\mathpzc{Mod}(\widetilde{Ch}(\mathpzc{E})).$$
\end{defn}

Unravelling the definitions, an $(R,d_{R})$-module is just an $R$-module $V$ equipped with a morphism $d_{V}:V\rightarrow V$ of degree $-1$ such that the following diagram commutes (recall that we are using the Koszul sign rule):
\begin{displaymath}
\xymatrix{
R\otimes V\ar[rrr]^{d_{R}\otimes Id_{V}+Id_{R}\otimes d_{V}}\ar[d]^{\mu_{V}} & & & R\otimes V\ar[d]^{\mu_{V}}\\
V\ar[rrr]^{d_{V}} & & & V. 
}
\end{displaymath}
 The category of $(R,d_{R})$-modules inherits a closed symmetric monoidal structure.

\begin{defn}
An \textbf{operad with derivation over }$(R,d_{R})$ is an object of $\mathpzc{Op}({}_{(R,d_{R})}\mathpzc{Mod}(\widetilde{Ch}(\mathpzc{E})))$.
\end{defn}

Cooperads with coderivations are defined dually. 
%FIX: derivations of operads, modules over operads.
%
%\begin{defn}
%Let $\mathfrak{P}$ be an operad. 
%\end{defn}
%For most of this section we follow \cite{arlettaz2000degustation}, Andr\'{e} Quillen Cohomology, Section 2.

\subsubsection{Relative Derivations}\label{subsubsec:relative}

Let us now introduce a relative notion of derivation.

\begin{defn}\label{defn:relder}
Let $(R,d_{R})$ be a ring with derivation, and let $(\mathfrak{P},d_{\mathfrak{P}})$ be an operad with derivation over $(R,d_{R})$.
 Let $A$ and $B$ be left $\mathfrak{P}$-modules, and let $\epsilon:A\rightarrow B$ be a morphism of $\mathfrak{P}$-modules. An $\epsilon$-\textbf{derivation from }$A$\textbf{ to }$B$ is a morphism $D:A\rightarrow B$ such that the following diagram commutes:
\begin{displaymath}
\xymatrix{
\mathfrak{P}\circ(A;A)\ar[rrr]^{d_{\mathfrak{P}}\circ(\epsilon;\epsilon)+Id_{\mathfrak{P}}\circ(\epsilon;D)}\ar[d]^{\gamma_{A}} &&& \mathfrak{P}(B;B)\ar[d]^{\gamma_{B}}\\
A\ar[rrr]^{D} & && B.\\
}
\end{displaymath}
\end{defn}

\begin{defn}\label{defn:extendderivation}
Let $V,A\in{}_{R}\mathpzc{Mod}(\mathpzc{Gr}_{\mathbb{Z}}(\mathpzc{Mod}_{\Sigma}(\mathpzc{E})))$, and let $\epsilon:V\rightarrow A$ be a morphism of $R$-modules. By an $\epsilon$-\textbf{derivation from }$V$\textbf{ to }$A$ over $(R,d_{R})$, we mean an $\epsilon$-derivation in the sense of Definition \ref{defn:relder}, where we regard $(R,d_{R})$ as an operad with derivation concentrated in arity $1$.
\end{defn}

Again, relative coderivations of left and right $\mathfrak{C}$-comodules are defined dually.

Now let $(R,d_{R})\in\mathpzc{Alg}_{\mathfrak{Comm}}(\widetilde{Ch}(\mathpzc{E}))$, let $(\mathfrak{P},d_{\mathfrak{P}})$ be an operad in ${}_{(R,d_{R})}\mathpzc{Mod}(\widetilde{Ch}(\mathpzc{E}))$, let $A$ be a graded left $\mathfrak{P}$-module in ${}_{R}\mathpzc{Mod}(\mathpzc{Gr}_{\mathbb{Z}}(\mathpzc{E}))$, and let $V\in{}_{R}\mathpzc{Mod}(\mathpzc{Gr}_{\mathbb{Z}}(\mathpzc{Mod}_{\Sigma}(\mathpzc{E})))$. Let $\epsilon:V\rightarrow A$ be a morphism of $R$-modules, $D:V\rightarrow A$ an $\epsilon$-derivation from $V$ to $A$ over $(R,d_{R})$, and $\phi:V\rightarrow A$ a degree $-1$ morphism of $R$-modules. Denote by
$$\tilde{d}_{(\epsilon;\phi)}:\mathfrak{P}\circ V\rightarrow A$$
the degree $-1$ morphism given by
$$\gamma_{A}\circ(Id_{\mathfrak{P}}\circ(\epsilon;\phi)).$$
Define
$$d_{D;\epsilon;\phi}:\mathfrak{P}\circ V\rightarrow A$$
by
$$d_{D;\epsilon;\phi}=\gamma_{A}(d_{\mathfrak{P}}\circ\epsilon+Id_{\mathfrak{P}}\circ(\epsilon; D))+\tilde{d}_{(\epsilon;\phi)}.$$
Note that $d_{D;\epsilon;\phi}$ is an $\overline{\epsilon}$-derivation over $(R,d_{R})$, where $\overline{\epsilon}:\mathfrak{P}\circ V\rightarrow A$ is the unique morphism of $\mathfrak{P}$-modules determined by $\epsilon:V\rightarrow A$. As in the absolute case proved in \cite{loday2012algebraic} Proposition 6.3.6, we have the following.

\begin{prop}\label{prop:twistder}
$d_{D;\epsilon;\phi}$ is the unique $\overline{\epsilon}$-derivation of left $(\mathfrak{P},d_{\mathfrak{P}})$-modules whose restriction to $V$ is $D+\phi$.
\end{prop}

Consider the special case where $A=\mathfrak{P}\circ C$ for some pre-differential graded $\Sigma$-module $(C,d_{C})$, where $\epsilon:C\rightarrow\mathfrak{P}\circ C$ is the morphism given by 
\begin{displaymath}
    \xymatrix{
    C\ar[r]^{i_{C}} &\mathfrak{P}(C).
    }
\end{displaymath}
Let $\phi:C\rightarrow\mathfrak{P}\circ C$ be a morphism of $R$-modules. Then $d_{d_{C};i_{C};\phi}:\mathfrak{P}\circ C\rightarrow\mathfrak{P}\circ C$ is given by the formula
$$d_{d_{C};i_{C};\phi}=d_{\mathfrak{P}}\circ Id_{C}+ Id_{\mathfrak{P}}\circ'd_{C}+(\gamma_{(1)}\circ Id_{C})(Id_{\mathfrak{P}}\circ'\phi).$$

\begin{rem}
There are of course identical statements for right $\mathfrak{P}$-modules and for $\mathfrak{C}$-comodules.
\end{rem}
There is an important special case of this result.
% Let $(\mathfrak{P},d_{\mathfrak{P}})$ be an operad with derivation over a ring $(R,d_{R})$ with derivation in $\mathpzc{E}$. Let $A$ be a $\mathfrak{P}$-module with derivation over $(R,d_{R})$ and let $V$ be any object of $\mathpzc{E}$. Suppose $\alpha:V\rightarrow A$ and $i:A\rightarrow V$ are morphisms in $\mathpzc{E}$. There is a unique induced morphism $R\otimes i:R\otimes V\rightarrow A$ of $R$-modules, and a unique $R\otimes i$-derivation $d_{\alpha}:R\otimes V\rightarrow A$. Applying the proposition again, there is a unique induced $R\otimes i$-derivation $\mathfrak{P}(R\otimes V)\rightarrow A$ of $\mathfrak{P}$-modules over $(R,d_{R})$, which we also denote by $d_{\alpha}$. 
%\begin{prop}
%Let $V$ be an object of $\mathpzc{E}$ and $(A,d_{A})$ an algebra with derivation. Let $i:V\rightarrow A$ and $\phi:V\rightarrow A$ be morphisms in $\mathpzc{E}$. Let $\tilde{i}:
%\end{prop}
%\begin{displaymath}
%\xymatrix{
%\mathfrak{P}(A;A)\ar[d]\ar[r]^{d_{\mathfrak{P}}}\circ(Id_{A};Id_{A})+Id_{\mathfrak{P}}\\
%\mathfrak{P}(A)\ar[r]^{}
%}
%\end{displaymath}
%\end{defn}

\begin{cor}\label{coproder}
Let $A$, $B$, and $C$ be left $\mathfrak{P}$-modules, and let $\epsilon_{A}:A\rightarrow C$ and $\epsilon_{B}:B\rightarrow C$ be morphisms of $\mathfrak{P}$-modules. Let $D_{A},\phi_{A}:A\rightarrow C$ and $D_{B},\phi_{B}:B\rightarrow C$ be degree $-1$ morphisms from $A$ to $C$ and from $B$ to $C$, respectively, where $\phi_{A}$ and $\phi_{B}$ are morphisms of $\mathfrak{P}$-modules, and $D_{A}$ and $D_{B}$ being an $\epsilon_{A}$-derivation and an $\epsilon_{B}$-derivation of $\mathfrak{P}$-modules, respectively. Then there is a unique $\epsilon_{A}\coprod\epsilon_{B}$-derivation of left $(\mathfrak{P},d_{\mathfrak{P}})$-modules $D:A\coprod B\rightarrow C$ whose restriction to $A$ is $D_{A}+\phi_{A}$ and whose restriction to $B$ is $D_{B}+\phi_{B}$.
\end{cor}
%In particular, it is both an $A$-module and a $B$-module. Suppose $d:A\rightarrow M$ and $\delta: B\rightarrow M$ are derivations. There is a unique derivation $d+\delta: A\coprod B\rightarrow M$ whose restriction to $A$ is $d$ and whose restriction to $B$ is $\delta$.

\begin{proof}
$A\coprod B$ can be constructed as a quotient of $\mathfrak{P}(A\oplus B)$. There is a unique $\overline{(\epsilon_{A},\epsilon_{B})}$-derivation of left $(\mathfrak{P},d_{\mathfrak{P}})$-modules $\overline{D}:\mathfrak{P}(A\oplus B)\rightarrow C$ whose restriction to $A$ is $D_{A}+\phi_{A}$ and whose restriction to $B$ is $D_{B}+\phi_{B}$. One checks that this descends to a morphism $A\coprod B\rightarrow C$.
\end{proof}

\subsubsection{The Twisted Composite Product}

Let $(\mathfrak{C},d_{\mathfrak{C}})$ be a cooperad in ${}_{(R,d_{R})}\mathpzc{Mod}(\widetilde{Ch}(\mathpzc{E}))$, let $(\mathfrak{P},d_{\mathfrak{P}})$ be an operad in ${}_{(R,d_{R})}\mathpzc{Mod}(\widetilde{Ch}(\mathpzc{E}))$, and let
$$\alpha:\mathfrak{C}\rightarrow\mathfrak{P}$$
be a degree $-1$ morphism of $R$-modules.

%
%Denote by $\mathbb{I}_{\mathfrak{C},\mathfrak{P}}$ the composition
%$$(\mathfrak{C},d_{\mathfrak{C}})\rightarrow (R,d_{R})\rightarrow (\mathfrak{P},d_{\mathfrak{P}})$$

%and let 
%$$\alpha:\mathfrak{C}\rightarrow\mathfrak{P}$$
%a degree $-1$ $\mathbb{I}_{\mathfrak{C},\mathfrak{P}}$-derivation of $(R,d_{R})$-modules. Then the composition

Write $\tilde{d}^{r}_{\alpha}\defeq\tilde{d}_{(Id;\alpha^{r})}$, where $\alpha^{r}$ is the composition
\begin{displaymath}
\xymatrix{
\alpha^{r}\defeq \mathfrak{C}\ar[r]^{\Delta_{(1)}} & \mathfrak{C}\circ_{(1)}\mathfrak{C}\ar[r]^{Id_{\mathfrak{C}}\circ_{(1)}\alpha} & \mathfrak{C}\circ_{(1)}\mathfrak{P}\ar[r] & \mathfrak{C}\circ\mathfrak{P},
}
\end{displaymath}
 
 and $\tilde{d}^{l}_{\alpha}\defeq\tilde{d}_{(Id;\alpha^{l})}$, where $\alpha^{l}$ is the composition

\begin{displaymath}
\xymatrix{
\alpha^{l}\defeq \mathfrak{C}\ar[r]^{\Delta} & \mathfrak{C}\circ\mathfrak{C}\ar[r]^{\alpha\circ Id_{\mathfrak{C}}} & \mathfrak{P}\circ\mathfrak{C}.
}
\end{displaymath}
Define
$$d^{l}_{\alpha}\defeq d_{\mathfrak{P}\circ\mathfrak{C}}-\tilde{d}^{l}_{\alpha},$$

and
$$d^{r}_{\alpha}\defeq d_{\mathfrak{C}\circ\mathfrak{P}}+\tilde{d}^{r}_{\alpha}.$$
Note that these are the unique derivations of free right (resp. left) $\mathfrak{P}$-modules on $\mathfrak{C}$ associated to the morphisms $\alpha^{r}$ and $-\alpha^{l}$, respectively. Denote by $\mathfrak{P}\circ_{\alpha}\mathfrak{C}$ the composite product $\mathfrak{P}\circ\mathfrak{C}$ equipped with the differential $d_{\alpha}^{l}$, and by $\mathfrak{C}\circ_{\alpha}\mathfrak{P}$ the composite product $\mathfrak{C}\circ\mathfrak{P}$ equipped with the differential $d_{\alpha}^{r}$.

%Let $\alpha:\mathfrak{C}\rightarrow\mathfrak{P}$ be a degree $-1$ morphism over $(R,d_{R})$. Denote by $d^{r}_{\alpha}$ the unique derivation on $\mathfrak{C}\circ\mathfrak{P}$ over $(R,d_{R})$ which extends the composition

Finally, we denote by $\mathfrak{P}\circ_{\alpha}\mathfrak{C}\circ_{\alpha}\mathfrak{P}$ the complex whose underlying graded $\Sigma$-module is given by $\mathfrak{P}\circ\mathfrak{C}\circ\mathfrak{P}$, with differential given by $d_{\mathfrak{P}\circ\mathfrak{C}\circ\mathfrak{P}}+Id_{\mathfrak{P}}\circ' \tilde{d}_{\alpha}^{r}-\tilde{d}_{\alpha}^{l}\circ Id_{\mathfrak{P}}$. This is called the \textbf{two-sided twisted composite product}.

%As in \cite{hirsh2012curved} Proposition 5.1.3. a direct computation shows that this is a complex. 

%\begin{lem}
%On $\mathfrak{P}\circ\mathfrak{C}$ the derivation $d_{\alpha}$ satisfies
%$$d_{\alpha}^{2}=d^{l}_{\partial(\alpha)+\alpha\star\alpha}$$
%On $\mathfrak{C}\circ\mathfrak{P}$ the derivation $d_{\alpha}$ satisfies
%$$d_{\alpha}^{2}=d^{r}_{\partial(\alpha)+\alpha\star\alpha}$$
%\end{lem}
%
%\begin{cor}
%If $\alpha$ is a twisting morphism then 
%$$\mathfrak{P}\circ_{\alpha}\mathfrak{C}\defeq(\mathfrak{P}\circ\mathfrak{C},d_{\alpha})$$
%and
%$$\mathfrak{C}\circ_{\alpha}\mathfrak{P}\defeq(\mathfrak{C}\circ\mathfrak{P},d_{\alpha})$$
%are chain complexes.
%\end{cor}

\subsection{(Co)operads in Chain Complexes}

In this section, we let $(\mathpzc{E},\otimes,k)$ be a symmetric monoidal additive category with kernels, cokernels, and countable coproducts. We fix an object $R\in\mathpzc{Alg}_{\mathfrak{Comm}}(Ch(\mathpzc{E}))$ and consider the monoidal additive category $\mathpzc{M}\defeq {}_{R}\mathpzc{Mod}(Ch(\mathpzc{E}))$.

\subsubsection{Twisting Morphisms}
Following \cite{loday2012algebraic} Section 6.4, let us now introduce twisting morphisms. Let $(\mathfrak{P},d_{\mathfrak{P}})\in\mathpzc{Op}(\mathpzc{M})$ and $(\mathfrak{C},d_{\mathfrak{C}})\in\mathpzc{coOp}^{stfp}(\mathpzc{M})$.

\begin{defn}
A degree $-1$ morphism in ${}_{R}\mathpzc{Mod}(\mathpzc{Mod}_{\Sigma})$, $\alpha:\mathfrak{C}\rightarrow\mathfrak{P}$,  is said to be a \textbf{twisting morphism} if 
$$\mathfrak{P}\circ_{\alpha}\mathfrak{C},$$
$$\mathfrak{C}\circ_{\alpha}\mathfrak{P},$$
and
$$\mathfrak{P}\circ_{\alpha}\mathfrak{C}\circ_{\alpha}\mathfrak{P}$$
are all complexes.
\end{defn}

Typically, the twisting morphisms we consider in additive categories are induced from ones in $Ch(\mathpzc{Ab})$ (or $Ch({}_{\Q}\mathpzc{Mod})$). To this end we record the following observation as a proposition.

\begin{prop}\label{preserveMC}
\begin{enumerate}
\item
Let $\mathpzc{D}$ and $\mathpzc{E}$ be monoidal additive categories and let $F:\mathpzc{D}\rightarrow\mathpzc{E}$ be a strong symmetric monoidal functor. Let $\alpha:\mathfrak{C}\rightarrow\mathfrak{P}$ be a twisting morphism in $Ch(\mathpzc{D})$. Then $F(\alpha):F(\mathfrak{C})\rightarrow F(\mathfrak{P})$ is a twisting morphism in $Ch(\mathpzc{E})$.
\item
Let $\mathpzc{E}$ be a monoidal additive category, $\alpha:\mathfrak{C}\rightarrow\mathfrak{P}$ a twisting morphism in $Ch(\mathpzc{E})$, and $(R,d_{R})\in\mathpzc{Alg}_{\mathfrak{Comm}}(Ch(\mathpzc{E}))$. Then 
$$Id_{R}\otimes\alpha:R\otimes\mathfrak{C}\rightarrow R\otimes\mathfrak{P}$$
is a twisting morphism in ${}_{(R,d_{R})}\mathpzc{Mod}(Ch(\mathpzc{E}))$.
\end{enumerate}
\end{prop}

%\begin{proof}
%This follows from the fact that $F$ induces a homomorphism of convolution Lie algebras, and homomorphisms of Lie algebras send Maurer-Cartan elements to Maurer-Cartan elements. 
%\end{proof}
%FIX: Unique derivation on operad. 
Twisting morphisms allow us to construct twisted differentials on the composite products $\mathfrak{C}\circ\mathfrak{P}$ and $\mathfrak{P}\circ\mathfrak{C}$, which we will need to define the bar-cobar adjunction.

\begin{example}\label{cofreetwist}
\begin{enumerate}
\item
Let $V$ be an object of $\mathpzc{Mod}_{\Sigma}(\mathpzc{M})$. The composite morphism
$$\mathfrak{T}^{c}(V[-1])\rightarrow V[-1]\rightarrow V\rightarrow\mathfrak{T}(V)$$
is a twisting morphism. Indeed, the proof of Lemma 7.4.1 in \cite{loday2012algebraic} for more general quadratic operads goes through here. 
\item 
Now, suppose we are working over $\mathbb{Q}$. The proof of Lemma 7.4.1 in \cite{loday2012algebraic} also works to show that for $V\in\mathpzc{M}$, the morphism
$$\overline{S}^{c}(V[-1])\rightarrow V[-1]\rightarrow V\rightarrow \overline{S}(V)$$
from the reduced (i.e. non-counital) cosymmetric coalgebra to the reduced (i.e. non-unital) symmetric algebra is a twisting morphism.
\end{enumerate}
\end{example}

%It can be shown as in \cite{loday2012algebraic} Section 6.4.1 that all three of the pre-differential graded modules $\mathfrak{P}\circ_{\alpha}\mathfrak{C}$, $\mathfrak{C}\circ_{\alpha}\mathfrak{P}$, and
%$\mathfrak{P}\circ_{\alpha}\mathfrak{C}\circ_{\alpha}\mathfrak{P}$ are complexes precisely if any one of them is. 

Twisting morphisms have descriptions in terms of Maurer-Cartan elements in the convolution Lie algebra. Let $(\mathfrak{C},\Delta,\epsilon)$ be a stably formally planar cooperad and $(\mathfrak{P},\gamma,\eta)$ an operad in $\mathpzc{M}$, and consider the object in $\mathpzc{Gr}_{\mathbb{N}_{0}}(Ch(\mathpzc{Ab}))$ given by
$$\mathbf{Hom}(\mathfrak{C},\mathfrak{P})(n)\defeq\mathbf{Hom}(\mathfrak{C}(n),\mathfrak{P}(n)).$$

It is a (right) $\Sigma$-module in $Ch(\mathpzc{Ab})$. For $f\in\mathbf{Hom}_{\mathpzc{M}}(\mathfrak{C}(n),\mathfrak{P}(n))_{m}$, $\sigma\in\Sigma_{n}$ acts on the right by
$$f\mapsto \sigma\circ f\circ\sigma^{-1}.$$

Using the methods of \cite{loday2012algebraic} Section 6.4.1, this $\Sigma$-module in $Ch(\mathpzc{Ab})$ can be made into an operad, called the \textbf{convolution operad} of $\mathfrak{C}$ and $\mathfrak{P}$. Recall that any $dg$-operad $\mathfrak{P}$ has an associated dg pre-Lie algebra whose underlying differentially graded abelian group is $\prod_{n}\mathfrak{P}(n)$. Denote by

$$\mathbf{Hom}_{\Sigma}(\mathfrak{C},\mathfrak{P})\defeq\prod_{n\ge0}\mathbf{Hom}_{\Sigma_{n}}(\mathfrak{C}(n),\mathfrak{P}(n))$$
the subobject of $\prod_{n\ge0}\mathbf{Hom}(\mathfrak{C}(n),\mathfrak{P}(n))$ consisting of $\Sigma$-equivariant morphisms.

\begin{prop}
The object $\mathbf{Hom}_{\Sigma}(\mathfrak{C},\mathfrak{P})$ is a sub dg pre-Lie algebra of the dg pre-Lie algebra associated to the convolution operad.
\end{prop}

\begin{defn}
The object $\mathbf{Hom}_{\Sigma}(\mathfrak{C},\mathfrak{P})$ is called the \textbf{convolution dg pre-Lie algebra}. The binary operation is denoted $\star$.
\end{defn}

As shown in \cite{loday2012algebraic} Chapter 11, a twisting morphism is precisely a Maurer-Cartan element in the convolution dg pre-Lie algebra $\mathbf{Hom}_{\Sigma}(\mathfrak{C},\mathfrak{P})$.

\begin{prop}
  For $\alpha,\beta:\mathfrak{C}\rightarrow\mathfrak{P}$ degree $-1$ morphisms, we have $\tilde{d}_{\alpha}^{r}\circ\tilde{d}_{\beta}^{r}=\tilde{d}_{\alpha\star\beta}^{r}$.
\end{prop}

\begin{proof}
This can be proven exactly as in \cite{loday2012algebraic} Proposition 1.6.2, using elementary string diagram arguments. 
\end{proof}

\begin{lem}[\cite{loday2012algebraic} Lemma 6.4.7]
Denote by $\partial$ the differential on $\mathbf{Hom}_{\Sigma}(\mathfrak{C},\mathfrak{P})$.
We have 
$$(d^{l}_{\alpha})^{2}=\tilde{d}^{l}_{\partial(\alpha)+\alpha\star\alpha},$$
and
  $$(d^{r}_{\alpha})^{2}=\tilde{d}^{r}_{\partial(\alpha)+\alpha\star\alpha}.$$  
\end{lem}

\begin{proof}
The proof proceeds exactly as in \cite{loday2012algebraic} Lemma 6.4.7 (and \cite{loday2012algebraic} Lemma 2.1.4).
\end{proof}

\subsubsection{Bar and Cobar Constructions for Algebras}\label{subsubsec:barcobar}
Here we recall from \cite{loday2012algebraic} Chapter 11 how twisting morphisms induce adjunctions of algebras and coalgebras. Let $\alpha:\mathfrak{C}\rightarrow\mathfrak{P}$ be a twisting morphism, where $\mathfrak{C}$ is stably formally planar. We construct an adjoint pair of functors

$$\adj{\Omega_{\alpha}}{\mathpzc{coAlg}^{conil}_{\mathfrak{C}}(\mathpzc{M})}{\mathpzc{Alg}_{\mathfrak{P}}(\mathpzc{M})}{B_{\alpha}}.$$

Let $C$ be a $\mathfrak{C}$-coalgebra. The underlying graded algebra of $\Omega_{\alpha}C$ is the free algebra $\mathfrak{P}(C)$ on $C$. The differential, however, is augmented using the twisting morphism. Namely, it is produced using the procedure of Proposition \ref{prop:twistder}, where $-\phi$ is given by the following morphism.

%Namely, it is the sum $d_{1}=d_{\mathfrak{P}}\circ Id_{C}+Id_{\mathfrak{P}}\circ'd_{C}$ and $-d_{2}$, where $d_{2}$ is the unique derivation extending the morphism
\begin{displaymath}
\xymatrix{
C\ar[r]^{\Delta} & \mathfrak{C}\circ C\ar[r]^{\alpha\circ Id_{C}} & \mathfrak{P}\circ C.
}
\end{displaymath}

The resulting differential can be written as $d_{1}-d_{2}$, where $d_{1}$ is the differential $d_{1}=d_{\mathfrak{P}}\circ Id_{C}+Id_{\mathfrak{P}}\circ'd_{C}$.

\begin{prop}[\cite{loday2012algebraic} Lemma 11.2.4]
The morphism $d_{1}-d_{2}$ is a square-zero derivation on $\mathfrak{P}(C)$. 
%There is a natural isomorphism of algebras with derivations
%$$(\mathfrak{P}(C),-d_{1}-d_{2})\cong((\mathfrak{P}\circ_{\alpha}\mathfrak{C})\circ^{\mathfrak{C}}C,d_{\alpha})$$
%In  particular $d_{1}+d_{2}$ is a square-zero derivation.
\end{prop}

\begin{defn}
$(\mathfrak{P}(C),d_{1}-d_{2})$ is called the \textbf{cobar construction of }$C$\textbf{ with respect to }$\alpha$.
\end{defn}

Now, let $A$ be a $\mathfrak{P}$-algebra in $\mathpzc{M}$. The underlying graded coalgebra of $B_{\alpha}A$ is the cofree coalgebra $\mathfrak{C}(A)$. Again, the differential is augmented using the twisting morphism. Namely, it is produced using the coalgebraic version of the procedure of Proposition \ref{prop:twistder}, where $\phi$ is given by the following morphism.

\begin{displaymath}
\xymatrix{
\mathfrak{C}\circ A\ar[r]^{\alpha\circ Id_{A}} & \mathfrak{P}\circ A\ar[r]^{\gamma_{A}} & A.
}
\end{displaymath}

The resulting differential may be written as $d_{1}+d_{2}$, where $d_{1}=d_{\mathfrak{C}}\circ Id_{A}+Id_{\mathfrak{C}}\circ' d_{A}$.

\begin{prop}[\cite{loday2012algebraic} Lemma 11.2.4]
The morphism $d_{1}+d_{2}$ is a square-zero coderivation on $\mathfrak{C}(A)$.
%There is a natural isomorphism
%$$(\mathfrak{C}(A),)\cong((\mathfrak{C}\circ_{\alpha}\mathfrak{P})\circ_{\mathfrak{P}} A,d_{\alpha})$$
%In particular, $d_{1}+d_{2}$ is a square-zero coderivation.
\end{prop}

\begin{defn}
We call $B_{\alpha}A\defeq(\mathfrak{C}(A),d_{1}+d_{2})$ the \textbf{bar construction of } $A$ \textbf{with respect to }$\alpha$.
\end{defn}

\begin{prop}[\cite{loday2012algebraic} Proposition 11.3.1]
There is an adjunction
$$\adj{\Omega_{\alpha}}{\mathpzc{coAlg}^{conil}_{\mathfrak{C}}(\mathpzc{M})}{\mathpzc{Alg}_{\mathfrak{P}}(\mathpzc{M})}{B_{\alpha}}.$$
\end{prop}

\begin{prop}\label{prop:(co)unitadmissible}
The unit $\nu_{\alpha}$ of the bar-cobar adjunction is a split monomorphism, and the counit $\epsilon_{\alpha}$ is a split epimorphism. 
\end{prop}

\begin{proof}
    We prove the $\epsilon_{\alpha}$ claim, the $\nu_{\alpha}$ claim being similar. Denote the counit of the cooperad $\mathfrak{C}$ by $\epsilon:\mathfrak{C}\rightarrow I$. Now, on underlying objects, $\epsilon_{\alpha}$ is the composite 
    \begin{displaymath}
        \xymatrix{
        \mathfrak{P}\circ\mathfrak{C}(A)\ar[rr]^{(Id_{\mathfrak{P}}\circ\epsilon)(Id_{A})}& &\mathfrak{P}\circ A\ar[r]^{\gamma_{A}} & A.
        }
    \end{displaymath}
    Since $\mathfrak{C}$ is coaugmented, $\epsilon$ is a split epimorphism. The morphism $\gamma_{A}$ is evidently a split epimorphism by unitality.
\end{proof}

In the context of the bar-cobar adjunction, particularly for Koszul duals of quadratic operads, it is convenient to introduce the \textit{shifting (co)operad}. If $V$ is an object of $\mathpzc{M}$ then one defines the $\Sigma$-module $\underline{\mathrm{End}}_{V}$ by
$$\underline{\mathrm{End}}_{V}(n)\defeq\underline{\mathrm{Hom}}(V^{\otimes n},V),$$
where the right action of $\Sigma_{n}$ is induced by the left action of $\Sigma_{n}$ on $V^{\otimes n}$.
This can be endowed with the structure of an operad, denoted by $\mathrm{End}_{V}$.

\begin{defn}
The \textbf{shifting operad}, denoted $\mathfrak{S}$, is the operad $\mathrm{End}_{R[1]}$. 
\end{defn} 

We also write $\mathfrak{S}^{-1}$ for the operad $\mathrm{End}_{R[-1]}$.

The object $\mathrm{End}^{c}_{R[-1]}\defeq(\mathrm{End}_{R[1]})^{\vee}$ may also be given the structure of a cooperad. This is called the \textbf{shifting cooperad} and we denote it by $\mathfrak{S}^{c}$.

\begin{rem}
    Although $\mathfrak{S}^{c}$ may not be formally planar, if $\mathfrak{C}$ is stably formally planar then so is the Hadamard tensor product $\mathfrak{S}^{c}\otimes_{H}\mathfrak{C}$.
\end{rem}

By \cite{loday2012algebraic} Exercise 5.11.4 we have the following result.

\begin{prop}
For any object $V\in {}_{R}\mathpzc{Mod}$ and any $\Sigma$-module $\mathfrak{P}$ there is an isomorphism, natural in $V$,
$$\mathfrak{P}(V)[-1]\cong (\mathfrak{S}\otimes_{H}\mathfrak{P})(V[-1]).$$
In particular, if $\mathfrak{P}$ is an operad, the shift functor induces an equivalence of categories
$$[-1]:\mathpzc{Alg}_{\mathfrak{P}}(\mathpzc{M})\rightarrow\mathpzc{Alg}_{\mathfrak{S}\otimes_{H}\mathfrak{P}}(\mathpzc{M}).$$
Similarly, if $\mathfrak{C}$ is a stably formally planar cooperad, the shift functor induces an equivalence of categories
$$[-1]:\mathpzc{coAlg}^{conil}_{\mathfrak{S}^{c}\otimes_{H}\mathfrak{C}}(\mathpzc{M})\rightarrow\mathpzc{coAlg}^{conil}_{\mathfrak{C}}(\mathpzc{M}).$$
\end{prop}

\subsection{Duality and Twisting Morphisms}\label{subsec:dualitytwist}

In this section, $\mathpzc{M}={}_{R}\mathpzc{Mod}(Ch(\mathpzc{E}))$, where $\mathpzc{E}$ is a closed monoidal category and $R$ is a unital commutative monoid in $Ch(\mathpzc{E})$. In particular, $\mathpzc{M}$ is closed monoidal as well. 
Let $\mathfrak{C}$ be a stably formally planar cooperad in $\mathpzc{M}$. The dualising functor $(-)^{\vee}:\mathpzc{M}\rightarrow\mathpzc{M}^{op}$ is lax monoidal, so it induces a functor
$$(-)^{\vee}:\mathpzc{coAlg}^{conil}_{\mathfrak{C}}\rightarrow(\mathpzc{Alg}_{\mathfrak{C}^{\vee}})^{op}.$$
Let $\alpha:\mathfrak{C}\rightarrow\mathfrak{P}$ be a twisting morphism. Denote by
$$\hat{C}_{\alpha}:\mathpzc{Alg}_{\mathfrak{P}}\rightarrow(\mathpzc{Alg}_{\mathfrak{C}^{\vee}})^{op}$$
the composition $(-)^{\vee}\circ B_{\alpha}$.

%We will study this functor in this section. In particular, we shall show  recover the interpretation in terms of the shifted tangent complex. Before constructing this adjoint functor, let us identify a partially defined subfunctor which will be useful for applications later. 
The hat notation is supposed to evoke comparisons with formal power series -- this will be evident for the case of duality between cocommutative coalgebras and Lie algebras. Indeed, the underlying $\Z$-graded algebra of 
$\hat{C}_{\alpha}(\mathfrak{g})$ is 
$$\prod_{n}(\mathfrak{C}(n)\otimes_{\Sigma_{n}}\mathfrak{g}^{\otimes n})^{\vee}.$$
 Consider the $\Z$-graded `polynomial' algebra
 $$\bigoplus_{n}(\mathfrak{C}(n)\otimes_{\Sigma_{n}}\mathfrak{g}^{\otimes n})^{\vee}.$$
  There is a natural morphism of algebras
   $$\bigoplus_{n}(\mathfrak{C}(n)\otimes_{\Sigma_{n}}\mathfrak{g}^{\otimes n})^{\vee}\rightarrow\prod_{n}(\mathfrak{C}(n)\otimes_{\Sigma_{n}}\mathfrak{g}^{\otimes n})^{\vee}.$$ 
  Under nice circumstances, established below, the differential on $\hat{C}_{\alpha}$ restricts to a differential on this algebra. Thus, we get a differentially graded algebra $C_{\alpha}(\mathfrak{g})$ together with a morphism
  $$C_{\alpha}(\mathfrak{g})\rightarrow \hat{C}_{\alpha}(\mathfrak{g}).$$

\begin{defn}\label{defn:separable}
Let $\alpha:\mathfrak{C}\rightarrow\mathfrak{P}$ be a twisting morphism. A $\mathfrak{P}$-algebra $\mathfrak{g}$ is said to be $\alpha$-\textbf{separable} if 
\begin{enumerate}
\item
the natural morphism 
$$(\mathfrak{C}(n))^{\vee}\otimes_{\Sigma_{n}}\mathfrak{g}^{\vee}\otimes\ldots\otimes\mathfrak{g}^{\vee}\rightarrow ((\mathfrak{C}(n))^{\vee}\otimes\mathfrak{g}^{\vee}\otimes\ldots\otimes\mathfrak{g}^{\vee})^{\Sigma_{n}}\rightarrow(\mathfrak{C}(n)\otimes_{\Sigma_{n}}\mathfrak{g}^{\otimes n})^{\vee}$$ is an isomorphism,
\item
the morphism $\mathfrak{C}^{\vee}\circ\mathfrak{g}^{\vee}\rightarrow(\mathfrak{C}\circ\mathfrak{g})^{\vee}$ is a monomorphism, and
%\item
%the morphism $\mathfrak{g}^{\vee}\otimes\mathfrak{g}^{\vee}\rightarrow (\mathfrak{g}\otimes\mathfrak{g})^{\vee}$ is an isomorphism
%\item
%the morphism $\mathfrak{P}^{\vee}\circ\mathfrak{g}^{\vee}\rightarrow(\mathfrak{P}\circ\mathfrak{g})^{\vee}$ is a monomorphism
\item
 the morphism $\mathfrak{g}^{\vee}\rightarrow(\mathfrak{P}\circ\mathfrak{g})^{\vee}$ factors through $\mathfrak{P}^{\vee}\circ\mathfrak{g}^{\vee}$. 
\end{enumerate}
\end{defn}

Note that if the natural morphism from the coproduct of a countable collection of objects to the product of a countable collection of objects is a monomorphism, then $(2)$ follows from $(1)$.

\begin{prop}\label{polykoszul}
If $\mathfrak{g}$ is a separable $\mathfrak{P}$-algebra, then the differential on $\hat{C}_{\alpha}(\mathfrak{g})$ restricts to a differential on $C_{\alpha}(\mathfrak{g})$. 
\end{prop}

\begin{proof}
By our assumptions, the morphism $C_{\alpha}(\mathfrak{g})\rightarrow \hat{C}_{\alpha}(\mathfrak{g})$ is a monomorphism. Moreover, the underlying algebra of $C_{\alpha}(\mathfrak{g})$ is free on $\mathfrak{g}^{\vee}$, so it suffices to check that the restriction of the differential to $\mathfrak{g}^{\vee}$ factors through $C_{\alpha}(\mathfrak{g})$. Then it will automatically square to zero. In fact, we shall check that the underlying graded morphism (i.e., after forgetting differentials) of $\mathfrak{g}^{\vee}\rightarrow C_{\alpha}(\mathfrak{g})\cong(\mathfrak{C}\circ_{\alpha}\mathfrak{g})^{\vee}$ factors through $\mathfrak{C}^{\vee}\circ\mathfrak{g}^{\vee}$.  This follows from the commutative diagram below
\begin{displaymath}
\xymatrix{
\mathfrak{g}^{\vee}\ar[r]^{}&(\mathfrak{P}\circ\mathfrak{g})^{\vee}\ar[r]^{\;\;\;(\alpha\circ Id_{\mathfrak{g}})^{\vee}} & (\mathfrak{C}\circ\mathfrak{g})^{\vee}\\
\mathfrak{g}^{\vee}\ar[r]\ar@{=}[u] & \mathfrak{P}^{\vee}\circ\mathfrak{g}^{\vee}\ar[r]^{\alpha^{\vee}\circ Id}\ar[u] & \mathfrak{C}^{\vee}\circ\mathfrak{g}^{\vee}.\ar[u]
}
\end{displaymath}
\end{proof}
%(\mathfrak{C}\circ \mathfrak{g})^{\vee}\ar[r] & (\mathfrak{C}\circ(\mathfrak{g};\mathfrak{P}\circ\mathfrak{g}))^{\vee}\ar[r] & ((\mathfrak{C}_{\circ_{(1)}}\mathfrak{P})\circ\mathfrak{g})^{\vee}\ar[r] & ((\mathfrak{C}_{\circ_{(1)}}\mathfrak{C})\circ\mathfrak{g})^{\vee}\ar[r] & (\mathfrak{C}\circ\mathfrak{g})^{\vee}\\
%\mathfrak{C}^{\vee}\circ\mathfrak{g}^{\vee}\ar[u]\ar[r] & \mathfrak{C}^{\vee}\circ(\mathfrak{g}^{\vee};\mathfrak{P}^{\vee}\circ\mathfrak{g}^{\vee})\ar[r] & 
%\begin{displaymath}
%\xymatrix{
%(\mathfrak{C}\circ \mathfrak{g})^{\vee}\ar[r] & (\mathfrak{C}\circ(\mathfrak{g};\mathfrak{P}\circ\mathfrak{g}))^{\vee}\ar[r] & ((\mathfrak{C}_{\circ_{(1)}}\mathfrak{P})\circ\mathfrak{g})^{\vee}\ar[r] & ((\mathfrak{C}_{\circ_{(1)}}\mathfrak{C})\circ\mathfrak{g})^{\vee}\ar[r] & (\mathfrak{C}\circ\mathfrak{g})^{\vee}\\
%\mathfrak{C}^{\vee}\circ\mathfrak{g}^{\vee}\ar[u]\ar[r] & \mathfrak{C}^{\vee}\circ(\mathfrak{g}^{\vee};\mathfrak{P}^{\vee}\circ\mathfrak{g}^{\vee})\ar[r] & 
%}
%
%\end{displaymath}

\subsection{Non-Symmetric Operads}

There is a non-symmetric version of the theory detailed above. Details can be found in \cite{loday2012algebraic} Section 5.9. Let $(\mathpzc{E},\otimes,k)$ be a symmetric monoidal additive category. Consider the category $\mathpzc{Gr}_{\mathbb{N}_{0}}(\mathpzc{E})$ of graded objects in $\mathpzc{E}$.

\begin{defn}
Let $M$ and $N$ be graded objects. The \textbf{non-symmetric composite product} of $M$ and $N$, denoted $M\circ_{ns} N$, is the graded object given by
$$(M\circ_{ns} N)(n)=\bigoplus_{k}M(k)\otimes N^{\otimes k}(n).$$
\end{defn}

\begin{defn}
The graded object $I$ is defined by $I(i)=0$ for $i\neq 1$ and $I(1)=k$.
\end{defn}

With these, the constructions and results analogous to those for symmetric operads above go through mutatis mutandis.\newline
\\
Denote the category of non-symmetric operads by $\mathpzc{Op}_{ns}$. As with any module category, there is an adjunction
$$\adj{k[\Sigma]\otimes(-)}{\mathpzc{Gr}_{\mathbb{N}_{0}}(\mathpzc{E})}{\Sigma-\mathpzc{Mod}}{|-|},$$
that induces an adjunction
$$\adj{(-)^{\Sigma}}{\mathpzc{Op}_{ns}}{\mathpzc{Op}}{|-|}.$$

%Algebras in non-symmetric categories

\section{Homotopy Theory of Algebras}\label{algcoalghom}

In this section, we recall some general results regarding the existence of model structures on categories of algebras over operads. Standard references for the homotopy theory of such algebras are \cite{hinich1997homological} and \cite{MR2016697}, but we will also need \cite{pavlov2018admissibility}, \cite{harper2010homotopy}, \cite{white2018bousfield}, and \cite{spitzweck2001operads}. Moreover, we will in fact require some minor generalisations of these results.

\subsection{The Monoid Axiom}

Before establishing the existence of model structures on certain categories of algebras, we need to recall some terminology concerning compatibility between model category structures and monoidal category structures.

\subsubsection{Pushout-products}

Let $\mathpzc{M}$ be a combinatorial model category equipped with a symmetric monoidal structure.

Let $\mathpzc{L}, \mathpzc{M}$, and $\mathpzc{N}$ be categories, where $\mathpzc{N}$ has finite colimits, and let $\overline{\otimes}:\mathpzc{L}\times\mathpzc{M}\rightarrow\mathpzc{N}$ be a bifunctor.
Suppose that $f:X\rightarrow Y$ is a morphism in $\mathpzc{L}$ and $g:X'\rightarrow Y'$ is a morphism in $\mathpzc{M}$. We define the \textbf{pushout-product morphism} (implicitly, relative to $\overline{\otimes}$) to be the unique morphism 
$$Y\overline{\otimes} X'\coprod_{X\overline{\otimes} X'}X\overline{\otimes} Y'\rightarrow Y\overline{\otimes} Y'$$ determined by the morphisms $Y\overline{\otimes} X'\rightarrow Y\overline{\otimes} Y'$ and $X\overline{\otimes} Y'\rightarrow Y\overline{\otimes} Y'$.

\begin{defn}

Let $\mathpzc{M}$ be a model category that is equipped with a symmetric monoidal structure. Let $\mathpzc{N}$ be a model category that is also an $\mathpzc{M}$-module, with tensor functor $\overline{\otimes}:\mathpzc{M}\times\mathpzc{N}\rightarrow\mathpzc{N}$. Let $\mathcal{S}$ be a class of morphisms in $\mathpzc{M}$ and $\mathcal{T}$ a class of morphisms in $\mathpzc{N}$.
\begin{enumerate}
\item
The pair $(\mathcal{S},\mathcal{T})$ is said to \textbf{satisfy the pushout-product axiom} if a pushout-product of a morphism in $\mathcal{S}$ and a morphism in $\mathcal{T}$ is a morphism in $\mathcal{T}$.
\item
The pair $(\mathcal{S},\mathcal{T})$ is said to \textbf{satisfy the weak pushout-product axiom} if whenever $s_{1}\Box\ldots\Box s_{n-1}\Box t$ is an iterated pushout-product with each $s_{i}$ a morphism in $\mathcal{S}$, and $t$ a morphism in $\mathcal{T}$ and one of the $s_{i}$ or $t$ is a weak equivalence, then $s_{1}\Box\ldots\Box s_{n-1}\Box t$ is a weak equivalence. 
\item
The category $\mathpzc{N}$ is said to be a \textbf{weak model} $\mathpzc{M}$-\textbf{module} if the pair $(Cof_{\mathpzc{M}},Cof_{\mathpzc{N}})$ satisfies the weak pushout-product axiom, where $Cof_{\mathpzc{M}}$ (resp. $Cof_{\mathpzc{N}}$) denotes the class of cofibrations in $\mathpzc{M}$ (resp. in $\mathpzc{N}$).
\item 
The category $\mathpzc{M}$ is said to be a \textbf{weak (non-unital) monoidal model category} if it is a weak model $\mathpzc{M}$-module when regarded as a module over itself.
\item
The category $\mathpzc{N}$ is said to be a \textbf{(non-unital) monoidal model} $\mathpzc{M}$-\textbf{module} if the pair $(Cof_{\mathpzc{M}},Cof_{\mathpzc{N}})$  satisfies the pushout-product axiom and the weak pushout-product axiom.
\item
The category $\mathpzc{M}$ is said to be a \textbf{(non-unital) monoidal model category} if it is a monoidal model $\mathpzc{M}$-module.
\item 
We say that $\mathpzc{N}$ \textbf{satisfies the unit axiom} as a $\mathpzc{M}$-module if, for any cofibrant replacement $C\rightarrow k$ of the monoidal unit of $\mathpzc{M}$, and any object $N$ of $\mathpzc{N}$, the morphism
$$C\overline{\otimes} N\rightarrow k\overline{\otimes} N\cong N$$
is an equivalence.
% if cofibrations satisfy the pushout-product axiom and the weak pushout-product axiom.
\end{enumerate}
\end{defn}

\begin{rem}

If $\mathpzc{N}$ is a weak monoidal model $\mathpzc{M}$-module, then the tensor product functor
$$\overline{\otimes}:\mathpzc{M}\times\mathpzc{N}\rightarrow\mathpzc{N}$$
is left deformable (or derivable) in the sense of homotopical categories (see e.g., \cite{Riehl} Definition 2.2.4). Thus, we can make sense of $\overline{\otimes}^{\mathbb{L}}$.
\end{rem}

%Following \cite{pavlov2018admissibility} there is a symmetric version as well. Let $\mathpzc{M}$ be a model category equipped with a symmetric monoidal structure, and let $s$ be a morphism in $\mathpzc{M}$. Let $n=(n_{1},\ldots,n_{e})$ be a multi-index, and $s=(s_{1},\ldots,s_{e})$ a family of morphisms in $\mathpzc{M}$. Write $\Sigma_{n}=\prod_{i}\Sigma_{n_{i}}$ and 
%$$s^{\Box n}=\Box_{i}s_{i}^{\Box n_{i}}.$$
%This is equipped with a (left) action of $\Sigma_{n}$. Now let $y$ be a morphism of (right) $\Sigma_{n}$-modules. We consider the $\Sigma_{n}$-equivariant pushout-product
%$$y\Box_{\Sigma_{n}}s^{\Box n}.$$

%\begin{defn}[\cite{pavlov2018admissibility} Definition 2.1 (vii)]
%A weak equivalence $y$ of right $\Sigma_{n}$-modules is said to be \textbf{symmetric flat} if $y\Box_{\Sigma_{n}}s^{\Box n}$ is a weak equivalence for any family of cofibrations $s$.
%\end{defn}

\subsubsection{Variants of the Monoid Axiom}

Here we give variations of the monoid axiom -- these are relevant when the pushout-product axiom is not satisfied by cofibrations.

\begin{defn}
     Let $\mathpzc{M}$ and $\mathpzc{N}$ be model categories, where $\mathpzc{M}$ is a symmetric monoidal category, and $\mathpzc{N}$ is an $\mathpzc{M}$-module with tensor functor
    $$\overline{\otimes}:\mathpzc{M}\times\mathpzc{N}\rightarrow\mathpzc{N}.$$
    
    \begin{enumerate}
        \item 
         We say that $\mathpzc{N}$ is \textbf{weakly tensored over} $\mathpzc{M}$ if
    \begin{enumerate}
        \item 
           the bifunctor $\overline{\otimes}$ satisfies the weak pushout-product axiom, 
           \item 
           any transfinite composition of pushouts of morphisms of the form $f\overline{\otimes}Id_{N}$ where $f$ is an acyclic cofibration in $\mathpzc{M}$ and $N$ is an arbitrary object of $\mathpzc{N}$ is a weak equivalence, and
          \item
        any transfinite composition of pushouts of morphisms of the form $Id_{M}\overline{\otimes}g$ where $g$ is an acyclic cofibration in $\mathpzc{N}$ and $M$ is an arbitrary object of $\mathpzc{M}$ is a weak equivalence.
           \end{enumerate}
           \item 
           We say that $\mathpzc{M}$ \textbf{satisfies the monoid axiom} if it is weakly tensored over itself.
          \item 
         We say that $\mathpzc{N}$ is \textbf{weakly pp-tensored over} $\mathpzc{M}$ if
    \begin{enumerate}
        \item 
           it is weakly tensored over $\mathpzc{M}$, 
           \item 
           any transfinite composition of pushouts of morphisms of the form $f\overline{\otimes}Id_{N}$ where $f$ is an iterated pushout-product of acyclic cofibrations in $\mathpzc{M}$, and $N$ is an arbitrary object of $\mathpzc{N}$, is a weak equivalence, and
              \item
        any transfinite composition of pushouts of morphisms of the form $Id_{M}\overline{\otimes}g$ where, 
        \begin{enumerate}
            \item 
            $g$ is an iterated pushout-product of acyclic cofibrations in $\mathpzc{M}$ with an acyclic cofibration in $\mathpzc{N}$, and
            \item
            $M$ is an arbitrary object of $\mathpzc{M}$,
        \end{enumerate}  
        is a weak equivalence.
           \end{enumerate}
           \item 
              We say that $\mathpzc{M}$  \textbf{satisfies the pp-monoid axiom} if it is weakly pp-tensored over itself.
   \end{enumerate}
 
\end{defn}

%\begin{convention}
%
%From now on, unless stated otherwise, the monoid axiom will be assumed.
%\end{convention}

Usually such tensoring will arise in the following way. We will have a strongly monoidal, colimit-preserving functor $L:\mathpzc{M}\rightarrow\mathpzc{N}$ between model categories that are closed symmetric monoidal as categories. We then define
$$\overline{\otimes}:\mathpzc{M}\times\mathpzc{N}\rightarrow\mathpzc{N}$$
by $(-)\overline{\otimes}(-)=L(-)\otimes_{\mathpzc{N}}(-)$, where $\otimes_{\mathpzc{N}}$ is the tensor product on $\mathpzc{N}$. In fact, we will essentially always be in the following situation. Let $\mathpzc{E}$ be a complete and cocomplete closed symmetric monoidal additive category. Let $R$ be a unital commutative monoid in $Ch(\mathpzc{E})$. Write $S=\mathrm{Hom}_{\mathpzc{E}}(k,k)$, where $k$ is the monoidal unit of $\mathpzc{E}$. Then $S$ is a unital commutative ring. There is a canonical colimit-preserving, strongly monoidal functor $L:{}_{S}\mathpzc{Mod}(\mathpzc{Ab})\rightarrow\mathpzc{E}$ that sends $S$ to the monoidal unit $k$, and extends by colimits. This extends to a colimit-preserving, strongly monoidal functor $L:Ch({}_{S}\mathpzc{Mod}(\mathpzc{Ab}))\rightarrow Ch(\mathpzc{E})$ and, by tensoring with $R$, a strongly monoidal colimit-preserving functor $L_{R}:Ch({}_{S}\mathpzc{Mod}(\mathpzc{Ab}))\rightarrow {}_{R}\mathpzc{Mod}(Ch(\mathpzc{E}))$. 

%More generally, if $\mathpzc{E}$ is enriched over ${}_{S}\mathpzc{Mod}(\mathpzc{Ab})$ for some $S\in\mathpzc{Alg}_{\mathfrak{Comm}}(\mathpzc{Ab})$, then there is also a strongly monoidal functor $L_{R}:Ch({}_{S}\mathpzc{Mod}(\mathpzc{Ab}))\rightarrow {}_{R}\mathpzc{Mod}(Ch(\mathpzc{E}))$.

\begin{prop}\label{prop:puretnesorLequiv}

       Let $L:\mathpzc{M}\rightarrow\mathpzc{N}$ be a strongly monoidal, colimit-preserving functor  between model categories that are closed symmetric monoidal as categories, such that the induced functor $\overline{\otimes}$ realises $\mathpzc{N}$ as a category weakly tensored over $\mathpzc{M}$. Then $L$ preserves weak equivalences between cofibrant objects.
\end{prop}

\begin{proof}
    We have $L(-)\cong(-)\overline{\otimes}k$, where $k$ is the monoidal unit of $\mathpzc{N}$. By assumption, $(-)\overline{\otimes}Id_{k}$ sends acyclic cofibrations to equivalences. By Ken Brown's Lemma, this functor also preserves weak equivalences between cofibrant objects. 
\end{proof}

\begin{prop}
    Let $\mathpzc{N}$ be weakly tensored over $\mathpzc{M}$. Let $R\in\mathpzc{M}$ be a unital associative monoid. The transferred model structure exists on ${}_{R}\mathpzc{Mod}(\mathpzc{N})$.
\end{prop}

\begin{proof}
    This is a very minor variation of \cite{schwede} Theorem 4.1/Remark 4.2. 
\end{proof}

Suppose that $\mathpzc{M}$ is a weak monoidal model category, and $\mathpzc{N}$ is weakly tensored over $\mathpzc{M}$. Let $R$ be an associative monoid in $\mathpzc{M}$. Then the transferred model structure exists on both $\mathpzc{Mod}_{R}(\mathpzc{M})$ and ${}_{R}\mathpzc{Mod}(\mathpzc{N})$. We get an induced functor
$$\overline{\otimes}_{R}:\mathpzc{Mod}_{R}(\mathpzc{M})\times{}_{R}\mathpzc{Mod}(\mathpzc{N})\rightarrow\mathpzc{N}$$
defined as the coequaliser of the two morphisms from $M\overline{\otimes}R\overline{\otimes}N\rightarrow M\overline{\otimes}N$ induced by the respective actions of $R$ on $M$ and $N$.

\begin{prop}\label{prop:tensorevedoverringRinduced}
    
  Suppose that $\mathpzc{M}$ is a weak monoidal model category, and $\mathpzc{N}$ is weakly tensored over $\mathpzc{M}$. Let $R$ be an associative monoid in $\mathpzc{M}$. Then the following hold.
  \begin{enumerate}
      \item 
        Any transfinite composition of pushouts of morphisms of the form $f\overline{\otimes}_{R}Id_{N}$ where $f$ is an acyclic cofibration in $\mathpzc{Mod}_{R}(\mathpzc{M})$ and $N\in{}_{R}\mathpzc{Mod}(\mathpzc{N})$ is a weak equivalence in $\mathpzc{N}$.
        %\item 
         % Any transfinite composition of morphisms of the form $f\overline{\otimes}_{R}Id_{N}$ where $f$ is an acyclic cofibration in $\mathpzc{Mod}_{R}(\mathpzc{M})$ and $N\in{}_{R}\mathpzc{Mod}(\mathpzc{N})$ is a weak equivalence in $\mathpzc{N}$.
                  \item 
          Any transfinite composition of pushouts of morphisms of the form $Id_{M}\overline{\otimes}_{R}g$ where $g$ is an acyclic cofibration in ${}_{R}\mathpzc{Mod}(\mathpzc{N})$ and $M\in\mathpzc{Mod}_{R}(\mathpzc{M})$ is a weak equivalence in $\mathpzc{N}$.
        %  \item 
        %  If $R$ is commutative, and $\mathpzc{N}$ is weakly $pp$-tensored over $\mathpzc{M}$, then ${}_{R}\mathpzc{Mod}(\mathpzc{N})$ is weakly tensored over $\mathpzc{Mod}_{R}(\mathpzc{M})$.
  \end{enumerate}

\end{prop}

\begin{proof}
\begin{enumerate}
    \item 
        Any such morphism is a transfinite composition of pushouts of morphisms of the form $(f\otimes R)\overline{\otimes}_{R}Id_{N}\cong f\overline{\otimes}Id_{|N|}$ where $f$ is an acyclic cofibration in $\mathpzc{M}$ and $|N|$ is an object of $\mathpzc{N}$. Such a morphism is a weak equivalence by assumption on $\overline{\otimes}$.
        \item 
This is identical to part (1) with the two variables interchanged.
%\item 
%We just need to verify the weak pushout-product axiom. First, we claim that an iterated pushout-product of cofibrations in ${}_{R}\mathpzc{Mod}(\mathpzc{N})$ is an iterated pushout-product of cofibrations in 

%We may take a generating cofibration (resp. a generating acyclic cofibration) of the form $f\otimes Id_{R}$ in $\mathpzc{Mod}_{R}(\mathpzc{M})$, with $f$ a cofibration (resp. an acyclic cofibration) in $\mathpzc{M}$, and a generating cofibration (resp. a generating acyclic cofibration) $Id_{R}\overline{\otimes}g$ in ${}_{R}\mathpzc{Mod}(\mathpzc{M})$ with $g$ a cofibration (resp. an acyclic cofibration) in $\mathpzc{N}$. Then $(f\otimes Id_{R})\Box_{R}(Id_{R}\overline{\otimes} g)\cong Id_{R}\overline{\otimes}(f\Box g)$. By assumption this is a weak equivalence.
\end{enumerate}

\end{proof}

\subsection{$K$-Transversality}

 The definitions and proofs below appear in \cite{dang}, but we record them here for completeness. 
 
\begin{defn}
Let $\mathpzc{M}$ be a model category equipped with a monoidal structure, and let $\mathcal{S}$ be a subcategory of $\mathpzc{M}$. An object $X$ of $\mathpzc{M}$ is said to be $K$-\textbf{transverse to} $\mathcal{S}$ if for any weak equivalence $f:A\rightarrow B$ in $\mathcal{S}$, the morphism $X\otimes A\rightarrow X\otimes B$ is a weak equivalence. The object $X$ is said to be $K$-\textbf{flat} if it is $K$-transverse to the class $\mathpzc{Mor}(\mathpzc{M})$ of all morphisms in $\mathpzc{M}$.
\end{defn}

\begin{rem}
    The terminology `$K$-flat' is standard in homological algebra, but perhaps not in homotopical algebra. Indeed, what we call a $K$-flat object in a monoidal model category is often just called a flat object (see, for example, \cite{white2017model} and \cite{white2018bousfield}). However, in a monoidal abelian category, a complex $F_{\bullet}$ is called $K$-\textbf{flat} if $F_{\bullet}\otimes X_{\bullet}$ is acyclic whenever $X_{\bullet}$ is acyclic. Since we will be dealing with both $K$-flat complexes and flat objects in the abelian category itself, to avoid confusion we adopt the term `$K$-flat' for general monoidal model categories.
\end{rem}

In model categories (or more generally, homotopical categories) the notion of derived functors makes sense. Suppose that there is a functor $Q:\mathpzc{M}\rightarrow\mathpzc{M}$, and a natural transformation $\eta:Q\rightarrow Id_{\mathpzc{M}}$, where each component of $\eta$ is a weak equivalence in $\mathcal{S}$. Suppose further that the functor $Q(-)\otimes Q(-):\mathpzc{M}\times \mathpzc{M}\rightarrow\mathpzc{M}$ is a left-derived functor of $\otimes$. Suppose that for any $Y\in\mathpzc{M}$, $Q(Y)$ is $K$-transverse to $\mathcal{S}$. Then, if $X$ is $K$-transverse to $\mathcal{S}$, for any object $A$ of $\mathpzc{M}$ the morphism $A\otimes^{\mathbb{L}}X\rightarrow A\otimes X$ is an equivalence.

\begin{rem}\label{rem:flatcat}
    The concept of a subcategory of flat objects in a monoidal model category appears in \cite{MR4721371} Definition 7.1.1. This is a full symmetric monoidal subcategory (containing the unit), $\mathpzc{M}^{\flat}$, such that 
    \begin{enumerate}
        \item 
        all cofibrant objects are in $\mathpzc{M}^{\flat}$, and
        \item 
         if $X$ is flat and $f$ is a weak equivalence in $\mathpzc{M}^{\flat}$, then $X\otimes f$ is a weak equivalence.
    \end{enumerate}
    In particular, as mentioned in loc. cit., if the unit of $\mathpzc{M}$ is cofibrant, then the subcategory of cofibrant objects is a subcategory of flat objects. 
\end{rem}

 In \cite{white2014monoidal} it is shown that flatness of cofibrant objects in some Bousfield localisation of $\mathpzc{M}$ is a necessary condition for this localisation to be a monoidal model category.

Let us give some operations under which $K$-transversality is stable. First, we need a definition.

\begin{defn}\label{defn:lprop}
\begin{enumerate}
\item
Let $\mathpzc{M}$ be a model category. A morphism $f:X\rightarrow Y$ in $\mathpzc{M}$ is said to be \textbf{left proper} if every pushout diagram of the following form
\begin{displaymath}
\xymatrix{
X\ar[d]^{f}\ar[r] & A\ar[d]\\
Y\ar[r] & P
}
\end{displaymath}
 is a homotopy pushout. One defines right properness of morphisms dually.
\item
Let $\mathpzc{M}$ be a model category equipped with a symmetric monoidal structure. For $\mathcal{C}$ a class of objects in $\mathpzc{M}$, a morphism $f:X\rightarrow Y$ in $\mathpzc{M}$ is said to be $\mathcal{C}$-\textbf{monoidally left proper} if $Id_{C}\otimes f$ is left proper for any $C\in\mathcal{C}$. 
\end{enumerate}
\end{defn}

\begin{rem}
    The left properness condition on a morphism is closely related to the $h$-cofibration condition, as formulated in \cite{white2014monoidal} Definition 8.1. In fact, if the model category $\mathpzc{M}$ is left proper (i.e., all cofibrations are left proper), then the conditions are equivalent. In what follows model categories will typically be manifestly left proper, and so the conditions do coincide. We find the formulation in terms of preservation of homotopy pushouts useful, as this is what we will need in the proofs below.  
\end{rem}

\begin{defn}

\begin{enumerate}
    \item 
    A weak monoidal model category $\mathpzc{M}$ is said to be a \textbf{weak} h\textbf{-monoidal model category} if any iterated pushout-product of cofibrations is left proper. 
    \item 
       A weak monoidal model category $\mathpzc{M}$ is said to be a \textbf{weak} h$^{\otimes}$-\textbf{monoidal model category} if any morphism of the form $Id_{C}\otimes f$, where $f$ is an iterated pushout-product of cofibrations, is left proper. 
\end{enumerate}

\end{defn}

In particular, if $\mathpzc{M}$ is a weak h-monoidal model category then cofibrations are $\mathcal{C}$-monoidally left proper, where $\mathcal{C}$ is the class of cofibrant objects. Indeed, when $C$ is cofibrant and $f$ is a cofibration, the morphism $Id_{C}\otimes f$ is the pushout-product of the cofibrations $0\rightarrow C$ and $f$. Thus, in this case, the model category $\mathpzc{M}$ is left proper provided the unit is cofibrant. 

\begin{defn}[cf. \cite{batanin2013homotopy} Definition 2.1]\label{defn:weaklyelmodel}

Let $\mathcal{S}$ be a collection of morphisms in a cocomplete model category $\mathpzc{M}$ (or more generally, a cocomplete homotopical category; see Definition \ref{defn:hmtpical}). The category $\mathpzc{M}$ is said to be \textbf{weakly} $\mathcal{S}$-\textbf{elementary} if, for any ordinal $\lambda$, and any natural transformation $\eta:X\rightarrow Y$ of cocontinuous functors in $\mathpzc{Fun}(\lambda,\mathpzc{M})^{cocont}$ such that each $\eta_{i}:X_{i}\rightarrow Y_{i}$ is a weak equivalence, and $X_{i\le j}$ and $Y_{i\le j}$ are in $\mathcal{S}$ for any $i\le j$ in $\lambda$, then $\colim_{\lambda}\eta$ is a weak equivalence.  
\end{defn}

%\begin{example}[\cite{kelly2016homotopy} Example A.3.18]
%
%\begin{enumerate}
%\item
%If $\mathcal{S}$ is the class of cofibrations then $\mathpzc{M}$ is weakly $\mathcal{S}$-elementary.
%\item
%If $\mathcal{S}$ is a class of weak equivalences such that any transfinite composition of morphisms in $\mathcal{S}$ is a weak equivalence, then it follows from the $2$-out-of-$3$ property that $\mathpzc{M}$ is weakly $\mathcal{S}$-elementary.
%\end{enumerate}
%\end{example}

\begin{lem}\label{Lem:pushouttransK-flat}
Let $\mathpzc{M}$ be a weak h-monoidal model category.
\begin{enumerate}
\item
Let $\mathcal{T}$ be a class of morphisms in $\mathpzc{M}$ and let $\mathcal{C}$ be the class of objects appearing as domains or codomains of morphisms in $\mathcal{T}$. Let
\begin{displaymath}
\xymatrix{
A\ar[d]^{f}\ar[r] & B\ar[d]\\
C\ar[r] & D
}
\end{displaymath}
be a pushout diagram, where $f$ is $\mathcal{C}$-monoidally left proper. If $A,B$, and $C$ are $K$-transverse to $\mathcal{T}$, then $D$ is $K$-transverse to $\mathcal{T}$.\\
\item
Let $\mathpzc{M}$ be weakly $\mathcal{S}$-elementary for some class of morphisms $\mathcal{S}$ in $\mathpzc{M}$. Let $\mathcal{T}$ be a class of morphisms in $\mathpzc{M}$. Suppose that if $f\in\mathcal{S}$ and $Q$ is any object appearing as the domain or codomain of a morphism in $\mathcal{T}$, then $Q\otimes f$ is in $\mathcal{S}$. Let 
\begin{displaymath}
\xymatrix{
X_{0}\ar[r] & X_{1}\ar[r] & \ldots\ar[r] & X_{\beta}\ar[r] & \ldots
}
\end{displaymath}
be a $\lambda$-indexed transfinite sequence where each $X_{\beta}$ is $K$-transverse to $\mathcal{T}$ and each morphism in the diagram is in $\mathcal{S}$. Then $\colim_{\beta<\lambda} X_{\beta}$ is $K$-transverse to $\mathcal{T}$. 
\end{enumerate}
\end{lem}

\begin{proof}
\begin{enumerate}
\item
Let $g:P\rightarrow Q$ be a weak equivalence in $\mathcal{T}$. Consider the following commutative diagram
\begin{displaymath}
\xymatrix{
& A\otimes Q\ar[dd]\ar[rr] & &B\otimes Q\ar[dd]\\
A\otimes P\ar[dd]\ar[ur]\ar[rr] & & B\otimes P\ar[dd]\ar[ur]\\
& C\otimes Q\ar[rr] && D\otimes Q.\\
C\otimes P\ar[ur]\ar[rr] &  & D\otimes P\ar[ur]
}
\end{displaymath}
Now, since $f$ is a $\mathcal{C}$-monoidally left proper morphism, the near and far faces of the cube are both homotopy pushouts. The morphisms connecting the top-left, bottom-left, and top-right vertices of the near and far faces are weak equivalences. Therefore, the morphism $D\otimes P\rightarrow D\otimes Q$ is an equivalence. Since $g:P\rightarrow Q$ was arbitrary, this implies that $D$ is $K$-transverse to $\mathcal{T}$. 
\item
Consider a diagram
\begin{displaymath}
\xymatrix{
X_{0}\ar[r] & X_{1}\ar[r] & \ldots\ar[r] & X_{\beta}\ar[r] & \ldots
}
\end{displaymath}
where each $X_{\beta}$ is $K$-transverse to $\mathcal{T}$ and each morphism in the diagram is in $\mathcal{S}$. Let $g:A\rightarrow B$ be a weak equivalence in $\mathcal{T}$.
%Consider the following commutative diagram
%\begin{displaymath}
%\xymatrix{
%Q\otimes^{\mathbb{L}} X_{0}\ar[r]\ar[d] & Q\otimes^{\mathbb{L}} X_{1}\ar[r]\ar[d] & \ldots\ar[r]\ar[d] & Q\otimes^{\mathbb{L}} X_{\beta}\ar[r]\ar[d] & \ldots
%Q\otimes X_{0}\ar[r] & Q\otimes X_{1}\ar[r] & \ldots\ar[r] & Q\otimes X_{\beta}\ar[r] & \ldots
%}
%\end{displaymath}
%All the vertical morphisms are equivalences since each $X_{\beta}$ is $K$-flat. Moreover, the colimits of both the top and bottom sequences are homotopy colimits. Thus, the morphism $\colim Q\otimes^{\mathbb{L}}X_{\beta}\rightarrow \colim Q\otimes X_{\beta}$ is an equivalence. 
In both diagrams
$$A\otimes X:\beta\mapsto A\otimes X_{\beta}$$
$$B\otimes X:\beta\mapsto B\otimes X_{\beta}$$
where $\alpha\le\beta$, the transition morphisms $A\otimes X_{\alpha}\rightarrow A\otimes X_{\beta}$ and $B\otimes X_{\alpha}\rightarrow B\otimes X_{\beta}$ are in $\mathcal{S}$. Each $g\otimes Id_{X_{\beta}}$ is an equivalence, so the colimit is as well.
\end{enumerate}
\end{proof}

\begin{lem}\label{lem:almost-monoidal}

   Let $\mathpzc{M}$ be a combinatorial weak h-monoidal model category whose class of cofibrations is $\mathrm{cof}$, and whose class of weak equivalences is $W$. Fix a generating set $I$ of cofibrations.
      Let $\mathrm{cof}^{\otimes}$ denote the weakly saturated class generated by finite iterated pushout-products of elements of $I$. If $\mathrm{cof}^{\otimes}\cap W$ is weakly saturated, then there is a combinatorial model category structure on $\mathpzc{M}$ with the same class of weak equivalences $W$, and with cofibrations $\mathrm{cof}^{\otimes}$ satisfying the pushout-product axiom such that $\mathrm{cof}^{\otimes}$ is precisely the weakly saturated class generated by iterated pushout-products of elements of $\mathrm{cof}$. In particular, if $\mathrm{cof}^{\otimes}$ also satisfies the weak pushout-product axiom, then $\mathpzc{M}$ is a (non-unital) monoidal model category. 
\end{lem}

\begin{proof}
We use \cite{lurie2006higher} Proposition A.2.6.10 (and Proposition A.2.6.8). Let $\mathrm{cof}^{\otimes}$ denote the weakly saturated class generated by finite pushout-products of elements of $I$. Condition $(1)$ of \cite{lurie2006higher} Proposition A.2.6.10 is then satisfied by construction. Condition (2) of \cite{lurie2006higher} Proposition A.2.6.10, that $\mathrm{cof}^{\otimes}\cap W$ is weakly saturated, is satisfied by assumption. Condition $(3)$, that the class $W$ is accessible, follows from the fact that $\mathpzc{M}$ is combinatorial. Condition $(4)$, that weak equivalences satisfy the $2$-out-of-$3$ property, follows from the fact that $\mathpzc{M}$ is a model category. Condition $(5)$, that any element with the right lifting property against $\mathrm{cof}^{\otimes}$ is a weak equivalence, follows from the fact that $\mathpzc{M}$ is a model category, and that $\mathrm{cof}\subset\mathrm{cof}^{\otimes}$. Now, as the left-hand part of a weak factorisation system, \cite{schwede} Lemma 3.5 implies that $\mathrm{cof}^{\otimes}$ satisfies the pushout-product axiom precisely if pushout-products of a generating set are contained in $\mathrm{cof}^{\otimes}$. This is true by assumption. Then iterated pushout-products of elements of $\mathrm{cof}$ are contained in $\mathrm{cof}^{\otimes}$. Thus, the weakly saturated class generated by iterated pushout-products of elements of $\mathrm{cof}$ is contained in $\mathrm{cof}^{\otimes}$. But $\mathrm{cof}^{\otimes}$ is generated as a weakly saturated class by a subset of the class of iterated pushout-products of elements of $\mathrm{cof}$. Hence, $\mathrm{cof}^{\otimes}$ must in fact equal the weakly saturated class generated by iterated pushout-products of elements of $\mathrm{cof}$. 
\end{proof}

\begin{rem}
In light of the above lemma, the notion of a weak monoidal model category may seem like needless generality. However, we have explicit examples (Example \ref{example:topalg}) in mind in which a very natural model structure is weak monoidal but not monoidal. 
\end{rem}

\subsection{Model Structures on Categories of Algebras}
%In \cite{kelly2016homotopy} we gave conditions on model structures on categories of chain complexes $Ch(\mathpzc{E})$ under which the operads $\mathfrak{Ass},\mathfrak{Comm}$, and $\mathfrak{Lie}$ are admissible. Namely we required that $\mathpzc{E}$ needs to be enriched over $\Q$, $Ch(\mathpzc{E})$ admits the small object argument, and the model structure has generating acyclic cofibrations of the form $0\rightarrow D^{n}(F)$. Using \cite{white2017model} it is easy to prove a vast generalisation of this. 
We are now in a position to formulate very moderate generalisations of results from \cite{white2017model} and \cite{white2018bousfield} regarding the existence of model structures on categories of algebras over operads.

In this section, $\mathpzc{M}$ will be a combinatorial model category that is also a symmetric monoidal category satisfying the monoid axiom. Let $\mathfrak{P}$ be an operad in $\mathpzc{M}$ (either symmetric or non-symmetric), and consider the free-forgetful adjunction

$$\adj{\mathrm{Free}_{\mathfrak{P}}(-)}{\mathpzc{M}}{\mathpzc{Alg}_{\mathfrak{P}}(\mathpzc{M})}{|-|_{\mathfrak{P}}}.$$

 Recall that the \textbf{transferred model structure} on $\mathpzc{Alg}_{\mathfrak{P}}(\mathpzc{M})$, if it exists, is the one for which weak equivalences (resp. fibrations) are morphisms $f:A\rightarrow B$ of algebras such that $|f|_{\mathfrak{P}}$ is a weak equivalence (resp. fibration). When there is no confusion, we will write $|-|$ instead of $|-|_{\mathfrak{P}}$.

\begin{defn}
An operad $\mathfrak{P}$ is said to be \textbf{admissible} if the transferred model structure exists on $\mathpzc{Alg}_{\mathfrak{P}}$. 
\end{defn}
By \cite{harper2010homotopy} Proposition 7.6, for $X$ a $\mathfrak{P}$-algebra, there is a $\Sigma$-module $\mathfrak{P}_{X}$ in $\mathpzc{M}$ such that for any $\Sigma$-module $Y$ in $\mathpzc{M}$, there is an isomorphism, natural in $X$ and $Y$,
 $$X\coprod(\mathfrak{P}\circ Y)\cong\mathfrak{P}_{X}\circ Y.$$

  As in \cite{harper2010homotopy} Definition 7.31, for $s:A\rightarrow B$ a morphism in $\mathpzc{M}$, we define $Q^{t}_{q}(s)$ for $t\ge 1$ and $0\le q\le t$ as follows. We have $Q_{0}^{t}(s)\defeq A^{\otimes t}$, $Q^{t}_{t}(s)\defeq B^{\otimes t}$, and for $0<q<t$, $Q^{t}_{q}(s)$ is defined by the pushout. 
 \begin{displaymath}
 \xymatrix{
 (A^{\otimes(t-q)}\otimes Q^{q}_{q-1}(s))^{\oplus\binom{t}{q}}\ar[d]\ar[r] & Q_{q-1}^{t}(s)\ar[d]\\
  (A^{\otimes(t-q)}\otimes B^{\otimes q})^{\oplus\binom{t}{q}}\ar[r] & Q^{t}_{q}(s),
 }
 \end{displaymath}
 where the top morphism is induced by projection, and the left-hand morphism is induced by the natural morphism $Q^{q}_{q-1}(s)\rightarrow B^{\otimes q}$. 
 \begin{prop}[\cite{harper2010homotopy}  Proposition 7.32]\label{prop:pushoutalg}
 Let $s:A\rightarrow B$ be a morphism in $\mathpzc{M}$, and let $X$ be a $\mathfrak{P}$-algebra. Consider a pushout diagram of algebras
 \begin{displaymath}
 \xymatrix{
 \mathfrak{P}(A)\ar[d]^{\mathfrak{P}(s)}\ar[r] & X\ar[d]\\
 \mathfrak{P}(B)\ar[r] & P.
 }
 \end{displaymath}
 Then $P$ is naturally isomorphic to a filtered colimit
 $$P\cong\colim_{n}X_{n},$$
 where $X_{0}=X$, and for $n\ge 1$ the morphism $X_{n-1}\rightarrow X_{n}$ is given by the pushout diagram in $\mathpzc{M}$
 \begin{displaymath}
 \xymatrix{
 \mathfrak{P}_{X}(n)\otimes_{\Sigma_{n}} Q^{n}_{n-1}(s)\ar[d]\ar[r] & X_{n-1}\ar[d]\\
 \mathfrak{P}_{X}(n)\otimes_{\Sigma_{n}} B^{\otimes n}\ar[r] & X_{n}.
 }
 \end{displaymath}
 \end{prop}
 The filtration $P\cong\colim_{n}X_{n}$ will be called the \textbf{standard filtration}.

%As explained in \cite{kelly2016homotopy} 6.3.1, the following definition is based on a comment after Remark 6.1.3 on Page 36 \cite{white2018bousfield}, where we are also relaxing the condition that $\mathpzc{C}$ must be a monoidal model category.

%\begin{defn}\label{defn:weakPalgeb}
%Let $\mathfrak{P}$ be a symmetric operad, and $\mathcal{S}$ a class of morphisms in $\mathpzc{M}$. A collection of morphisms $\mathcal{S}$ in $\mathpzc{M}$ is said to \textbf{satisfy the weak }$\mathfrak{P}$-\textbf{algebra axiom relative to }$X$ if for any $s\in\mathcal{S}$
%\begin{enumerate}
%\item
%the square below is a homotopy pushout for any $n\ge1$:
 %\begin{displaymath}
%\xymatrix{
 %\mathfrak{P}_{X}(n)\otimes_{\Sigma_{n}} Q_{n-1}^{n}(s)\ar[d]\ar[r] & X_{n-1}\ar[d]\\
 %\mathfrak{P}_{X}(n)\otimes_{\Sigma_{n}} B^{\otimes n}\ar[r] & X_{n},
 %}
 %\end{displaymath}
 %\item
%and $\mathpzc{M}$ is weakly $\mathcal{S}$-elementary. 
%\end{enumerate}
%$\mathfrak{P}$ is said to \textbf{satisfy the weak }$\mathfrak{P}$-\textbf{algebra axiom} if it satisfies the weak $\mathfrak{P}$-algebra axiom relative to $X$ for all $X\in\mathpzc{Alg}_{\mathfrak{P}}(\mathpzc{M})$.
%\end{defn}

\begin{rem}
There is also a version for non-symmetric operads, in which one removes any reference to $\Sigma_{n}$ above. See, for example, \cite{harper2010homotopy} Proposition 7.28 for the construction of $\mathfrak{P}_{X}$ in this case.
\end{rem}

This description of certain pushouts in $\mathpzc{Alg}_{\mathfrak{P}}(\mathpzc{M})$ allows us to give a criterion for checking admissibility of operads. Indeed, we have the following.

\begin{prop}
    Suppose that there is a class $\mathcal{S}$ of weak equivalences in $\mathpzc{M}$ that is closed under pushouts and transfinite compositions, and contains all morphisms of the form $\mathfrak{P}_{X}(n)\otimes_{\Sigma_{n}} Q^{n}_{n-1}(s)\rightarrow\mathfrak{P}_{X}(n)\otimes_{\Sigma_{n}} B^{\otimes n}$, where $s:A\rightarrow B$ is an acyclic cofibration in $\mathpzc{M}$, and $X$ is any $\mathfrak{P}$-algebra. Then the operad $\mathfrak{P}$ is admissible.
\end{prop}

The result above is essentially a rewording of \cite{white2018bousfield} Theorem 6.1.1.

\begin{cor}\label{cor:ppadmiss}
    If $\mathpzc{M}$ satisfies the pp-monoid axiom, then any non-symmetric operad $\mathfrak{P}$ in $\mathpzc{M}$ is admissible. 
\end{cor}

\begin{proof}
  By tensoring with $\Sigma$, we may regard $\mathfrak{P}$ as a symmetric operad.  Let $s$ be an acyclic cofibration. Using an inductive argument, the proof of \cite{harper2010homotopy} Proposition 7.23 shows that for any $t$ the morphism $Q^{t}_{t-1}(s)\rightarrow B^{\otimes t}$ is an iterated pushout-product of acyclic cofibrations. Another inductive argument also proves this for $Q^{t}_{q-1}(s)\rightarrow Q^{t}_{q}(s)$ for $q<t$. In the diagram defining each step of the standard filtration we may remove the $\Sigma_{n}$ by freeness. Then we can apply the pp-monoid axiom. 
\end{proof}

Admissibility of symmetric operads is a thornier issue. We have the following useful trick for categories enriched over $\mathbb{Q}$ (i.e., as a $1$-category $\mathpzc{M}$ is enriched over $\mathbb{Q}$-vector spaces). It follows immediately from the fact that for any right $\Sigma_{n}$-module $X$ and any morphism $f$ of left $\Sigma_{n}$-modules, $X\otimes_{\Sigma_{n}}f$ is a retract of $X\otimes f$.

\begin{cor}
Let $\mathpzc{M}$ be a combinatorial model category that is enriched over $\mathbb{Q}$ and is also a symmetric monoidal category. Suppose that there is a class $\mathcal{S}$ of weak equivalences in $\mathpzc{M}$ that is closed under pushouts, transfinite composition, and retracts, and contains all morphisms of the form $\mathfrak{P}_{X}(n)\otimes Q^{n}_{n-1}(s)\rightarrow  \mathfrak{P}_{X}(n)\otimes B^{\otimes n}$, where $s:A\rightarrow B$ is an acyclic cofibration in $\mathpzc{M}$, and $X$ is any $\mathfrak{P}$-algebra. Then $\mathfrak{P}$ is admissible. In particular, if $\mathpzc{M}$ satisfies the pp-monoid axiom then, by the proof of Corollary \ref{cor:ppadmiss}, any operad is admissible. 
\end{cor}

\begin{example}
     Suppose that an additive model category
     $\mathpzc{M}$ has a generating set of acyclic cofibrations of the form $0\rightarrow C$, where $C$ is trivially cofibrant. Suppose further that $C^{\otimes n}$ is stably acyclic as a $\Sigma_{n}$-module, that is, $X\otimes_{\Sigma_{n}}C^{\otimes n}$ is acyclic for any right $\Sigma_{n}$-module $X$ (stable acyclicity holds, for example, if $\mathpzc{M}$ is enriched over $\mathbb{Q}$ and the tensor product of two trivially cofibrant objects is trivially cofibrant). Finally, suppose that a countable direct sum of trivial objects is trivial. Then any operad is admissible. Consider a pushout of the morphism $\mathfrak{P}(0)\rightarrow\mathfrak{P}(C)$ along a morphism $\mathfrak{P}(0)\rightarrow X$. 
 Note that for $0<q<t$ we have $Q^{t}_{q}(s)\cong0$, and for $q=t$ we have $Q^{t}_{t}(s)\cong C^{\otimes t}$. Thus, $X_{n}\cong X_{n-1}\oplus\mathfrak{P}_{X}(n)\otimes_{\Sigma_{n}}C^{\otimes n}$. Hence, the pushout $P$ is the direct sum of the $\mathfrak{P}_{X}(n)\otimes_{\Sigma_{n}}C^{\otimes n}$, and is itself acyclic.  Removing the $\Sigma_{n}$ actions, the same claim holds for non-symmetric operads. Let us give two precise examples here.
\begin{enumerate}
    \item 
    Consider a (possibly one-sided) exact category  $\mathpzc{E}$ with enough projectives, and in which direct sums are exact. Consider the category $s\mathpzc{E}$ of simplicial objects in $\mathpzc{E}$, equipped with the projective model structure of \cite{kelly2016homotopy} Theorem 4.76.
    The category $s\mathpzc{E}$ has a generating set of acyclic cofibrations given by $\{0\rightarrow P[U]\}$, where $U$ is a contractible simplicial set. In particular, $U$ has an augmentation $U_{-1}$, and extra degeneracies $s^{n}_{0}$ for $n\ge -1$. Hence, $P[U]$ also has an augmentation and extra degeneracies. Note that for $0<q<t$ we have $Q^{t}_{q}(s)\cong0$, and for $q=t$ we have $Q^{t}_{t}(s)\cong P^{\otimes t}[U^{\times t}]$. The object $P^{\otimes t}[U^{\times t}]$ has an augmentation $P^{\otimes t}[U_{-1}^{\times t}]$, and also extra degeneracies as a $\Sigma_{t}$-module. This makes $P^{\otimes t}[U^{\times t}]$ contractible as a $\Sigma_{t}$-module in $s\mathpzc{E}$. Then $\mathfrak{P}_{X}(t)\otimes_{\Sigma_{t}}P^{\otimes t}[U^{\times t}] $ is also contractible. Note that for this proof we do not need to use that $s\mathpzc{E}$ has a cocommutative Hopf interval, nor do we require the existence of a symmetric monoidal fibrant replacement functor.
\item 
In this example the utility is that we do not require our category to satisfy the pushout-product axiom. This is of benefit even in the non-symmetric case. Let $\mathpzc{E}$ be a locally presentable, closed symmetric monoidal, exact category with exact countable direct sums. Suppose further that it has a generating set of projectives. We equip $ Ch_{\le0}(\mathpzc{E})$ with the projective model structure furnished by \cite{kelly2016homotopy} Theorem 4.67. This is not a monoidal model category due to the existence of cofibrations of the form $S^{0}(P)\rightarrow 0$ for $P$ a projective. Nonetheless, every non-symmetric operad is admissible, and over $\mathbb{Q}$ every operad is admissible. Indeed, the generating acyclic cofibrations are still of the form $0\rightarrow D^{n}(P)$ for $n\le0$ and $P$ projective. The tensor product of finitely many such objects is still acyclic. 
\end{enumerate}
\end{example}

We finish with a result describing properties of underlying objects of cofibrant $\mathfrak{P}$-algebras. 
\begin{prop}\label{underlyingKflat}

Let $\mathpzc{M}$ be a combinatorial monoidal model category. Let $\mathfrak{P}$ be an admissible operad in $\mathpzc{M}$. 
\begin{enumerate}
\item
Let $\mathcal{S}$ be a class of left proper morphisms containing cofibrations, closed under pushout-products, closed under tensoring with objects of $\mathpzc{M}$, and such that $\mathpzc{M}$ is weakly $\mathcal{S}$-elementary. Let $\mathfrak{P}$ be an operad such that, whenever the underlying object of an algebra $X$ is $K$-transverse to some class of morphisms $\mathcal{T}$, then $\mathfrak{P}_{X}(n)$ is a retract of a free $\Sigma_{n}$-module on an object  of $\mathpzc{M}$ that is $K$-transverse to $\mathcal{T}$. Suppose that $f:X\rightarrow Y$ is a morphism of $\mathfrak{P}$-algebras that is obtained as a transfinite composition of pushouts of morphisms of the form $\mathfrak{P}(g):\mathfrak{P}(A)\rightarrow\mathfrak{P}(B)$, where $g\in\mathcal{S}$, and both $A$ and $B$ are $K$-transverse to $\mathcal{T}$.
If the underlying object of $X$ is $K$-transverse to $\mathcal{T}$, then so is the underlying object of $Y$. 
\item
Let $\mathfrak{P}$ be such that $\mathfrak{P}_{X}$ is $\Sigma$-cofibrant whenever $X$ is cofibrant as an object of $\mathpzc{M}$. If $f:X\rightarrow Y$ is a cofibration of $\mathfrak{P}$-algebras with $X$ cofibrant as an object of $\mathpzc{M}$, then $f$ is a cofibration as a morphism in $\mathpzc{M}$.
\item
Let $\mathfrak{P}$ be such that $\mathfrak{P}_{X}$ is $\Sigma$-cofibrant whenever $X$ is a cofibrant $\mathfrak{P}$-algebra. If $f:X\rightarrow Y$ is a cofibration of $\mathfrak{P}$-algebras with $X$ cofibrant as a $\mathfrak{P}$-algebra, then $f$ is a cofibration as a morphism in $\mathpzc{M}$.
 \end{enumerate}
\end{prop}

\begin{proof}
    \begin{enumerate}
        \item 
        The proof of this part follows from Lemma \ref{Lem:pushouttransK-flat}, Proposition \ref{prop:pushoutalg}, and the proof of Corollary \ref{cor:ppadmiss}.
        \item 
        This is similar to part (1).
        \item
             This is similar to part (1).
    \end{enumerate}
\end{proof}

\begin{rem}
Parts (2) and (3) of Proposition \ref{underlyingKflat} are essentially \cite{carmona2025enveloping} Proposition 4.3.
\end{rem}

\begin{example}\label{ex:appofabove}
    
    \begin{enumerate}
               \item 
    Let $\mathpzc{M}$ be a combinatorial weak $h^{\otimes}$-monoidal model category that is weakly elementary and enriched over $\mathbb{Q}$. Suppose further that there is a generating set of cofibrations with $K$-flat domain and codomain. Consider the commutative operad $\mathfrak{Comm}$. We have $\mathfrak{Comm}_{X}(n)=X$ with the trivial $\Sigma_{n}$-action. If $X$ is $K$-flat, then $X$ is a retract of the $K$-flat $\Sigma_{n}$-module $\Sigma_{n}\otimes X$. Thus, the conditions of $(1)$ are satisfied when $\mathcal{S}$ is the class of cofibrations. 
        \item 
           Suppose that $\mathpzc{M}$ is a monoidal model category. If $\mathfrak{P}$ is cofibrant, then the conditions of $(2)$ are satisfied by \cite{carmona2025enveloping} Proposition 3.34. If $\mathfrak{P}$ is $\Sigma$-cofibrant, then the conditions of $(3)$ are satisfied by \cite{berger2008derived} Proposition 2.3.
    \end{enumerate}
 
\end{example}

\subsection{Quasi-Categories and Localisation of Relative Categories}
As we shall see, Koszul duality can be cleanly formulated using $(\infty,1)$-categories rather than model categories. In particular, this setup will be useful when showing that a certain functor has an $(\infty,1)$-categorical right adjoint using the $(\infty,1)$-adjoint functor theorem. However, the proof of Koszul duality uses $1$-categorical homotopy theory techniques, specifically relative categories. Hence it is fruitful to switch between the two perspectives.  For concreteness, we will fix quasi-categories as our model for $(\infty,1)$-categories.

\begin{defn}
A \textbf{relative category}, or a \textbf{category with weak equivalences}, is a pair $(\mathpzc{C},\mathcal{W})$ where $\mathpzc{C}$ is a category and $\mathcal{W}$ is a wide subcategory of $\mathpzc{C}$ containing all isomorphisms. Moreover, $\mathcal{W}$ is required to satisfy the $2$-out-of-$3$ property: if $f:A\rightarrow B$ and $g:B\rightarrow C$ are morphisms in $\mathpzc{C}$, and two of $f,g,g\circ f$ are in $\mathcal{W}$, then so is the third. 
\end{defn}

There is the stronger notion of a homotopical category.

\begin{defn}\label{defn:hmtpical}
A \textbf{homotopical category} is a pair $(\mathpzc{M},\mathcal{W})$ where $\mathpzc{M}$ is a category, and $\mathcal{W}$ is a class of morphisms in $\mathpzc{M}$ called \textit{weak equivalences} containing all identity morphisms and satisfying the 2-out-of-6 property: for morphisms $f:W\rightarrow X$, $g:X\rightarrow Y$, $h:Y\rightarrow Z$, the morphisms $h\circ g$ and $g\circ f$ are weak equivalences if and only if $f,g,h$ and $h\circ g\circ f$ are. 
\end{defn}

%\begin{rem}
%If $\mathpzc{M}$ is a homotopical category, and $\mathcal{W}$ is the wide subcategory of weak equivalences in $\mathpzc{M}$, then $(\mathpzc{M},\mathcal{W})$ is a relative category.
%\end{rem}

To a relative category $(\mathpzc{C},\mathcal{W})$ we can associate a quasi-category $\mathbf{C}$ by Dwyer-Kan localisation \cite{dwyer1980simplicial}. If $\mathpzc{C}$ is a combinatorial simplicial model category, and $\mathcal{W}$ is its subcategory of weak equivalences, then $\mathbf{C}$ is a locally presentable $(\infty,1)$-category by Proposition A.3.7.6 in \cite{lurie2006higher}.

\begin{defn}
If $(\mathpzc{C},\mathcal{W})$ and $(\mathpzc{C}',\mathcal{W}')$ are relative categories, then a functor $F:\mathpzc{C}\rightarrow\mathpzc{C}'$ is said to be a \textbf{relative functor} if $F(f)\in\mathcal{W}'$ whenever $f\in\mathcal{W}$. A \textbf{relative adjunction} between relative categories is an adjunction
$$\adj{F}{\mathpzc{C}}{\mathpzc{C}'}{G},$$
such that $F$ and $G$ are both relative functors. A relative adjunction is said to be a \textbf{relative equivalence} if the unit and counit of the adjunction are both component-wise weak equivalences.
\end{defn}

Note that relative adjunctions do not always induce adjunctions of $(\infty,1)$-categories between the Dwyer-Kan localisations. However, Corollary 3.6 in \cite{dwyer1980calculating} 
 says the following.

\begin{thm}\label{prerelequivuc}
Let
$$\adj{F}{\mathpzc{C}}{\mathpzc{C}'}{G}$$ 
be a relative equivalence. Then there is an induced adjoint equivalence of $(\infty,1)$-categories
$$\adj{\textbf{F}}{\mathbf{C}}{\mathbf{C}'}{\textbf{G}}.$$ 
\end{thm}

We will use this result in the following form.

\begin{cor}\label{relequivuc2}

    Let $\mathpzc{C}$ and $\mathpzc{D}$ be relative categories, and let
$$\adj{F}{\mathpzc{C}}{\mathpzc{D}}{G}$$ 
be an adjunction between the underlying categories. Suppose that there are inclusions of subcategories $i:\widetilde{\mathpzc{D}}\rightarrow\mathpzc{D}$ and $j:\widetilde{\mathpzc{C}}\rightarrow\mathpzc{C}$ such that 
\begin{enumerate}
    \item 
    there is functor $Q:\mathpzc{D}\rightarrow\widetilde{\mathpzc{D}}$ that preserves weak equivalences, and is equipped with a natural transformation $\gamma:i\circ Q\rightarrow Id_{\mathpzc{D}}$ that is component-wise a weak equivalence,
    \item 
    for $c\in\widetilde{\mathpzc{C}}$, $F(c)\in\widetilde{\mathpzc{D}}$,
    \item 
    for $d\in\widetilde{\mathpzc{D}}$, $G(d)\in\widetilde{\mathpzc{C}}$,
    \item 
    the functor $F$ preserves weak equivalences between objects in $\widetilde{\mathpzc{C}}$,
    \item 
    the functor $G$ preserves weak equivalences between objects in $\widetilde{\mathpzc{D}}$,
    \item 
    for all $d\in\widetilde{\mathpzc{D}}$,  the counit $\epsilon_{d}:FG(d)\rightarrow d$ is an equivalence, and 
    \item 
    for all $c\in\widetilde{\mathpzc{C}}$, the unit morphism
    $\eta_{c}:c\rightarrow GF(c)$ is a weak equivalence.
\end{enumerate}

Let $\widetilde{\mathbf{D}}$ denote the localisation of the relative category that is obtained as the full relative subcategory of $\mathpzc{D}$ consisting of objects of $\widetilde{\mathpzc{D}}$. Similarly, let $\widetilde{\mathbf{C}}$ denote the localisation of the relative category that is obtained as the full relative subcategory of $\mathpzc{C}$ consisting of objects in $\widetilde{\mathpzc{C}}$.
Then the following hold.
\begin{enumerate}
\item 
The natural functor $\textbf{i}:\widetilde{\mathbf{D}}\rightarrow\mathbf{D}$ is an equivalence.
    \item 
   The adjunction
   $$\adj{F}{\widetilde{\mathpzc{C}}}{\widetilde{\mathpzc{D}}}{G}$$
   is a relative equivalence.
\end{enumerate}
\end{cor}

\begin{proof}
\begin{enumerate}
    \item 
    Since the functors $i:\widetilde{\mathpzc{D}}\rightarrow\mathpzc{D}$ and $Q:\mathpzc{D}\rightarrow\widetilde{\mathpzc{D}}$ preserve weak equivalences, we get induced functors 
  $$\textbf{i}:\widetilde{\mathbf{D}}\rightarrow\mathbf{D}$$
   and
    $$\textbf{Q}:\mathbf{D}\rightarrow\widetilde{\mathbf{D}}.$$
By \cite{dwyer1980calculating} Proposition 3.5, the natural transformation $\gamma$ induces a natural equivalence $\textbf{i}\circ\textbf{Q}\rightarrow Id_{\mathbf{D}}$. By composing on the right with $i$, we also get a natural transformation that is an object-wise weak equivalence $\gamma\circ i:i\circ Q\circ i\rightarrow i$ of functors $\widetilde{\mathpzc{D}}\rightarrow\mathpzc{D}$. Since $i$ is fully faithful and reflects weak equivalences, this gives a natural transformation that is an object-wise weak equivalence $\tilde{\gamma}:Q\circ i\rightarrow Id_{\widetilde{\mathpzc{D}}}$. Thus, we again get a natural equivalence $\textbf{Q}\circ\textbf{i}\rightarrow Id_{\widetilde{\mathbf{D}}}$. 
   \item 
   Clearly the adjunction
   $$\adj{F}{\mathpzc{C}}{\mathpzc{D}}{G}$$ 
   restricts to a relative equivalence
   $$\adj{F}{\widetilde{\mathpzc{C}}}{\widetilde{\mathpzc{D}}}{G}.$$ 
\end{enumerate}
\end{proof}

Let us give a further corollary.

\begin{cor}\label{relequivuc}

    Let $\mathpzc{C}$ and $\mathpzc{D}$ be relative categories, and let
$$\adj{F}{\mathpzc{C}}{\mathpzc{D}}{G}$$ 
be an adjunction between the underlying categories. Suppose that $F$ both preserves and reflects weak equivalences. Let $\widetilde{\mathpzc{D}}\subset\mathpzc{D}$ be a full subcategory such that 
\begin{enumerate}
    \item 
    there is a functor $Q:\mathpzc{D}\rightarrow\widetilde{\mathpzc{D}}$ that preserves weak equivalences, and is equipped with a natural transformation $\gamma:i\circ Q\rightarrow Id_{\mathpzc{D}}$ that is component-wise a weak equivalence, and
    \item 
    for all $d\in\widetilde{\mathpzc{D}}$ and all $n\in\mathbb{N}_{0}$, the morphism $(FG)^{n}(\epsilon_{d})$ is an equivalence, where $\epsilon$ is the counit.
\end{enumerate}
Let $\widetilde{\mathbf{D}}$ denote the localisation of the relative category that is obtained as the full relative subcategory of $\mathpzc{D}$ consisting of objects of $\widetilde{\mathpzc{D}}$. Similarly, let $\textbf{G}(\widetilde{\mathbf{D}})$ denote the localisation of the relative category that is obtained as the full relative subcategory of $\mathpzc{C}$ consisting of objects in the essential image of $G:\widetilde{\mathpzc{D}}\rightarrow\mathpzc{C}$.
Then 
\begin{enumerate}
\item 
the natural functor $\textbf{i}:\widetilde{\mathbf{D}}\rightarrow\mathbf{D}$ is an equivalence, and
    \item 
the functor $\textbf{G}:\widetilde{\mathbf{D}}\rightarrow\textbf{G}(\widetilde{\mathbf{D}})$ is an equivalence.
\end{enumerate}
\end{cor}

\begin{proof}
   Let $\widetilde{\mathpzc{C}}$ denote the full subcategory of $\mathpzc{C}$ consisting of objects of the form $G(FG)^{n}(d)$ for $d\in\widetilde{\mathpzc{D}}$ and $n\in\mathbb{N}_{0}$. Let $\widetilde{\mathpzc{D}}'$ denote the full subcategory of $\widetilde{\mathpzc{D}}$ consisting of objects of the form $(FG)^{n}(d)$ for $d\in\widetilde{\mathpzc{D}}$ and $n\in\mathbb{N}_{0}$.
    
    All that remains to check is that $F$ preserves weak equivalences between such objects, and that for $c\in\widetilde{\mathpzc{C}}$, the unit $\eta_{c}$ is an equivalence. Note that $Q$ restricts to a replacement functor
    $$Q:\widetilde{\mathpzc{D}}'\rightarrow\widetilde{\mathpzc{D}}$$
    so that $\widetilde{\mathbf{D}}'\cong\widetilde{\mathbf{D}}$. Hence, we may replace $\widetilde{\mathpzc{D}}$ by $\widetilde{\mathpzc{D}}'$. Now we need to verify the assumptions of Corollary \ref{relequivuc2}. By definition of $\widetilde{\mathpzc{C}}$ and $\widetilde{\mathpzc{D}}'$, assumptions $(2)$ and $(3)$ are satisfied. Condition $(4)$ follows from the fact that $F$ is assumed to preserve all weak equivalences.

We observe that our assumptions and the unit-counit equations imply that each morphism
 $$\epsilon_{(FG)^{n}(d)}=(FG)^{n}(\epsilon_{d}):(FG)^{n+1}d\rightarrow (FG)^{n}d$$
 is an equivalence for $n\ge0$. Thus, condition $(6)$ is satisfied.
  
  Now, let $A\cong(FG)^{m}(d_{1})$ and $B\cong (FG)^{n}(d_{2})$ with $d_{1},d_{2}\in\widetilde{\mathpzc{D}}$, and let $f:A\rightarrow B$ be a weak equivalence. We have a commutative diagram
   \begin{displaymath}
      \xymatrix{
      FG(A)\ar[d]^{\epsilon_{A}}\ar[r]^{FG(f)} & FG(B)\ar[d]^{\epsilon_{B}}\\
       A\ar[r]^{FG(f)} & B.
       }
   \end{displaymath}
   The vertical morphisms are equivalences, and the top horizontal morphism is an equivalence. Hence, the bottom morphism is also an equivalence. Since $F$ reflects weak equivalences, this precisely means that $G(f)$ is a weak equivalence in $\mathpzc{C}$. Hence, $G$ preserves equivalences between objects of the form $(FG)^{n}(d)$ for $d\in\widetilde{\mathpzc{D}}$. This gives condition $(5)$. 
   
   Next, consider the unit morphism $\eta_{G(FG)^{n}(d)}:GFG(FG)^{n}(d)\rightarrow G(FG)^{n}(d)$. The unit-counit equations and the 2-out-of-3 property imply that $F(\eta_{G(FG)^{n}(d)})$ is an equivalence. Hence $\eta_{G(FG)^{n}(d)}$ is an equivalence, and condition $(7)$ is satisfied. This concludes the proof.
\end{proof}

Often in our applications of these two corollaries, $\mathpzc{D}$ will be a model category, and $\widetilde{\mathpzc{D}}$ will be the subcategory of cofibrant objects.

\subsubsection{The Quasi-Category of Algebras}\label{subsubsec:QACAlg}

Let $\mathpzc{M}$ be a combinatorial \textit{monoidal} model category satisfying the monoid axiom, and $\mathfrak{P}$ an admissible operad in $\mathpzc{M}$. The category $\mathpzc{Alg}_{\mathfrak{P}}(\mathpzc{M})$ is a (combinatorial) model category when equipped with the transferred model structure. We denote the associated $(\infty,1)$-category by $\textbf{Alg}_{\mathfrak{P}}(\mathpzc{M})$.

We get an adjunction of $(\infty,1)$-categories
 $$\adj{\mathrm{Free}_{\mathfrak{P}}}{\mathrm{L^{H}}(\mathpzc{M})}{\textbf{Alg}_{\mathfrak{P}}(\mathpzc{M})}{|-|}.$$

Consider the $(\infty,1)$-monad $\textbf{T}_{\mathfrak{P}}=|-|\circ\mathrm{Free}_{\mathfrak{P}}:\mathrm{L^{H}}(\mathpzc{M})\rightarrow \mathrm{L^{H}}(\mathpzc{M})$ constructed in \cite{higheralgebra} Section 4.7, and the $(\infty,1)$-category $\textbf{Alg}_{\textbf{T}_{\mathfrak{P}}}(\mathrm{L^{H}}(\mathpzc{M}))$ of algebras over this monad. There is a factorisation of the forgetful functor $|-|:\textbf{Alg}_{\mathfrak{P}}(\mathpzc{M})\rightarrow\mathrm{L^{H}}(\mathpzc{M})$ as
$$\textbf{Alg}_{\mathfrak{P}}(\mathpzc{M})\rightarrow\textbf{Alg}_{\textbf{T}_{\mathfrak{P}}}(\mathrm{L^{H}}(\mathpzc{M}))\rightarrow\mathrm{L^{H}}(\mathpzc{M})$$
(again, by \cite{higheralgebra} Section 4.7).

\begin{defn}

    An admissible operad $\mathfrak{P}$ is said to be \textbf{weakly rectifiable} if the functor $\textbf{Alg}_{\mathfrak{P}}(\mathpzc{M})\rightarrow\textbf{Alg}_{\textbf{T}_{\mathfrak{P}}}(\mathrm{L^{H}}(\mathpzc{M}))$ is an equivalence of $(\infty,1)$-categories.
\end{defn}

Under some mild additional assumptions (it suffices that $\mathpzc{M}$ is a monoidal model category, and $\mathfrak{P}$ is level-wise cofibrant) we may also regard $\mathfrak{P}$ as an $(\infty,1)$-operad on the $(\infty,1)$-category $\mathrm{L^{H}}(\mathpzc{M})$. To this one can associate the $(\infty,1)$-category of $(\infty,1)$-algebras in $\mathrm{L^{H}}(\mathpzc{M})$ over the $(\infty,1)$-operad $\mathfrak{P}$. This category is denoted $\textbf{Alg}_{\mathfrak{P}}(\mathrm{L^{H}}(\mathpzc{M}))$. For general model categories it is not obvious when the category $\mathpzc{Alg}_{\mathfrak{P}}(\mathpzc{M})$  presents $\textbf{Alg}_{\mathfrak{P}}(\mathrm{L^{H}}(\mathpzc{M}))$, i.e., when the $(\infty,1)$-categories $\mathrm{L^{H}}(\mathpzc{Alg}_{\mathfrak{P}}(\mathpzc{M}))$ and $\textbf{Alg}_{\mathfrak{P}}(\mathpzc{M})$ are equivalent. There is always a natural morphism
$$\textbf{Alg}_{\mathfrak{P}}(\mathpzc{M})\rightarrow\textbf{Alg}_{\mathfrak{P}}(\mathrm{L^{H}}(\mathpzc{M})).$$

\begin{defn}

An operad $\mathfrak{P}$ is said to be \textbf{rectifiably admissible} if:
\begin{enumerate}
    \item 
    it is admissible,
    \item 
    it presents an $(\infty,1)$-operad,
    \item 
    the adjunction of $(\infty,1)$-categories
    $$\adj{\mathrm{Free}_{\mathfrak{P}}}{\mathrm{L^{H}}(\mathpzc{M})}{\textbf{Alg}_{\mathfrak{P}}(\mathpzc{M})}{|-|}$$
    is monadic, and
    \item
    the morphism $\textbf{Alg}_{\mathfrak{P}}(\mathpzc{M})\rightarrow\textbf{Alg}_{\mathfrak{P}}(\mathrm{L^{H}}(\mathpzc{M}))$ is an equivalence of $(\infty,1)$-categories.
\end{enumerate}
%, any $n\in\mathbb{N}_{0}$, and any cofibration $s\in\mathpzc{M}$, the morphism $q(n)\Box_{\Sigma_{n}}s^{\Box n}$ is a weak equivalence. 
\end{defn}

In particular, rectifiably admissible operads are weakly rectifiable. In \cite{MR4721371} White and Yau give very general conditions under which an operad is rectifiably admissible. Recall the definition of a subcategory of flat objects (Remark \ref{rem:flatcat}).

\begin{thm}[\cite{MR4721371} Theorem 7.3.1]
Suppose that $\mathpzc{M}$ is a cofibrantly generated monoidal model category that admits a subcategory of flat objects $\mathpzc{M}^{\flat}$, and $\mathfrak{P}$ is a $\Sigma$-cofibrant symmetric operad valued in $\mathpzc{M}^{\flat}$. Then the natural morphism
$$\mathrm{L^{H}}(\mathpzc{Alg}_{\mathfrak{P}}(\mathpzc{M}))\rightarrow\textbf{Alg}_{\mathfrak{P}}(\mathpzc{M})$$
is an equivalence of $(\infty,1)$-categories.
\end{thm}

The proof in \cite{MR4721371} in fact works for coloured operads, but we only need the $1$-coloured version in this work. We will typically be in the situation where the unit of $\mathpzc{M}$ is flat, so we can apply the theorem above using the subcategory of cofibrant objects by Remark \ref{rem:flatcat}.

\begin{example}

   Consider the model category $\mathpzc{M}=s\mathpzc{Ab}$ of simplicial abelian groups, equipped with the projective model structure. This is a monoidal model category satisfying the monoid axiom. The commutative operad $\mathfrak{Comm}$ is weakly rectifiable and admissible, but not rectifiably admissible. In fact, the $(\infty,1)$-category presented by $\mathpzc{Alg}_{\mathfrak{Comm}}(s\mathpzc{Ab})$ is the category of \textbf{connective derived commutative rings} introduced (in general settings) in \cite{raksit2020hochschild} Section 4. On the other hand, the category of commutative monoid objects in $\mathrm{L^{H}}(s\mathpzc{Ab})$ is the category of (connective) $E_{\infty}$-monoid objects.
\end{example}

\subsubsection{Generation under Homotopy Sifted Colimits}\label{siftedgeneration}
The technical subtleties of the above discussion actually have a crucial consequence; they provide us with a very convenient set of generators for the $(\infty,1)$-category $\textbf{Alg}_{\mathfrak{P}}(\mathrm{L^{H}}(\mathpzc{M}))$, which we will now discuss.

Let $\mathbf{C}$ be a cocomplete $(\infty,1)$-category and $\mathbf{C}_{0}$ a full subcategory. We denote by $\mathcal{P}_{\Sigma}(\mathbf{C}_{0})$ the free cocompletion of $\mathbf{C}_{0}$ by homotopy sifted colimits, i.e., by filtered colimits and geometric realisations. There is a natural functor
$$\mathcal{P}_{\Sigma}(\mathbf{C}_{0})\rightarrow\mathbf{C}.$$
Let $\textbf{T}$ be a monad on $\mathbf{C}$ that preserves homotopy sifted colimits, and let $\mathbf{C}^{\textbf{T}}$ be its category of Eilenberg-Moore algebras. Consider the corresponding adjunction.
$$\adj{\mathrm{Free}_{\textbf{T}}}{\mathbf{C}}{\mathbf{C}^{\textbf{T}}}{|-|_{\textbf{T}}}.$$
 Let $\mathrm{Free}_{\textbf{T}}(\mathbf{C}_{0})$ denote the full subcategory of $\mathbf{C}^{\textbf{T}}$ spanned by free $\textbf{T}$-algebras on $\mathbf{C}_{0}$. Consider the diagram
 \begin{displaymath}
 \xymatrix{
 \mathcal{P}_{\Sigma}(\mathbf{C}_{0})\ar[d]\ar[r] & \mathcal{P}_{\Sigma}(\mathrm{Free}_{\textbf{T}}(\mathbf{C}_{0}))\ar[d]\\
 \mathbf{C}\ar[r] & \mathbf{C}^{\textbf{T}}.
 }
 \end{displaymath}
 Suppose that the left-hand vertical morphism is essentially surjective. By Proposition 4.7.3.14 in \cite{higheralgebra}, every object in $\mathbf{C}^{\textbf{T}}$ can be obtained as a colimit of a simplicial diagram of objects in the image of $\mathrm{Free}_{\textbf{T}}$. Therefore, by the discussion above, every object in $\mathbf{C}^{\textbf{T}}$ can be obtained as a homotopy sifted colimit of objects in $\mathrm{Free}_{\textbf{T}}(\mathbf{C}_{0})$. In particular, the functor  $\mathcal{P}_{\Sigma}(\mathrm{Free}_{\textbf{T}}(\mathbf{C}_{0}))\rightarrow\mathbf{C}^{\textbf{T}}$
is essentially surjective.

We are particularly interested in the case where $\mathbf{C}$ is presented by a Kan complex-enriched or dg-enriched monoidal model category $\mathpzc{C}$, and $\mathfrak{P}$ is a weakly rectifiable admissible operad in $\mathpzc{C}$. Since $\mathpzc{C}$ is cofibrantly generated, $\mathbf{C}$ is generated under homotopy sifted colimits by some small subcategory $\mathbf{C}_{0}$ of cofibrant objects in $\mathpzc{C}$. Therefore, $\textbf{Alg}_{\mathfrak{P}}(\mathpzc{C})$ is generated under homotopy sifted colimits by the full subcategory of free $\mathfrak{P}$-algebras on objects in $\mathbf{C}_{0}$.

\begin{rem}
The argument given above is a significant component of the one in \cite{hennion2015tangent} Proposition 1.2.2, which shows that the category of chain complexes of vector spaces over a field and the category of Lie algebras over it are in fact the formal completions of subcategories of certain compact objects by homotopy sifted colimits.
\end{rem}
\subsubsection{Presentation by Subcategories}

Let $\mathfrak{P}$ be an admissible operad. For technical reasons in Section \ref{seccoopkosz} it will be useful to consider certain relative subcategories of $\mathpzc{Alg}_{\mathfrak{P}}(\mathpzc{M})$ that present the same $(\infty,1)$-category.

\begin{prop}\label{infinitysame}
Let $\mathpzc{M}$ be a combinatorial model category with a symmetric monoidal model structure. Let $\mathfrak{P}$ be an admissible operad in $\mathpzc{M}$, and let $\mathpzc{J}\subset\mathpzc{Alg}_{\mathfrak{P}}(\mathpzc{M})$ be a full subcategory containing all cofibrant objects, which we regard as a relative subcategory. Then the inclusion
$$\mathpzc{J}\rightarrow\mathpzc{Alg}_{\mathfrak{P}}(\mathpzc{M})$$
induces an equivalence of $(\infty,1)$-categories.
\end{prop}

\begin{proof}
This follows from an easy argument dual to the proof of Corollary 4.5 in \cite{dwyer1980calculating} (as in Corollary \ref{relequivuc2} Part (1)).
\end{proof}

 Let $\mathpzc{Alg}_{\mathfrak{P}}^{|c|}(\mathpzc{M})$ denote the full subcategory of $\mathpzc{Alg}_{\mathfrak{P}}(\mathpzc{M})$ consisting of algebras whose underlying object in $\mathpzc{M}$ is cofibrant. Let $\mathpzc{Alg}_{\mathfrak{P}}^{c}(\mathpzc{M})$ denote the full subcategory of $\mathpzc{Alg}_{\mathfrak{P}}(\mathpzc{M})$ consisting of cofibrant algebras. Finally, let $\mathpzc{Alg}_{\mathfrak{P}}^{|K|}(\mathpzc{M})$ denote the full subcategory of $\mathpzc{Alg}_{\mathfrak{P}}(\mathpzc{M})$ consisting of algebras whose underlying object in $\mathpzc{M}$ is $K$-flat. We regard these categories as relative categories in the obvious way. In what follows, when it does not cause confusion, we will often suppress the underlying category $\mathpzc{M}$ when writing categories of algebras over an operad.

\begin{cor}

\begin{enumerate}
\item
The inclusion
$$\mathpzc{Alg}_{\mathfrak{P}}^{c}\rightarrow\mathpzc{Alg}_{\mathfrak{P}}$$
induces an equivalence of $(\infty,1)$-categories. 
\item
Let $\mathpzc{M}$ be a (unital) monoidal model category, and let $\mathfrak{P}$ be such that $\mathfrak{P}(0)$ is cofibrant, and $\mathfrak{P}_{X}$ is $\Sigma$-cofibrant whenever $X$ is cofibrant as an object of $\mathpzc{M}$. Then any cofibrant $\mathfrak{P}$-algebra $X$ is cofibrant as an object of $\mathpzc{M}$ by Proposition \ref{underlyingKflat}. Consequently, the functor
$$\mathpzc{Alg}_{\mathfrak{P}}^{|c|}\rightarrow\mathpzc{Alg}_{\mathfrak{P}}$$
induces an equivalence of $(\infty,1)$-categories. If $\mathpzc{M}$ is $KC$-monoidal, then the functor
$$\mathpzc{Alg}_{\mathfrak{P}}^{|K|}\rightarrow\mathpzc{Alg}_{\mathfrak{P}}$$
induces an equivalence of $(\infty,1)$-categories. 
\end{enumerate}
\end{cor}

\subsubsection{A Useful Diversion: the Hurewicz Model Structure}

We conclude this section with an example of a model category structure on algebras over an operad in some category of chain complexes that almost, but not quite, fits into the picture above. 
Namely, let $\mathpzc{E}$ be a locally presentable closed symmetric monoidal additive category. We consider $\mathpzc{E}$ as an exact category with the split exact structure. Since every object of $\mathpzc{E}$ is projective, and arbitrary sums of split exact sequences are split exact, \cite{kelly2016homotopy} Theorem 4.59 implies that the projective model structure exists on $Ch(\mathpzc{E})$. This is called the \textit{Hurewicz model structure}, and was established in \cite{MarekGolasiński1982}.
Let $\mathfrak{P}$ be an operad in $Ch(\mathpzc{E})$.
\begin{prop}\label{prop:hurewiczmodel}
Suppose that each $\mathfrak{P}(n)$ is such that, for any $\Sigma_{n}$-module $M$, the natural map 
$$\mathfrak{P}(n)\otimes_{\Sigma_{n}}M\rightarrow\mathfrak{P}(n)\otimes M$$
is an isomorphism.
    Let $A\rightarrow B$ be a morphism in $\mathpzc{Alg}_{\mathfrak{P}}(Ch(\mathpzc{E}))$. There exists a factorisation
    \begin{displaymath}
        \xymatrix{
            A\ar[r]^{i} & \tilde{B}\ar[r]^{p} & B
        }
    \end{displaymath}
  such that
  \begin{enumerate}
      \item 
      the underlying morphism of $p$ is a degree-wise split surjection,
      \item 
      the morphism $i$ has the left lifting property against all morphisms $C\rightarrow D$ of $\mathfrak{P}$-algebras that are degree-wise split surjective as morphisms of underlying complexes, and
       \item 
       $i$ is a homotopy equivalence of underlying complexes.
  \end{enumerate}  Consequently, the right transferred model structure of the Hurewicz model structure exists on $\mathpzc{Alg}_{\mathfrak{P}}(Ch(\mathpzc{E}))$.
\end{prop}
\begin{proof}
    We follow (and generalise) the proof in \cite{hess2017necessary} Proposition 6.2.1, using Quillen's path object argument \cite{quillen2006homotopical} Section II.4. Note that the Hurewicz model structure on $Ch(\mathpzc{E})$ is often not cofibrantly generated. However, it is accessible. That is, there are functorial factorisations, for which the factorisation functors are accessible functors. In the Hurewicz model structure, this is easily verified using the mapping cylinder and mapping cocylinder constructions, as in \cite{hess2017necessary} Example 4.1.1. Moreover, note that all objects of $Ch(\mathpzc{E})$ are fibrant. Consider the contractible complex $cocone(Id_{B})$ and the free $\mathfrak{P}$-algebra, $\mathrm{Free}_{\mathfrak{P}}(cocone(Id_{B}))$. We have the following morphism of $\mathfrak{P}$-algebras
    $$\mathrm{Free}_{\mathfrak{P}}(cocone(Id_{B}))\rightarrow \mathrm{Free}_{\mathfrak{P}}(B)\rightarrow B.$$
   As maps of underlying complexes, these are both split surjections. Consider the induced morphism 
   $$A\coprod \mathrm{Free}_{\mathfrak{P}}(cocone(Id_{B}))\rightarrow B.$$ This is again split surjective as a morphism of complexes. 
   Next, we claim that the canonical morphism $A\rightarrow A\coprod \mathrm{Free}_{\mathfrak{P}}(cocone(Id_{B}))$ is a homotopy equivalence. The standard filtration of Proposition \ref{prop:pushoutalg}, and the proof of \cite{harper2010homotopy} Proposition 7.23, shows that this morphism is the colimit of a countable tower of iterated pushout-products of degree-wise split monomorphisms with contractible cokernel. Such pushout-products are themselves degree-wise split monomorphisms with contractible cokernel. Since the split model structure has enough injectives (indeed, every object is injective), the colimit of a countable tower of such morphisms is again a degree-wise split monomorphism with contractible cokernel.
\end{proof}

\section{Koszul Categories}\label{homotopexact}
We now introduce the class of model categories in which we will be primarily interested: additive model categories in which we have a rich theory of homotopical algebra.

\subsection{Exact Categories}
The model categories we consider will come from additive categories equipped with exact structures. Following \cite{hovey}, \cite{christensen2002quillen}, and \cite{gillespie}, we developed homotopy theory in exact categories in detail in \cite{kelly2016homotopy}. Here we briefly recall some essential definitions from the theory of exact categories. Details about exact categories can be found in \cite{Buehler}. From now on, we shall assume that $\mathpzc{E}$ is weakly idempotent complete.

 A \textbf{kernel-cokernel pair} in $\mathpzc{E}$ is a pair of composable morphisms $(i,p)$, $i:A\rightarrow B$, $p:B\rightarrow C$ such that $i=\textrm{ker}(p)$ and $p=\textrm{coker}(i)$. If $\mathcal{Q}$ is a class of kernel-cokernel pairs and $(i,p)\in\mathcal{Q}$, then we say that $i$ is an \textbf{admissible monomorphism} and $p$ is an \textbf{admissible epimorphism} with respect to $\mathcal{Q}$. Finally, we say that a morphism $f$ is \textbf{admissible} if it can be decomposed as $i\circ p$ where $p$ is an admissible epimorphism and $i$ is an admissible monomorphism.

\begin{defn}
A \textbf{Quillen exact structure} on an additive category $\mathpzc{E}$ is a collection $\mathcal{Q}$ of kernel-cokernel pairs such that:
\begin{enumerate}
\item
isomorphisms are both admissible  monomorphisms and admissible epimorphisms,
\item
both the collection of admissible  monomorphisms and the collection of admissible epimorphisms  are closed under composition,
\item
if
\begin{displaymath}
\xymatrix{
A\ar[d]\ar[r]^{f} & B\ar[d]\\
X\ar[r]^{f'} & Y
}
\end{displaymath}
is a pushout diagram, and $f$ is an admissible  monomorphism, then $f'$ is as well, and
\item
if
\begin{displaymath}
\xymatrix{
A\ar[d]\ar[r]^{f'} & B\ar[d]\\
X\ar[r]^{f} & Y
}
\end{displaymath}
is a pullback diagram, and $f$ is an admissible epimorphism, then $f'$ is as well.
\end{enumerate}
\end{defn}

From now on, we let $\textbf{RegMon}$ denote the class of regular monomorphisms in $\mathpzc{E}$ (i.e., morphisms that are kernels of their cokernels), and $\textbf{AdMon}$ denote the class of admissible monomorphisms. Note that $\textbf{AdMon}\subset\textbf{RegMon}$.

Let $(\mathpzc{E},\mathcal{Q})$ be an exact category. We call a null sequence
\begin{displaymath}
\xymatrix{
0\ar[r] & A\ar[r]^{i} & B\ar[r]^{p} & C\ar[r] & 0
}
\end{displaymath}
\textbf{short exact} if $(i,p)$ is a kernel-cokernel pair in $\mathcal{Q}$. We will use the notion of kernel-cokernel pair and short exact sequence interchangeably. When it is not likely to cause confusion, we will abbreviate $(\mathpzc{E},\mathcal{Q})$ to $\mathpzc{E}$. When studying exact categories, it is natural to consider so-called exact functors:

\begin{defn}
Let $(\mathpzc{E},\mathcal{Q})$ and  $(\mathpzc{F},\mathcal{R})$ be exact categories. A functor $F:\mathpzc{E}\rightarrow\mathpzc{F}$ is said to be \textbf{exact} (with respect to $\mathcal{Q}$ and $\mathcal{R}$) if for any short exact sequence
$$0\rightarrow X\rightarrow Y\rightarrow Z\rightarrow 0$$
in $\mathcal{Q}$,
$$0\rightarrow F(X)\rightarrow F(Y)\rightarrow F(Z)\rightarrow 0$$
is a short exact sequence in $\mathcal{R}$.
\end{defn}

\begin{defn}
Let $\mathpzc{E}$ be an exact category. An object $P\in\mathpzc{E}$ is said to be \textbf{projective} if the functor
$$\mathrm{Hom}(P,-):\mathpzc{E}\rightarrow\mathpzc{Ab}$$
is exact.
\end{defn}

%\begin{rem}
%Almost all of the results of this paper can be formulated using the more general notion of a left exact category (see e.g., \cite{bazzoni2013one}). 
%\end{rem}

%\begin{defn}
%Let $(\mathpzc{E},\mathcal{P})$ be an exact category. An \textbf{exact subcategory} of $(\mathpzc{E},\mathcal{P})$ is an exact category $(\mathpzc{F},\mathcal{Q})$ where $\mathpzc{F}$ is a subcategory of $\mathpzc{E}$ and the inclusion  functor is exact.
%\end{defn}

On any additive category one can define the split exact structure for which the kernel-cokernel pairs are the split exact sequences. Any exact category contains this as an exact subcategory. At the other extreme, we have quasi-abelian exact structures.

\begin{defn}\label{quas}
An additive category $\mathpzc{E}$ with all kernels and cokernels is said to be \textbf{quasi-abelian} if the class $\mathpzc{qac}$ of all kernel-cokernel pairs forms an exact structure on $\mathpzc{E}$.
\end{defn}

\begin{rem}
If an additive category has all kernels and cokernels, then it always has a maximal exact structure. This is due to Sieg and Wegner \cite{sieg2011maximal}. However, this maximal exact structure may not consist of \textit{all} kernel-cokernel pairs.
\end{rem}

Let $X\in Ch(\mathpzc{E})$ be a chain complex.

\begin{defn}
$X$ is said to be \textbf{acyclic} if, for each $n\in\mathbb{Z}$, the morphism $d_{n}:X_{n}\rightarrow Z_{n-1}X$ is an admissible epimorphism, where $Z_{n-1}X\defeq ker(d_{n-1}:X_{n-1}\rightarrow X_{n-2})$. 
\end{defn}

\begin{defn}
Let $\mathpzc{E}$ be an exact category. A morphism $f:X\rightarrow Y$ in $Ch(\mathpzc{E})$ is said to be a \textbf{quasi-isomorphism} if $cone(f)$ is acyclic.
\end{defn}

%We are going to make the technical assumptions that $\mathpzc{E}$ is complete and cocomplete, and that for any morphism $f:X\rightarrow Y$ in $\mathpzc{E}$, the morphism $X\rightarrow im(f)$ is an epimorphism (though not necessarily admissible). This is in particular the case when $\mathpzc{E}$ is a quasi-abelian category, or indeed for any exact category whose underlying additive category is quasi-abelian. 

%\section{Model Structures on Categories of Algebras}

\subsection{(Monoidal) Model Structures on Exact Categories}
The class $\mathcal{W}$ of quasi-isomorphisms satisfies the 2-out-of-6 property. Thus, the class $\mathcal{W}$ makes $Ch(\mathpzc{E})$ into a homotopical category. If $\mathpzc{C}$ is a model category and $\mathcal{W}$ is its class of weak equivalences, then $(\mathpzc{C},\mathcal{W})$ is a homotopical category. It will be important for us that a model structure exists on $Ch(\mathpzc{E})$ for which the class of weak equivalences coincides with the class of quasi-isomorphisms.

In \cite{kelly2016homotopy} Section 2.4 we introduced the notion of a monoidal exact category. Essentially, this is an exact category $\mathpzc{E}$ equipped with an additive associative bifunctor $\otimes:\mathpzc{E}\times\mathpzc{E}\rightarrow\mathpzc{E}$ that commutes with colimits in each variable. Closed monoidal exact categories, where the assumption that $\otimes$ commutes with colimits automatically holds, are discussed in \cite{vst2012exact}. Note that if a monoidal exact category $\mathpzc{E}$ is locally presentable, then it is automatically closed monoidal by the adjoint functor theorem.

\begin{defn}
A \textbf{monoidal exact category} is an exact category $\mathpzc{E}$, equipped with a bifunctor $\otimes:\mathpzc{E}\times\mathpzc{E}\rightarrow\mathpzc{E}$, and a unit object $k$, such that $(\mathpzc{E},\otimes,k)$ is a monoidal additive category. 
\end{defn}
In a monoidal exact category we have the notions of flatness and purity.

\begin{defn}
Let $(\mathpzc{E},\otimes,k)$ be a monoidal exact category. An object $X$ of $\mathpzc{E}$ is said to be \textbf{flat} if for any exact sequence 
$$0\rightarrow E\rightarrow F\rightarrow G\rightarrow 0,$$
the sequences 
$$0\rightarrow X\otimes E\rightarrow X\otimes F\rightarrow X\otimes G\rightarrow 0,$$
and
$$0\rightarrow E\otimes X\rightarrow F\otimes X\rightarrow G\otimes X\rightarrow 0$$
are exact. 
\end{defn}

\begin{defn}
Let $(\mathpzc{E},\otimes)$ be a monoidal exact category, and let $\mathpzc{C}$ be a class of objects in $\mathpzc{E}$.
A sequence
$$0\rightarrow E\rightarrow F\rightarrow G\rightarrow 0$$
is said to be $\mathpzc{C}$-\textbf{pure} if for any object $X$ of  $\mathpzc{C}$ the sequences
$$0\rightarrow X\otimes E\rightarrow X\otimes F\rightarrow X\otimes G\rightarrow 0,$$
and
$$0\rightarrow E\otimes X\rightarrow F\otimes X\rightarrow G\otimes X\rightarrow 0$$
are exact. A monomorphism $f:E\rightarrow F$ is said to be $\mathpzc{C}$-\textbf{pure} if it is an admissible monomorphism occurring as the kernel of a $\mathcal{C}$-pure exact sequence. When $\mathpzc{C}=\mathrm{Ob}(\mathpzc{E})$, we write \textbf{pure} instead of $\mathrm{Ob}(\mathpzc{E})$-pure.
 \end{defn}

The class of all pure monomorphisms is denoted \textbf{PureMon}. For example, by \cite{kelly2016homotopy} Lemma 2.76, if $E\rightarrow F$ is an admissible monomorphism with flat cokernel, then it is pure.

It is straightforward to prove the following (see  \cite{kelly2016homotopy} Section 2.4.1). Indeed, it is a consequence of the fact that the class of $\mathcal{C}$-pure exact sequences defines an exact structure on $\mathpzc{E}$ by \cite{kelly2016homotopy} Proposition 2.77.

\begin{prop}\label{3pure}
Let
\begin{displaymath}
\xymatrix{
& 0\ar[d] & 0\ar[d] & 0\ar[d] &\\
0\ar[r] & A\ar[r]\ar[d] &  B\ar[r]\ar[d] & C\ar[r]\ar[d] & 0\\
0\ar[r] & X\ar[r]\ar[d] & Y\ar[r]\ar[d] & Z\ar[r]\ar[d] & 0\\
0\ar[r] & K\ar[r]\ar[d] & L\ar[r]\ar[d] & M\ar[r]\ar[d] & 0\\
& 0 & 0 & 0 &
}
\end{displaymath}
be a diagram in which all columns and two of the rows, or all rows and two of the columns are $\mathcal{C}$-pure exact. Then the remaining row or column is $\mathcal{C}$-pure exact. 
\end{prop}

\begin{cor}\label{cor:exteflatpure}

    Let 
    $$0\rightarrow X\rightarrow Y\rightarrow Z\rightarrow 0$$
    be a pure exact sequence with two of the objects $X,Y,Z$ flat. Then the third object is flat. In particular, the kernel of an admissible epimorphism between flat objects is flat. 
\end{cor}

The following result will be extremely useful.

\begin{prop}\label{prop:purepushoutsum}
    
    Let 
    $$0\rightarrow A\rightarrow B\rightarrow C\rightarrow 0$$
    and
    $$0\rightarrow E\rightarrow F\rightarrow G\rightarrow 0$$
    be pure exact sequences. Then the sequence
    $$0\rightarrow A\otimes F+ B\otimes E\rightarrow B\otimes F\rightarrow C\otimes G\rightarrow 0$$
    is pure exact, where $A\otimes F+ B\otimes E\rightarrow B\otimes F$ is the image of the natural morphism $(A\otimes F)\oplus (B\otimes E)\rightarrow B\otimes F$.
\end{prop}

\begin{proof}
    This is contained in the proof of \cite{kelly2016homotopy} Proposition 5.62 (2). In fact, it is essentially an incarnation of the pushout-product property for pure monomorphisms, as proved in \cite{vst2012exact} Theorem 8.11.
\end{proof}

We record the following immediate corollary, which follows by induction.

\begin{cor}\label{cor:itpushoutprodpure}
    Let $\mathpzc{E}$ be a finitely cocomplete monoidal exact category. Let $f_{1},\ldots,f_{k}$ be a collection of pure monomorphisms. Write $C_{i}=coker(f_{i})$. Then the pushout-product of all of the $f_{i}$ is a pure monomorphism with cokernel $C_{1}\otimes\ldots\otimes C_{k}$.
\end{cor}

\begin{rem}
 
    Let 
    $$0\rightarrow A\otimes F+B\otimes E\rightarrow B\otimes F\rightarrow C\otimes G\rightarrow0$$
    be the exact sequence from the proposition above. Let $Y$ be an arbitrary object of $\mathpzc{E}$. Then, by purity, we have $Y\otimes(A\otimes F+B\otimes E)\cong Y\otimes (A\otimes F)+Y\otimes (B\otimes E)$, and the sequence
       $$0\rightarrow  Y\otimes (A\otimes F)+Y\otimes (B\otimes E)\rightarrow Y\otimes (B\otimes F)\rightarrow Y\otimes (C\otimes G)\rightarrow0$$
       is pure exact.
\end{rem}

We shall require a technical assumption on $\mathpzc{E}$. %Let $\mathcal{I}$ be a filtered category, and let $\mathcal{S}$ be a class of morphisms in $\mathpzc{E}$. Denote by $\mathpzc{Fun}_{\mathcal{S}}^{cocont}(\mathcal{I},\mathpzc{E})$ the category of cocontinuous functors from $F:\mathcal{I}\rightarrow\mathpzc{E}$ such that for any morphism $\alpha$ in $\mathcal{I}$ the morphism $F(\alpha)$ is in $\mathcal{S}$.

\begin{defn}
Let $\mathcal{S}$ be a class of morphisms in an exact category $\mathpzc{E}$ that is closed under finite sums. The category $\mathpzc{E}$ is said to be \textbf{weakly} $\mathcal{S}$-\textbf{elementary} if, for any ordinal $\lambda$, the functor $\colim :\mathpzc{Fun}_{\mathcal{S}}^{cocont}(\lambda,\mathpzc{E})\rightarrow\mathpzc{E}$ exists, and the functor
$$\colim :\mathpzc{Fun}_{ Ch(\mathcal{S})}^{cocont}(\lambda, Ch(\mathpzc{E}))\rightarrow Ch(\mathpzc{E})$$
sends acyclic objects to acyclic objects, where $ Ch(\mathcal{S})$ is the class of morphisms that degree-wise are in $\mathcal{S}$.
\end{defn}

\begin{rem}\label{rem:selementarydouble}

Let $\mathcal{S}$ be a class of morphisms in $\mathpzc{E}$ that is closed under finite sums. Suppose there exists a model structure on $Ch(\mathpzc{E})$ in which the weak equivalences are the quasi-isomorphisms. Then $\mathpzc{E}$ is weakly $\mathcal{S}$-elementary if and only if the model structure on $Ch(\mathpzc{E})$ is weakly $Ch(\mathcal{S})$-elementary, in the sense of Definition \ref{defn:weaklyelmodel}.  
%and the class of cofibrations is contained in $Ch(\mathcal{S})$. Then \cite{kelly2016homotopy} Proposition 4.44/ Corollary 4.45 says that $\mathpzc{E}$ is weakly $\mathcal{S}$-elementary as an exact category precisely if $Ch(\mathpzc{E})$ is weakly $Ch(\mathcal{S})$-elementary as a model category.
\end{rem}

\begin{prop}\label{prop:admonweaklyelementaryws}

    Let $\mathpzc{E}$ be a cocomplete exact category. If $\mathpzc{E}$ is weakly $\textbf{AdMon}$-elementary, then the class $\textbf{AdMon}$ is weakly saturated.
\end{prop}

\begin{proof}
    By definition, the class $\textbf{AdMon}$ is stable under pushout. The Obscure Axiom implies that the class $\textbf{AdMon}$ is closed under retracts. By \cite{kelly2016homotopy} Proposition 2.104, the class $\textbf{AdMon}$ is closed under transfinite compositions.
\end{proof}

\begin{prop}\label{prop:admonwepuremonwe}
 
   Let $\mathpzc{E}$ be a cocomplete closed symmetric monoidal exact category that is weakly $\textbf{PureMon}$-elementary. Then the pure exact structure $\mathpzc{E}_{\otimes}$ is weakly $\textbf{AdMon}$-elementary.
\end{prop}

\begin{proof}
    Let $F\in \mathpzc{Fun}_{\textbf{PureMon}}^{cocont}(\lambda,\mathpzc{E})$ be such that each $F(\alpha)$ is pure acyclic. For any object $X$ of $\mathpzc{E}$, $X\otimes F$ is in $\mathpzc{Fun}_{\textbf{PureMon}}^{cocont}(\lambda,\mathpzc{E})$, and each $X\otimes F(\alpha)$ is pure acyclic. Therefore, $\colim  X\otimes F\cong X\otimes\colim  F$ is acyclic in $\mathpzc{E}$. This implies that $\colim  F$ is pure acyclic.
\end{proof}

\begin{prop}\label{prop:puremontransfer}

   Let $\mathpzc{E}$ be a cocomplete closed symmetric monoidal exact category that is weakly $\textbf{PureMon}$-elementary. Let $R$ be a commutative monoid in $\mathpzc{E}$. Then ${}_{R}\mathpzc{Mod}(\mathpzc{E})$ is weakly $\textbf{PureMon}$-elementary. Moreover, if $g:A\rightarrow B$ is a pure monomorphism in $\mathpzc{E}$, then $Id_{R}\otimes g:R\otimes A\rightarrow R\otimes B$ is a pure monomorphism in ${}_{R}\mathpzc{Mod}(\mathpzc{E})$.
\end{prop}

\begin{proof}
    Observe that the forgetful functor $|-|:{}_{R}\mathpzc{Mod}(\mathpzc{E})\rightarrow \mathpzc{E}$ commutes with colimits, and preserves and reflects exactness. It therefore suffices to prove that if $f:X\rightarrow Y$ is a pure monomorphism in ${}_{R}\mathpzc{Mod}(\mathpzc{E})$, then $|f|$ is a pure monomorphism in $\mathpzc{E}$. Indeed, take any object $Z$ of $\mathpzc{E}$, and consider $R\otimes Z$. Then $|f\otimes_{R}(Id_{R\otimes Z})|=|f|\otimes Id_{Z}$ is an admissible monomorphism.

    The proof of the second claim is similar.
\end{proof}

Since pure monomorphisms are stable under tensoring with any object, by taking $\mathcal{S}=\textbf{PureMon}$, we immediately get the following as a consequence of Lemma \ref{Lem:pushouttransK-flat}.

\begin{cor}

Let $\mathpzc{M}\cong{}_{R}\mathpzc{Mod}(Ch(\mathpzc{E}))$ be a model category in which the weak equivalences are the quasi-isomorphisms. Suppose that $Ch(\mathpzc{E})$ is weakly $\textbf{PureMon}$-elementary.
Let $B$ be an object in $\mathpzc{M}$. Suppose that $B=\colim_{\mathcal{I}}B_{i}$ is a transfinite extension, where $B_{0}$ is $K$-flat and each $B_{i}\rightarrow B_{i+1}$ is a pure monomorphism with $K$-flat cokernel. Then $B$ is $K$-flat.  
\end{cor}
%
%
%%Lemma \ref{standard} has another convenient consequence.

%\begin{proof}
%Let $f:M\rightarrow N$ be an equivalence in $\mathpzc{M}$. Then 
%$$(Id_{B}\otimes f: B\otimes M\rightarrow B\otimes N)\cong\colim_{i\in\mathcal{I}}(Id_{B_{i}}\otimes f:B_{i}\otimes M\rightarrow B_{i}\otimes N).$$
%Now each $Id_{B_{i}}\otimes f$ is an equivalence by assumption. For each $i$, $B_{i}\otimes M\rightarrow B_{i+1}\otimes M$ and $B_{i}\otimes N\rightarrow B_{i+1}\otimes N$ are pure monomorphisms. Since $\mathpzc{M}$ is weakly $\textbf{PureMon}$-elementary, the colimit $\colim_{i\in\mathcal{I}}(Id_{B_{i}}\otimes f:B_{i}\otimes M\rightarrow B_{i}\otimes N)$ is also a weak equivalence.
%\end{proof}

Typically, we will assume that our exact categories are weakly \textbf{PureMon}-elementary. Often our exact categories will satisfy a much stronger property than being weakly elementary for some class of (admissible) monomorphisms.

\begin{defn}[\cite{kelly2016homotopy} Definition 2.97]
An exact category $\mathpzc{E}$ is said to be \textbf{elementary} if it has a projective generating set $\mathcal{P}$ consisting of tiny projectives. That is:
\begin{enumerate}
\item
each $P\in\mathcal{P}$ is projective,
\item
each $P\in\mathcal{P}$ is \textit{tiny}, in that the functor
$$\mathrm{Hom}(P,-):\mathpzc{E}\rightarrow\mathpzc{Ab}$$
commutes with filtered colimits, and
\item
 for any object $E$ of $\mathpzc{E}$ there is an admissible epimorphism $\bigoplus_{i\in\mathcal{I}}P_{i}\rightarrow E$, where each $P_{i}\in\mathcal{P}$. 
\end{enumerate}
\end{defn}

In an elementary exact category $\mathpzc{E}$ all filtered colimits are exact by \cite{kelly2016homotopy} Proposition 2.101. In particular, $\mathpzc{E}$ is weakly $\mathcal{S}$-elementary for $\mathcal{S}$ the class of all morphisms.

\subsection{$K$-Flatness and $K$-Cotorsion in Monoidal Exact Categories}\label{Kcotorproj}

Let us now focus on homotopical categories of the form $(Ch(\mathpzc{E}),\mathcal{W})$, where $\mathpzc{E}$ is an exact category and $\mathcal{W}$ is the class of quasi-isomorphisms.

%\begin{defn}
%Let $\mathpzc{M}$ be a homotopical category enriched over $Ch(\mathpzc{Ab})$. 
%\end{defn}

\begin{prop}\label{boundedKflat}
Let $\mathpzc{E}$ be a weakly $\textbf{PureMon}$-elementary exact category. Let $X_{\bullet}$ be an $(\aleph_{0};\textbf{PureMon})$-extension of bounded-below complexes of flat objects, that is, $X_{\bullet}=\colim_{n\in\mathbb{N}} X_{\bullet}^{n}$ where each $X_{\bullet}^{n}$ consists of flat objects and $X_{\bullet}^{n}\rightarrow X_{\bullet}^{n+1}$ is a pure monomorphism. Then $X_{\bullet}$ is $K$-flat. 
\end{prop}
%Let $X_{\bullet}$ be a complex such that $Z_{n}X\rightarrow X_{n}$ is in $\mathcal{S}$ for all sufficiently negative $n$, and both $Z_{n}$ and $X_{n}$ are flat for all $n$. Let $Y_{\bullet}$ be any complex. If either $X_{\bullet}$ and $Y_{\bullet}$ is acyclic then so is $X_{\bullet}\otimes Y_{\bullet}$. 
%Suppose that $X_{\bullet}$ is acyclic and that $Z_{n}X$ is flat for each $n\in\Z$. Let $F_{\bullet}$ be any complex. Then $F_{\bullet}\otimes X_{\bullet}$ 

\begin{proof}
First, observe that any transfinite extension of $K$-flat objects along pure monomorphisms is $K$-flat. Indeed, let $\colim_{\lambda\in\Lambda}C_{\lambda}$ be such an extension, and $f:U\rightarrow V$ be a weak equivalence. In particular, $cone(f)$ is acyclic. Now, we have $cone(\colim_{\lambda\in\Lambda}(Id_{C_{\lambda}}\otimes f))\cong\colim_{\lambda\in\Lambda}(C_{\lambda}\otimes cone(f))$. There is a pure exact sequence
$$0\rightarrow C_{\lambda}\rightarrow C_{\lambda+1}\rightarrow \mathrm{coker}(C_{\lambda}\rightarrow C_{\lambda+1})\rightarrow0,$$
and thus a pure exact sequence
$$0\rightarrow C_{\lambda}\otimes cone(f)\rightarrow C_{\lambda+1}\otimes cone(f)\rightarrow \mathrm{coker}(C_{\lambda}\rightarrow C_{\lambda+1})\otimes cone(f)\rightarrow0.$$
Hence, the sequence $\lambda\mapsto C_{\lambda}\otimes cone(f)$ is a transfinite composition of degree-wise pure monomorphisms.
Since $C_{\lambda}\otimes cone(f)$ and $C_{\lambda+1}\otimes cone(f)$ are acyclic, so is $\mathrm{coker}(C_{\lambda}\rightarrow C_{\lambda+1})\otimes cone(f)$. Therefore, each $C_{\lambda}\otimes cone(f)\rightarrow C_{\lambda+1}\otimes cone(f)$ is a weak equivalence. By the $\textbf{PureMon}$-elementary assumption, the colimit is acyclic. 

Hence, in the setup of our proposition, we may in fact assume that $X_{\bullet}$ is bounded-below. Now, $X_{\bullet}$ is obtained from the complex $S^{0}(X_{0})$ by a transfinite composition of pushouts of the form
\begin{displaymath}
\xymatrix{
S^{k}(F)\ar[r]\ar[d] & A\ar[d]\\
D^{k+1}(F)\ar[r] & B,
}
\end{displaymath}
where $F$ is a flat object. Similarly to Lemma \ref{Lem:pushouttransK-flat} we reduce to proving the claim when $X$ is either of the form $S^{k}(F)$ or $D^{k+1}(F)$ for $F$ flat. Since the subcategory of $K$-flat objects is stable under pure extensions, we in fact reduce to the $S^{k}(F)$ case. Finally, since shifts of $K$-flat objects are $K$-flat, we may assume that our object $X$ is concentrated in degree $0$. So let $X$ be flat and $f:U\rightarrow V$ a quasi-isomorphism. Then $cone(Id_{X}\otimes f)\cong X\otimes cone(f)$. Since $X$ is flat, $X\otimes cone(f)$ is acyclic. 
%The proof is by induction on the length of $X_{\bullet}$. If $X_{\bullet}$ is concentrated in degree $0$ then the result is clear. Suppose the claim has been proven when $X_{\bullet}$ has length $k$ for some $k\ge 0$. Suppose that $X_{\bullet}$ is has length $k+1$. 
\end{proof}
%First note that $X_{\bullet}$ is a $(\aleph_{0};\mathcal{S})$-extension of bounded-below complexes satisfying the hypotheses of the proposition. Since $\mathpzc{E}$ is weakly $(\aleph_{0};\mathcal{S})$-elementary, we may in fact assume that $X_{\bullet}$ is bounded-below. 
Dually, one defines cotorsion objects.

\begin{defn}\label{defn:Kcotors}
Suppose that $(\mathpzc{M},\mathcal{W})$ is a homotopical category with a closed monoidal structure such that the internal hom functor $\underline{\mathrm{Hom}}(-,O)$ is right derivable for any object $O$ of $\mathpzc{M}$. If $\mathfrak{O}$ is a class of objects in $\mathpzc{M}$, an object $X_{\bullet}$ is said to be $\mathfrak{O}$-K-cotorsion if the morphism $\underline{\mathrm{Hom}}(X_{\bullet},O)\rightarrow\mathbb{R}\underline{\mathrm{Hom}}(X_{\bullet},O)$ is an equivalence for any $O\in\mathfrak{O}$. When $\mathfrak{O}$ is the class of finite products of copies of the monoidal unit $k$, we say that $X_{\bullet}$ is \textbf{finitely }$K$-\textbf{cotorsion}.
\end{defn}

Now, suppose that $\mathpzc{E}$ is a closed monoidal exact category, and consider the homotopical category $(Ch(\mathpzc{E}),\mathcal{W}_{qis})$, where $\mathcal{W}_{qis}$ is the class of quasi-isomorphisms. Let us assume that $\underline{\mathrm{Hom}}$ is right-derivable.
Notice that to check that a complex $X_{\bullet}$ is finitely $K$-cotorsion, it is necessary and sufficient to check that $\underline{\mathrm{Hom}}(X_{\bullet},k)\rightarrow\mathbb{R}\underline{\mathrm{Hom}}(X_{\bullet},k)$ is an equivalence. In particular, the functor $\underline{\mathrm{Hom}}(-,k)$ preserves weak equivalences between finitely $K$-cotorsion objects. A proof similar to that of Proposition \ref{boundedKflat} gives the following.

\begin{prop}\label{boundedbelowKcotors}
Any bounded-below complex of $\mathfrak{O}$-cotorsion objects is $\mathfrak{O}$-K-cotorsion. 
\end{prop}

%Moreover, if $\mathpzc{E}$ has enough projectives and $X_{\bullet}$ is a complex such that for sufficiently negative $n$ $Z_{n}X\rightarrow X_{n}$ is an admissible morphism, and both $B_{n}=im(d_{n})$ and $X_{n}$ are $\mathfrak{O}$-cotorsion for all $n$, then $X_{\bullet}$ is $\mathfrak{O}$-K-cotorsion.
%
%\begin{proof}
%The first part is similar to Proposition \ref{boundedKflat}. The second part is also similar, but there is a slight subtlety. Since $\mathpzc{E}$ has enough projectives, projective limits in which connecting morphisms are admissible epimorphisms are exact. Indeed we may test exactness by passing to abelian groups using $\mathrm{Hom}(P,-)$ for $P$ projective, and these send our projective systems to Mittag-Leffler systems. Now we write $X_{\bullet}=\colim \tau_{\ge n}X_{\bullet}$. Then $\underline{\mathrm{Hom}}(X_{\bullet},O)\cong lim_{\leftarrow}\underline{\mathrm{Hom}}(\tau_{\ge n}X_{\bullet},O)$. Each $\underline{\mathrm{Hom}}(\tau_{\ge n}X_{\bullet},O)$ is acyclic by the first part. Now by the long exact sequence (\cite{Buehler} Section 12) each $Z_{n}X$ is $\mathfrak{O}$-cotorsion, and the sequence 
%$$0\rightarrow\underline{\mathrm{Hom}}(B_{n},O)\rightarrow\underline{\mathrm{Hom}}(X_{n},O)\rightarrow\underline{\mathrm{Hom}}(Z_{n},O)\rightarrow0$$
%is short exact. Therefore, the system $lim_{\leftarrow}\underline{\mathrm{Hom}}(\tau_{\ge n}X_{\bullet},O)$ is a Mittag-Leffler system. 
%%is a Mittag-Lefler system, and so we get
%%$$\underline{\mathrm{Hom}}(X_{\bullet},O)\cong lim_{\leftarrow}\underline{\mathrm{Hom}}(\tau_{\ge n}X_{\bullet},O)$$
%\end{proof}

\begin{prop}\label{prop:elexsuff}

Let $\mathpzc{E}$ be a closed symmetric monoidal exact category with enough projectives and exact countable coproducts. Let $X_{\bullet}$ be a complex in $Ch(\mathpzc{E})$ such that, for sufficiently negative $n$, $Z_{n}X\rightarrow X_{n}$ is an admissible morphism, and both $B_{n+1}=im(d_{n})$ and $X_{n}$ are $\mathfrak{O}$-cotorsion for all $n$. Then $X_{\bullet}$ is $\mathfrak{O}$-K-cotorsion.
\end{prop}

\begin{proof}
We know the claim holds for bounded-below complexes by Proposition \ref{boundedbelowKcotors}. Since $\mathpzc{E}$ has enough projectives, projective limits of countable towers in which connecting morphisms are admissible epimorphisms (Mittag-Leffler systems) are exact. Indeed, we may test exactness by passing to abelian groups using $\mathrm{Hom}(P,-)$ for $P$ projective, and these send our projective systems to Mittag-Leffler systems. We write $X_{\bullet}=\colim \tau_{\ge n}X_{\bullet}$. Then $\underline{\mathrm{Hom}}(X_{\bullet},O)\cong \lim_{\leftarrow}\underline{\mathrm{Hom}}(\tau_{\ge n}X_{\bullet},O)$. Each $\underline{\mathrm{Hom}}(\tau_{\ge n}X_{\bullet},O)$ is acyclic since $\tau_{\ge n}X_{\bullet}$ is bounded-below. By the long exact sequence (\cite{Buehler} Section 12) each $Z_{n}X$ is $\mathfrak{O}$-cotorsion, and the sequence 
$$0\rightarrow\underline{\mathrm{Hom}}(B_{n},O)\rightarrow\underline{\mathrm{Hom}}(X_{n},O)\rightarrow\underline{\mathrm{Hom}}(Z_{n},O)\rightarrow0$$
is short exact. Therefore, the limit $\lim_{\leftarrow}\underline{\mathrm{Hom}}(\tau_{\ge n}X_{\bullet},O)$ is the limit of a Mittag-Leffler system. Now, exactness of countable coproducts implies that the natural morphism
$$\mathbb{L}\colim \tau_{\ge n}X_{\bullet}\rightarrow X_{\bullet}$$
is an equivalence.

%Consequently, we have 
%$$\mathbb{R}lim_{\leftarrow}\mathbb{R}\underline{\mathrm{Hom}}(\tau_{\ge n}X_{\bullet},O)\cong \mathbb{R}lim_{\leftarrow}\underline{\mathrm{Hom}}(\tau_{\ge n}X_{\bullet},O)\cong lim_{\leftarrow}\underline{\mathrm{Hom}}(\tau_{\ge n}X_{\bullet},O)\cong\underline{\mathrm{Hom}}(X_{\bullet},O).$$

Consequently, we have
$$\mathbb{R}\underline{\mathrm{Hom}}(\colim \tau_{\ge n}X_{\bullet},O)\cong\mathbb{R}\lim_{\leftarrow}\mathbb{R}\underline{\mathrm{Hom}}(\tau_{\ge n}X_{\bullet},O)\cong \lim_{\leftarrow}\underline{\mathrm{Hom}}(\tau_{\ge n}X_{\bullet},O),$$
where the last line follows from the fact that the projective system is Mittag-Leffler, and each $\tau_{\ge n}X_{\bullet}$ is $\mathfrak{O}$-K-cotorsion.  
%is a Mittag-Lefler system, and so we get
%$$\underline{\mathrm{Hom}}(X_{\bullet},O)\cong lim_{\leftarrow}\underline{\mathrm{Hom}}(\tau_{\ge n}X_{\bullet},O)$$
\end{proof}

\subsubsection{Duality}

For a closed monoidal model category $\mathpzc{M}$ with monoidal unit $R$, we denote by $(-)^{\vee}:\mathpzc{M}\rightarrow\mathpzc{M}^{op}$ the functor $\underline{\mathrm{Hom}}(-,R)$. Since $\mathpzc{M}$ is closed monoidal, this is derivable to a functor
$$\mathbb{R}(-)^{\vee}\defeq\mathbb{R}\underline{\mathrm{Hom}}(-,R):Ho(\mathpzc{M})\rightarrow Ho(\mathpzc{M})^{op}.$$
By abuse of notation, we denote by $\mathbb{R}(-)^{\vee\vee}:Ho(\mathpzc{M})\rightarrow Ho(\mathpzc{M})$ the composition $\mathbb{R}(-)^{\vee}\circ \mathbb{R}(-)^{\vee}$. Note that in $Ho(\mathpzc{M})$ there is a natural morphism $J_{V}:V\rightarrow\mathbb{R}(V)^{\vee\vee}$.

\begin{defn}\label{defn:htpyref}
An object $V$ is said to be \textbf{homotopically reflexive} if the morphism $J_{V}:V\rightarrow \mathbb{R}(V)^{\vee \vee}$ is an isomorphism in $Ho(\mathpzc{M})$. 
\end{defn}

If $V$ and its dual are both finitely $K$-cotorsion, then $\mathbb{R}(V)^{\vee\vee}$ is equivalent to $V^{\vee\vee}$, and such an object is homotopically reflexive if and only if the morphism $V\rightarrow V^{\vee\vee}$ is a weak equivalence.

\subsection{Koszul Categories}
We are now ready to define what we mean by a Koszul category.

\begin{defn}\label{defn:koszcat}

A \textbf{Koszul category} is a weak monoidal model category $\mathpzc{M}$ of the form ${}_{R}\mathpzc{Mod}(Ch(\mathpzc{E}))$ where:
\begin{enumerate}
\item
$\mathpzc{E}$ is a complete and cocomplete symmetric monoidal exact category in which the class of pure monomorphisms is stable under transfinite extensions, and the monoidal structure on $Ch(\mathpzc{E})$ is the one induced from $\mathpzc{E}$,
\item
$R$ is a commutative monoid in $Ch(\mathpzc{E})$,
\item
$Ch(\mathpzc{E})$ is equipped with a combinatorial model structure, and this model structure satisfies the monoid axiom,
\item
$\mathpzc{M}$ is equipped with the transferred model structure,
\item
weak equivalences in the model structure on $Ch(\mathpzc{E})$ are thick: if
\begin{displaymath}
\xymatrix{
0\ar[r] & X\ar[d]\ar[r] & Y\ar[d]\ar[r] & Z\ar[d]\ar[r] & 0\\
0\ar[r] & U\ar[r] & V\ar[r] & W\ar[r] & 0
}
\end{displaymath}
is a diagram in $Ch(\mathpzc{E})$ in which the top and bottom rows are exact, and any two of the vertical morphisms are weak equivalences, then the third morphism is a weak equivalence,
\item 
the canonical tensoring of $\mathpzc{M}$ over $Ch(\mathpzc{Ab})$
realises $\mathpzc{M}$ as a model category weakly tensored over $Ch(\mathpzc{Ab})$,
%if $C\in\mathpzc{M}$ is (trivially) cofibrant, then for all $n\in\mathbb{Z}$,  $C[n]\rightarrow cone(Id_{C[n]})$ is a (trivial) cofibration,
%\item 
%the model category $\mathpzc{M}$ is weakly $\textbf{AdMon}_{\mathpzc{Cof}}$-elementary, where $\mathpzc{Cof}$ is the class of cofibrant objects, 
\item 
cofibrations in $\mathpzc{M}$ are pure monomorphisms,
\item 
$\mathpzc{M}$ is weakly $\textbf{PureMon}$-elementary,
\item
and $\mathpzc{M}$ satisfies the $pp$-monoid axiom.
\end{enumerate}
A Koszul category is said to be \textbf{closed} if it is a closed monoidal category.
\end{defn}

We \textit{do not assume} the full pushout-product axiom for the monoidal structure on $\mathpzc{M}$. In particular, we do not assume that the tensor product of cofibrant objects is cofibrant.

\begin{prop}

    Let $\mathpzc{M}={}_{R}\mathpzc{Mod}(Ch(\mathpzc{E}))$ be a Koszul category, and let 
    $$0\rightarrow X\rightarrow Y\rightarrow Z\rightarrow 0$$
    be an exact sequence with $Z$ trivial. Then $X\rightarrow Y$ is a weak equivalence.
\end{prop}

\begin{proof}
    Consider the diagram of exact sequences
    \begin{displaymath}
        \xymatrix{
        0\ar[r] & X\ar[r]\ar[d] & X\ar[r]\ar[d] & 0\ar[d]\ar[r] & 0\\
        0\ar[r] & X\ar[r] & Y\ar[r] & Z\ar[r] & 0.
        }
    \end{displaymath}
    The first vertical morphism is trivially a weak equivalence, and the third vertical morphism is a weak equivalence by assumption. Thus, the middle vertical morphism is a weak equivalence, as required.
\end{proof}

\begin{prop}\label{prop:adminleftprop}
    
    Let
    \begin{displaymath}
        \xymatrix{
        X\ar[d]^{f}\ar[r]^{w} & A\ar[d]^{f'}\\
        Y\ar[r]^{w'} & B
        }
    \end{displaymath}
    be a pushout diagram in a Koszul category $\mathpzc{M}$, where $f$ is an admissible monomorphism, and $w$ is a weak equivalence. Then $w'$ is a weak equivalence. In particular, admissible monomorphisms are left proper in $\mathpzc{M}$. 
\end{prop}

\begin{proof}
    We have a commutative diagram
    \begin{displaymath}
        \xymatrix{
        X\ar[d]^{f}\ar[r]^{w} & A\ar[d]^{f'}\\
        Y\ar[r]^{w'}\ar[d] & B\ar[d]\\
        C\ar@{=}[r] & C
        }
    \end{displaymath}
    in which both columns are exact sequences. By thickness of weak equivalences, $w'$ is a weak equivalence.
\end{proof}

\begin{prop}
    
    Let 
    $$0\rightarrow X\rightarrow Y\rightarrow Z\rightarrow 0$$
    be a pure exact sequence in a Koszul category $\mathpzc{M}$. If two of the three entries are $K$-flat, then so is the third. 
\end{prop}

\begin{proof}
    Let $f:A\rightarrow B$ be a weak equivalence. We get a commutative diagram of exact sequences
        \begin{displaymath}
        \xymatrix{
        0\ar[r] & A\otimes X\ar[r]\ar[d] & A\otimes Y\ar[r]\ar[d] & A\otimes Z\ar[d]\ar[r] & 0\\
        0\ar[r] & B\otimes X\ar[r] & B\otimes Y\ar[r] & B\otimes Z\ar[r] & 0.
        }
    \end{displaymath}
    Thickness implies that if two of the vertical morphisms are weak equivalences, then so is the third. 
\end{proof}

%\begin{proof}
 %   As a cocomplete additive category, it is tensored over abelian groups. Note that, tensoring with $\mathbb{Z}^{\oplus n}\otimes M\cong M^{\oplus n}$, where $M\in\mathpzc{M}$. Let $f:X\rightarrow Y$ be a cofibration in $\mathpzc{M}$.
%\end{proof}

\begin{defn}\label{defn:pureK}
A Koszul category is said to be
\begin{enumerate}
\item
 \textbf{C-monoidal} if it satisfies the pushout-product axiom, i.e., it is a (non-unital) monoidal model category;
\item
\textbf{K-monoidal} if cofibrant objects are $K$-flat;
 \item
 \textbf{KC-monoidal} if it is both K-monoidal and C-monoidal;
 \item
 \textbf{pure} if the class of morphisms that are both pure monomorphisms and weak equivalences is weakly saturated.
 \end{enumerate}
\end{defn}

\begin{rem}
     
    \begin{enumerate}
        \item     Note that a Koszul category is pure if and only if any pushout of a pure monomorphism that is also a weak equivalence is a weak equivalence.
        \item 
        A KC-monoidal model category automatically satisfies the unit axiom, since the monoidal unit is trivially $K$-flat.
    \end{enumerate}

\end{rem}

%FIX: Similar to https://arxiv.org/pdf/1012.1394.pdf. May need compatibility with monoidal structure?
%
%$X_{\bullet}$ is an \textbf{PureMon}-extension of bounded acyclic complexes $Y_{\bullet}$ with each $Z_{n}Y$ flat. Therefore, we may assume that $X_{\bullet}$ is of this form. Moreover, we may assume that $X_{\bullet}$ is concentrated in non-negative degrees. We prove the claim by induction of the length of $X_{\bullet}$. First suppose that $X_{\bullet}$ has length $2$. Then it is of the form $D^{1}(L)$ for some flat object $L$, and $F_{\bullet}\otimes D^{1}(L)\cong cone(Id_{F_{\bullet}})\otimes L$ which is acyclic since $L$ is flat. Now suppose that $X_{\bullet}$ is of length $k+1$ for some $k\ge 1$. Then there is an exact sequence
%$$0\rightarrow D^{k+1}(X_{k+1})\rightarrow X_{\bullet}\rightarrow \tau_{\le k} X_{\bullet}\rightarrow 0$$
%Since $ \tau_{\le k} X_{\bullet}$ consists of flat objects the sequence is pure exact. Tensoring with $F_{\bullet}$ gives an exact sequence.
%$$0\rightarrow F_{\bullet}\otimes D^{k+1}(X_{k+1})\rightarrow F_{\bullet}\otimes X_{\bullet}\rightarrow F_{\bullet}\otimes\tau_{\le k} X_{\bullet}\rightarrow 0$$
%The first and last complexes are acyclic by the base case and inductive step respectively. Therefore, the middle one is acyclic. 

\subsection{Construction of Koszul Categories} 
In this subsection, we will give a method of constructing a large class of Koszul categories. These will be called \textbf{basic Koszul categories}. In the present paper, the vast majority of categories of interest will be constructed from basic Koszul categories in a way set out below. However, the full generality of Koszul categories in principle allows one to work with certain stable Bousfield localisations. This will be useful in future work on Koszul duality for quasi-coherent sheaves on stacks.

\subsubsection{Basic Koszul Categories}
\begin{defn}\label{basic Koszul}
A \textbf{basic Koszul category} is a complete and cocomplete monoidal exact category $(\mathpzc{E},\otimes)$, together with a model structure on $Ch(\mathpzc{E})$ such that
\begin{enumerate}
\item
$\mathpzc{E}$ is weakly $\textbf{PureMon}$-elementary, and transfinite compositions of pure monomorphisms are pure monomorphisms,
\item
the weak equivalences are the quasi-isomorphisms,
\item
the model structure is combinatorial,
\item
there is a set of objects $\{G_{i}\}_{i\in\mathcal{I}}$ and a set of pure monomorphisms $\{f_{l}\}_{l\in\mathcal{L}}$ such that 
$$\mathfrak{G}=\{S^{n}(G_{i})\rightarrow D^{n+1}(G_{i})\}_{n\in\Z,i\in\mathcal{I}}\cup\{0\rightarrow D^{n}(G_{i})\}_{n\in\Z,i\in\mathcal{I}}\cup\{S^{n}(f_{l})\}_{n\in\Z,l\in\mathcal{L}}$$
is a set of generating cofibrations, 
\item
and there is a set of pure monomorphisms $\{g_{k}\}_{k\in\mathcal{K}}$ in $\mathpzc{E}$ such that $\{D^{n}(g_{k})\}_{n\in\Z,k\in\mathcal{K}}$ is a collection of generating acyclic cofibrations.
\end{enumerate}
A basic Koszul category is said to be \textbf{projective} if one can take $\{f_{l}\}_{l\in\mathcal{L}}=\emptyset$, and can take all $g_{k}$ to be of the form $0\rightarrow F$ for $F$ projective. 
\end{defn}

The meaning behind the term \textit{projective} is that such a model structure behaves very much like a projective model structure, in that the cofibrations are degree-wise split. 

%The collection
%$$\mathfrak{G}=\{S^{n}(G_{i})\rightarrow D^{n+1}(G_{i})\}_{n\in\Z,i\in\mathcal{I}}\cup\{0\rightarrow D^{n}(G_{i})\}_{n\in\Z,i\in\mathcal{I}}\cup\{S^{n}(f_{l})\}_{n\in\Z,l\in\mathcal{L}}$$
%is called the \textbf{cofibrancy data} of $\mathpzc{E}$. 

\begin{rem}\label{rem:puritydw}

    Let $\mathpzc{E}$ be a basic Koszul category. Since split monomorphisms are pure, condition (1) implies that taking direct sums is an exact functor. This in particular means that a sequence
    $$0\rightarrow X_{\bullet}\rightarrow Y_{\bullet}\rightarrow Z_{\bullet}\rightarrow 0$$
    in $Ch(\mathpzc{E})$ is pure exact if and only if, for each $n\in\mathbb{Z}$,
    $$0\rightarrow X_{n}\rightarrow Y_{n}\rightarrow Z_{n}\rightarrow 0$$
    is pure exact in $\mathpzc{E}$. Indeed, if the sequence is pure exact in $Ch(\mathpzc{E})$, then tensoring with complexes of the form $S^{0}(W)$ gives that the term-wise sequences are pure exact.  In the other direction, suppose that the sequence of complexes is term-wise pure exact, and let $W_{\bullet}$ be a complex. We need to check that each sequence
    $$0\rightarrow\bigoplus_{i+j=k}W_{i}\otimes X_{j}\rightarrow\bigoplus_{i+j=k}W_{i}\otimes Y_{j}\rightarrow \bigoplus_{i+j=k}W_{i}\otimes Z_{j}\rightarrow 0$$
    is exact. By degree-wise pure exactness, for fixed $i$ and $j$ the sequence 
     $$0\rightarrow W_{i}\otimes X_{j}\rightarrow W_{i}\otimes Y_{j}\rightarrow W_{i}\otimes Z_{j}\rightarrow 0$$
     is exact. The direct sum is then exact by assumption.
\end{rem}

\begin{prop}\label{prop:cokercofib}
%FIX:finish, reference. 
Let $\mathpzc{E}$ be a basic Koszul category. Suppose that $f:X\rightarrow Y$ is a (trivial) cofibration in $Ch(\mathpzc{E})$. Then $f$ is a pure monomorphism and $coker(f)$ is a (trivially) cofibrant object. 
\end{prop}

\begin{proof}
We first note that $f$ is a retract of a transfinite composition of pushouts of morphisms in 
$$\{S^{n}(G_{i})\rightarrow D^{n+1}(G_{i})\}_{n\in\Z,i\in\mathcal{I}}\cup\{0\rightarrow D^{n}(G_{i})\}_{n\in\Z,i\in\mathcal{I}}\cup\{S^{n}(f_{l})\}_{n\in\Z,l\in\mathcal{L}},$$
and is therefore a pure monomorphism. 
 Let $C$ be the cokernel of $f:X\rightarrow Y$. Suppose that 
\begin{displaymath}
\xymatrix{
0\ar[d]\ar[r] & A\ar[d]\\
C\ar[r] & B
}
\end{displaymath}
is a commutative diagram and that $A\rightarrow B$ is a trivial fibration. Then there is a commutative diagram
\begin{displaymath}
\xymatrix{
X\ar[d]^{f}\ar[r] & A\ar[d]\\
Y\ar[r] & B
}
\end{displaymath}
where the morphism $X\rightarrow A$ is the zero morphism. Thus, there is a lift $h:Y\rightarrow A$ in this diagram. But $h\circ f=0$, so this gives a lift $\tilde{h}:C\rightarrow A$ in the first diagram.
The same proof works for trivial cofibrations.
\end{proof}

\begin{prop}
    
    Let $\mathpzc{E}$ be a basic Koszul category, and let $f:X\rightarrow Y$ be an acyclic cofibration in the transferred model structure on $Ch(\mathpzc{E})$. Then $f$ is a pure monomorphism whose cokernel is a retract of a transfinite extension, by pure monomorphisms, of objects of the form $D^{n}(C)$, with $n\in\mathbb{Z}$ and $C\in\mathpzc{E}$. 
\end{prop}

\begin{proof}
    The cofibration $f$ is a retract of a transfinite composition of pushouts of morphisms of the form $D^{n}(A)\rightarrow D^{n}(B)$, where $A\rightarrow B$ is a pure monomorphism with cokernel $C$. From here, the claim is obvious.
\end{proof}

\begin{lem}\label{lem:almostpurebkosz}
Let $\mathpzc{E}$ be a basic Koszul category. Then $Ch(\mathpzc{E})$ is a weak $h^{\otimes}$-monoidal model category (non-unital) and a pure Koszul category that is weakly $pp$-tensored over $Ch(\mathpzc{E})$.
\end{lem}

\begin{proof}

We check all the axioms of Definition \ref{defn:koszcat}. Axiom (1) is clear, axioms (2) and (4) are vacuous, and axioms (3) and (8) are by assumption. Axiom (5) follows from \cite{kelly2016homotopy} Proposition 4.27.

Now, for axiom (7), condition (4) furnishes a generating set of cofibrations consisting of pure monomorphisms. By Proposition \ref{prop:admonweaklyelementaryws} the class of pure monomorphisms is weakly saturated. Consequently, any cofibration is a pure monomorphism.

Now we prove that $Ch(\mathpzc{E})$ is a weak $h^{\otimes}$-monoidal model category. Let $g:X\rightarrow Y$ be any pure monomorphism, and $f$ an acyclic cofibration. Then $coker(f)$ is a retract of a transfinite extension, by pure monomorphisms, of objects of the form $D^{n}(C)$, with $n\in\mathbb{Z}$ and $C\in\mathpzc{E}$. Therefore, by Corollary \ref{cor:itpushoutprodpure}, the pushout-product of $f$ with $g$ is a pure monomorphism with cokernel $coker(f)\otimes coker(g)$. Now, $coker(f)\otimes coker(g)$ is a retract of a transfinite extension, by pure monomorphisms, of objects of the form
$D^{n}(C)\otimes coker(g)$. Such an object is just a shift of $C\otimes cone(Id_{coker(g)})$. Since $cone(Id_{coker(g)})$ is homotopy equivalent to the zero complex, any shift $C\otimes cone(Id_{coker(g)})[m]$ is also homotopy equivalent to the zero complex, and hence is acyclic. By $\textbf{PureMon}$-elementarity,  $coker(f)\otimes coker(g)$ is acyclic. Now, by \cite{kelly2016homotopy} Lemma 4.36, the pushout-product is a quasi-isomorphism. This establishes that $Ch(\mathpzc{E})$ is a weak monoidal model category. Now, any iterated pushout-product of cofibrations is pure by the fact that cofibrations are pure monomorphisms, as we have seen, and by Corollary \ref{cor:itpushoutprodpure}. Observe that $Id_{X}\otimes h$ is a pure monomorphism for any object $X$ and any pure monomorphism $h$. By Remark \ref{rem:leftproperadm}, we see that $Ch(\mathpzc{E})$ is a weak $h^{\otimes}$-monoidal model category. 

Next, let us prove axiom (9), the pp-monoid axiom. As we have seen, if $g$ is any pure monomorphism, and $f$ an acyclic cofibration, then the pushout-product of $f$ with $g$ is a pure monomorphism with acyclic cokernel. Note that tensoring a morphism with $Id_{X}$ is the same as taking the pushout-product with the pure monomorphism $0\rightarrow X$. It follows that any morphism of the form $Id_{X}\otimes h$, with $h$ being an iterated pushout-product of acyclic cofibrations, is a pure monomorphism with acyclic cokernel. By condition (1) and Remark \ref{rem:selementarydouble}, transfinite extensions of morphisms that are both pure monomorphisms and quasi-isomorphisms are still pure monomorphisms and quasi-isomorphisms. This completes this part of the proof.

Finally, the fact that $Ch(\mathpzc{E})$ is weakly $pp$-tensored over $Ch(\mathpzc{Ab})$ is essentially identical to the $pp$-monoid axiom proof.

%It suffices to prove that the pushout-product of a generating cofibration (resp. a generating acyclic cofibration) in $Ch(\mathpzc{Ab})$ with a pure monomorphism with acyclic cokernel in $Ch(\mathpzc{E})$ (resp. a pure monomorphism in $Ch(\mathpzc{E})$) is a pure monomorphism with acyclic cokernel. Take a generating cofibration $0\rightarrow D^{n}(\mathbb{Z})$ of $Ch(\mathpzc{Ab})$ and any pure monomorphism $g:X\rightarrow Y$ in $Ch(\mathpzc{E})$. The pushout-product of $f$ and $g$ is the morphism $D^{n}(\mathbb{Z})\otimes X\rightarrow D^{n}(\mathbb{Z})\otimes Y.$ This is a pure monomorphism with cokernel $D^{n}(\mathbb{Z})\otimes coker(g)\cong cone(Id_{coker(g)[n]})$ being acyclic. Now, take a generating cofibration $S^{n}(\mathbb{Z})\rightarrow D^{n+1}(\mathbb{Z})$ in $Ch(\mathpzc{Ab})$, and a pure monomorphism $g:X\rightarrow Z$ in $Ch(\mathpzc{E})$ with acyclic cokernel. The pushout-product of $f$ and $g$ is the morphism 
%$$cone(f)[n+1]\rightarrow cone(Id_{Y})[n+1].$$
%This is a pure monomorphism with cokernel $coker(g)[n+1]$. Since $coker(g)$ is acyclic, so is $coker(g)[n+1]$. Hence $cone(f)[n+1]\rightarrow cone(Id_{Y})[n+1]$ is a pure monomorphism with acyclic cokernel. Any pushout of such a morphism is also a pure monomorphism with acyclic cokernel. By weak $\textbf{PureMon}$-elementarity, any transfinite composition of such morphisms is an equivalence. 
\end{proof}
%\begin{prop}
%Suppose that $Ch(\mathpzc{E})$ is hereditary Koszul. Let $R\in\mathpzc{Alg}_{\mathfrak{Comm}}(Ch(\mathpzc{E})$. Equip $\mathpzc{M}={}_{R}\mathpzc{Mod}(Ch(\mathpzc{E}))$ with the transferred model structure. Then $\mathpzc{M}$ is hereditary Koszul.
%\end{prop}
%
%\begin{proof}
%We first claim 
%\end{proof}

%  for example if $\mathpzc{E}$ is (monoidal) elementary (\cite{kelly2016homotopy}) and $Ch(\mathpzc{E})$ is equipped with the projective model structure.

\begin{prop}\label{prop:KoszulK}
Let $\mathpzc{E}$ be a basic Koszul category, and let $G_{i}$ and $f_{l}$ be as in Definition \ref{basic Koszul}.
Let $\mathcal{C}\subset \mathpzc{E}$ be a subcategory. Let $\tilde{\mathcal{C}}$ denote the class of complexes $X\in Ch(\mathpzc{E})$ such that for each $n$, $X_{n}$, $ker(d_{n}^{X})$, and $coker(d_{n}^{X})$ are in $\mathcal{C}$. Let $\mathcal{S}$ be the class of morphisms of the form $X\rightarrow 0$ where $X\in\tilde{\mathcal{C}}$. If each $G_{i}$ is $K$-transverse to $\mathcal{S}$, and the cokernel of each $f_{l}$ is $K$-transverse to $\mathcal{S}$, then for $C$ a cofibrant object, and $X$  any acyclic object in $\tilde{\mathcal{C}}$, we have that $C\otimes X$ is trivial. In particular, if each $G_{i}$ is flat and the cokernel of each $f_{l}$ is flat, then $Ch(\mathpzc{E})$ is $K$-monoidal. 
\end{prop}

\begin{proof}
The proof of Proposition 4.56 in \cite{kelly2016homotopy} works in this generality.
\end{proof}

\begin{lem}\label{lem:koszulmonoidalinduced}
   Let $\mathpzc{E}$ be a basic Koszul category such that in the set $\mathfrak{G}$, all $G_{i}$ are flat, and all $coker(f_{l})$ are flat. Then there is a (non-unital) monoidal model category structure on $ Ch(\mathpzc{E})$ with the same weak equivalences, but with more cofibrations.
\end{lem}

\begin{proof}

We use Lemma \ref{lem:almost-monoidal}. Since $ Ch(\mathpzc{E})$ is weakly $\textbf{PureMon}$-elementary, by \cite{kelly2024flat}  Proposition 2.23, pushout-products of pure monomorphisms are pure monomorphisms. The class $\textbf{PureMon}\cap W$, where $W$ is the class of quasi-isomorphisms, is weakly saturated by the proof of Lemma \ref{lem:almostpurebkosz}. Now, let $f:X\rightarrow Y$ be a morphism in $\mathrm{cof}^{\otimes}$. By the proof of \cite{vst2012exact} Theorem 8.11, $f$ is an admissible monomorphism with cokernel a tensor product of cofibrant objects. By Proposition \ref{prop:KoszulK}, cofibrant objects are $K$-flat. Now, let $g$ be a pushout-product of two morphisms in $\mathrm{cof}^{\otimes}$, at least one of which is acyclic. Then $g$ is an admissible monomorphism with cokernel $C\otimes D$, where $C$ is $K$-flat, and $D$ is acyclic. Hence, $C\otimes D$ is acyclic, as required. 
\end{proof}

Let $\mathpzc{E}$ be a basic Koszul category and consider the  model category $Ch(\mathpzc{E})$. Since, by assumption, the model structure on $Ch(\mathpzc{E})$ is monoidal and satisfies the monoid axiom, for any unital commutative monoid $R\in\mathpzc{Alg}_{\mathfrak{Comm}}(Ch(\mathpzc{E}))$ the transferred model structure exists on ${}_{R}\mathpzc{Mod}(Ch(\mathpzc{E}))$. Moreover, this model structure is monoidal and satisfies the monoid axiom (Theorem 4.1 in \cite{schwede}). 

%All of the arguments above work again for the transferred model structure on ${}_{R}\mathpzc{Mod}(Ch(\mathpzc{E}))$, for $R\in\mathpzc{Alg}_{\mathfrak{Comm}}(Ch(\mathpzc{E}))$. Thus, we get the following.

\begin{lem}
Let $\mathpzc{E}$ be a basic Koszul category, and $R\in\mathpzc{Alg}_{\mathfrak{Comm}}(Ch(\mathpzc{E}))$. Then, with the induced model structure, ${}_{R}\mathpzc{Mod}$ is a weak $h^{\otimes}$-monoidal, pure, Koszul category that is weakly $pp$-tensored over $Ch(\mathpzc{E})$. It is $K$-monoidal if each $G_{i}$ is flat and the cokernel of each $f_{l}$ is flat.
\end{lem}

\begin{proof}

    Axioms (1), (2), (3), (4), and (7) are trivial. Axiom (8) follows from Proposition \ref{prop:admonwepuremonwe}, Proposition \ref{prop:puremontransfer}, and the fact that ${}_{R}\mathpzc{Mod}(Ch(\mathpzc{E}))$ has a generating set of pure monomorphisms of the form $Id_{R}\otimes g$ for $g$ a pure monomorphism in $Ch(\mathpzc{E})$.
    
Next, let $g:X\rightarrow Y$ be any pure monomorphism, and $f$ an acyclic cofibration. Then the pushout-product of $f$ with $g$ is a pure monomorphism with cokernel given by a transfinite extension, by pure monomorphisms, of objects of the form $(R\otimes D^{n}(C))\otimes_{R} coker(g)$. With this fact, the proof proceeds exactly as in the proofs of Lemma \ref{lem:almostpurebkosz} and Proposition \ref{prop:KoszulK}.

\end{proof}

\begin{defn}
If $\mathpzc{E}$ is a basic Koszul category and $R\in\mathpzc{Alg}_{\mathfrak{Comm}}(Ch(\mathpzc{E}))$, we call a Koszul category of the form ${}_{R}\mathpzc{Mod}(Ch(\mathpzc{E}))$, equipped with the transferred model structure, an \textbf{exact Koszul category}.
\end{defn}

\begin{rem}\label{rem:leftproperadm}
 By the proof of \cite{kelly2016homotopy} Proposition 4.46, admissible monomorphisms are left proper in exact Koszul categories.
 \end{rem}

%\begin{prop}
%Let $\mathpzc{E}$ be a basic Koszul category and let $R\in\mathpzc{Alg}_{\mathfrak{Comm}}(Ch(\mathpzc{E}))$. A morphism $f:X\rightarrow Y$ whose cokernel is a free $R$-module on a (trivially) cofibrant object of $Ch(\mathpzc{E})$ is (trivially) cofibrant.
%\end{prop}
%
%\begin{proof}
%Let $f:X\rightarrow Y$ have cokernel $R\otimes C$, where $C$ is cofibrant in $Ch(\mathpzc{E})$. Consider a commutative diagram
%\begin{displaymath}
%\xymatrix{
%X\ar[d]^{f}\ar[r] & W\ar[d]^{g}\\
%Y\ar[r] & Z
%}
%\end{displaymath}
%where $g$ is an acyclic fibration. 
%\end{proof}

\subsubsection{Strong Koszul Categories}

\begin{defn}
A Koszul category $\mathpzc{M}={}_{R}\mathpzc{Mod}(Ch(\mathpzc{E}))$ is said to be \textbf{strong} if the model structure is transferred from a basic Koszul category structure on $Ch(\mathpzc{E})$, and a morphism $X\rightarrow Y$ in $\mathpzc{M}$ is a (trivial) fibration if and only if it is an admissible epimorphism whose kernel is a (trivially) fibrant object. 
\end{defn}

\begin{prop}
Let $\mathpzc{M}$ be a Koszul category. Suppose that a morphism $g:G\rightarrow C$ is a (trivial) fibration if and only if it is an admissible epimorphism with (trivially) fibrant kernel. Then any morphism $f:X\rightarrow Y$ that is a monomorphism with (trivially) cofibrant cokernel is a (trivial) cofibration. The converse holds provided that $\mathpzc{M}\cong{}_{R}\mathpzc{Mod}(Ch(\mathpzc{E}))$, where $Ch(\mathpzc{E})$ is a basic Koszul category and $\mathpzc{M}$ is equipped with the transferred model structure.
\end{prop}

\begin{proof}
First, we show that if $F$ is trivially fibrant (resp. fibrant) in $\mathpzc{M}$ and $C$ is cofibrant (resp. trivially cofibrant), then $\mathrm{Ext}^{1}(C,F)=0$. Let
$$0\rightarrow F\rightarrow G\rightarrow C\rightarrow 0$$
be an exact sequence. Then $G\rightarrow C$ is a trivial fibration/fibration. Since the morphism $0\rightarrow C$ is a cofibration/trivial cofibration, there is a lift in the diagram
\begin{displaymath}
\xymatrix{
0\ar[d]\ar[r] & G\ar[d]\\
C\ar[r] & C.
}
\end{displaymath}
This says precisely that the sequence is split, so $\mathrm{Ext}^{1}(C,F)=0$. The result now follows from Lemma 5.14 in \cite{vst2012exact}. 

 The partial converse follows from Proposition \ref{prop:cokercofib}. 
\end{proof}

In the terminology of \cite{vst2012exact}, this means that the weak factorisation systems determining the model structure on an exact Koszul category are \textbf{exact weak factorisation systems}. Equivalently, the weak factorisation systems arise from cotorsion pairs (see \cite{vst2012exact} Section 6).

\begin{cor}
If $Ch(\mathpzc{E})$ is a strong Koszul category and $R\in\mathpzc{Alg}_{\mathfrak{Comm}}(Ch(\mathpzc{E}))$, then $\mathpzc{M}={}_{R}\mathpzc{Mod}(Ch(\mathpzc{E}))$ is strong Koszul.
\end{cor}

The fact that the weak factorisation systems are determined by cotorsion pairs also has the following important consequences, which follow from the general theory of exact model structures (see e.g., \cite{vst2012exact}, \cite{kelly2016homotopy}).

\begin{cor}\label{middlecofibrant}
If $\mathpzc{M}$ is strong Koszul and $0\rightarrow X\rightarrow Y\rightarrow Z\rightarrow 0$ is an exact sequence such that $X$ and $Z$ are (trivially) cofibrant, then $Y$ is also (trivially) cofibrant. 
\end{cor}

\begin{cor}
If $\mathpzc{M}$ is strong Koszul, then it is $C$-monoidal precisely if the tensor product of two cofibrant objects is cofibrant.
\end{cor}

\begin{prop}
    Let $\mathpzc{M}$ be a strong Koszul category such that in the set of generating cofibrations $\mathfrak{G}$, the objects $G_{i}$ are all flat, and the cokernels of all $f_{l}$ are flat. Then the non-unital monoidal model structure furnished by the proof of Lemma \ref{lem:koszulmonoidalinduced} is in fact a unital monoidal model structure.
\end{prop}

\begin{proof}
    Let $C\rightarrow R$ be a cofibrant resolution of the unit. The morphism $C\rightarrow R$ is an admissible epimorphism with trivially fibrant kernel. Thus, we have an exact sequence
    $$0\rightarrow L\rightarrow C\rightarrow R\rightarrow 0$$
    with $L$ acyclic. Since $R$ is flat, this is a pure exact sequence. Let $X$ be cofibrant. Then
    $$0\rightarrow L\otimes X\rightarrow C\otimes X\rightarrow X\rightarrow 0$$
    is an exact sequence. By Proposition \ref{prop:KoszulK}, $L\otimes X$ is acyclic. Hence $C\otimes X\rightarrow X$ is a weak equivalence, as required.
    
    %We need to prove that $L\otimes X$ is acyclic. First, suppose that $X$ is trivially cofibrant. Then $C\otimes X$ is acyclic, so $L\otimes X$ is also acyclic. This implies that tensoring with $L$ preserves weak equivalences between cofibrant objects. Indeed, if $f:C\rightarrow D$ is a trivial cofibration between cofibrant objects, then it is a pure monomorphism with trivially cofibrant cokernel. Thus, $X\otimes C\rightarrow X\otimes D$ is a pure monomorphism with acyclic cokernel. On the other hand, if $f:C\rightarrow D$ is a trivial fibration, then it is an admissible epimorphism with trivially fibrant kernel $K$. In particular, $f$ is a split monomorphism, and so $K$ is trivially cofibrant as well. Since admissible monomorphisms with cofibrant cokernels are pure, $X\otimes f$ is an admissible epimorphism with kernel $X\otimes K$. This object is acyclic, so $X\otimes f$ is a weak equivalence. 

 %   Thi means that we are free to choose a particular resolution $L$ of $k$. 
\end{proof}

We will occasionally have need of a stricter notion than strength.

\begin{defn}
A strong Koszul category $\mathpzc{M}$ is said to be \textbf{hereditary} if whenever $f:X\rightarrow Y$ is an admissible epimorphism between (trivially) cofibrant objects, then $ker(f)$ is (trivially) cofibrant.
\end{defn}
In the language of cotorsion pairs this is equivalent to the condition that the cotorsion pairs determining the weak factorisation systems are hereditary.

\begin{prop}
Suppose that $Ch(\mathpzc{E})$ is a hereditary Koszul category. If $R\in\mathpzc{Alg}_{\mathfrak{Comm}}(Ch(\mathpzc{E}))$, then $\mathpzc{M}={}_{R}\mathpzc{Mod}(Ch(\mathpzc{E}))$ is hereditary Koszul.
\end{prop}

\begin{proof}
By Lemma 6.17 in \cite{vst2012exact}, it suffices to show that if $f:X\rightarrow Y$ is an admissible epimorphism where $X$ and $Y$ are (trivially) fibrant, then $ker(f)$ is (trivially) fibrant. But $ker(f)$ is (trivially) fibrant precisely if its underlying complex is, and this is true since $Ch(\mathpzc{E})$ is hereditary. 
\end{proof}

The utility of the hereditary property is the following $2$-out-of-$3$ result for cofibrations.

\begin{prop}\label{23cofibhered}
Let $f:X\rightarrow Y$ and $g:Y\rightarrow Z$ be morphisms in a hereditary Koszul category $\mathpzc{M}$. Suppose that $g\circ f$ and $g$ are (acyclic) cofibrations. Then $f$ is an (acyclic) cofibration.
\end{prop}

\begin{proof}
The Obscure Axiom (\cite{Buehler} Proposition 2.16) implies that $f$ is an admissible monomorphism. Moreover, the Snake Lemma implies that there is an exact sequence
$$0\rightarrow coker(f)\rightarrow coker(g\circ f)\rightarrow coker(g)\rightarrow 0.$$
Now, $coker(g\circ f)$ and $coker(g)$ are (trivially) cofibrant objects by strength. Therefore, $coker(f)$ is (trivially) cofibrant by the hereditary property. 
\end{proof}

Another useful property is \textbf{projectivity}.

\begin{defn}
A strong Koszul category is said to be \textbf{projective} if it arises from a projective basic Koszul category.  
\end{defn}

Note that projective Koszul categories are hereditary.

 If $\mathpzc{E}$ is projectively monoidal (see \cite{kelly2016homotopy} Definition 2.71) and is weakly $\textbf{PureMon}$-elementary, and $Ch(\mathpzc{E})$ is equipped with the projective model structure, then $\mathpzc{M}={}_{R}\mathpzc{Mod}(Ch(\mathpzc{E}))$ is projective Koszul by \cite{kelly2016homotopy} Theorem 4.59 and Theorem 4.69. Examples of this include $\mathpzc{E}={}_{R}\mathpzc{Mod}(\mathpzc{Ab})$ for any ring $R$, $\mathpzc{E}=Ind(Ban_{R})$ for $R$ a Banach ring, and $\mathpzc{E}=Ind(Ban^{nA}_{R})$ for $R$ a non-Archimedean Banach ring. In the final example, the superscript `nA' indicates that we require the norms on our Banach modules themselves to be non-Archimedean.

 \begin{rem}\label{rem:removelementary}
     In \cite{kelly2016homotopy} Theorem 4.69, the statement contains the requirement that $\mathpzc{E}$ be weakly elementary, but the proofs of the results on which it relies only require that it be weakly $\textbf{PureMon}$-elementary.
 \end{rem}

\begin{defn}
A Koszul category $\mathpzc{M}={}_{R}\mathpzc{Mod}(Ch(\mathpzc{E}))$ is said to be \textbf{elementary} if $\mathpzc{E}$ is an elementary exact category, and the model structure on $Ch(\mathpzc{E})$ is the projective one. 
\end{defn}

%When the weak factorisation systems defining the model structure on $Ch(\mathpzc{E})$ arise from cotorsion pairs (see \cite{vst2012exact}) then a basic Koszul category is strong. 

%If in this case cotorsion pairs are hereditary then $Ch(\mathpzc{E})$ is hereditary Koszul.
\subsection{Examples}

The definition of Koszul category may appear unnecessarily general. Let us give some examples that we will pursue both in this work and in future work.

\begin{example}
As mentioned above, if $\mathpzc{E}$ is a projectively monoidal elementary exact category and $R$ is a commutative monoid in $Ch(\mathpzc{E})$, then ${}_{R}\mathpzc{Mod}(Ch(\mathpzc{E}))$ is an elementary Koszul category when $Ch(\mathpzc{E})$ is equipped with the projective model structure, and ${}_{R}\mathpzc{Mod}(Ch(\mathpzc{E}))$ is equipped with the transferred model structure. Moreover, in this instance, $\mathpzc{M}$ is in fact KC-monoidal. This class of examples includes the category of modules over any ring, ${}_{R}\mathpzc{Mod}(\mathpzc{Ab})$.
\end{example}

\begin{example}\label{example:flatmodel}
   
    Let $\mathpzc{E}$ be a locally presentable closed symmetric monoidal exact category with exact filtered colimits and enough flat objects. Since filtered colimits are exact, $\mathpzc{E}$ is automatically weakly $\textbf{AdMon}$-elementary, and therefore also weakly $\textbf{PureMon}$-elementary, by Proposition \ref{prop:admonwepuremonwe}.
    
    Then, by \cite{kelly2024flat} Theorem 1.1/Corollary 6.3, the flat model structure exists on $Ch(\mathpzc{E})$. In this model structure, a cofibration is a degree-wise pure monomorphism with dg-flat cokernel (see e.g., \cite{kelly2024flat} Definition 4.23). A dg-flat object is, in particular, degree-wise flat. Hence, cofibrations are pure monomorphisms by \cite{kelly2016homotopy} Lemma 2.76. By \cite{kelly2016homotopy} Proposition 4.49, there is a generating set of (acyclic) cofibrations as required by axioms (4) and (5) of a basic Koszul category. By \cite{kelly2024flat} Proposition 3.28 (which is just a slightly more precise statement of what is proven in \cite{saorin2011exact} Proposition 2.7 (2)) and Corollary \ref{cor:exteflatpure}, one can choose the domains and codomains of the generating (acyclic) cofibrations to themselves be flat objects. Consequently, dg-flat objects are $K$-flat.

   % By \cite{kelly2024flat} Proposition 3.28 (that is just a slightly more precise statement of what %is proven in \cite{saorin2011exact} Proposition 2.7 (2)), Corollary \ref{cor:exteflatpure} 
    
   % In this model structure the weak equivalences are the quasi-isomorphisms. Moreover, by the same result, the model structure is combinatorial, and $Ch(\mathpzc{E})$ is a monoidal model category satisfying the monoid axiom.  of a basic Koszul category is furnished by \cite{kelly2016homotopy} Proposition 4.49. 
    
  %The cofibrations

  %  such that the flat cotorsion pair $(\mathrm{Flat},\mathrm{Cotor})$ is a dg-compatible cotorsion pair on $Ch(\mathpzc{E})$, in the sense of \cite{kelly2016homotopy} Definition 4.38. Then, for any commutative monoid object $R$ in $Ch(\mathpzc{E})$, ${}_{R}\mathpzc{Mod}(Ch(\mathpzc{E}))$ is a $KC$-monoidal, pure Koszul category. 

    This includes myriad examples:
    \begin{enumerate}
        \item 
        The category $\mathcal{O}_{X}\mathpzc{Mod}(Ch(Shv(X)))$ of $\mathcal{O}_{X}$-modules in sheaves on a topological space $X$ (in fact, one can take sheaves valued in a monoidal elementary exact category by \cite{kelly2024flat}).
        \item 
        Let $\mathcal{X}$ be a geometric stack over a ring $R$. By \cite{estrada2014derived} Theorem 8.1, $Ch(QCoh(\mathcal{X}))$ may be equipped with the flat model structure. This includes, for example, global quotient stacks.
    \end{enumerate}
\end{example}

\begin{example}\label{ex:pureflatmodel}
   
    Let $\mathpzc{E}$ be any closed symmetric monoidal exact category with exact filtered colimits. Consider the category $\mathpzc{E}_{\otimes}$, which has the same underlying category as $\mathpzc{E}$, but is equipped with the pure exact structure. Then every object is flat, so we can regard $\mathpzc{E}_{\otimes}$ as a basic Koszul category, when we equip $Ch(\mathpzc{E}_{\otimes})$ with the flat model structure. Note that this coincides with the injective model structure (wherein every object is cofibrant). If $\mathpzc{E}$ already had enough flats, then we can also endow $Ch(\mathpzc{E})$ with the flat model structure. 
Consider the adjunction
$$\adj{Id}{Ch(\mathpzc{E})}{Ch(\mathpzc{E}_{\otimes})}{Id}.$$
This is a Quillen adjunction. Indeed, a cofibration on the left-hand side is an admissible monomorphism with dg-flat cokernel. In particular, such a morphism is degree-wise pure. Thus, the left-hand identity functor preserves cofibrations. Moreover, a trivially cofibrant object in the flat model structure is necessarily acyclic in the pure exact structure. Therefore, the left-hand identity functor also preserves acyclic cofibrations. Clearly, the right-hand identity functor is exact. It follows that the adjunction above is a Quillen coreflection.
\end{example}

\begin{example}\label{example:topalg}
Here we give examples related to topological algebras. We will revisit these examples in detail in future work.
\begin{enumerate}
    \item 
    Consider a category of the form $Pro(Norm_{k})$ where $k$ is either a spherically complete non-Archimedean Banach field or $k=\mathbb{C}$. This category can be equipped with the injective tensor product $\otimes_{\epsilon}$ (\cite{MR1888309} Chapter 3). The category $Norm_{k}$ has enough injectives, and these injectives are strongly flat (\cite{houzel2006seminaire}, page 212, Theorem 1). However, the tensor product of injectives need not be injective. Now, $Norm_{k}^{op}$ has enough projectives. By \cite{qacs} Proposition 2.1.17, $Ind(Norm_{k}^{op})\cong Pro(Norm_{k})^{op}$ is an elementary quasi-abelian category. Since the projectives are flat, $Ch(Ind(Norm_{k}^{op}))$ is a $K$-monoidal Koszul category. However, it is only a weak h-monoidal model category, not a monoidal model category, as the tensor product of two projectives is not projective. To solve this, we can instead consider the flat model structure on $Ch(Ind(Norm_{k}^{op}))$, and this model structure is $KC$-monoidal.
\item

Consider instead $Pro(Ban_{k})$, with monoidal structure given by extension by projective limits of the completed injective tensor product $\hat{\otimes}_{\epsilon}$ on Banach spaces. In this case we have to be slightly careful. Injectives are still flat, but tensoring with them does not preserve kernels, due to the classical fact that the completion functor does not preserve kernels. Let $\mathcal{I}$ be a cogenerating category of Banach injectives. It can be taken to be $\mathcal{I}=\{\ell^{\infty}(V):V\textrm{ is a set}\}$, where $\ell^{\infty}(V)$ is the space of bounded sequences indexed by $V$, and we should also restrict the allowed sizes of the sets for smallness. The category $Pro(Ban_{k})$ is quasi-abelian by \cite{dcfapp} Proposition 3.3.2. It has two natural abelian embeddings, the left heart and the right heart (see \cite{qacs} Section 1.2). For pro-objects, it turns out that the right heart is more suitable. By passing to opposite categories, \cite{qacs} Proposition 2.1.14 says that the right heart of the quasi-abelian category $Pro(Ban_{k})$, denoted $RH(Pro(Ban_{k}))$, is $(\mathcal{P}_{\Sigma}(\mathcal{I}^{op}))^{op}$. Here $\mathcal{P}_{\Sigma}(-)$ stands for the $1$-categorical free-sifted cocompletion. By the opposite of \cite{qacs} Proposition 2.1.17, we have that $RH(Pro(Ban_{k}))\cong Pro(RH(Ban_{k}))$. By the opposite of \cite{qacs} Proposition 1.3.24, $RH(Ban_{k})$ also has enough injectives, and in fact, all of these come from the image of an injective in $Ban_{k}$, under the canonical inclusion $i:Ban_{k}\rightarrow RH(Ban_{k})$. Finally, $Pro(Ban_{k})$ and $RH(Pro(Ban_{k}))$ then have enough injectives. Indeed, if $\lim_{\mathcal{J}}X_{j}$ is an object of either pro-category, then we can pick an admissible monomorphism $X_{j}\rightarrow I_{j}$ for each $j$, where $I_{j}$ is an injective object of $Ban_{k}$. Then there is an admissible monomorphism $\lim_{\mathcal{J}}X_{i}\rightarrow\prod_{j\in\mathrm{Ob}(\mathcal{J})}I_{j}$, and the object on the right-hand side is injective in the pro-categories. Finally, we observe that projective limits, and hence countable products, are exact in the pro-categories.  By the opposite of \cite{kelly2016homotopy} Theorem 4.59, we can therefore equip both $Ch(Pro(Ban_{k}))$ and $Ch(RH(Pro(Ban_{k})))$ with the injective model structures. As both pro-categories have the same injectives, these model categories are Quillen equivalent. We can then define a monoidal structure on $RH(Pro(Ban_{k}))$ by extending the functor $\hat{\otimes}_{\epsilon}:\mathcal{I}\times\mathcal{I}\rightarrow Pro(Ban_{k})\rightarrow RH(Pro(Ban_{k}))$ by cosifted limits. By construction this functor commutes with all cosifted limits. It is also an additive functor, so it commutes with finite products. Hence it commutes with all limits, and by the adjoint functor theorem, the tensor product functor has a \textit{left adjoint} in each variable. This monoidal structure is therefore coclosed. Then $Ch(Pro(RH(Ban_{k}))^{op})$ is a Koszul category. It is neither $K$-monoidal nor $C$-monoidal.
\item

Over $\mathbb{C}$, the opposite of the monoidal category $RH(Pro(Ban_{\mathbb{C}}))$ is closely related to the symmetric monoidal category $\mathcal{M}Cpl$ of $\mathcal{M}$-complete vector spaces of Clausen and Scholze (\cite{clausenscholze2} Lecture IV). This is a cocomplete category by \cite{clausenscholze2} Corollary 4.15, and is a quasi-abelian category by \cite{stempfhuber2024bornological} Corollary 8.3.7. By \cite{stempfhuber2024bornological} Theorem 8.3.6 there is an equivalence
$$Ind^{m}(Smith)\cong\mathcal{M}Cpl$$
between the free cocompletion of the category $Smith$ of Smith spaces (see \cite{stempfhuber2024bornological} for a definition) by monomorphic filtered colimits, and $\mathcal{M}Cpl$. The category $Smith$ has enough projectives, given by objects of the form $\mathcal{M}(S)$ where $S$ ranges over $\kappa$-small extremally disconnected profinite sets for some fixed strongly inaccessible $\kappa$ (see \cite{clausenscholze2} Lecture IV and its appendix, or \cite{stempfhuber2024bornological} Section 8.3.2 for details here). Moreover, by \cite{clausenscholze1} Proposition 4.9 and Proposition 4.12, the category $\mathcal{M}Cpl$ is equipped with a symmetric monoidal structure $\otimes_{\pi}$ that restricts to one on $Smith$, and commutes with colimits in each variable. All this means that the category $\mathcal{M}Cpl$ is a monoidal quasi-elementary exact category when equipped with the quasi-abelian exact structure. Its left heart is then $\mathcal{P}_{\Sigma}(\{\mathcal{M}(S)\})$, where the collection ranges over all $\kappa$-small extremally disconnected $S$. 
We consider the left adjoint functor
$$C:LH(Ind^{m}(Smith))\rightarrow Ind^{m}(Smith)$$
to the fully faithful inclusion 
$$i:Ind^{m}(Smith)\rightarrow LH(Ind^{m}(Smith))$$
given by computing cokernels. Define a symmetric monoidal structure $\otimes$ on $\mathcal{P}_{\Sigma}(\{\mathcal{M}(S)\})$ as follows. On $\{\mathcal{M}(S)\}$, it is defined to be 
$$i((-)\otimes_{\pi}(-)),$$
and then we extend by sifted colimits. By construction, $C$ is strongly monoidal for these monoidal structures. By \cite{kelly2016homotopy} Proposition 4.63, the adjunction
$$\adj{C}{Ch(\mathcal{P}_{\Sigma}(\{\mathcal{M}(S)\}))}{Ch(Ind^{m}(Smith))}{i}$$
is a Quillen equivalence when both sides are equipped with the projective model structures, and it is even a monoidal Quillen equivalence. As mentioned in \cite{clausenscholze2} Theorem 3.8, a theorem of Smith \cite{smith1952pontrjagin} states that the category $Ban_{\mathbb{C}}$ is equivalent to $Smith^{op}$ via a certain construction called \textit{stereotype duality}. Under this duality, a projective $\mathcal{M}(S)$ is sent to an injective Banach space. Moreover, the collection $\mathcal{I}$ of stereotype duals of Smith spaces of the form $\mathcal{M}(S)$ cogenerates $Ban_{\mathbb{C}}$. Therefore, $LH(\mathcal{M}Cpl)^{op}\cong LH(Ind^{m}(Smith))^{op}\cong RH(Pro(Ban_{\mathbb{C}}))$. Finally, this is even a strongly monoidal equivalence by \cite{clausenscholze2} Proposition 4.12. Note that in this example it is not clear how to boost this monoidal structure to an $\infty$-categorical monoidal structure on the $\infty$-categorical localisation of $Ch(LH(\mathcal{M}Cpl))$. In forthcoming work with Ben-Bassat and Hekking we will address this.
\end{enumerate}
\end{example}
The second example above is the main reason for not assuming the pushout-product axiom for cofibrations.

In both of the above examples, the weak equivalences in the model structures are quasi-isomorphisms in the underlying exact category. There is an example that we will focus on in a future paper in which the homotopical structure will be a localisation of the exact one. Let $\mathcal{I}$ be a small category, and let $\mathpzc{ExCat}^{L}$ denote the $2$-category of exact categories whose morphisms are left adjoints with exact right adjoints. By a presheaf of exact categories on $\mathcal{I}$, we mean a pseudofunctor
$$F:\mathcal{I}^{op}\rightarrow\mathpzc{ExCat}^{L}.$$
Suppose that each $F(i)$ is elementary. By equipping each $Ch(F(i))$ with the projective model structure, we get a \textit{left Quillen presheaf}
$$Ch(F):\mathcal{I}^{op}\rightarrow\mathpzc{Mod}^{L}$$
in the sense of \cite{MR2771591} Definition 2.21,
where $\mathpzc{Mod}^{L}$ is the $2$-category whose objects are model categories, and whose morphisms are left Quillen functors. The category $\mathpzc{Sect}^{L}(Ch(F))$ of sections of the Grothendieck construction (\cite{MR2771591} Definition 2.21) may be equipped with a model category structure that presents the homotopy limit of this presheaf of model categories (\cite{MR2771591} Theorem 5.25). Note that \cite{MR2771591} Theorem 5.25 only furnishes a right model structure in general, as it is constructed as a right Bousfield localisation; however, a full model structure is guaranteed in our setting as the model category that we must localise is itself right (and indeed also left) proper. For various monoidal structures on $\mathpzc{Sect}^{L}(Ch(F))$, we will show that this is a Koszul category. This, for example, will include quasi-coherent sheaves on (pre-)stacks.

\subsection{Strict $t$-Structures on Koszul Categories}

In order to formulate a version of connective Koszul duality it will be convenient to have a compatible notion of $t$-structure.

\begin{defn}
A \textbf{strict }$t$-\textbf{structure} on a stable combinatorial model category $\mathpzc{M}$ is, for each $n\in\mathbb{Z}$, a collection of subcategories
$$\ldots\mathpzc{M}_{\ge n+1}\subset\mathpzc{M}_{\ge n}\subset\ldots\subset\mathpzc{M}$$
such that
\begin{enumerate}
\item
each $\mathpzc{M}_{\ge n}$ is a combinatorial model category,
\item
each inclusion $\mathpzc{M}_{\ge n}\rightarrow\mathpzc{M}$ is a Quillen coreflection and preserves equivalences, and
\item
there is a $t$-structure $(\textbf{M}_{\ge0},\textbf{M}_{\le0})$ on the $(\infty,1)$-category $\mathrm{L^{H}}(\mathpzc{M})$ (see \cite{higheralgebra} Chapter 1.2) such that the essential image of the inclusion $\mathrm{L^{H}}(\mathpzc{M}_{\ge n})\rightarrow\mathrm{L^{H}}(\mathpzc{M})$ is $\textbf{M}_{\ge n}$.
\end{enumerate}
We will write the data of a strict $t$-structure on $\mathpzc{M}$ as a tuple $(\mathpzc{M}_{\ge n})_{n\in\mathbb{Z}}$.
%If $\mathpzc{M}$ is a monoidal model category, then a strict $t$-structure on $\mathpzc{M}$ is said to be \textbf{monoidal} if the functor
%$$\otimes:\textbf{M}_{\ge m}\times\textbf{M}_{\ge n}\rightarrow\textbf{M}$$
%factors through $\textbf{M}_{\ge m+n}$.
\end{defn}

We also need some compatibility with the monoidal structure.

\begin{defn}
Let $\mathpzc{M}$ be a model category and $\otimes$ a unital symmetric monoidal functor on $\mathpzc{M}$. A strict $t$-structure $(\mathpzc{M}_{\ge n})_{n\in\mathbb{Z}}$ on $\mathpzc{M}$ is said to be \textbf{compatible with the monoidal structure} if 
\begin{enumerate}
\item
the monoidal unit is contained in $\mathpzc{M}_{\ge0}$,
\item
and $X\otimes Y\in\mathpzc{M}_{\ge0}$ for any $X,Y\in\mathpzc{M}_{\ge 0}$ (cf. \cite{raksit2020hochschild} Definition 3.3.1).
\end{enumerate}
\end{defn}

Let us give some relevant examples. Let $\mathpzc{E}$ be a basic Koszul category, and let $Ch_{\ge n}(\mathpzc{E})$ denote the full subcategory of $Ch(\mathpzc{E})$ consisting of objects concentrated in degrees $\ge n$. There is an obvious inclusion functor $Ch_{\ge n}(\mathpzc{E})\rightarrow Ch(\mathpzc{E})$. It has both a left and right adjoint, but we will only be concerned with the right adjoint $\tau_{\ge n}$ defined in the introduction.

\begin{thm}\label{connective}
Consider the adjunction
$$\adj{i_{\ge n}}{Ch_{\ge n}(\mathpzc{E})}{Ch(\mathpzc{E})}{\tau_{\ge n}}.$$
The left-transferred model structure exists on $Ch_{\ge n}(\mathpzc{E})$. Moreover, for $n=0$, the underlying category is monoidal, and the model category structure satisfies the $pp$-monoid axiom, provided $Ch(\mathpzc{E})$ does.
\end{thm}

\begin{proof}
We are going to use \cite{MR3380069} Theorem 2.23. Let $f:X\rightarrow Y$ be a morphism in $Ch_{\ge n}(\mathpzc{E})$ that has the right lifting property against all morphisms in $Ch_{\ge n}(\mathpzc{E})$ that are cofibrations in $Ch(\mathpzc{E})$. We claim that this morphism is in fact an acyclic fibration in $Ch(\mathpzc{E})$. We need to check that such a morphism has the right lifting property with respect to all morphisms in the set of generating cofibrations
$$\mathfrak{G}=\{S^{m}(G_{i})\rightarrow D^{m+1}(G_{i})\}_{m\in\Z,i\in\mathcal{I}}\cup\{0\rightarrow D^{m}(G_{i})\}_{m\in\Z,i\in\mathcal{I}}\cup\{S^{m}(f_{l})\}_{m\in\Z,l\in\mathcal{L}}.$$
For $m\ge n$ the morphisms $S^{m}(G_{i})\rightarrow D^{m+1}(G_{i}),\;0\rightarrow D^{m+1}(G_{i})$, and $S^{m}(f_{l})$ are in $Ch_{\ge n}(\mathpzc{E})$, so $f$ clearly has the right lifting property against them. Finding a lift in either of the commutative diagrams
\begin{displaymath}
 \xymatrix{
0\ar[d]\ar[r] & X\ar[d]^{f}\\
D^{n}(G_{i})\ar[r] & Y,
 }   
\end{displaymath}
or
\begin{displaymath}
 \xymatrix{
S^{n-1}(G_{i})\ar[d]\ar[r] & X\ar[d]^{f}\\
D^{n}(G_{i})\ar[r] & Y
 }   
\end{displaymath}
is equivalent to finding a lift in a commutative diagram of the form
\begin{displaymath}
 \xymatrix{
0\ar[d]\ar[r] & X\ar[d]^{f}\\
S^{n}(G_{i})\ar[r] & Y.
 }   
\end{displaymath}

Indeed, this follows because $X$ and $Y$ are concentrated in degrees $\ge n$. As we have seen, we can always find such a lift. When $m<n-1$, the lifting problem is trivial, because for such $m$, and for any degree $k$ such that either $X_{k}$ or $Y_{k}$ is non-zero, the domain and codomain of the generating cofibrations are zero in that degree.
For $n=0$, the monoidal structure on $Ch(\mathpzc{E})$ restricts to one on $Ch_{\ge0}(\mathpzc{E})$. The $pp$-monoid axiom simply follows from the facts that $i_{\ge0}$ is strongly monoidal, commutes with all colimits, reflects both cofibrations and acyclic cofibrations, and that $Ch(\mathpzc{E})$  satisfies the pp-monoid axiom.
\end{proof}

By \cite{henrard2021left} Section 3, there is a $t$-structure on $Ch(\mathpzc{E})$, called the \textbf{left }$t$-\textbf{structure}. Moreover, the functors $\tau_{\ge n}$ are the truncation functors for this $t$-structure. It follows that
$$(Ch_{\ge n}(\mathpzc{E}))_{n\in\mathbb{Z}}$$
is a strict $t$-structure on $Ch(\mathpzc{E})$.

\begin{prop}
    Let $\mathpzc{E}$ be a basic Koszul category and let $g:X\rightarrow Y$ be an acyclic fibration in $Ch(\mathpzc{E})$. Then $i_{\ge n}\circ\tau_{\ge n}(g)$ is an acyclic fibration in $Ch(\mathpzc{E})$.
\end{prop}

\begin{proof}
    We need to check the right lifting property against all generating cofibrations. Now, since $\tau_{\ge n}$ is right Quillen, the morphism $\tau_{\ge n}(g)$ has the right lifting property against all morphisms
    $$\mathfrak{G}=\{S^{m}(G_{i})\rightarrow D^{m+1}(G_{i})\}_{m\ge n,i\in\mathcal{I}}\cup\{0\rightarrow D^{m}(G_{i})\}_{m> n,i\in\mathcal{I}}\cup\{S^{m}(f_{l})\}_{m\ge n,l\in\mathcal{L}}.$$

The morphism $i_{\ge n}\circ\tau_{\ge n}(g)$ also clearly has the right lifting property against all morphisms of the form
$$\mathfrak{G}=\{S^{m}(G_{i})\rightarrow D^{m+1}(G_{i})\}_{m<n-1,i\in\mathcal{I}}\cup\{0\rightarrow D^{m}(G_{i})\}_{m\le n,i\in\mathcal{I}}\cup\{S^{m}(f_{l})\}_{m<n,l\in\mathcal{L}}.$$
It remains to check the right lifting property against morphisms 
$$S^{n-1}(G_{i})\rightarrow D^{n}(G_{i})$$
and 
$$0\rightarrow D^{n}(G_{i}).$$
But, as in the theorem above, this amounts to finding a lift in the diagram
\begin{displaymath}
 \xymatrix{
0\ar[d]\ar[r] & \tau_{\ge n}X\ar[d]^{\tau_{\ge n}g}\\
S^{n}(G_{i})\ar[r] & \tau_{\ge n}Y.
 }   
\end{displaymath}
This in turn is equivalent to finding a lift in the diagram
\begin{displaymath}
 \xymatrix{
0\ar[d]\ar[r] & ker(d_{n}^{X})\ar[d]^{g_{n}}\\
G_{i}\ar[r] & ker(d_{n}^{Y}).
 }   
\end{displaymath}
Finally, this is equivalent to a lift in the diagram
\begin{displaymath}
 \xymatrix{
0\ar[d]\ar[r] & X\ar[d]^{g}\\
S^{n}(G_{i})\ar[r] & Y,
 }   
\end{displaymath}
which we know to exist since $g$ is an acyclic fibration.
\end{proof}

%\begin{defn}
%A monoidal model category $\mathpzc{M}_{\ge n}$ of the form ${}_{R}\mathpzc{Mod}(Ch(\mathpzc{E})_{\ge n})$ where $\mathpzc{E}$ is a connective basic Koszul category, $Ch(\mathpzc{E})$ is equipped with the corresponding model structure, $Ch(\mathpzc{E})_{\ge n}$ is equipped with the right-transferred model structure $R\in\mathpzc{Alg}_{\mathfrak{Comm}}(Ch_{\ge n}(\mathpzc{E}))$, and  ${}_{R}\mathpzc{Mod}(Ch(\mathpzc{E})_{\ge n})$ is equipped with the left-transferred model structure, and  $n\in\Z$ is called an $n$-\textbf{connective Koszul category}.
%\end{defn}

\subsection{Filtrations and Gradings in Koszul Categories}

Many of the algebraic objects we study will be equipped with filtrations.
In \cite{kelly2016homotopy} Chapter 5 we studied in detail categories of filtered objects in exact categories and in monoidal exact categories. Fix a class of morphisms $\mathcal{S}$ in a category $\mathpzc{E}$. Let us recall the definition of a filtered object. 
\begin{defn}
Let $A$ be an object of a category $\mathpzc{C}$ and $\mathcal{S}$ a class of morphisms in $\mathpzc{C}$. An $\mathcal{S}$-\textbf{subobject} of $A$ is a morphism $s:A'\rightarrow A$ in $\mathcal{S}$. An $\mathcal{S}$-\textbf{filtration} of $A$ consists of a collection of $\mathcal{S}$-subobjects of $A$, $\{\alpha_{i}:A_{i}\rightarrow A\}_{i\in\mathbb{N}_{0}}$, together with morphisms $a_{i}:A_{i}\rightarrow A_{i+1}$ in $\mathcal{S}$ such that $\alpha_{i+1}\circ a_{i}=\alpha_{i}$. An $\mathcal{S}$-\textbf{filtered object} of $\mathpzc{C}$ is a tuple of data $(A_{top},\alpha_{i},a_{i})$ where $A_{top}$ is an object of $\mathpzc{C}$ and $(\alpha_{i},a_{i})$ is an $\mathcal{S}$-filtration of $A_{top}$. A \textbf{morphism of} $\mathcal{S}$-\textbf{filtered objects} $g:(A_{top},\alpha_{i},a_{i})\rightarrow(B_{top},\beta_{i},b_{i})$ consists of a collection of morphisms $\{g_{i}:A_{i}\rightarrow B_{i}\}_{i\in\mathbb{N}_{0}}$, and $g_{top}:A_{top}\rightarrow B_{top}$ such that $g_{i+1}\circ a_{i}=b_{i}\circ g_{i}$ and $g_{top}\circ \alpha_{i}=\beta_{i}\circ g_{i}$ for all $i\in\mathbb{N}_{0}$. The collection of $\mathcal{S}$-filtered objects and morphisms of $\mathcal{S}$-filtered objects can then be organised into an additive category $\mathpzc{Filt}_{\mathcal{S}}(\mathpzc{C})$.
\end{defn}

\begin{defn}
A filtered object $(A_{top},\alpha_{i},a_{i})$ of a category $\mathpzc{C}$ is said to be \textbf{exhaustive} if $A_{top}$ together with the morphisms $\alpha_{i}:A_{i}\rightarrow A_{top}$ is a direct limit of the diagram
\begin{displaymath}
\xymatrix{
A_{0}\ar[r]^{a_{0}} & A_{1}\ar[r]^{a_{1}} & A_{2}\ar[r] &\ldots.
}
\end{displaymath}
The full subcategory of $\mathpzc{Filt}_{\mathcal{S}}(\mathpzc{C})$ on objects equipped with an exhaustive filtration will be denoted by $\overline{\mathpzc{Filt}}_{\mathcal{S}}(\mathpzc{C})$. 
\end{defn}

\begin{rem}

    Note that, in a countably cocomplete category $\mathpzc{C}$ with a class $\mathcal{S}$ of morphisms closed under countable transfinite composition, we may identify $\overline{\mathpzc{Filt}}_{\mathcal{S}}(\mathpzc{C})$ with the category of functors $\mathpzc{Fun}_{\mathcal{S}}(\mathbb{N}_{0}^{\le},\mathpzc{C})$ from the poset $\mathbb{N}_{0}^{\le}$ of natural numbers (including $0$), equipped with the usual increasing ordering, such that $F(n\rightarrow n+1)\in\mathcal{S}$ for all $n$. Writing such a functor as $A=(A_{n},a_{n})$, with $a_{n}:A_{n}\rightarrow A_{n+1}$ a morphism in $\mathcal{S}$, the equivalence sends $(A_{n},a_{n})$ to $(\colim_{n}A_{n},\alpha_{n},a_{n})$, where $\alpha_{n}:A_{n}\rightarrow\colim_{n}A_{n}$ is the canonical morphism.
\end{rem}

There is a functor
$$gr:\mathpzc{Filt}_{\mathcal{S}}(\mathpzc{E})\rightarrow\mathpzc{Gr}_{\mathbb{N}_{0}}(\mathpzc{E}),$$
called the \textbf{associated graded functor}, of which we will make extensive use. It sends a filtered object $(A_{top},\alpha_{i},a_{i})$ to $\bigoplus_{n=0}^{\infty}coker(A_{n-1}\rightarrow A_{n})=\bigoplus_{n=0}^{\infty}gr_{n}(A)$, where $coker(A_{n-1}\rightarrow A_{n})$ is defined by the pushout diagram
\begin{displaymath}
    \xymatrix{
    A_{n-1}\ar[d]\ar[r] & A_{n}\ar[d]\\
    \bullet\ar[r] & coker(A_{n-1}\rightarrow A_{n}).
    }
\end{displaymath}
(and we interpret $coker(A_{-1}\rightarrow A_{0})=A_{0}$).
Here, $\bullet$ is the terminal object.

We also let $(-)_{top}:\mathpzc{Filt}_{\mathcal{S}}\rightarrow\mathpzc{M}$ denote the functor sending a filtered object $(A_{top},\alpha_{i},a_{i})$ to $A_{top}$.

\subsubsection{Infinity Categories of Filtered and Graded Objects}

 Let $\mathpzc{M}$ be a model category. If $\mathcal{W}$ is the wide subcategory of $\mathpzc{M}$ on weak equivalences, then $(\mathpzc{M},\mathcal{W})$ is a relative category. Recall that $\mathrm{L^{H}}(\mathpzc{M})$ denotes the $(\infty,1)$-categorical localisation of $\mathpzc{M}$ at $\mathcal{W}$. Consider the category $\mathpzc{Filt}_{\mathcal{S}}(\mathpzc{M})$ where $\mathcal{S}$ is some class of monomorphisms. Let $\mathcal{W}_{f}$ denote the class of filtered weak equivalences -- that is, morphisms $f:(A_{top},\alpha_{i},a_{i})\rightarrow (B_{top},\beta_{i},b_{i})$ of filtered objects such that $f_{top}$ is a weak equivalence, and each $f_{i}$ for $0\le i<\infty$ is a weak equivalence. Then $(\mathpzc{Filt}_{\mathcal{S}}(\mathpzc{M}),\mathcal{W}_{f})$ is also a relative category. Denote the $(\infty,1)$-categorical localisation of this homotopical category by $\textbf{Filt}_{\mathcal{S}}(\mathrm{L^{H}}(\mathpzc{M}))$. The category $\overline{\textbf{Filt}}_{\mathcal{S}}(\mathrm{L^{H}}(\mathpzc{M}))$ is defined similarly.

\begin{rem}\label{rem:htpfiltqis}

    Suppose that $\mathpzc{M}$ is weakly $\mathcal{S}$-elementary as a model category. Suppose that $f:(A_{top},\alpha_{n},a_{n})\rightarrow(B_{top},\beta_{n},b_{n})$ is a filtered morphism between exhaustively $\mathcal{S}$-filtered objects, where each $f_{n}$ is a weak equivalence. Then $f_{top}:A_{top}\rightarrow B_{top}$ is a weak equivalence.
\end{rem}

\begin{rem}
      Suppose that $\mathpzc{M}$ is weakly $\mathcal{S}$-elementary as a model category, and that $\mathcal{S}$ contains the class of cofibrations. Let $(A_{top},\alpha_{n},a_{n})$ be an exhaustively $\mathcal{S}$-filtered object. Then the natural morphism $\mathbb{L}\colim_{n}A_{n}\rightarrow \colim_{n}A_{n}$ is an equivalence. In particular, if $f:(A_{top},\alpha_{n},a_{n})\rightarrow (B_{top},\beta_{n},b_{n})$ is a morphism between two such objects, such that $f_{n}$ is an equivalence for all $0\le n<\infty$, then $f_{top}:A_{top}\rightarrow B_{top}$ is also an equivalence. In other words, the identification between $\mathpzc{Fun}_{\mathcal{S}}(\mathbb{N}_{0}^{\le},\mathpzc{M})$ and $\overline{\mathpzc{Filt}}_{\mathcal{S}}(\mathpzc{M})$ also identifies the relative category structures on both sides, where on the left-hand side, a natural transformation is a weak equivalence precisely if it is a component-wise weak equivalence.
\end{rem}

\begin{rem}
       Let $\mathpzc{M}$ be a combinatorial model category. Suppose that $\mathcal{S}$ contains the class of cofibrations. Consider the category $\mathpzc{Fun}(\mathbb{N}_{0}^{\le},\mathpzc{M})$ of functors from the poset of non-negative integers to $\mathpzc{M}$. Equip this with the projective model structure for functors. The central result of \cite{gwilliam2018enhancing} is that this gives a model for the $(\infty,1)$-category of filtered objects $\textbf{Fun}(N(\mathbb{N}_{0}^{\le}),\mathrm{L^{H}}(\mathpzc{M}))$ (where $N(-)$ denotes the nerve). For any cofibrant object $(C_{n})$, each morphism $C_{n}\rightarrow C_{n+1}$ is a cofibration, and hence is in $\mathcal{S}$. Thus, the cofibrant replacement functor in fact lands in $\overline{\mathpzc{Filt}}_{\mathcal{S}}(\mathpzc{M})$. Hence, we in fact get $\textbf{Fun}(N(\mathbb{N}_{0}^{\le}),\mathrm{L^{H}}(\mathpzc{M}))\cong\mathrm{L^{H}}(\overline{\mathpzc{Filt}}_{\mathcal{S}}(\mathpzc{M}))$.
\end{rem}

 As in Chapter 5 in \cite{kelly2016homotopy}, the functors $(-)_{n}:\mathpzc{Filt}_{\mathcal{S}}(\mathpzc{M})\rightarrow \mathpzc{M}$ for $0\le n<\infty$, and $(-)_{top}:\mathpzc{Filt}_{\mathcal{S}}(\mathpzc{M})\rightarrow \mathpzc{M}$ all induce functors of $(\infty,1)$-categories after localisation. 
 
 Now, consider the category $\mathpzc{Gr}_{\mathbb{N}_{0}}(\mathpzc{M})=\mathpzc{Fun}(\mathbb{N}_{0},\mathpzc{M})$, where $\mathbb{N}_{0}$ is regarded as a discrete category. Denote by $\mathcal{W}_{gr}$ the class of graded weak equivalences in $\mathpzc{Gr}_{\mathbb{N}_{0}}(\mathpzc{M})$. Then $(\mathpzc{Gr}_{\mathbb{N}_{0}}(\mathpzc{M}), \mathcal{W}_{gr})$ is also a homotopical category. In fact, it is naturally a combinatorial model category, in which all three classes of morphisms are defined level-wise. Provided that $\mathcal{S}$ consists of left proper morphisms, the functor $gr:\mathpzc{Filt}_{\mathcal{S}}(\mathpzc{M})\rightarrow \mathpzc{Gr}_{\mathbb{N}_{0}}(\mathpzc{M})$ also induces a functor of $(\infty,1)$-categories after localisation. 
 
 %If $\mathpzc{M}$ is weakly $\mathcal{S}$-elementary, then $\overline{\textbf{Filt}}_{\mathcal{S}}(\mathrm{L^{H}}(\mathpzc{M}))$ may be identified with the localisation of the full subcategory of $\mathpzc{Filt}_{\mathcal{S}}$ consisting of those objects $(A_{n},a_{n},\alpha_{n})$ such that the natural morphism $colim_{n}A_{n}\rightarrow A_{top}$ is an isomorphism. Under this equivalence, the functor $(-)_{top}\cong\colim_{n}$ is homotopical, and induces a functor
%$$\overline{\textbf{Filt}}_{\mathcal{S}}(\mathrm{L^{H}}(\mathpzc{M}))\rightarrow\mathrm{L^{H}}(\mathpzc{M}).$$

\subsubsection{Specialisation to Koszul Categories}

%{prop:adminleftprop}

Let us now specialise to the case where $\mathpzc{M}={}_{R}\mathpzc{Mod}(Ch(\mathpzc{E}))$ is a Koszul category. We consider a class $\mathcal{S}$ of admissible monomorphisms in $\mathpzc{M}$.

Constructing limits and colimits in these categories of filtered objects is a somewhat subtle enterprise. However, for cokernels we have the following result, which is in fact all we will really need. Using the $3\times 3$ lemma for exact categories (\cite{Buehler} Corollary
3.6), one has the following.

\begin{prop}\label{filteredcokernels}
Let $g:(A_{top},\alpha_{i},a_{i})\rightarrow(B_{top},\beta_{i},b_{i})$ be a morphism in $\mathpzc{Filt}_{\textbf{AdMon}}(\mathpzc{M})$ where each $g_{i}$ for $0\le i\le\infty$ (where $\infty$ stands for $top$ here) is an admissible monomorphism. Suppose that, for each $0\le i\le\infty$, the morphisms $\gamma_{i}:coker(g_{i})\rightarrow coker(g_{top})$ and $\overline{b}_{i}:coker(g_{i})\rightarrow coker(g_{i+1})$ are admissible monomorphisms. Then the object $(coker(g_{top}),\gamma_{i},\overline{b}_{i})$ is a cokernel of $g$. If $g$ is a morphism in $\overline{\mathpzc{Filt}}_{\textbf{AdMon}}(\mathpzc{M})$ then $(coker(g_{i}),\overline{b}_{i})$ is also the cokernel in $\overline{\mathpzc{Filt}}_{\textbf{AdMon}}(\mathpzc{M})$. Finally, if each $g_{i}$ is a pure monomorphism, and each $\gamma_{i}$ is a pure monomorphism, then this is also the cokernel in $\mathpzc{Filt}_{\textbf{PureMon}}(\mathpzc{M})$ (resp.  $\overline{\mathpzc{Filt}}_{\textbf{PureMon}}(\mathpzc{M})$).
\end{prop}

\begin{defn}
 An object $(A_{top},\alpha_{i},a_{i})$ of $\overline{\mathpzc{Filt}}_{\mathcal{S}}(\mathpzc{M})$ is said to be \textbf{filtered cofibrant} if each  $A_{i}$ is cofibrant, and each of the morphisms $a_{i}:A_{i}\rightarrow A_{i+1}$ is a cofibration. A morphism $f:A\rightarrow B$ between filtered objects is said to be a \textbf{filtered cofibration} if it is an admissible monomorphism with filtered cofibrant cokernel and is a cofibration at each level of the filtration.
\end{defn}

\begin{prop}\label{prop:gradedfiltprop}

Let $\mathpzc{M}$ be a Koszul category, and let $f:A\rightarrow B$ be a morphism in $\mathpzc{Filt}_{\mathcal{S}}(\mathpzc{M})$.
\begin{enumerate}
\item 
The functor
$gr:\mathpzc{Filt}_{\mathcal{S}}(\mathpzc{M})\rightarrow\mathpzc{Gr}(\mathpzc{M})$
preserves weak equivalences.
\item 
If $f$ is a morphism in $\overline{\mathpzc{Filt}}_{\mathcal{S}}(\mathpzc{M})$, and $\mathpzc{M}$ is weakly $\mathcal{S}$-elementary (e.g., if $\mathcal{S}\subset\textbf{PureMon}$), then the functor $gr:\mathpzc{Filt}_{\mathcal{S}}(\mathpzc{M})\rightarrow\mathpzc{Gr}(\mathpzc{M})$ reflects weak equivalences.
\item 
If $f$ is a filtered cofibration, then each $gr_{n}(f)$ is a cofibration.
    \item 
If $\mathpzc{M}$ is strong Koszul, the morphism $f$ is a filtered cofibration if and only if $gr(f)$ is a degree-wise admissible monomorphism with graded cofibrant cokernel.
\end{enumerate}
\end{prop}

\begin{proof}
\begin{enumerate}
\item 
By Proposition \ref{prop:adminleftprop}, the functor
$$gr:\mathpzc{Filt}_{\mathcal{S}}(\mathpzc{M})\rightarrow\mathpzc{Gr}(\mathpzc{M})$$
preserves weak equivalences.
    \item 
  This follows exactly as in \cite{kelly2016homotopy} Proposition 5.73, using that the category of acyclic objects is thick.
    \item
    Suppose $f$ is a filtered cofibration. By Proposition \ref{filteredcokernels} $gr(f)$ is a degree-wise admissible monomorphism with graded cofibrant cokernel. 
    
    \item 
    This follows by an easy inductive argument using Corollary \ref{middlecofibrant}.
\end{enumerate}

\end{proof}

Later we will need the following useful result, inspired by part of the proof of Proposition 2.2.12 in \cite{lurie2011derived}.

\begin{prop}\label{inductgraded}

Let $\textbf{N}$ be an $(\infty,1)$-category, and $\mathcal{S}$ a class of admissible monomorphisms in $\mathpzc{M}$ containing all cofibrations. Suppose that $\mathpzc{M}$ is a weakly $\mathcal{S}$-elementary Koszul category, and $F:\textbf{N}\rightarrow\mathrm{L^{H}}(\overline{\mathpzc{Filt}}_{\mathcal{S}}(\mathpzc{M}))$ a functor. If $\textbf{gr}\circ F$ and $(-)_{0}\circ F$ preserve homotopy sifted colimits, then so does $F$. Moreover, the composite functor 
$$(-)_{top}\circ F:\textbf{N}\rightarrow\mathrm{L^{H}}(\mathpzc{M})$$
also commutes with homotopy sifted colimits.
\end{prop}

\begin{proof}
 It suffices to show that each $(-)_{n}\circ F$ preserves homotopy sifted colimits, for $1\le n<\infty$. The proof is an easy induction. By assumption it is true for $n=0$. We suppose it has been shown for $n=k$. Now, since admissible monomorphisms are left proper, there is a homotopy fibre sequence of functors 
$$(-)_{k}\circ F \rightarrow (-)_{k+1}\circ F \rightarrow \textbf{gr}_{k+1}\circ F.$$
Since the left- and right-hand functors preserve homotopy sifted colimits, so does the middle functor.

The final claim follows from the fact that $(-)_{top}\cong \colim_{n}\cong\mathbb{L}\colim_{n}$, and this functor preserves all colimits, since colimits commute with each other.
\end{proof}

%\begin{rem}
%In \cite{kelly2016homotopy} we denoted this functor by $(-)_{top}$. However, we shall be applying this functor to formally planar cooperads, for which the notation $\mathfrak{C}_{top}$ already has an meaning in terms of homotopy formally planar cooperads.
%\end{rem} 

%As in \cite{kelly2016homotopy} Section 5.3, provided that $\mathpzc{E}$ is a weakly $\textbf{AdMon}$-elementary exact category, we then identify $\overline{\mathpzc{Filt}}_{\textbf{AdMon}}(\mathpzc{M})\cong\mathpzc{Fun}(\mathbb{N}_{0}^{<},\mathpzc{M})$ with the category of diagrams over the poset $\mathbb{N}_{0}^{\le}$ of natural numbers with the increasing ordering, by identifying $(A_{top},\alpha_{i},a_{i})$ with 
%$$(A_{n},a_{i})=A_{0}\rightarrow A_{1}\rightarrow\ldots.$$ 

%Note that if $\mathpzc{M}=Ch(\mathpzc{E})$ then a morphism $f$ of filtered objects is a filtered cofibration if and only if it an admissible monomorphism with filtered cofibrant cokernel.

\begin{cor}
    Let $\mathpzc{M}$ be a hereditary Koszul category. If $f:X\rightarrow Y$ and $g:Y\rightarrow Z$ are morphisms in $\overline{\mathpzc{Filt}}_{\textbf{AdMon}}(\mathpzc{M})$ such that $g\circ f$ and $g$ are filtered cofibrations, then $f$ is a filtered cofibration.
\end{cor}

\begin{proof}
    We have $gr(g\circ f)=gr(g)\circ gr(f)$. Both $gr(g\circ f)$ and $gr(g)$ are graded cofibrations, i.e., each $gr_{n}(g\circ f)$ and each $gr_{n}(g)$ is a cofibration. Thus, each $gr_{n}(f)$ is a cofibration by Proposition \ref{23cofibhered}. By Proposition \ref{prop:gradedfiltprop} above, $f$ is a filtered cofibration.
\end{proof}

%\begin{cor}\label{twooutofthreefiltcofib}
%Let $\mathpzc{M}$ be a hereditary Koszul category. If $f:X\rightarrow Y$ and $g:Y\rightarrow Z$ are morphisms in $\overline{\mathpzc{Filt}}_{\textbf{AdMon}}(\mathpzc{M})$ such that $g\circ f$ and $g$ are filtered cofibrations then so is $f$. 
%\end{cor}

For $1\le j\le k$, let $A^{j}=(A_{top}^{j},\alpha^{j},a^{j})$ be in $\mathpzc{Filt}_{\textbf{RegMon}}(\mathpzc{M})$. We define an object $\bigotimes_{j=1}^{k}A^{j}$ of $\mathpzc{Filt}_{\textbf{RegMon}}(\mathpzc{M})$ by 
$$(\bigotimes_{j=1}^{k}A^{j})_{top}=\bigotimes_{j=1}^{k}A^{j}_{top}$$
and
$$(\bigotimes_{j=1}^{k}A^{j})_{n}=im(\bigoplus_{i_{1}+\ldots+i_{k}\le n}A^{1}_{i_{1}}\otimes\ldots\otimes A^{k}_{i_{k}}\rightarrow \bigotimes_{j=1}^{k}A^{j}_{top}).$$

\begin{rem}

   In general, this tensor product does not define a monoidal structure, due to a lack of associativity. However, it follows from Proposition \ref{prop:purepushoutsum} that, when restricted to $\mathpzc{Filt}_{\textbf{PureMon}}(\mathpzc{M})$, the tensor product is associative. The restriction to $\overline{\mathpzc{Filt}}_{\textbf{PureMon}}(\mathpzc{M})$ is also associative.  
\end{rem}

A useful result from \cite{kelly2016homotopy}, specifically Proposition 5.62, says that the associated graded functor is strongly monoidal when restricted to $\mathpzc{Filt}_{\textbf{PureMon}}(\mathpzc{M})$. We will make use of this in the sequel.

We let $\overline{\mathpzc{Filt}}_{\textbf{PureMon}}^{K}(\mathpzc{M})$ denote the full subcategory of $\overline{\mathpzc{Filt}}_{\textbf{PureMon}}(\mathpzc{M})$ consisting of filtered objects $A$ such that for each $i\in\mathbb{N}$, $ gr_{i}(A)$ is $K$-flat.

\subsubsection{$\Sigma_{n}$-Filtrations}
As discussed above, for associativity reasons, in general, the tensor product defined above does not endow $\overline{\mathpzc{Filt}}_{\textbf{RegMon}}(\mathpzc{M})$ with a monoidal structure. However, there is a natural $\Sigma_{n}$ action on the filtered object $\bigotimes_{j=1}^{n}A^{j}$ that is functorial. This brings us to a discussion of $\Sigma_{n}$-filtrations. Note that there is an obvious equivalence of categories between the category $\overline{\mathpzc{Filt}}_{\textbf{AdMon}}(\mathpzc{Mod}_{\Sigma_{n}})$ of filtered $\Sigma_{n}$-modules and the category of $\Sigma_{n}$-modules in $\overline{\mathpzc{Filt}}_{\textbf{AdMon}}(\mathpzc{Mod})$.

\begin{defn}
An object $(A_{top},\alpha_{i},a_{i})$ of $\overline{\mathpzc{Filt}}_{\textbf{AdMon}}(\mathpzc{Mod}_{\Sigma_{n}})$ is said to \textbf{have admissibly filtered coinvariants} if the morphisms $(\alpha_{i})_{\Sigma_{n}}$ and $(a_{i})_{\Sigma_{n}}$ obtained by applying the functor of coinvariants are admissible monomorphisms. The full subcategory of $\overline{\mathpzc{Filt}}_{\textbf{AdMon}}(\mathpzc{Mod}_{\Sigma_{n}})$ consisting of filtered $\Sigma_{n}$-modules that have admissibly filtered coinvariants is denoted $\overline{\mathpzc{Filt}}_{\textbf{AdMon}_{\Sigma_{n}}}(\mathpzc{Mod}_{\Sigma_{n}})$.
\end{defn}

The following is clear.

\begin{prop}\label{prop:coniditonsadfilt}
\begin{enumerate}
\item
Let $(A_{top},\alpha_{i},a_{i})$ be an object of $\overline{\mathpzc{Filt}}_{\textbf{AdMon}}(\mathpzc{Mod})$. Then the free object $(\Sigma_{n}\otimes A_{top},\Sigma_{n}\otimes\alpha_{i},\Sigma_{n}\otimes a_{i})$ has admissibly filtered coinvariants.
\item
The category $\overline{\mathpzc{Filt}}_{\textbf{AdMon}_{\Sigma_{n}}}(\mathpzc{Mod}_{\Sigma_{n}})$ is closed under taking summands in $\overline{\mathpzc{Filt}}_{\textbf{AdMon}}(\mathpzc{Mod}_{\Sigma_{n}})$.
\item
If $(A_{top},\alpha_{i},a_{i})$ is in $\overline{\mathpzc{Filt}}_{\textbf{AdMon}_{\Sigma_{n}}}(\mathpzc{Mod}_{\Sigma_{n}})$, then the natural morphism $gr(A)_{\Sigma_{n}}\rightarrow gr(A_{\Sigma_{n}})$ is an isomorphism.
\end{enumerate}
\end{prop}

Let us now set about defining a property of $\Sigma_{n}$-modules that guarantees admissibility of filtered coinvariants, and is reasonably stable under useful operations.

\begin{prop}\label{prop:pureInd}

 Let $H$ be a finite group, and $\phi:G\rightarrow H$ the inclusion of a subgroup. 
A sequence of $G$-modules
  $$0\rightarrow X\rightarrow Y\rightarrow Z\rightarrow 0$$
 is pure as an exact sequence in $\mathpzc{M}$ if and only if
    $$0\rightarrow Ind_{G}^{H}X\rightarrow Ind_{G}^{H}Y\rightarrow Ind_{G}^{H}Z\rightarrow 0$$
    is pure exact as a sequence in $\mathpzc{M}$.
\end{prop}

\begin{proof}
    As a sequence in $\mathpzc{M}$,
      $$0\rightarrow Ind_{G}^{H}X\rightarrow Ind_{G}^{H}Y\rightarrow Ind_{G}^{H}Z\rightarrow 0$$
      is just obtained by tensoring 
        $$0\rightarrow X\rightarrow Y\rightarrow Z\rightarrow 0$$
        with $\bigoplus_{[h]\in H\big\slash G}k$, where $k$ is the monoidal unit. If the former sequence is pure exact in $\mathpzc{M}$ then, as a retract of it (in $\mathpzc{M}$), the latter is pure exact as well. If the latter sequence is pure exact in $\mathpzc{M}$, the former must also be, as it is obtained from the latter by tensoring with an object. 
\end{proof}

\begin{defn}

    Let $G$ be a group, and consider the category ${}_{G}\mathpzc{Mod}$. We say that a sequence 
    $$0\rightarrow X\rightarrow Y\rightarrow Z\rightarrow 0$$
    in ${}_{G}\mathpzc{Mod}$
    is \textbf{coinvariant-pure exact} if
    \begin{enumerate}
        \item 
        it is pure exact as a sequence in $\mathpzc{M}$, and
        \item 
        it is $Triv$-pure exact, where $Triv\subset{}_{G}\mathpzc{Mod}$ consists of objects of $\mathpzc{M}$ equipped with the trivial action of $G$.
    \end{enumerate}
   
\end{defn}

\begin{prop}
     
     Consider a sequence
     $$0\rightarrow X\rightarrow Y\rightarrow Z\rightarrow 0.$$
     The following are equivalent.
     \begin{enumerate}
         \item 
         The sequence is coinvariant-pure exact.
         \item 
         The sequence is contained in the class of $\mathcal{S}$-pure exact sequences, where $\mathcal{S}$ is the union of the classes of free $G$-modules and $Triv$.
         \item 
         Both sequences
         $$0\rightarrow X\rightarrow Y\rightarrow Z\rightarrow 0$$
         and
         $$0\rightarrow X_{G}\rightarrow Y_{G}\rightarrow Z_{G}\rightarrow 0$$
         are pure exact as sequences in $\mathpzc{M}$.
     \end{enumerate}
\end{prop}

\begin{proof}
%The equivalence of the first two statements is obvious. Let 
 %        $$0\rightarrow X\rightarrow Y\rightarrow Z\rightarrow 0$$
%be a sequence of $G$-modules.

%Let us first prove (1) $\Rightarrow$ (2).

%Now we prove (2) $\Leftrightarrow$ (3).
Let 
         $$0\rightarrow X\rightarrow Y\rightarrow Z\rightarrow 0$$
be a sequence of $G$-modules. Let $M\in\mathpzc{M}$. Consider the right $G$-module $ k[G]\otimes M$. For a right $k[G]$-module $N$, we have $N\otimes_{k[G]}(k[G]\otimes M)\cong N\otimes M$ in $\mathpzc{M}$. Regarding $M$ as a trivial (left) $G$-representation, we have $N\otimes_{k[G]}M\cong N_{G}\otimes M$. Thus, tensoring over $k[G]$ with $k[G]\otimes M$ gives the sequence
             $$X\otimes M\rightarrow Y\otimes M\rightarrow Z\otimes M,$$
             and tensoring with $M$ over $k[G]$ equipped with the trivial action gives
         the sequence
            $$X_{G}\otimes M\rightarrow Y_{G}\otimes M\rightarrow Z_{G}\otimes M.$$

%Hence, both of these sequences are exact for all $M$ if and only if the original sequence is coinvariant-pure exact.

The equivalence of (1), (2), and (3) now follows easily.
\end{proof}

 \begin{example} 
 Suppose that $\mathpzc{M}$ is enriched over $\mathbb{Q}$ and that $G$ is a finite group. Then a sequence of $G$-modules
 $$0\rightarrow X\rightarrow Y\rightarrow Z\rightarrow 0$$
 is coinvariant-pure if and only if it is pure as a sequence in $\mathpzc{M}$.
 \end{example}

By \cite{kelly2016homotopy} Proposition 2.77 we get the following.

\begin{cor}\label{prop:coinvpureexact}
     The collection of coinvariant-pure exact sequences defines an exact structure on ${}_{G}\mathpzc{Mod}$. The same is true for the class of $Triv$-pure sequences.
\end{cor}

\begin{prop}

    Let $(A_{top},\alpha_{i},a_{i})$ be a filtered object in $\mathpzc{Mod}_{\Sigma_{n}}$. Suppose that each $A_{i}\rightarrow A_{i+1}$ is a coinvariant-pure monomorphism of $\Sigma_{n}$-modules. Then the filtration has admissibly filtered coinvariants.
\end{prop}

\begin{proof}
    We have a pure exact sequence
    $$0\rightarrow A_{i}\rightarrow A_{i+1}\rightarrow gr_{i+1}(A)\rightarrow 0.$$
    Consider the trivial representation of $\Sigma_{n}$ on the monoidal unit $k$. Tensoring with this object gives the exact sequence
    $$0\rightarrow (A_{i})_{\Sigma_{n}}\rightarrow (A_{i+1})_{\Sigma_{n}}\rightarrow (gr_{i+1}(A))_{\Sigma_{n}}\rightarrow 0.$$
\end{proof}

\begin{prop}
    
    Let $G\rightarrow H$ be an injective homomorphism of finite groups, and let 
    $$0\rightarrow X\rightarrow Y\rightarrow Z\rightarrow 0$$
    be a coinvariant-pure exact sequence in ${}_{G}\mathpzc{Mod}$. Then 
    $$0\rightarrow \mathrm{Ind}_{G}^{H}X\rightarrow\mathrm{Ind}_{G}^{H}Y\rightarrow \mathrm{Ind}_{G}^{H}Z\rightarrow 0$$
    is a coinvariant-pure exact sequence in ${}_{H}\mathpzc{Mod}$. 
\end{prop}

\begin{proof}
    Under the equivalences ${}_{G}\mathpzc{Mod}\cong{}_{k[G]}\mathpzc{Mod}$ and ${}_{H}\mathpzc{Mod}\cong{}_{k[H]}\mathpzc{Mod}$, the functor $\mathrm{Ind}_{G}^{H}$ is given by $k[H]\otimes_{k[G]}(-)$. Then, for $M$ an object of ${}_{k[H]}\mathpzc{Mod}$, we have $M\otimes_{k[H]}(k[H]\otimes_{k[G]}(-))\cong M\otimes_{k[G]}(-)$. The sequence
    $$0\rightarrow k[H]\otimes_{k[G]}X\rightarrow k[H]\otimes_{k[G]}Y\rightarrow k[H]\otimes_{k[G]}Z\rightarrow 0$$
 is pure exact in $\mathpzc{M}$ by Remark \ref{rem:ind}. Taking $H$ coinvariants gives the sequence
 $$0\rightarrow X_{G}\rightarrow Y_{G}\rightarrow Z_{G}\rightarrow 0$$
 that is pure exact in $\mathpzc{M}$ by assumption.
  
 % Consider an $H$-module of the form $M\otimes k[H]$. Tensoring this with the above sequence gives the sequence
  %  $$0\rightarrow M\otimes X_{G}\rightarrow M\otimes Y_{G}\rightarrow M\otimes Z_{G}\rightarrow 0$$
  %  that is exact by assumption. Now, consider a trivial $H$-module $M$. Tensoring with our exact sequence gives the exact sequence 
   %     $$0\rightarrow M_{H}\otimes X_{G}\rightarrow M_{H}\otimes Y_{G}\rightarrow M_{H}\otimes Z_{G}\rightarrow 0$$
    
   % is pure exact as a sequence of $k[H]$-modules. This sequence is pure exact as a sequence in $\mathpzc{M}$ by Remark \ref{rem:ind}.
\end{proof}

\begin{prop}
         Let $G$ and $H$ be groups. Let 
        $$0\rightarrow X\rightarrow Y\rightarrow Z\rightarrow 0$$
        be a coinvariant-pure exact sequence of $G$-modules, and let 
            $$0\rightarrow A\rightarrow B\rightarrow C\rightarrow 0$$
            be a coinvariant-pure exact sequence of $H$-modules. Then the sequence
            $$0\rightarrow (X\otimes B+Y\otimes A)\rightarrow Y\otimes B \rightarrow Z\otimes C\rightarrow 0$$
            is coinvariant-pure exact as a $G\times H$-module, where the groups act component-wise on the tensor product.
 \end{prop}

 \begin{proof}
     The sequence is pure exact in $\mathpzc{M}$ by Proposition \ref{prop:purepushoutsum}. It therefore suffices to prove that $Y\otimes B\rightarrow Z\otimes C$ is a $Triv$-pure epimorphism of $G\times H$-representations. We will prove that $Y\otimes B\rightarrow Z\otimes B$ is $Triv$-pure. Then, by symmetry, $Z\otimes B\rightarrow Z\otimes C$ is $Triv$-pure. We then conclude by observing that $Triv$-pure epimorphisms are stable under composition by Corollary \ref{prop:coinvpureexact}. Thus, we reduce to proving that, if 
      $$0\rightarrow X\rightarrow Y\rightarrow Z\rightarrow 0$$
      is a coinvariant-pure exact sequence of $G$-modules, and $C$ is an arbitrary $H$-module, then 
           $$0\rightarrow X\otimes C\rightarrow Y\otimes C\rightarrow Z\otimes C\rightarrow 0$$
           is a $Triv$-pure exact sequence of $G\times H$-modules.% This amounts to observing that the sequence is pure exact in $\mathpzc{M}$, 
           
    But the $G\times H$-coinvariant sequence is
             $$0\rightarrow X_{G}\otimes C_{H}\rightarrow Y_{G}\otimes C_{H}\rightarrow Z_{G}\otimes C_{H}\rightarrow 0.$$
            This is pure exact in $\mathpzc{M}$, since
             $$0\rightarrow X_{G}\rightarrow Y_{G}\rightarrow Z_{G}\rightarrow 0$$
             is pure exact in $\mathpzc{M}$.
        
 \end{proof}

 \begin{defn}
 
     A $G$-module $Z$ is said to be \textbf{stably coinvariant-purifying} if for any exact sequence of $G$-modules
     $$0\rightarrow A\rightarrow B\rightarrow C\rightarrow 0$$
     that is pure exact as a sequence in $\mathpzc{M}$, the sequence
     $$0\rightarrow Z\otimes A\rightarrow Z\otimes B\rightarrow Z\otimes C\rightarrow 0$$
     is coinvariant-pure exact, where $G$ acts on the tensor product diagonally. 
 \end{defn}

 \begin{example}\label{example:freestcfp}
The class of stably coinvariant-purifying objects contains free objects, and is stable under retracts and filtered colimits. 
\begin{enumerate}
    \item 
    When $\mathpzc{M}$ is enriched over $\mathbb{Q}$ and $G$ is a finite group, any $G$-module is stably coinvariant-purifying.
\item
Consider a strong Koszul category of the form $Ch(\mathpzc{E})$. Then any complex of $G$-modules that is degree-wise a retract of a free $G$-module is stably coinvariant-purifying. In the projective Koszul case this includes all cofibrant $G$-modules. 

\end{enumerate}

%        By Proposition \ref{prop:coniditonsadfilt}, if $\mathpzc{M}$ is a $\Q$-Koszul category then any admissibly filtered $\Sigma_{n}$-module has admissibly filtered coinvariants. Indeed, in this case, every $\Sigma_{n}$-module is a retract of a free one. 
\end{example}

 \begin{prop}\label{prop:Gdiagpure}
        Let $G$ be a group. Let 
        $$0\rightarrow X\rightarrow Y\rightarrow Z\rightarrow 0$$
        be a pure exact sequence of $G$-modules with $Z$ stably coinvariant-purifying, and let
            $$0\rightarrow A\rightarrow B\rightarrow C\rightarrow 0$$
            be a sequence of $G$-modules that is pure as an exact sequence in $\mathpzc{M}$. Then the sequence 
        $$0\rightarrow (X\otimes B+Y\otimes A)\rightarrow Y\otimes B \rightarrow Z\otimes C\rightarrow 0$$
        is $G$-coinvariant-pure as a sequence of $G$-modules, where $G$ acts diagonally. 
 \end{prop}

 \begin{proof}
The sequence
      $$0\rightarrow (X\otimes B+Y\otimes A)\rightarrow Y\otimes B \rightarrow Z\otimes C\rightarrow 0$$
      is pure exact in $\mathpzc{M}$. It suffices to prove that $Z\otimes B\rightarrow Z\otimes C$ and $Y\otimes C\rightarrow Z\otimes C$ are $G$-coinvariant-pure epimorphisms. They are certainly pure epimorphisms when regarded as morphisms in $\mathpzc{M}$.  The morphism $Z\otimes B\rightarrow Z\otimes C$ is $G$-coinvariant-pure by the assumption that $Z$ is stably coinvariant-purifying.

      On the other hand, let $M\in\mathpzc{M}$. Consider the $G$-module $C\otimes M$, where $M$ is equipped with the trivial $G$-action, and $G$ acts diagonally. The sequence
      $$0\rightarrow X\otimes_{k[G]}(C\otimes M)\rightarrow Y\otimes_{k[G]}(C\otimes M)\rightarrow Z\otimes_{k[G]}(C\otimes M)\rightarrow 0$$
      is exact in $\mathpzc{M}$. But this is just
      $$0\rightarrow ((X\otimes C)_{G})\otimes M\rightarrow((Y\otimes C)_{G})\otimes M\rightarrow((Z\otimes C)_{G})\otimes M\rightarrow0.$$
      In other words, the sequence
      $$0\rightarrow(X\otimes C)_{G}\rightarrow (Y\otimes C)_{G}\rightarrow (Z\otimes C)_{G}\rightarrow 0$$
      is pure exact in $\mathpzc{M}$, as required. 
      
 \end{proof}

\subsubsection{$\Sigma$-Filtrations}

Let us apply coinvariant-purity to the composite product of filtered $\Sigma$-modules.

\begin{defn}\label{defn:stablycoinvariantpure}
\begin{enumerate}
    \item 
    We say that a sequence 
    $$0\rightarrow X\rightarrow Y\rightarrow Z\rightarrow 0$$
    in $\mathpzc{Mod}_{\Sigma}$
    is \textbf{coinvariant-pure} if each sequence
       $$0\rightarrow X(n)\rightarrow Y(n)\rightarrow Z(n)\rightarrow 0$$
       is coinvariant-pure as a sequence of $\Sigma_{n}$-representations. 
       \item 
       We say that an object $Z$ of $\mathpzc{Mod}_{\Sigma}$ is \textbf{stably coinvariant-purifying} if each $Z(n)$ is stably coinvariant-purifying as a $\Sigma_{n}$-module. 
\end{enumerate}
\end{defn}

\begin{prop}\label{prop:compproductpure}
 Let 
  $$0\rightarrow X\rightarrow Y\rightarrow Z\rightarrow 0$$
  be a sequence of $\Sigma$-modules such that for each $n$, 
    $$0\rightarrow X(n)\rightarrow Y(n)\rightarrow Z(n)\rightarrow 0$$
is a pure exact sequence of $\Sigma_{n}$-modules, where $Z(n)$  is stably coinvariant-purifying. Let 
  $$0\rightarrow A\rightarrow B\rightarrow C\rightarrow 0$$
  be a sequence of $\Sigma$-modules where each sequence
    $$0\rightarrow A(n)\rightarrow B(n)\rightarrow C(n)\rightarrow 0$$
    is pure exact in $\mathpzc{M}$. Then each sequence
    \begin{align*}
            0&\rightarrow\bigoplus_{m\ge0}(X(m)\otimes_{\Sigma_{m}}(B\otimes\ldots\otimes B)(n)+Y(m)\otimes_{\Sigma_{m}}(A\otimes B\otimes\ldots\otimes B)(n)+\ldots+Y(m)\otimes_{\Sigma_{m}}(B\otimes\ldots \otimes B\otimes A)(n))\\
            &\rightarrow (Y\circ B)(n)\\
            &\rightarrow (Z\circ C)(n)\\
            &\rightarrow 0
    \end{align*}

    is pure exact in $\mathpzc{M}$. 
\end{prop}

\begin{proof}
    The sequence is the sum over $m$ of the sequences 
       \begin{align*}
       0&\rightarrow(X(m)\otimes_{\Sigma_{m}}(B\otimes B\otimes\ldots\otimes B)(n)+Y(m)\otimes_{\Sigma_{m}}(A\otimes B\otimes\ldots\otimes B)(n)+\ldots+Y(m)\otimes_{\Sigma_{m}}(B\otimes\ldots \otimes B\otimes A)(n))\\
       &\rightarrow Y(m)\otimes_{\Sigma_{m}}(B^{\otimes m})(n)\\
       &\rightarrow Z(m)\otimes_{\Sigma_{m}}(C^{\otimes m})(n)\\
       &\rightarrow0.
       \end{align*}
   
    Since direct sums are exact, we only need to show each of these sequences is pure exact. Fix an $m$. Now, by Proposition \ref{prop:pureInd} and Proposition \ref{prop:purepushoutsum}, each sequence
        $$0\rightarrow (A\otimes B\otimes\ldots\otimes B)(n)+\ldots+(B\otimes\ldots \otimes B\otimes A)(n)\rightarrow (B^{\otimes m})(n)\rightarrow (C^{\otimes m})(n)\rightarrow0$$
        is pure exact in $\mathpzc{M}$. We conclude by applying Proposition \ref{prop:Gdiagpure} to the sequences
          $$0\rightarrow X\rightarrow Y\rightarrow Z\rightarrow 0$$
          and
           $$0\rightarrow (A\otimes B\otimes\ldots\otimes B)(n)+\ldots+(B\otimes\ldots\otimes B\otimes A)(n)\rightarrow (B^{\otimes m})(n)\rightarrow (C^{\otimes m})(n)\rightarrow0,$$
           and by taking coinvariants.
\end{proof}

\begin{defn}\label{defn:stabcoinvapurefilt}
    \begin{enumerate}
        \item 
            A filtered object $M\in\overline{\mathpzc{Filt}}_{\textbf{AdMon}}(\mathpzc{Mod}_{\Sigma})$ is said to be \textbf{stably coinvariant-purifyingly filtered} if each sequence
    $$0\rightarrow M_{i}(n)\rightarrow M_{i+1}(n)\rightarrow gr_{i+1}(M)(n)\rightarrow 0$$
    is pure exact as a sequence of $\Sigma_{n}$-modules, and $gr_{i+1}(M)$ is stably coinvariant-purifying.
    \item 
          A filtered object $N\in\overline{\mathpzc{Filt}}_{\textbf{AdMon}}(\mathpzc{Mod}_{\Sigma})$ is said to be \textbf{purely filtered} if each sequence
    $$0\rightarrow N_{i}\rightarrow N_{i+1}\rightarrow gr_{i+1}(N)\rightarrow 0$$
    is pure exact as a sequence in $\mathpzc{M}$.
    \end{enumerate}
\end{defn}

By Proposition \ref{prop:compproductpure}, we get the following result.

\begin{cor}\label{cor:grmn}
  Let $M\in\overline{\mathpzc{Filt}}_{\textbf{AdMon}}(\mathpzc{Mod}_{\Sigma})$ be stably coinvariant-purifyingly filtered, and let $N\in\overline{\mathpzc{Filt}}_{\textbf{AdMon}}(\mathpzc{Mod}_{\Sigma})$ be purely filtered. Then we have a natural isomorphism
  $$gr(M\circ N)\cong gr(M)\circ gr(N).$$
\end{cor}

\begin{prop}\label{prop:Kflatcirc1}
    Let $M$ be a $\Sigma$-module such that each $M(n)$ is $K$-flat as a right $\Sigma_{n}$-module. Let $f:X\rightarrow Y$ be a weak equivalence of $\Sigma$-modules such that for each $n$, $X(n)$ and $Y(n)$ are $K$-flat as objects of $\mathpzc{M}$. Then $Id_{M}\circ f:M\circ X\rightarrow M\circ Y$ is an equivalence.
\end{prop}

\begin{proof}
    We have
    $$(Id_{M}\circ f)(n)=\bigoplus_{m\ge0}Id_{M}\otimes_{\Sigma_{m}}(f)^{\otimes m}(n):\bigoplus_{m\ge 0}M(m)\otimes_{\Sigma_{m}}(X)^{\otimes m}(n)\rightarrow\bigoplus_{m\ge 0}M(m)\otimes_{\Sigma_{m}}(Y)^{\otimes m}(n).$$
    Since direct sums preserve equivalences, it suffices to prove that each 
    $$M(m)\otimes_{\Sigma_{m}}(X)^{\otimes m}(n)\rightarrow M(m)\otimes_{\Sigma_{m}}(Y)^{\otimes m}(n)$$
    is an equivalence. Since $M(m)$ is $K$-flat as a right $\Sigma_{m}$-module, we reduce to proving that $(f)^{\otimes m}(n)$ is a weak equivalence. Let $i_{1},\ldots,i_{m}$ be such that $i_{1}+\ldots+i_{m}=n$. Then 
    $$X(i_{1})\otimes\ldots\otimes X(i_{m})\rightarrow Y(i_{1})\otimes\ldots\otimes Y(i_{m})$$ 
    is a weak equivalence since, for all $i$, $X(i)$ and $Y(i)$ are $K$-flat. Then 
    $$Ind_{\Sigma_{i_{1}}\times\ldots\times\Sigma_{i_{m}}}^{\Sigma_{n}}X(i_{1})\otimes\ldots \otimes X(i_{m})\rightarrow Ind_{\Sigma_{i_{1}}\times\ldots\times\Sigma_{i_{m}}}^{\Sigma_{n}}Y(i_{1})\otimes\ldots \otimes Y(i_{m})$$
    is a weak equivalence since, after forgetting to $\mathpzc{M}$, $Ind_{\Sigma_{i_{1}}\times\ldots\times\Sigma_{i_{m}}}^{\Sigma_{n}}$ is just tensoring with $\bigoplus_{[g]\in\Sigma_{n}\big\slash\Sigma_{i_{1}}\times\ldots\times\Sigma_{i_{m}}}R$. This amounts to taking a finite direct sum, and this preserves equivalences. 
\end{proof}

\begin{prop}\label{prop:Kflatcirc2}
    Let $M$ be a $\Sigma$-module such that each $M(n)$ is $K$-flat as a right $\Sigma_{n}$-module, and let $X$ be any $\Sigma$-module such that each $X(n)$ is $K$-flat as an object of $\mathpzc{M}$. Then each $(M\circ X)(n)$ is $K$-flat as an object of $\mathpzc{M}$.
\end{prop}

\begin{proof}
    First, each $Ind_{\Sigma_{i_{1}}\times\ldots\times\Sigma_{i_{k}}}^{\Sigma_{n}}X(i_{1})\otimes\ldots \otimes X(i_{k})$ is $K$-flat, where $i_{1}+\ldots+i_{k}=n$. Indeed, it is the tensor product of all of the $X(i_{j})$ with $\bigoplus_{\Sigma_{n}\big\slash\Sigma_{i_{1}}\times\ldots\times\Sigma_{i_{k}}}k$. Let $f:U\rightarrow V$ be a weak equivalence in $\mathpzc{M}$. We regard $\bigoplus_{i_{1}+\ldots+i_{k}=n}(X(i_{1})\otimes\ldots\otimes X(i_{k}))\otimes U$ and $\bigoplus_{i_{1}+\ldots+i_{k}=n}(X(i_{1})\otimes\ldots\otimes X(i_{k}))\otimes V$ as $\Sigma_{k}$-modules, where $\Sigma_{k}$ acts on $U$ and $V$ trivially. Now
    $$\bigoplus_{i_{1}+\ldots+i_{k}=n}(X(i_{1})\otimes\ldots\otimes X(i_{k}))\otimes U\rightarrow\bigoplus_{i_{1}+\ldots+i_{k}=n}(X(i_{1})\otimes\ldots\otimes X(i_{k}))\otimes V$$
    is an equivalence of $\Sigma_{k}$-modules. Then 
    $$(M(k)\otimes_{\Sigma_{k}}\bigoplus_{i_{1}+\ldots+i_{k}=n}(X(i_{1})\otimes\ldots\otimes X(i_{k})))\otimes U\rightarrow (M(k)\otimes_{\Sigma_{k}}\bigoplus_{i_{1}+\ldots+i_{k}=n}(X(i_{1})\otimes\ldots\otimes X(i_{k})))\otimes V$$
    is also an equivalence. This completes the proof. 
\end{proof}

\begin{prop}\label{prop:Kflatcirc3}
    Let $M$ be a $\Sigma$-module such that each $M(n)$ is $K$-flat as a left $\Sigma_{n}$-module. Let $f:X\rightarrow Y$ be a weak equivalence of $\Sigma$-modules such that, for all $n$, $X(n)$ and $Y(n)$ are $K$-flat as objects of $\mathpzc{M}$. Then $f\circ Id_{M}:X\circ M\rightarrow Y\circ M$ is an equivalence.
\end{prop}

\begin{proof}
    It suffices to prove that each $Ind_{\Sigma_{i_{1}}\times\ldots\times\Sigma_{i_{k}}}^{\Sigma_{n}}M(i_{1})\otimes\ldots \otimes M(i_{k})$ is $K$-flat as a left $\Sigma_{n}$-module. Now, we have 
    $$U\otimes_{\Sigma_{n}}(Ind_{\Sigma_{i_{1}}\times\ldots\times\Sigma_{i_{k}}}^{\Sigma_{n}}V)\cong U\otimes_{\Sigma_{i_{1}}\times\ldots\times\Sigma_{i_{k}}}V,$$
    so it suffices to prove that $M(i_{1})\otimes\ldots\otimes M(i_{k})$ is $K$-flat as a left $\Sigma_{i_{1}}\times\ldots\times\Sigma_{i_{k}}$-module. But 
    $$U\otimes_{\Sigma_{i_{1}}\times\ldots\times\Sigma_{i_{k}}}M(i_{1})\otimes\ldots \otimes M(i_{k})\cong (\ldots((U\otimes_{\Sigma_{i_{1}}}M(i_{1}))\otimes_{\Sigma_{i_{2}}}M(i_{2}))\ldots\otimes_{\Sigma_{i_{k}}}M(i_{k})),$$
    and this is an iterated tensor product with $K$-flat modules at each step. Thus, this functor preserves weak equivalences. 
    
\end{proof}

\subsection{Standard Cofibrations}\label{subsec:standard}
Now let us analyse more closely the transferred model structures on categories of algebras over operads in exact Koszul categories. In particular, we will study certain classes of cofibrations, called standard cofibrations, in strong Koszul categories. Though the idea is the same, because complexes in a general exact category do not split, the discussion of standard cofibrations as defined in \cite{vallette2014homotopy} Section 2.4 is significantly more involved.

In fact, we shall study much more general classes of morphisms. Let $\mathpzc{E}$ be a weakly $\textbf{PureMon}$-elementary, monoidal exact category. Let $\{G_{i}\}$ be a collection of objects in $\mathpzc{E}$ and $\{f_{l}\}_{l\in\mathcal{L}}$ a collection of admissible monomorphisms in $\mathpzc{E}$. Denote by $\mathfrak{G}$ the collection of morphisms
$$\{S^{n}(G_{i})\rightarrow D^{n+1}(G_{i})\}_{n\in\Z,i\in\mathcal{I}}\cup\{0\rightarrow D^{n+1}(G_{i})\}_{n\in\Z,i\in\mathcal{I}}\cup\{S^{n}(f_{l})\}_{n\in\Z,l\in\mathcal{L}}.$$
Denote by $preCell(R;\mathfrak{G})$ the collection of morphisms obtained by transfinite composition of pushouts of morphisms in
$$\{R\otimes S^{n}(G_{i})\rightarrow R\otimes D^{n+1}(G_{i})\}_{n\in\Z,i\in\mathcal{I}}\cup\{0\rightarrow R\otimes D^{n+1}(G_{i})\}_{n\in\Z,i\in\mathcal{I}},$$
and by $Cell(R;\mathfrak{G})$ the collection of morphisms obtained by transfinite composition of pushouts of morphisms in 
$$\{R\otimes S^{n}(G_{i})\rightarrow R\otimes D^{n+1}(G_{i})\}_{n\in\Z,i\in\mathcal{I}}\cup\{0\rightarrow R\otimes D^{n+1}(G_{i})\}_{n\in\Z,i\in\mathcal{I}}\cup\{R\otimes S^{n}(f_{l})\}_{n\in\Z,l\in\mathcal{L}}.$$
Similarly, we define $(pre)Cell_{\ge0}(R;\mathfrak{G})$, $(pre)Cell_{>0}(R;\mathfrak{G})$, where in the collections above we impose that $n\ge0$ and $n>0$, respectively. For an object $V\in\mathpzc{M}$, we will say $V\in preCell(R,\mathfrak{G})$, if $(0\rightarrow V)\in preCell(R,\mathfrak{G})$ in the sense above. We will also write $G\in\mathfrak{G}$ if $G\in\{G_{i}\}$.
%Let $\mathfrak{O}$ be a thick subcategory of $\mathpzc{E}$ which is closed under transfinite extensions. 

%there is a set of objects $\{G_{i}\}_{i\in\mathcal{I}}$ and a set of pure monomorphisms $\{f_{l}\}_{l\in\mathcal{L}}$ such that 
%$$\{S^{n}(G_{i})\rightarrow D^{n+1}(G_{i})\}_{n\in\Z,i\in\mathcal{I}}\cup\{0\rightarrow D^{n}(G_{i})\}_{n\in\Z,i\in\mathcal{I}}\cup\{S^{n}(f_{l})\}_{n\in\Z,l\in\mathcal{L}}$$
%is a set of generating cofibrations. 

\begin{defn}\label{defn:sullivan}
\begin{enumerate}
\item
Denote by $Sull_{\mathfrak{P}}(\mathpzc{M};\mathfrak{G})$ the class of morphisms of $\mathfrak{P}$-algebras that are obtained as a retract of a transfinite composition of pushouts of morphisms of the form $\mathfrak{P}(f)$ for $f\in preCell(R;\mathfrak{G})$, and by $Cof_{\mathfrak{P}}(\mathpzc{M};\mathfrak{G})$ the class of morphisms of $\mathfrak{P}$-algebras that are obtained as a retract of a transfinite composition of pushouts of morphisms of the form $\mathfrak{P}(f)$ for $f\in Cell(R;\mathfrak{G})$.
\item
For $A$ a $\mathfrak{P}$-algebra, denote by $Sull^{A}_{\mathfrak{P}}(\mathpzc{M};\mathfrak{G})$ (resp. $Cof^{A}_{\mathfrak{P}}(\mathpzc{M};\mathfrak{G})$) the class of morphisms $f\in Sull_{\mathfrak{P}}(\mathpzc{M};\mathfrak{G})$ (resp. $Cof_{\mathfrak{P}}(\mathpzc{M};\mathfrak{G})$) whose domain is $A$. We equivalently regard these as full subcategories of the undercategory ${}_{A\slash}\mathpzc{Alg}_{\mathfrak{P}}(\mathpzc{M})$.
\item
Define the class of \textbf{Sullivan models} to be $Sull^{\mathfrak{P}(0)}_{\mathfrak{P}}(\mathpzc{M};\mathfrak{G})$.
\end{enumerate}
We similarly define $Sull^{A}_{\ge0,\mathfrak{P}}(\mathpzc{M};\mathfrak{G})$, $Sull^{A}_{>0,\mathfrak{P}}(\mathpzc{M};\mathfrak{G})$, $Cof^{A}_{\ge0,\mathfrak{P}}(\mathpzc{M};\mathfrak{G})$, and $Cof^{A}_{>0,\mathfrak{P}}(\mathpzc{M};\mathfrak{G})$. The class $Sull^{A}_{\ge0,\mathfrak{P}}(\mathpzc{M};\mathfrak{G})$ (resp. $Sull^{A}_{>0,\mathfrak{P}}(\mathpzc{M};\mathfrak{G})$) is called the class of $\mathfrak{G}$-\textbf{non-negatively graded Sullivan models relative to} $A$ (resp. $\mathfrak{G}$-\textbf{positively graded Sullivan models relative to} $A$).
\end{defn}

Note that if $\mathpzc{E}={}_{k}\mathpzc{Vect}$ is the category of vector spaces over a field $k$ of characteristic $0$, and $\mathfrak{P}=\mathfrak{Comm}$, then a morphism in  $Sull_{\mathfrak{Comm}}^{A}(Ch(\mathpzc{E});\mathfrak{G})$ is a relative Sullivan algebra as in \cite{hess2007rational} Section 2.

%\begin{prop}
%Let $R\in\mathpzc{Alg}_{\mathfrak{Comm}}(\mathpzc{M})$, and regard it as an operad concentrated in arity $1$. 
%\end{prop}
%%\begin{enumerate}
%\item
%$\mathfrak{P}(R\otimes S^{n}(X)\rightarrow R\otimes D^{n+1}(X))$ where $X\in\mathfrak{O}$ and $n\in\mathbb{Z}$.\\
%\item
%$\mathfrak{P}(0\rightarrow R\otimes D^{n}(X))$ where $X\in\mathfrak{O}$ and $n\in\mathbb{Z}$.
%\end{enumerate}
%Denote by $Cof_{\mathfrak{P}}(\mathfrak{G})$ the class of morphisms of $\mathfrak{P}$-algebras obtained by transfinite composition of pushouts of morphisms of the form above, as well as those of the form $\mathfrak{P}(R\otimes S^{n}(X)\rightarrow R\otimes S^{n}(Y))$ where $coker(X\rightarrow Y)\in\mathfrak{O}$. 

Let $V$ be an object of $\mathpzc{M}$, $A$ an object of $\mathpzc{Alg}_{\mathfrak{P}}(\mathpzc{M})$, and $\alpha:V\rightarrow A$ be a degree $-1$ morphism of $R$-modules. There is then an induced morphism of graded objects
$$V\rightarrow A\rightarrow A\coprod\mathfrak{P}(V).$$
By the construction in Proposition \ref{prop:twistder} we get an induced derivation
$$d_{A}+d_{V}+d_{\alpha}:A\coprod\mathfrak{P}(V)\rightarrow A\coprod\mathfrak{P}( V)$$
whose restriction to $V$ is $d_{V}+\alpha$. We denote the algebra equipped with this derivation by $A\coprod_{\alpha}\mathfrak{P}(V)$.

If $f:V\rightarrow W$ is a morphism of complexes, and $\alpha:W[-1]\rightarrow A$ a morphism of $R$-modules, then we get an induced morphism that commutes with differentials
$$f_{\alpha}:A\coprod_{\alpha\circ f}\mathfrak{P}(V)\rightarrow A\coprod_{\alpha}\mathfrak{P}(W),$$
which on underlying objects is just $Id_{A}\coprod\mathfrak{P}(f)$.

\begin{prop}
Suppose that $\alpha:V[-1]\rightarrow A$ 
is a morphism in $\mathpzc{M}$, i.e., it commutes with differentials. Then $A\coprod_{\alpha}\mathfrak{P}(V)$ is a chain complex.
\end{prop}

\begin{proof}
The derivation $(A\coprod\mathfrak{P}(V))[-1]\rightarrow A\coprod\mathfrak{P}(V)$ is induced from the derivation 
$$d_{\mathfrak{P}}\circ Id_{A\oplus V}+d_{A}+ Id_{\mathfrak{P}}\circ'(\alpha+d_{V}):\mathfrak{P}(A\oplus V)\rightarrow\mathfrak{P}(A\oplus  V)[-1].$$
 Therefore, it suffices to show that this derivation squares to $0$. This is a straightforward computation.
\end{proof}
In particular, if $\alpha:V[-1]\rightarrow A$ commutes with differentials, then $A\coprod_{\alpha}\mathfrak{P}(V)$ is naturally an object of $\mathpzc{M}$. The crucial lemma is the following, and is a generalisation of \cite{vallette2014homotopy} Lemma 2.7.

\begin{lem}\label{standard}
Let $f:V\rightarrow W$ be a morphism in $preCell(R;\mathfrak{G})$ and let $A\in\mathpzc{Alg}_{\mathfrak{P}}(\mathpzc{M})$. Let $\alpha:W[-1]\rightarrow A$ be a morphism of $R$-modules. Then the induced morphism
$$f_{\alpha}:A\coprod_{\alpha\circ f}\mathfrak{P}(V)\rightarrow A\coprod_{\alpha}\mathfrak{P}(W)$$
is in $Sull_{\mathfrak{P}}(\mathpzc{M};\mathfrak{G})$. If $\mathfrak{G}$ is the set of generating cofibrations of a strong Koszul category from Definition \ref{basic Koszul}, and $f:V\rightarrow W$ a morphism in $Cell(R;\mathfrak{G})$, then 
$$f_{\alpha}:A\coprod_{\alpha\circ f}\mathfrak{P}(V)\rightarrow A\coprod_{\alpha}\mathfrak{P}(W)$$
is in $Cof_{\mathfrak{P}}(\mathpzc{M};\mathfrak{G})$.
%Let $V\in\mathfrak{O}$. The morphism $A\rightarrow A\coprod_{\alpha}\mathfrak{P}(V)$ is a $\mathfrak{O}$-cofibration of $\mathfrak{P}$-algebras.
\end{lem}

First, let us prove some auxiliary results. Let $(V_{\bullet},d_{V})$ and $(W_{\bullet},d_{W})$ be objects of $Ch(\mathpzc{E})$, and let $f:V_{\bullet}\rightarrow W_{\bullet}$ be a morphism of chain complexes. Suppose there are degree $-1$ morphisms of $R$-modules $\nu:V_{\bullet}\rightarrow A_{\bullet}$ and $\omega:W_{\bullet}\rightarrow A_{\bullet}$ such that $\omega\circ f=\nu$, and both
$$A\coprod_{\nu}\mathfrak{P}(V)\;\textrm{ and }A\coprod_{\omega}\mathfrak{P}(W)$$
are complexes. Then the morphism $Id_{A}\coprod\mathfrak{P}(f)$ induces a morphism of chain complexes

$$f_{\nu,\omega}:A\coprod_{\nu}\mathfrak{P}(V)\rightarrow A\coprod_{\omega}\mathfrak{P}(W).$$

Let $F:\mathcal{I}\rightarrow\mathpzc{M}$ be a diagram. For $i\in\mathcal{I}$ let $F(i)=V^{i}$. Define $V=\colim_{i\in\mathcal{I}}V^{i}$. Suppose there is a degree $-1$ morphism $\alpha:\textrm{colim}V^{i}_{\bullet}\rightarrow A$. Composing with the morphisms $f^{i}:V^{i}\rightarrow \textrm{colim}V^{i}_{\bullet}$ gives degree $-1$ morphisms $\alpha^{i}=\alpha\circ f^{i}:V^{i}\rightarrow A$.

%Suppose that for each $i$, $A\coprod_{\alpha_{i}}\mathfrak{P}(V^{i})$ is a chain complex.

\begin{prop}\label{transfinitestandard}
There is an isomorphism of $\mathfrak{P}$-algebras with derivation
$$\textrm{colim}(A\coprod_{\alpha^{i}}\mathfrak{P}(V^{i}))\cong A\coprod_{\alpha}\mathfrak{P}(V)$$
induced by the morphisms $f^{i}$.

In particular, if each $A\coprod_{\alpha^{i}}\mathfrak{P}(V^{i})$ is a chain complex, then so is $A\coprod_{\alpha}\mathfrak{P}(V)$.
\end{prop}

\begin{proof}
By the above remarks, there is a morphism of algebras with derivation $\textrm{colim}(A\coprod_{\alpha^{i}}\mathfrak{P}(V^{i}))\rightarrow A\coprod_{\alpha}\mathfrak{P}(V)$. To see that it is an isomorphism, we may forget the differentials and $R$-module structure, in which case it reduces to the fact that coproducts and colimits commute.
The claim regarding the colimit being a chain complex simply follows from the fact that the category of $\mathfrak{P}$-algebras with square-zero derivation is closed under colimits inside the category of all $\mathfrak{P}$-algebras with derivation.
\end{proof}

Let $f:V\rightarrow W$ and $\nu:V\rightarrow A$ be degree $0$ morphisms and let $\omega:W\rightarrow A$ be a degree $-1$ morphism of $R$-modules. There is an induced degree $-1$ morphism $\nu+\omega: cone(f)\rightarrow A$. There is also a degree $-1$ morphism
\begin{displaymath}
\xymatrix{
V[-1]\ar[r]^{(\nu,-f)} & A\oplus W\ar[r] & A\coprod_{-\omega}\mathfrak{P}(W),
}
\end{displaymath}
denoted by $\nu\cup (-f)$.

\begin{prop}\label{standardcone}
Suppose that $\omega:W\rightarrow A$ is a degree $-1$ morphism of $R$-modules such that the induced morphism $W[-1]\rightarrow A$ commutes with differentials, and that $\nu$ satisfies
$$\omega_{n}\circ f_{n}=d^{A}_{n}\circ\nu_{n}-\nu_{n-1}d_{n}^{V}.$$
Then
\begin{enumerate}
\item
$\nu\cup(-f)$ is a morphism of chain complexes, and
\item
there is an isomorphism of $\mathfrak{P}$-algebras with derivation, natural in all arguments,
$$A\coprod_{\nu+\omega}\mathfrak{P}( cone(f))\cong (A\coprod_{-\omega}\mathfrak{P}(W))\coprod_{\nu\cup -f}\mathfrak{P}(V[-1]).$$
\end{enumerate}
\end{prop}

\begin{proof}
The first part is a direct computation. For the second part, let us forget the differentials and $R$-module structure for the moment. Then we have
\begin{align*}
(A\coprod_{-\omega}\mathfrak{P}(W))\coprod_{\nu\cup (-f)}\mathfrak{P}( V[-1]) \cong A\coprod\mathfrak{P}( W)\coprod\mathfrak{P}( V[-1])
\cong  A\coprod\mathfrak{P}(W\oplus V[-1])\\
=A\coprod\mathfrak{P}( cone(f))=A\coprod_{\nu+\omega}\mathfrak{P}( cone(f)).
\end{align*}
We need to check that this isomorphism preserves the differentials. Again, this is a direct computation.
\end{proof}

\begin{prop}\label{extracofibs}
Let $\mathpzc{M}$ be strong Koszul, let $i:S^{n}(X)\rightarrow S^{n}(Y)$ be a cofibration in $Ch(\mathpzc{E})$, and let $\alpha:S^{n}(Y)[-1]\rightarrow A$ be a morphism of complexes. Then $p:A\coprod_{\alpha\circ i}\mathfrak{P}(R\otimes S^{n}(X))\rightarrow A\coprod_{\alpha}\mathfrak{P}(R\otimes S^{n}(Y))$ is a cofibration.
\end{prop}

\begin{proof}

Let
\begin{displaymath}
\xymatrix{
A\coprod_{\alpha\circ i}\mathfrak{P}(R\otimes S^{n}(X))\ar[d]^{p}\ar[r]^{\;\;\;\;\;\;\;\;\;\;\;f} & M\ar[d]^{q}\\
A\coprod_{\alpha}\mathfrak{P}(R\otimes S^{n}(Y))\ar[r]^{\;\;\;\;\;\;\;\;\;\;\;g} & N
}
\end{displaymath}
be a diagram with $q$ an acyclic fibration. It suffices to find a lift in the following diagram of complexes
\begin{displaymath}
\xymatrix{
A\oplus_{\alpha\circ i}R\otimes S^{n}(X)\ar[d]^{\tilde{p}}\ar[r]^{\;\;\;\;\;\;\;\;\;\;\;\tilde{f}} & M\ar[d]^{q}\\
A\oplus_{\alpha}R\otimes S^{n}(Y)\ar[r]^{\;\;\;\;\;\;\;\;\;\;\;\tilde{g}} & N,
}
\end{displaymath}
where the morphisms decorated with $\tilde{-}$ denote the restrictions. The object $A\oplus_{\alpha\circ i}R\otimes S^{n}(X)$ is, by definition, the subobject of $A\coprod_{\alpha\circ i}\mathfrak{P}(R\otimes S^{n}(X))$ whose underlying object is $A\oplus R\otimes S^{n}(X)$ (and similarly for $A\oplus_{\alpha} R\otimes S^{n}(Y)$). 
%The reason it suffices to find this lift is due to the existence of compatible admissible epimorphisms
%$$\mathfrak{P}(A\oplus_{\alpha\circ i}(R\otimes S^{n}(X))\rightarrow A\coprod_{\alpha\circ i}\mathfrak{P}(R\otimes S^{n}(X))$$
%and 
%$$\mathfrak{P}(A\oplus_{\alpha\circ i}(R\otimes S^{n}(Y))\rightarrow A\coprod_{\alpha}\mathfrak{P}(R\otimes S^{n}(Y)),$$

But the morphism $A\oplus_{\alpha\circ i}R\otimes S^{n}(X)\rightarrow A\oplus_{\alpha}R\otimes S^{n}(Y)$ is an admissible monomorphism whose cokernel is the cofibrant object $R\otimes coker(S^{n}(X)\rightarrow S^{n}(Y))$. In particular, it is a cofibration. 
%\begin{displaymath}
%\xymatrix{
%X\ar[r]\ar[d]^{i} & M_{n}\ar[d]^{q_{n}}\\
%Y\ar[r] & N_{n}
%}
%\end{displaymath}
%We claim that the induced morphism on graded objects $A\coprod_{\alpha}S^{n}(Y)\rightarrow M$ is a morphism of complexes. It suffices to check what happens when we restrict to $S^{n}(Y)$. Indeed the composition $Y\rightarrow (A\coprod_{\alpha\circ i}S^{n}(Y))_{n}\rightarrow (A\coprod_{\alpha\circ i}S^{n}(Y))_{n-1}\rightarrow M_{n-1}$ is given by $f_{n-1}\circ i^{A}_{n-1}\circ \alpha$. The composition $Y\rightarrow M_{n}\rightarrow M_{n-1}$ is given by $d_{n}^{M}\circ r$.
\end{proof}

\begin{proof}[Proof of Lemma \ref{standard}]
%In fact we are going to show the following. 
By Proposition \ref{transfinitestandard} it is sufficient to show that this is the case for the morphisms $R\otimes S^{n}(P)\rightarrow R\otimes D^{n+1}(P)$, $0\rightarrow R\otimes D^{n}(P)$, and in the general $Cof_{\mathfrak{P}}$-case, $R\otimes S^{n}(X)\rightarrow R\otimes S^{n}(Y)$. First, note that $R\otimes D^{n+1}(P)= cone(Id_{R\otimes S^{n}(P)})$. Therefore, by Proposition \ref{standardcone} and Proposition \ref{extracofibs} we reduce to showing that, given a degree $-1$ morphism $\alpha:S^{n}(P)\rightarrow A$, the natural morphism $i_{A}:A\rightarrow A\coprod_{\alpha}\mathfrak{P}(R\otimes S^{n}(P))$ is in $Sull_{\mathfrak{P}}(\mathpzc{M};\mathfrak{G})$, or $Cof_{\mathfrak{P}}(\mathpzc{M};\mathfrak{G})$ for the second case. But we have the pushout diagram
\begin{displaymath}
\xymatrix{
\mathfrak{P}(R\otimes S^{n-1}(P))\ar[r]^{\;\;\;\;\gamma_{A}\mathfrak{P}(R\otimes \alpha[-1])}\ar[d]^{Id_{\mathfrak{P}}\circ(Id_{R}\otimes j^{n-1}_{P})} & A\ar[d]^{i_{A}}\\
\mathfrak{P}(R\otimes D^{n}(P))\ar[r] & A\coprod_{\alpha}\mathfrak{P}(R\otimes S^{n}(P)).
}
\end{displaymath}
This completes the proof.
%The left-hand vertical morphism is a cofibration, so as a pushout of a cofibration the right-hand morphism is also a cofibration.

\end{proof}
This result has numerous applications. For the purposes of this paper, the important consequence is Proposition \ref{factorfibcofib} later. However, let us use it to show that certain interesting classes of algebras are cofibrant. Following \cite{loday2012algebraic} B.6.6 we define triangulated quasi-free algebras.

\begin{defn}
A $\mathfrak{P}$-algebra $A$ is said to be $\mathfrak{G}$-\textbf{quasi-free} if there is an object $V$ with  $V\in preCell(R,\mathfrak{G})$ such that, after forgetting differentials, the underlying graded algebra of $A$ is isomorphic to the underlying graded algebra of $\mathfrak{P}(R\otimes V)$. A quasi-free algebra $A$ is said to be $\mathfrak{G}$-\textbf{triangulated} if the underlying graded object of $V$ can be written as $V\cong\bigoplus_{i=0}^{\infty}V_{i}$, where
\begin{enumerate}
    \item
    after forgetting differentials, each $V_{i}$ is in $preCell(R,\mathfrak{G})$, 
    \item
    for $n\ge 1$, $d_{A}|_{\mathfrak{P}(R\otimes V_{n})}$ factors through $\mathfrak{P}(R\otimes\bigoplus_{i=0}^{n-1} V_{i})$, and 
    \item
    $d_{A}|_{V_{0}}=0$.
\end{enumerate}
\end{defn} 
% is equipped with $\mathfrak{G}$-filtration $V_{i}$ such that $d_{A}|_{V_{i}}$ factors through $\mathfrak{P}(V_{i-1})$. 

\begin{cor}\label{triangquasifree}
A $\mathfrak{G}$-triangulated quasi-free $\mathfrak{P}$-algebra is a $\mathfrak{G}$-Sullivan model.
\end{cor}

\begin{proof}
Let $A$ be $\mathfrak{G}$-triangulated. After forgetting differentials, we may write $A=\mathfrak{P}(V)\cong \mathfrak{P}(R\otimes \bigoplus_{i=0}^{\infty}V_{i})$. 
For $0\le n<\infty$, let $A_{n}=\mathfrak{P}(R\otimes \bigoplus_{i=0}^{n}V_{i})$. This is a subobject of $A$ in $\mathpzc{Alg}_{\mathfrak{P}}(\mathpzc{M})$. Moreover, $A\cong \colim A_{n}$. Now, $A_{0}$ is just the free algebra on $R\otimes V_{0}$, where $V_{0}$ is regarded as a complex in $Ch(\mathpzc{E})$ with trivial differential. Moreover, as an object in $Ch(\mathpzc{E})$, $V_{0}$ is a direct sum of objects of the form $S^{n}(G)$ where $G\in\mathfrak{G}$. Thus, $\mathfrak{P}(0)\rightarrow A_{0}$ is in $Sull_{\mathfrak{P}}(\mathpzc{M},\mathfrak{G})$. 

If we let $\alpha_{n+1}$ be the degree $-1$ morphism $\alpha_{n+1}\defeq d|_{V_{n+1}}:V_{n+1}\rightarrow A_{n}$, then $A_{n+1}=A_{n}\coprod_{\alpha_{n+1}}\mathfrak{P}(R\otimes V_{n+1})$. Thus, by Lemma \ref{standard}, $A_{n}\rightarrow A_{n+1}$ is in $Sull_{\mathfrak{P}}(\mathpzc{M},\mathfrak{G})$. Hence, $0\rightarrow A$ is a transfinite composition of morphisms in $Sull_{\mathfrak{P}}(\mathpzc{M},\mathfrak{G})$, and therefore is in $Sull_{\mathfrak{P}}(\mathpzc{M},\mathfrak{G})$.

%Thus, we just need to show that $A_{0}$ is a $\mathfrak{G}$- and, for each $n\ge 0$, the morphism $A_{n}\rightarrow A_{n+1}$ is a cofibration. The algebra $A_{0}$ is just the free algebra on $R\otimes V_{0}$, where $V_{0}$ is regarded as a complex in $Ch(\mathpzc{E})$ with trivial differential. Moreover, as an object in $Ch(\mathpzc{E})$, $V_{0}$ is a direct sum of objects of the form $S^{n}(G)$ where $G\in\mathfrak{G}$. Hence $A_{0}$ is clearly a $\mathfrak{G}$-Sullivan model. If we let $\alpha_{n+1}$ be the degree $-1$ morphism $\alpha_{n+1}\defeq d|_{V_{n+1}}:V_{n+1}\rightarrow A_{n}$, then $A_{n+1}=A_{n}\coprod_{\alpha_{n+1}}\mathfrak{P}(R\otimes V_{i})$. This is a morphism in $Sull_{\mathfrak{P}}(\mathpzc{M};\mathfrak{G})$. The morphism $0\rightarrow A$ is a transfinite composition of such morphisms and hence is also in $Sull_{\mathfrak{P}}(\mathpzc{M};\mathfrak{G})$. 
\end{proof}

In particular, quasi-free algebras on bounded-below objects of $preCell(R,\mathfrak{G})$ are $\mathfrak{G}$-triangulated. Immediately we get the following result.

%\begin{defn}
%Let $R\in\mathpzc{Alg}_{\mathfrak{Comm}}(Ch(\mathpzc{E}))$ and let $M$ be an $R$-module. We say that $M$ is $\mathfrak{G}$-\textbf{non-negatively graded} if there is a $\mathfrak{G}$-triangulation of $M$, $M=R\otimes lim_{\rightarrow_{n\ge 0}}V_{n}$, where $V_{n}$ is a non-negatively graded complex.
%\end{defn}

%The property $\mathfrak{G}$-triangulation is, in some sense, transitive. 
%
%\begin{prop}
%Let $R\in\mathpzc{Alg}_{\mathfrak{Comm}}(Ch(\mathpzc{E}))$ and let $M$ be a $\mathfrak{G}$-trinagulated $R$-module. Let $\mathfrak{P}$ be an operad in $\mathpzc{M}\defeq{}_{R}\mathpzc{Mod}(Ch(\mathpzc{E}))$. 
%\end{prop}
%\begin{defn}
%An object $M$ of ${}_{R}\mathpzc{Mod}$ is said to be \textbf{concentrated in degrees} $\ge n$ (resp. $\le n$) if 
%\end{defn}

%\begin{proof}
%It is clear that the morphism $0\rightarrow A$ is in
%\end{proof}

\begin{cor}\label{quasifreecofib}
If $\mathpzc{M}={}_{R}\mathpzc{Mod}(Ch(\mathpzc{E}))$, $R$ and $\mathfrak{P}$ are concentrated in non-negative degrees, $V$ is an object concentrated in non-negative degrees such that each $V_{n}$ is a direct sum of elements of $\mathfrak{G}$, and $A\in\mathpzc{Alg}_{\mathfrak{P}}(\mathpzc{M})$ is quasi-free on $V$, then $A$ is $\mathfrak{G}$-triangulated.
%Let $V$ be an object of $\mathpzc{M}$ which is $\mathfrak{G}$-triangulated as an $R$-module, and let $A$ be an algebra which is quasi-free on $V$. Then $A$ is a $\mathfrak{G}$-Sullivan model. 
\end{cor}

\begin{proof}
Write $V=\bigoplus V_{i}[-i]$, where $V_{i}$ is an object of $\mathfrak{G}$. The differential $d_{A}|_{V_{n}[-n]}$ must clearly factor through $\mathfrak{P}(R\otimes\bigoplus_{i=0}^{n-1}V_{i}[-i])$. This proves the claim.
\end{proof}

In the setting of this corollary we say that $A$ is $\mathfrak{G}$-non-negatively graded.

%\begin{prop}
%Let $f:A\rightarrow B$ be a morphism of commutative algebras. Suppose that there is an $A$-module filtration on $B$ by pure monomorphisms such that $gr(B)$ is $K$-flat as an $A$-module. Then $f$ is $K$-flat.
%\end{prop}
%
%\begin{proof}
%Let $f:M\rightarrow N$ be an equivalence of $A$-modules. Equip $M$ and $N$ with the trivial filtration. This is a filtration by pure monomorphisms. Therefore, $gr(Id_{B}\otimes_{A}f)\cong Id_{gr(B)}\otimes_{A} f$. Since $gr(B)$ is a $K$-flat $A$-module this is an equivalence, so $Id_{B}\otimes_{A}f$ is an equivalence. 
%\end{proof}
%\begin{prop}
%Let $g:X\rightarrow Y$ be an admissible monomorphism of complexes such that $C\defeq coker(g)$ is $K$-flat. Then $S(g)$ is a $K$-flat morphism of algebras. 
%\end{prop}
%
%\begin{proof}
%Let $M$ be an $S(X)$-module. 
%\end{proof}

%\begin{cor}
%If $\mathpzc{M} $ is a strong $\Q$-Koszul category then cofibrations in $\mathpzc{Alg}_{\mathfrak{Comm}}(\mathpzc{M})$ are $K$-flat.
%\end{cor}

\subsection{Example: Commutative Algebras}\label{subsec:commalg}

Before proceeding to coalgebras and Koszul duality, let us conclude this section with a detailed look at the homotopy theory of commutative algebras in certain $\Q$-Koszul categories. We will make a connection with (affine) derived geometry and formal geometry by using our results above to analyse cotangent complexes. These results will prove useful later when we discuss operadic Koszul duality. Unless stated otherwise, throughout this subsection we shall assume that our Koszul category $\mathpzc{M}$ is $C$-monoidal.

\subsubsection{HA Contexts and the Cotangent Complex}\label{seccotangentcomplex}
Recall that, in \cite{toen2004homotopical}, To{\"e}n and Vezzosi introduce an abstract categorical framework in which one can `do' homotopical algebra, namely a homotopical algebra context. Let us recall the truncated definition. The word `truncated' here means that, for the category $\mathsf{C}_{0}$ in \cite{toen2004homotopical}, we always take $\mathsf{C}=\mathsf{C}_{0}$).

\begin{defn}\label{Defn:HA context}
Let $\mathpzc{M}$ be a combinatorial symmetric monoidal model category. We say that $\mathpzc{M}$ is a \textbf{homotopical algebra context} (or HA context) if for any $A\in\mathpzc{Alg}_{\mathfrak{Comm}}(\mathpzc{M})$, the following hold:
\begin{enumerate}
\item
the model category $\mathpzc{M}$ is proper and pointed, and, for any two objects $X$ and $Y$ in $\mathpzc{M}$, the natural morphisms
$$QX\coprod QY\rightarrow X\coprod Y\rightarrow RX\times RY$$
are equivalences;
\item
$Ho(\mathpzc{M})$ is an additive category;
\item
with the transferred model structure and monoidal structure $-\otimes_{A}-$, the category ${}_{A}\mathpzc{Mod}$ is a combinatorial, proper, symmetric monoidal model category;
\item
for any cofibrant object $M\in{}_{A}\mathpzc{Mod}$ the functor
$$-\otimes_{A}M:{}_{A}\mathpzc{Mod}\rightarrow{}_{A}\mathpzc{Mod}$$
preserves equivalences;
\item
with the transferred model structures, $\mathpzc{Alg}_{\mathfrak{Comm}}({}_{A}\mathpzc{Mod})$ and $\mathpzc{Alg}_{\mathfrak{Comm}_{nu}}({}_{A}\mathpzc{Mod})$ are combinatorial, proper model categories;
\item
if $B$ is cofibrant in $\mathpzc{Alg}_{\mathfrak{Comm}}({}_{A}\mathpzc{Mod})$ then the functor
$$B\otimes_{A}-:{}_{A}\mathpzc{Mod}\rightarrow{}_{B}\mathpzc{Mod}$$
preserves equivalences.
\end{enumerate}
\end{defn}
%FIX: DO NOT LEED LEFT PROPER

The following is Theorem 6.42 in \cite{kelly2016homotopy}.

\begin{thm}\label{HAcont}
Let $(\mathpzc{E},\otimes,\underline{\mathrm{Hom}},k)$ be a locally presentable closed projectively monoidal exact category that is $\textbf{AdMon}$-elementary and whose hom sets are $\mathbb{Q}$-vector spaces. Then $Ch_{\ge0}(\mathpzc{E})$ and $Ch(\mathpzc{E})$ are homotopical algebra contexts. 
\end{thm}

\begin{rem}
    In \cite{kelly2016homotopy} Theorem 6.42 the assumption that $\mathpzc{E}$ is $\mathbb{Q}$-enriched was intended, but mistakenly omitted. Indeed, the results on which that theorem relies (namely \cite{kelly2016homotopy} Theorem 6.33) do include that assumption explicitly. When $\mathpzc{E}$ is not $\mathbb{Q}$-enriched, but the symmetric powers $Sym^{n}(P)$ are projective for any projective $P$, the category of simplicial objects $s\mathpzc{E}$ is still an HA context. This fact is related to presentations of the connective part of the derived algebraic contexts of Raksit \cite{raksit2020hochschild}. 
\end{rem}

All our HA contexts of interest will be additive as categories, and we assume this from now on.

By \cite{toen2004homotopical} Section 1.2 a homotopical algebra context has sufficient structure to define the relative cotangent complex of a morphism $f:A\rightarrow B$ in $\mathpzc{Alg}_{\mathfrak{Comm}}(\mathpzc{M})$. Let us briefly recall the discussion here. For a commutative monoid $B$, write $\mathpzc{Alg}_{\mathfrak{Comm}}^{aug}({}_{B}\mathpzc{Mod})\defeq \mathpzc{Alg}_{\mathfrak{Comm}}({}_{B}\mathpzc{Mod})_{\big\slash B}$ for the category of augmented commutative $B$-algebras. There is an adjunction
$$\adj{K}{\mathpzc{Alg}_{\mathfrak{Comm}^{nu}}({}_{B}\mathpzc{Mod})}{\mathpzc{Alg}_{\mathfrak{Comm}}^{aug}({}_{B}\mathpzc{Mod})}{I}$$
where $K$ is the trivial extension functor and $I$ sends an augmented algebra $C\rightarrow B$ to the kernel of the morphism $C\rightarrow B$. This is both an equivalence of categories (since $\mathpzc{M}$ is additive) and a Quillen equivalence of model categories. There is also a Quillen adjunction
$$\adj{Q}{\mathpzc{Alg}_{\mathfrak{Comm}^{nu}}({}_{B}\mathpzc{Mod})}{{}_{B}\mathpzc{Mod}}{Z},$$
where $Q(C)$ is defined by the pushout
\begin{displaymath}
\xymatrix{
C\otimes_{B}C\ar[d]\ar[r] & C\ar[d]\\
\bullet\ar[r] & Q(C),
}
\end{displaymath}
and $Z$ just equips a module $M$ with the non-unital commutative monoid structure whose multiplication morphism is the zero morphism. For a commutative monoid $B$ and a $B$-module $M$ we denote by $B\ltimes M$ the \textbf{square-zero extension} of $B$ by $M$. It is a $B$-algebra that also lives over $B$. The underlying $B$-module is $B\coprod M$. The multiplication 
$$(B\coprod M)\otimes (B\coprod M)\cong B\otimes B\coprod B\otimes M\coprod B\otimes M\coprod M\otimes M\rightarrow B\coprod M$$
is defined:
\begin{enumerate}
\item
on the $B\otimes B$ component using the multiplication morphism on $B$, 
\item 
on the $B\otimes M$ components using the $B$-module structure on $M$,
\item 
and on $M\otimes M$ using the zero morphism (recall that a homotopical algebra context is a pointed category).
\end{enumerate}
Now, given a morphism $f:A\rightarrow B$ in $\mathpzc{Alg}_{\mathfrak{Comm}}(\mathpzc{M})$, we define the \textbf{relative cotangent complex} by $$\mathbb{L}_{B\big\slash A}\defeq\mathbb{L}Q\mathbb{R}I(B\otimes_{A}^{\mathbb{L}}B).$$
It is shown in \cite{toen2004homotopical} that $\mathbb{L}_{B\big\slash A}$ corepresents the functor of $(\infty,1)$-categories 
$${}_{B}\textbf{Mod}\rightarrow\textbf{sSet},\;M\mapsto \textbf{Map}_{\textbf{Alg}_{\mathfrak{Comm}}({}_{A}\mathpzc{Mod})\big\slash B}(B,B\ltimes M).$$ Here $\textbf{Map}$ is the simplicial mapping space. We also write $\mathbb{L}_{B}\defeq\mathbb{L}_{B\big\slash k}$ where $k$ is the monoidal unit. Let $C$ be any $A$-algebra, and consider the category $\textbf{Alg}_{\mathfrak{Comm}}({}_{A}\mathpzc{Mod})\big\slash C$.

In our additive setting, the construction of $\mathbb{L}_{B\big\slash A}$ is more straightforward.

There is a functor
$$\textbf{Alg}_{\mathfrak{Comm}}({}_{A}\mathpzc{Mod})\big\slash C\rightarrow{}_{C}\textbf{Mod},\; B\mapsto\mathbb{L}_{B\big\slash A}\otimes^{\mathbb{L}}_{B}C.$$
It is left adjoint to the functor sending a $C$-module $M$ to the square-zero extension $C\ltimes M$, where we forget the $C$-module structure on $M$ to a $B$-module structure, and then forget to an $A$-module structure. When $C=A=k$, so that $\textbf{Alg}_{\mathfrak{Comm}}({}_{A}\mathpzc{Mod})\big\slash k=\textbf{Alg}_{\mathfrak{Comm}}^{aug}$, we denote this functor by $\mathbb{L}_{0}$.  

\begin{prop}\label{maximalideal}

Let $A\rightarrow C$ be a morphism in $\mathpzc{Alg}_{\mathfrak{Comm}}(\mathpzc{M})$. Then the functor ${}_{C}\mathpzc{Mod}\rightarrow\mathpzc{Alg}_{\mathfrak{Comm}}({}_{A}\mathpzc{Mod}(\mathpzc{M}))_{\big\slash C}$, $M\mapsto C\ltimes M$ is right Quillen. Denoting the left adjoint by $B\mapsto C\otimes_{B}\Omega_{B\big\slash A}$, we have $\mathbb{L}(C\otimes_{B}\Omega_{B\big\slash A})\cong C\otimes^{\mathbb{L}}_{B}\mathbb{L}_{B\big\slash A}$. 

If $A=C$, then $A\otimes_{B}\Omega_{ B\big\slash A}\cong coker(I\otimes I\rightarrow I)$, where $I=ker(B\rightarrow A)$ is the augmentation ideal. Consequently, if $B$ is cofibrant, then this computes $\mathbb{L}_{0}(B)$.
%In particular, if $A=B$, the left adjoint is given by $C\mapsto \Omega^{1}_{C}=coker(I\otimes I\rightarrow I)$, where $I=ker(C\rightarrow B)$ is the augmentation ideal.  Finally, if $C\rightarrow k$ is augmented over the unit, then $\mathbb{L}_{0}(A)\cong  coker(I\otimes I\rightarrow I)[1]$.
%Suppose further that $\mathbb{L}_{0}(A)$ is cofibrant. Then $\mathbb{L}_{A}\cong A\otimes\mathbb{L}_{0}(A)$. 
\end{prop}

\begin{proof}

The functor $M\mapsto C\ltimes M$ commutes with all limits and colimits. By the adjoint functor theorem, a left adjoint exists. Now, suppose that $M\rightarrow N$ is a fibration (resp. acyclic fibration) of $C$-modules. Then the morphism underlying $C\ltimes M\rightarrow C\ltimes N$ is $C\times M\rightarrow C\times N$ (since we are working in an additive category!). This is a fibration (resp. acyclic fibration) of $A$-modules. Hence, $C\ltimes M\rightarrow C\ltimes N$ is a fibration (resp. an acyclic fibration) of augmented $C$-algebras under $A$. Thus, $C\ltimes(-)$ is right Quillen. By adjunction, the derived functor of the left adjoint corepresents the same functor as $B\mapsto C\otimes^{\mathbb{L}}_{B}\mathbb{L}_{B\big\slash A}$. In particular, these two objects are naturally equivalent.

Now, consider the case where $A=C$, and $B$ is cofibrant. We may assume everything is underived. We need to show that the functor sending $B$ to the $A$-module $coker(I\otimes I\rightarrow I)$ is left adjoint to the square-zero extension functor $M\mapsto A\ltimes M$. Let $f:B\rightarrow A\ltimes M$ be a morphism of augmented algebras. This induces a morphism $I\rightarrow M$ of augmentation ideals that clearly descends to a morphism $\tilde{f}:coker(I\otimes I\rightarrow I)\rightarrow M$. Conversely, suppose we are given a morphism $g: coker(I\otimes I\rightarrow I)\rightarrow M$. Consider the morphism of modules $B\cong A\coprod I\rightarrow A\coprod I\big\slash I^{2}\rightarrow A\coprod M$. This is in fact a morphism of algebras $B\rightarrow A\ltimes M$. These morphisms on hom sets are inverse, realising the adjunction.
%Now, let $C\rightarrow I$ be a cofibrant replacement in $\mathpzc{Alg}_{\mathfrak{Comm}}^{nu}(\mathpzc{M})$ of $I$. Then the augmented algebra $k\ltimes C$ is cofibrant in $\mathpzc{Alg}_{\mathfrak{Comm}}^{aug}(\mathpzc{M})$, and $k\ltimes C\rightarrow A$ is an equivalence. Since both $k\ltimes C$ and $A$ are cofibrant, we have $\Omega^{1}_{k\ltimes C}\cong\Omega^{1}_{A}$. It therefore suffices to prove that $\Omega^{1}_{k\ltimes C}\cong\mathbb{L}_{k\big\slash k\ltimes C}$. Now, $k$ is a retract of $k\ltimes C$, and is therefore cofibrant as a $C$-algebra. We therefore have $k\otimes^{\mathbb{L}}_{k\ltimes C}k$ 

%We compute $\mathbb{L}_{k\big\slash A}$. Since $A$ is a cofibrant, we may assume everything is underived. By Proposition \ref{prop:L0shift}, 
%For the second assertion note that $A$ is cofibrant in the slice category over $A$. Thus, we need to show that
%$$Hom_{\mathpzc{Alg}_{\mathpzc{E}}\big\slash A}(A,A\ltimes M)\cong Hom_{{}_{A}\mathpzc{Mod}}(A\otimes\mathbb{L}_{0}(A),M)\cong Hom_{\mathpzc{E}}(\mathbb{L}_{0}(A),M)$$
%$$\;\;\;\;\;\;\;\;\;\;\;\;\;\;\;\;\;\;\;\;\;\;\;\cong Hom_{E}(I\big\slash I^{2},M)$$
%This computation is entirely similar to the first part.
\end{proof}

\begin{cor}
If $B$ is a cofibrant augmented algebra, then the unit of the adjunction
$$\adj{\mathbb{L}_{0}}{\textbf{Alg}_{\mathfrak{Comm}}^{aug}}{{}_{k}\textbf{Mod}}{k\ltimes(-)}$$
is the natural morphism 
$$B\rightarrow k\ltimes coker(I\otimes I\rightarrow I).$$
\end{cor}

We will make use of the following facts constituting Proposition 1.2.1.6 in \cite{toen2004homotopical}.

%\begin{prop}\label{cotangentfacts}
%\begin{enumerate}
%\item
%Let $f:A\rightarrow B$ and $g:B\rightarrow C$ be morphisms of algebras. Then there is a homotopy cofibre sequence in ${}_{C}\mathpzc{Mod}$.
%$$\mathbb{L}_{B\big\slash A}\otimes_{B}^{\mathbb{L}} C\rightarrow\mathbb{L}_{C\big\slash A}\rightarrow\mathbb{L}_{C\big\slash B}$$
%\item
%If 
%\begin{displaymath}
%\xymatrix{
%A\ar[d]\ar[r] &B\ar[d]\\
%A'\ar[r] &B'
%}
%\end{displaymath}
%is a homotopy pushout in $\mathpzc{Alg}_{\mathfrak{Comm}}(\mathpzc{E})$ then the natural morphism $\mathbb{L}_{B\big\slash A}\otimes_{B}^{\mathbb{L}}B'\rightarrow\mathbb{L}_{B'\big\slash A'}$ is an equivalence.
%\end{enumerate}
%\end{prop}

\begin{prop}\label{prop:L0shift}
There is a natural equivalence of functors $\mathbb{L}_{0}(-)[-1]\cong \mathbb{L}_{k\big\slash (-)}$, where $\mathbb{L}_{0}(-)[-1]$ is the suspension of $\mathbb{L}_{0}(-)$. 
\end{prop}

\begin{proof}
Consider the composition $k\rightarrow A\rightarrow k$. By \cite{toen2004homotopical} Proposition 1.2.1.6, we get a homotopy cofibre sequence
$$\mathbb{L}_{0}(A)\rightarrow 0\rightarrow\mathbb{L}_{k\big\slash A}.$$
This proves the claim.
\end{proof}

\begin{rem}
    Note that, in the context of a Koszul category ${}_{R}\mathpzc{Mod}(Ch(\mathpzc{E}))$ with model structure transferred from one on $Ch(\mathpzc{E})$, and whose weak equivalences are the quasi-isomorphisms, the suspension functor coincides with negative shift by \cite{kelly2016homotopy} Proposition 4.48.
\end{rem}

\subsubsection{Homotopy Pushouts in Koszul Categories}

As is evident from results quoted in the previous section, knowing how to compute homotopy pushouts might help with computing relative cotangent complexes. Let us see how to do this in a $\mathbb{Q}$-Koszul category ${}_{R}\mathpzc{Mod}(Ch(\mathpzc{E}))$, where $Ch(\mathpzc{E})$ is basic Koszul, by using Lemma \ref{standard}. First, let us give a general definition and technical proposition.

\begin{defn}
A morphism $f:A\rightarrow B$ of commutative algebras in a model category $\mathpzc{M}$ is said to be $K$-\textbf{flat} if $B$ is $K$-flat as an $A$-module.
\end{defn}

We will be interested in cases where cofibrations are $K$-flat.

\begin{lem}\label{lem:cofibKflat}
    Let $\mathpzc{M}$ be a combinatorial model category equipped with a symmetric monoidal structure, such that the transferred model structure exists on $\mathpzc{Alg}_{\mathfrak{Comm}}({}_{A}\mathpzc{Mod}(\mathpzc{M}))$ for any commutative monoid object $A$. Let $g:A\rightarrow B$ be a cofibration in $\mathpzc{Alg}_{\mathfrak{Comm}}(\mathpzc{M})$. Then $g$ is a cofibration in $\mathpzc{Alg}_{\mathfrak{Comm}}({}_{A}\mathpzc{Mod}(\mathpzc{M}))$.
\end{lem}

\begin{proof}
    The morphism $g:A\rightarrow B$ may be written as a transfinite composition of pushouts of the form 
    \begin{displaymath}
        \xymatrix{
        S(X)\ar[d]^{S(f)}\ar[r] & A_{\lambda}\ar[d]\\
        S(Y)\ar[r] & A_{\lambda+1},
        }
    \end{displaymath}
    where $f:X\rightarrow Y$ is a cofibration in $\mathpzc{M}$, and $A_{0}=A$. Now, each $A_{\lambda}$ is an $A$-algebra. Thus, $S(X)\rightarrow A_{\lambda}$ factors through $S(X)\rightarrow A\otimes S(X)\rightarrow A_{\lambda}$. Consider the rectangle
    \begin{displaymath}
        \xymatrix{
        S(X)\ar[d]^{S(f)}\ar[r] & A\otimes S(X)\ar[d]^{Id_{A}\otimes S(f)}\ar[r] & A_{\lambda}\ar[d]\\
        S(Y)\ar[r] & A\otimes S(Y)\ar[r] & A_{\lambda+1}.
        }
    \end{displaymath}
    The rectangle is a pushout in $\mathpzc{Alg}_{\mathfrak{Comm}}(\mathpzc{M})$, and the left-hand square is a pushout in $\mathpzc{Alg}_{\mathfrak{Comm}}(\mathpzc{M})$. Thus, the right-hand square is a pushout in $\mathpzc{Alg}_{\mathfrak{Comm}}(\mathpzc{M})$. Thus, $g:A\rightarrow B$ is in fact a transfinite composition in $\mathpzc{Alg}_{\mathfrak{Comm}}({}_{A}\mathpzc{Mod})$ of pushouts along morphisms of the form $Id_{A}\otimes S(f)$, where $f:X\rightarrow Y$ is a cofibration in $\mathpzc{M}$. In particular, it is a cofibration in $\mathpzc{Alg}_{\mathfrak{Comm}}({}_{A}\mathpzc{Mod})$.
\end{proof}

\begin{cor}

    Let $\mathpzc{M}$ be a monoidal model category satisfying the monoid axiom. Suppose that $\mathpzc{M}$ is enriched over $\mathbb{Q}$. In particular, the transferred model structure exists on $\mathpzc{Alg}_{\mathfrak{Comm}}({}_{A}\mathpzc{Mod}(\mathpzc{M}))$ for any commutative monoid object $A$. Let $g:A\rightarrow B$ be a cofibration in $\mathpzc{Alg}_{\mathfrak{Comm}}(\mathpzc{M})$. Then $B$ is cofibrant as an $A$-module. In particular, if $\mathpzc{M}$ is $K$-monoidal, then $B$ is $K$-flat as an $A$-module. 
\end{cor}

\begin{proof}
This follows from Lemma \ref{lem:cofibKflat}, Proposition \ref{underlyingKflat} (3), Example \ref{ex:appofabove} (2), and the fact that $\mathfrak{Comm}(n)=A$ is $\Sigma$-cofibrant as an $A$-module.
\end{proof}

\begin{prop}\label{htpypushoutKflatmodel}
Suppose that $\mathpzc{M}$ is a combinatorial model category equipped with a symmetric monoidal structure, such that the transferred model structure exists on $\mathpzc{Alg}_{\mathfrak{Comm}}(\mathpzc{M})$, and cofibrations between cofibrant objects in $\mathpzc{Alg}_{\mathfrak{Comm}}(\mathpzc{M})$ are $K$-flat. The homotopy pushout of a diagram
\begin{displaymath}
\xymatrix{
A\ar[d]\ar[r] & B\\
C
}
\end{displaymath}
is equivalent in $\mathpzc{M}$ (in fact as a $(B,C)$-bimodule) to $B\otimes_{A}^{\mathbb{L}}C$. In particular, if $B\otimes_{A}^{\mathbb{L}}C\rightarrow B\otimes_{A}C$ is an equivalence, then $B\otimes_{A}C$ is a homotopy pushout of the diagram above.
\end{prop}
% a diagram of algebras such that $B\otimes_{A}^{\mathbb{L}}C\rightarrow B\otimes_{A}C$ is an equivalence. Then the diagram
%\begin{displaymath}
%\xymatrix{
%A\ar[d]\ar[r] & B\ar[d]\\
%C\ar[r] &  B\otimes_{A}C
%}
%\end{displaymath}
%is a homotopy pushout diagram.

\begin{proof}
Let 
\begin{displaymath}
\xymatrix{
\tilde{A}\ar[d]\ar[r] & \tilde{B}\ar[d]\\
\tilde{C}\ar[r] &  P
}
\end{displaymath}
be a diagram presenting the homotopy pushout of the original diagram. There are equivalences $A\rightarrow\tilde{A}$, $B\rightarrow\tilde{B}$, and $C\rightarrow\tilde{C}$. Moreover, we may assume that $\tilde{A}$, $\tilde{B}$ and $\tilde{C}$ are cofibrant and that $\tilde{A}\rightarrow\tilde{B}$ is a cofibration. Then 
$$P\cong \tilde{B}\otimes_{\tilde{A}}\tilde{C}\cong \tilde{B}\otimes^{\mathbb{L}}_{\tilde{A}}\tilde{C}\cong B\otimes^{\mathbb{L}}_{A}C.$$
\end{proof}

%From now on, we fix a $\mathbb{Q}$-Koszul category $\mathpzc{M}={}_{R}\mathpzc{Mod}(Ch(\mathpzc{E}))$, 

%where $Ch(\mathpzc{E})$ is basic Koszul and $\mathpzc{M}$ is equipped with the transferred model structure. 

%.
%We will also fix a set of generating cofibrations $\mathfrak{G}$, corresponding to a collection of objects $\{G_{i}\}$, and a collection of morphisms $\{f_{l}\}$, as in Definition \ref{basic Koszul} (4).

\begin{prop}\label{sullivanKflat}
 
Let $\mathpzc{M}$ be a $\Q$-Koszul category, and let $f:A\rightarrow B$ be in $Cof_{\mathfrak{Comm}}(\mathpzc{M};\mathfrak{G})$, where $\{G_{i}\}_{i\in\mathcal{I}}$ consists of $K$-flat objects, and the elements of $\{f_{l}\}_{l\in\mathcal{L}}$ all have $K$-flat domain and codomain, and are pure (recall the notation from Definition \ref{defn:sullivan} here). Then $B$ is $K$-flat as an $A$-module.
\end{prop}

\begin{proof}

We use Proposition \ref{underlyingKflat} by taking $\mathcal{S}$ to be the class of pure monomorphisms. As we have seen, this class consists of left proper morphisms, is closed under pushout-product, and is closed under tensoring with any object.

%This is an immediate consequence of (the proof of) Proposition \ref{underlyingKflat} and Lemma \ref{lem:cofibKflat}.
%As an $A$-module, $B$ is a retract of $A\otimes_{\alpha}S(V)$ for some $K$-flat complex $V$. Without loss of generality we may assume that $B\cong A\otimes_{\alpha}S(V)$. Endow $A\otimes_{\alpha}S(V)$ with the filtration induced by the grading on $S(V)$ by $n$-th symmetric powers. This is a filtration by pure monomorphisms in the category ${}_{A}\mathpzc{Mod}(\mathpzc{M})$. The associated graded module is $A\otimes S(V)$. Since $V$ is $K$-flat, $S(V)$ is as well. Therefore, $A\otimes S(V)$ is free on a $K$-flat object of $\mathpzc{M}$, and so is a $K$-flat object of ${}_{A}\mathpzc{Mod}(\mathpzc{M})$. This suffices to prove the claim.
\end{proof}

\begin{cor}
If $\mathpzc{M}\cong{}_{R}\mathpzc{Mod}(Ch(\mathpzc{E}))$ is projective $\Q$-Koszul, with $\{G_{i}\}$ consisting of $K$-flat objects (recall notation and terminology from Definition \ref{basic Koszul}), then cofibrations in $\mathpzc{Alg}_{\mathfrak{Comm}}(\mathpzc{M})$ are $K$-flat morphisms. 
 \end{cor}

%\begin{prop}
%FIX: CHECK Let $A$ be a cofibrant algebra, and $A\rightarrow B$ a cofibration of coalgebras. Then $B$ is cofibrant as an $A$-module.
%\end{prop}
%
%\begin{proof}
%
%\end{proof}
The utility of this corollary is that in projective $\Q$-Koszul categories we can compute cotangent complexes very easily, even for certain algebras that are not necessarily cofibrant.

\begin{prop}\label{cotangentqfree}
Let $\mathpzc{M}$ be a $C$-monoidal Koszul category enriched over $\mathbb{Q}$ in which cofibrations between cofibrant objects of $\mathpzc{Alg}_{\mathfrak{Comm}}(\mathpzc{M})$ are $K$-flat. Let $A\rightarrow k$ be an augmented commutative monoid that is obtained from $k$ as a transfinite composition of pushouts of morphisms of the form $S(X\rightarrow Y)$, where $X\rightarrow Y$ is a pure monomorphism between $K$-flat objects. Then $\mathbb{L}_{0}(A)\cong (I\big\slash I^{2})$, where $I$ is the augmentation ideal of $A$.
%Let $A$ be a $\mathfrak{G}$-Sullivan model where $\{G_{i}\}$ consists of flat objects. Then $\mathbb{L}_{0}(A)\cong I\big\slash I^{2}$, where $I$ is the augmentation ideal of $A$.
\end{prop}
% and $\mathbb{L}_{A}\cong A\otimes I\big\slash I^{2}$

\begin{proof}
We need to show that $\mathbb{L}_{0}(A)\cong I\big\slash I^{2}$.
We break the proof into several steps.
First, suppose that $A\cong S(V)$ is free on a $K$-flat object $V$. Pick a cofibrant resolution $W\rightarrow V$ of $V$. Then $S(W)\rightarrow S(V)$ is a weak equivalence. Thus, $\mathbb{L}_{0}(S(V))\cong\mathbb{L}_{0}(S(W))\cong W\cong V$.

%Moreover, $\mathbb{L}(A)\cong \mathbb{L}(S(W))\cong S(W)\otimes W\cong A\otimes V$. 

Now, let $A\in\mathpzc{Alg}_{\mathfrak{Comm}}(\mathpzc{M})$ and suppose that $\mathbb{L}_{0}(A)\cong I\big\slash I^{2}$.
% and $\mathbb{L}_{A}\cong A\otimes V$. 
 Let $f:X\rightarrow Y$ be a pure monomorphism between $K$-flat objects. 
 Consider the pushout
\begin{displaymath}
\xymatrix{
S(X)\ar[d]\ar[r] & A\ar[d]\\
S(Y)\ar[r] & A'.
}
\end{displaymath}
Now, $S(X)\rightarrow S(Y)$ is a $K$-flat morphism by Proposition \ref{sullivanKflat}, so the diagram above is a homotopy pushout by Proposition \ref{htpypushoutKflatmodel}. Therefore, the diagram
\begin{displaymath}
\xymatrix{
X\ar[d]\ar[r] & I\big\slash I^{2}\ar[d]\\
Y\ar[r] & \mathbb{L}_{0}(A')
}
\end{displaymath}
is a homotopy pushout diagram in $\mathpzc{M}$. Since $X\rightarrow Y$ is an admissible monomorphism, this is in fact just presented by the ordinary pushout diagram by Proposition 4.46 in \cite{kelly2016homotopy} and Proposition \ref{prop:cokercofib}. The functor $(A\rightarrow k)\mapsto coker(I\otimes I\rightarrow I)$ is a left adjoint, and therefore commutes with colimits. Applying this fact to the pushout of augmented algebras above gives a pushout diagram
\begin{displaymath}
    \xymatrix{
    X\ar[d]\ar[r] & I\big\slash I^{2}\ar[d]\\
    Y\ar[r] &I'\big\slash I'^{2}.
    }
\end{displaymath}
Comparing the pushout diagrams gives $\mathbb{L}_{0}(A')\cong I'\big\slash I'^{2}$. 
Now let $\lambda$ be an ordinal, and let $F:\lambda\rightarrow\mathpzc{Alg}_{\mathfrak{Comm}}(\mathpzc{M})$ be a cocontinuous functor such that each algebra $A_{j}\defeq F(j)$ satisfies $\mathbb{L}_{0}(A_{j})\cong I_{j}\big\slash I_{j}^{2}$, and each morphism $A_{i}\rightarrow A_{j}$ is a pure monomorphism in $\mathpzc{M}$. The functor $\colim :\mathpzc{Fun}_{\textbf{PureMon}}(\lambda,\mathpzc{Alg}_{\mathfrak{Comm}}(\mathpzc{M}))\rightarrow\mathpzc{Alg}_{\mathfrak{Comm}}(\mathpzc{M})$ is exact and commutes with colimits. Therefore, 
$$\mathbb{L}_{0}(\colim A_{j})\cong \colim \mathbb{L}_{0}(A_{j})\cong \colim I_{j}\big\slash I_{j}^{2}\cong I\big\slash I^{2}$$
where $I$ is the augmentation ideal of $A=\colim A_{j}$. The last step follows because $\otimes$ commutes with colimits and the morphism $\colim I_{j}\rightarrow I$ is an isomorphism, which follows from the Obscure Axiom and the fact that the functor $\colim :\mathpzc{Fun}_{\textbf{PureMon}}(\lambda,\mathpzc{M})\rightarrow\mathpzc{M}$ is exact. 
\end{proof}

%Then 

%$$P\cong \tilde{B}\otimes^{\mathbb{L}}_{\tilde{A}}\tilde{C}\cong C\otimes^{\mathbb{L}}_{A}B\cong C\otimes_{A}B$$

%
%By Theorem \ref{standard} if $\mathpzc{M}$ is a strong $\Q$-Koszul category then any cofibration in $\mathpzc{Alg}_{\mathfrak{Comm}}(\mathpzc{M})$ is $K$-flat. 
%BOUNDED FOR FUCK'S SAKE
\subsubsection{Formal Completions and Cotangent Complexes}
Here we assume $\mathpzc{M}={}_{R}\mathpzc{Mod}(Ch(\mathpzc{E}))$ is a $KC$-monoidal $\Q$-Koszul category.
Let $A\in\mathpzc{Alg}_{\mathfrak{Comm}}(\mathpzc{M})$ be $\mathfrak{G}$-triangulated, $A=\colim_{n\ge 0}S(V_{n})$. In particular, the universal property of the colimit furnishes a morphism $S(V_{0})\rightarrow A$ in $\mathpzc{Alg}_{\mathfrak{Comm}}(\mathpzc{M})$. Consider the commutative monoid $\hat{S}(V_{0})\defeq\prod_{n=0}^{\infty}Sym^{n}(V_{0})$. There is a natural morphism of $R$-modules $S(V_{0})\rightarrow\hat{S}(V_{0})$. We define the algebra $\hat{A}$ to be the pushout $\hat{A}\defeq\hat{S}(V_{0})\otimes_{S(V_{0})} A$. Note that, in general, this construction depends on the presentation of $A$ as a $\mathfrak{G}$-triangulated object.

\begin{defn}
 We say that an object $V$ of $\mathpzc{M}$ is \textbf{decent} if there is an equivalence
$$\mathbb{L}_{\hat{S}(V)\big\slash S(V)}\otimes^{\mathbb{L}}_{\hat{S}(V)}R\cong 0$$
in $\mathpzc{M}$.
\end{defn}

\begin{prop}
 Let $A$ be a $\mathfrak{G}$-triangulated commutative algebra, where $\{G_{i}\}$ consists of flat objects. If $V_{0}$ is decent, then the morphism $\mathbb{L}_{0}(A)\rightarrow\mathbb{L}_{0}(\hat{A})$ is an equivalence.
\end{prop}

\begin{proof}
Consider the homotopy cofibre sequence
$$\mathbb{L}_{\hat{A}\big\slash A}\otimes^{\mathbb{L}}_{\hat{A}}R\rightarrow\mathbb{L}_{R\big\slash A}\rightarrow\mathbb{L}_{R\big\slash\hat{A}}$$
furnished by \cite{toen2004homotopical} Proposition 1.2.1.6.
We claim that $\mathbb{L}_{\hat{A}\big\slash A}\otimes^{\mathbb{L}}_{\hat{A}}R\cong 0$, which would prove the result. Now, using Proposition \ref{completionhpush} and \cite{toen2004homotopical} Proposition 1.2.1.6 gives that $\mathbb{L}_{\hat{S}(V_{0})\big\slash S(V_{0})}\otimes^{\mathbb{L}}_{\hat{S}(V_{0})}\hat{A}\cong\mathbb{L}_{\hat{A}\big\slash A}$. Therefore, 
$$\mathbb{L}_{\hat{A}\big\slash A}\otimes^{\mathbb{L}}_{\hat{A}}R\cong\mathbb{L}_{\hat{S}(V_{0})\big\slash S(V_{0})}\otimes^{\mathbb{L}}_{\hat{S}(V_{0})}\hat{A}\otimes^{\mathbb{L}}_{\hat{A}}R\cong\mathbb{L}_{\hat{S}(V_{0})\big\slash S(V_{0})}\otimes^{\mathbb{L}}_{\hat{S}(V_{0})}R.$$
 Since $V_{0}$ is decent, the term on the right is equivalent to $0$.
\end{proof}

\begin{prop}\label{completionhpush}
Let $A$ be a $\mathfrak{G}$-triangulated commutative algebra, where $\{G_{i}\}$ consists of flat objects. Then the diagram below is a homotopy pushout
\begin{displaymath}
\xymatrix{
S(V_{0})\ar[r]\ar[d] & \hat{S}(V_{0})\ar[d]\\
A\ar[r] & \hat{A}.
}
\end{displaymath}
\end{prop}

\begin{proof}
The morphism $S(V_{0})\rightarrow A$ is clearly in $Sull(\mathpzc{M};\mathfrak{G})$ and is therefore $K$-flat. By Proposition \ref{htpypushoutKflatmodel} the claim follows.
\end{proof}

Let us give our main example of this construction. Suppose that $\mathpzc{M}={}_{R}\mathpzc{Mod}(Ch(\mathpzc{E}))$ where $R$ is non-negatively graded. Let $A$ be quasi-free on an object of the form $R\otimes V$, where $V$ is a $\mathfrak{G}$-non-negatively graded complex in $Ch(\mathpzc{E})$ such that each $V_{n}$ is in $\mathfrak{G}$. 

The object $A$ is $\mathfrak{G}$-non-negatively graded. After forgetting differentials, the underlying algebra of $A$ is 
$$S_{R}(R\otimes V)\cong R\otimes S(V)\cong R\otimes S(V_{0})\otimes S(V\big\slash V_{0}),$$
where $S_{R}(R\otimes V)$ denotes the free commutative algebra taken in ${}_{R}\mathpzc{Mod}(Ch(\mathpzc{E}))$. Then, after forgetting differentials, 
$$\hat{A}\cong \hat{S}_{R}(R\otimes V_{0})\otimes S(V\big\slash V_{0}).$$
We claim that under certain circumstances this is isomorphic to $\hat{S}_{R}(R\otimes V)$. This is addressed in the next subsection.

\subsubsection{$\aleph_{1}$-Filtered Objects}

\begin{defn}\label{defn:alephfiltered}
Let $\mathpzc{M}$ be a complete monoidal category and let $\mathfrak{O}$ be a class of objects in $\mathpzc{M}$. An object $V$ of  $\mathpzc{M}$ is said to be \textbf{formally} $\aleph_{1}$-\textbf{filtered relative to} $\mathfrak{O}$ if the canonical morphism $V\otimes\prod_{n=1}^{\infty}(W_{n})\rightarrow \prod_{n=1}^{\infty}(V\otimes W_{n})$ is an isomorphism for any countable collection $\{W_{n}\}$ of objects of $\mathfrak{O}$. If $\mathfrak{O}=\mathrm{Ob}(\mathpzc{M})$, then $V$ is said to be \textbf{formally }$\aleph_{1}$-\textbf{filtered}.
\end{defn}

The terminology above is inspired by \cite{ben2020fr} Section 5. The following is clear:

\begin{prop}
Let $\mathfrak{O}$ be a class of objects in a complete additive monoidal category $\mathpzc{M}$.
\begin{enumerate}
\item
If $W$ is a summand of an object $V$ that is formally $\aleph_{1}$-filtered relative to $\mathfrak{O}$, then $W$ is formally $\aleph_{1}$-filtered relative to $\mathfrak{O}$.
\item
If $V$ is formally $\aleph_{1}$-filtered relative to $\mathfrak{O}$, and $W$ is formally $\aleph_{1}$-filtered relative to $V\otimes\mathfrak{O}\defeq \{V\otimes O:O\in\mathfrak{O}\}$, then $W\otimes V$ is  formally $\aleph_{1}$-filtered relative to $\mathfrak{O}$.
\end{enumerate}
\end{prop}

\begin{example}\label{aleph1examples}
\begin{enumerate}
\item
If $\mathpzc{M}={}_{R}\mathpzc{Mod}$ where $R$ is a ring, then any $\aleph_{1}$-filtered colimit of finitely presented $R$-modules is formally $\aleph_{1}$-filtered. 
In particular, if $R$ is Noetherian, then any $\aleph_{1}$-filtered colimit of finitely generated modules is formally $\aleph_{1}$-filtered.
\item
If $\mathpzc{M}={}_{R}\mathpzc{Mod}(Ind(Ban_{R}))$ or $\mathpzc{M}={}_{R}\mathpzc{Mod}(CBorn_{R})$ for $R$ a Banach ring, then Corollary 5.13 in \cite{ben2020fr} says that any bornological Fr\'{e}chet space (see Definition \ref{defn:bfs} later) is formally $\aleph_{1}$-filtered.  
\end{enumerate}
\end{example}

\begin{defn}\label{defn:alephfiltered2}
Let $\mathpzc{M}={}_{R}\mathpzc{Mod}(Ch(\mathpzc{E}))$. For an object $V$ of $Ch(\mathpzc{E})$, we denote by $R\otimes V^{\otimes}\otimes V^{sym}$ the collection of objects in $\mathpzc{M}$: $\{R_{i}\otimes (V^{\otimes m})_{j}\otimes (Sym^{n}(V))_{k}:m,n\in\mathbb{N},i,j,k\in\mathbb{Z}\}$. 
\end{defn}

%\begin{defn}
%Let $V$ be an object of a monoidal additive category $\mathpzc{M}$. We say that $V$ is \textbf{monoidally self weakly }$\aleph_{1}$-\textbf{filtered} if $V$ is formally $\aleph_{1}$-filtered relative to the class $\{V^{m}\otimes Sym^{n}:m,n\in\mathbb{N}_{0}\}$. 
%\end{defn}

\begin{prop}\label{completiontwoways}
Suppose that $\mathpzc{M}$ is enriched in $\mathbb{Q}$-vector spaces. Let $A\in\mathpzc{Alg}_{\mathfrak{Comm}}(\mathpzc{M})$ be quasi-free on an object $V=\bigoplus_{i=0}^{\infty}V_{i}[-i]\in\underline{Gr}_{\mathbb{N}_{0}}(\mathpzc{E})$. Consider the presentation of $A$ as a $\mathfrak{G}$-triangulated algebra, $A=\colim A_{n}$, where the underlying graded algebra of $A_{n}$ is $S_{R}(R\otimes\bigoplus_{i=0}^{n}V_{i})$. Suppose that, for all $0\le i<\infty$, $V_{i}$ is formally $\aleph_{1}$-filtered relative to $R\otimes V^{\otimes}\otimes V^{sym}$. Then the natural morphism of graded objects $\hat{A}\rightarrow\hat{S}_{R}(R\otimes V_{0})\otimes_{R} S_{R}(R\otimes V\big\slash R\otimes V_{0})\cong\hat{S}_{R}(R\otimes V_{0})\otimes S(V\big\slash V_{0})$ is an isomorphism.
\end{prop}

\begin{proof}
Consider the graded object
\begin{align*}
\hat{S}_{R}(R\otimes V)&=\prod_{n=0}^{\infty}R\otimes Sym^{n}(V)\\
&\cong\prod_{n=0}^{\infty}R\otimes\bigoplus_{i+j=n}Sym^{i}(V_{0})\otimes Sym^{j}(V\big\slash V_{0})\\
&\cong\prod_{n=0}^{\infty}\bigoplus_{j=0}^{n}(R\otimes Sym^{n-j}(V_{0})\otimes Sym^{j}(V\big\slash V_{0})).
\end{align*}
To make life easy for ourselves we will declare $Sym^{m}(W)=0$ for any object $W$ and any $m<0$. Thus, we may write the above as 
$$\prod_{n=0}^{\infty}\prod_{j=0}^{\infty}(R\otimes Sym^{n-j}(V_{0})\otimes Sym^{j}(V\big\slash V_{0}))\cong\prod_{j=0}^{\infty}\prod_{n=0}^{\infty}(R\otimes Sym^{n-j}(V_{0})\otimes Sym^{j}(V\big\slash V_{0})).$$
We claim that the natural morphism
$$\bigoplus_{j=0}^{\infty}\prod_{n=0}^{\infty}(R\otimes Sym^{n-j}(V_{0})\otimes Sym^{j}(V\big\slash V_{0}))\rightarrow\prod_{j=0}^{\infty}\prod_{n=0}^{\infty}(R\otimes Sym^{n-j}(V_{0})\otimes Sym^{j}(V\big\slash V_{0}))$$
is an isomorphism. The object $Sym^{j}(V\big\slash V_{0})$ is a retract of $(V\big\slash V_{0})^{\otimes j}$. Since $V\big\slash V_{0}$ is homologically concentrated in strictly positive degrees, $(V\big\slash V_{0})^{\otimes j}$ will be concentrated in degrees $\ge j$. Thus, for each $m$, $(V\big\slash V_{0})^{\otimes j}_{m}\neq0$ only for $j\le m$. The same is then true of the symmetric powers $Sym^{j}(V\big\slash V_{0})$. Then 

\begin{align*}
    (\prod_{j=0}^{\infty}\prod_{n=0}^{\infty}(R\otimes Sym^{n-j}(V_{0})\otimes Sym^{j}(V\big\slash V_{0})))_{m}&\cong \prod_{j=0}^{\infty}\prod_{n=0}^{\infty}(R\otimes Sym^{n-j}(V_{0})\otimes Sym^{j}(V\big\slash V_{0}))_{m}\\
    &\cong\bigoplus_{j=0}^{m}\prod_{n=0}^{\infty}(R\otimes Sym^{n-j}(V_{0})\otimes Sym^{j}(V\big\slash V_{0}))_{m}\\
    &\cong\bigoplus_{j=0}^{\infty}\prod_{n=0}^{\infty}(R\otimes Sym^{n-j}(V_{0})\otimes Sym^{j}(V\big\slash V_{0}))_{m}.
\end{align*}
This proves the claim. Next, by the $\aleph_{1}$ condition, $\bigoplus_{j=0}^{\infty}\prod_{n=0}^{\infty}(R\otimes Sym^{n-j}(V_{0})\otimes Sym^{j}(V\big\slash V_{0}))$ is isomorphic to
$$\bigoplus_{j=0}^{\infty}((\prod_{n=0}^{\infty}R\otimes Sym^{n-j}(V_{0}))\otimes Sym^{j}(V\big\slash V_{0}))\cong\bigoplus_{j=0}^{\infty}\hat{S}_{R}(R\otimes V_{0})\otimes Sym^{j}(V\big\slash V_{0})\cong\hat{S}_{R}(R\otimes V_{0})\otimes S(V\big\slash V_{0}).$$
%We also have 
%\begin{align*}
%\bigoplus_{j=0}^{\infty}\prod_{n=j}^{\infty}((R\otimes Sym^{n-j}(V_{0}))\otimes Sym^{j}%(V\big\slash V_{0}))_{m}\cong\bigoplus_{j=0}^{m}\prod_{n=j}^{\infty}((R\otimes Sym^{n-%j}(V_{0}))\otimes Sym^{j}(V\big\slash V_{0}))_{m}.
%\end{align*}
%Because the sum from $j=0$ to $m$ is finite we get
%$$\prod_{n=0}^{\infty}\bigoplus_{j=0}^{m}((R\otimes Sym^{n-j}(V_{0}))\otimes Sym^{j}(V\big\slash V_{0}))_{m}\cong\bigoplus_{j=0}^{m}\prod_{n=j}^{\infty}((R\otimes Sym^{n-j}(V_{0}))\otimes Sym^{j}(V\big\slash V_{0}))_{m}$$
%for all $m$, and hence

%This is isomorphic to
%$$\prod_{n=0}^{\infty}\bigoplus_{j=0}^{m}(R\otimes Sym^{n-j}(V_{0}))\otimes Sym^{j}(V\big\slash V_{0})\cong\bigoplus_{j=0}^{\infty}\prod_{n=j}^{\infty}((R\otimes Sym^{n-j}(V_{0}))\otimes Sym^{j}(V\big\slash V_{0})).$$

%In homological degree $m$ this is
%$$\bigoplus_{p+q=m}\prod_{n=0}^{\infty}\bigoplus_{i+j=n}((R_{p}\otimes Sym^{i}(V_{0}))\otimes Sym^{j}_{q}(V\big\slash V_{0}))$$
%Note that since the sum is finite we may bring it outside the product. 

%By the $\aleph_{1}$-condition this is equivalent to 
%$$\bigoplus_{p+q=m}(\prod_{n=0}^{\infty}\bigoplus_{i+j=n}R_{p}\otimes Sym^{i}(V_{0}))\otimes Sym^{j}_{q}(V\big\slash V_{0})$$
%This is precisely
%The degree $k$ term is $\prod_{n=0}^{\infty}\bigoplus_{i+j=k}R_{i}\otimes Sym_{j}^{n}(V)$. Since the sum is finite this is isomorphic to 
%$$\bigoplus_{i=0}^{k}\prod_{n=0}^{\infty}R_{k-i}\otimes Sym_{i}^{n}(V)$$
%Now 
%$$Sym_{i}^{n}(V)\cong $$
\end{proof}

%As an $R$-module this is just $R\otimes S(V)=\bigoplus_{n=0}^{\infty} R\otimes Sym^{n}(V)$. Moreover, $\hat{S}_{R}(R\otimes V_{0})\otimes_{S_{R}(V_{0})} S_{R}(R\otimes V)\cong \hat{S}_{R}(R\otimes V_{0})\otimes S_{R}(R\otimes(V\big\slash V_{0}))$. We claim that this is isomorphic to $\hat{S}_{R}(R\otimes V)$.  

%However, we have he following. 

%\begin{prop}
%Suppose that $A$ is $\mathfrak{G}$-non-negatively graded. 
%\end{prop}
%
%\begin{proof}
%Write $A=lim_{\rightarrow_{n\ge 0}}S(R\otimes V_{n})$ be a presentation as a $\mathfrak{G}$-non-negatively graded algebra. Then after forgetting differentials the underlying 
%\end{proof}

\section{Cooperadic Koszul Duality}\label{seccoopkosz}

%Generally speaking, a decent homotopy theory of coalgebras over a formally planar cooperad is bootstrapped from the homotopy theory of algebras over an operad using a twisting morphism. 

%In this entire section we shall assume that all operads and cooperads are reduced, that is as symmetric sequences the $0$-ary part is trivial.

\subsection{Filtered Cooperads and Filtered Coalgebras}
The homotopy theory of coalgebras is much more involved. For the purposes of Koszul duality, we need to consider coinvariant-pure stably formally planar cooperads and filtered coalgebras over them. In many cases of interest, such as $\mathfrak{C}=\mathfrak{coComm}^{ncu}$, the filtration is induced by a \textbf{weight grading}, $\mathfrak{C}=\bigoplus_{n=0}^{\infty}\mathfrak{C}^{n}$.

%The category $\overline{\mathpzc{Filt}}_{\textbf{PureMon}}(\mathpzc{M})$ is not in general a symmetric monoidal category, but for our purposes this is not a problem. Rather we make the following definition.

\begin{defn}\label{filtcoop}

A \textbf{coinvariant-pure stably formally planar cooperad} is a stably coinvariant-purifyingly filtered  (Definition \ref{defn:stabcoinvapurefilt}) object $\mathfrak{C}=(\mathfrak{C}_{top},\gamma_{n},c_{n})$ of
 $\overline{\mathpzc{Filt}}_{\textbf{PureMon}}(\mathpzc{Mod}_{\Sigma}(\mathpzc{M}))$ (i.e., it is also purely filtered), with each $\mathfrak{C}_{n}$ and each $gr_{n}(\mathfrak{C})$ being stably formally planar (recall Definition \ref{defn:stfp}),
 together with morphisms of filtered objects
 $$\Delta:\mathfrak{C}\rightarrow\mathfrak{C}\circ\mathfrak{C}$$
 $$\epsilon:\mathfrak{C}\rightarrow I$$
 satisfying the obvious coassociativity and counitality conditions.
 %(where $I$ has the trivial filtration),
% such that the induced morphisms
 %$$\mathfrak{C}_{top}\rightarrow\mathfrak{C}_{top}\circ\mathfrak{C}_{top}$$
 %$$\mathfrak{C}_{top}\rightarrow I$$
 %endow $\mathfrak{C}_{top}$ with the structure of a stably formally planar cooperad. 
 
 %together with a formally planar cooperad structure on $\mathfrak{C}_{top}$ such that 
%\begin{enumerate}
%\item
%$\mathfrak{C}\circ_{ns}\mathfrak{C}$ is has admissibly filtered coinvariants,
%\item
%and the morphisms $\mathfrak{C}_{top}\rightarrow\mathfrak{C}_{top}\circ\mathfrak{C}_{top}$ and $\mathfrak{C}_{top}\rightarrow I$ preserve filtrations, where $I$ is endowed with the trivial filtration. 
%\end{enumerate}
\end{defn}
%Note it follows that $\mathfrak{C}\circ\mathfrak{C}$ is also exhaustively filtered.  For posterity we will also define filtered operads.

\begin{defn}\label{filtop}

A \textbf{coinvariant-pure formally planar operad} is a stably coinvariant-purifyingly filtered  (Definition \ref{defn:stabcoinvapurefilt}) object $\mathfrak{P}=(\mathfrak{P}_{top},\gamma_{n},c_{n})$ of
 $\overline{\mathpzc{Filt}}_{\textbf{PureMon}}(\mathpzc{Mod}_{\Sigma}(\mathpzc{M}))$ (i.e., it is also purely filtered), with each $F_{n}\mathfrak{P}$ and each $gr_{n}(\mathfrak{P})$ being stably formally planar (recall Definition \ref{defn:stfp}),
 together with morphisms of filtered objects
 $$\mu:\mathfrak{P}\circ\mathfrak{P}\rightarrow\mathfrak{P}$$
 $$\eta:I\rightarrow\mathfrak{P}$$
 satisfying the obvious associativity and unitality conditions.
 %such that the induced morphisms
 %$$\mu:(\mathfrak{P})_{top}\circ(\mathfrak{P})_{top}\rightarrow(\mathfrak{P})_{top}$$
 %$$\epsilon:I\rightarrow(\mathfrak{P})_{top}$$
 %endow $(\mathfrak{P})_{top}$ with the structure of an operad.
\end{defn}

Coinvariant-pure stably formally planar cooperads and coinvariant-pure stably formally planar operads obviously arrange into categories. We denote these by $\mathpzc{coOp}^{pure}(\mathpzc{M})$ and $\mathpzc{Op}^{pure}(\mathpzc{M})$, respectively.

\begin{defn}\label{filtcoalg}

Let $\mathfrak{C}=(\mathfrak{C}_{top},\gamma_{n},c_{n})$ be a coinvariant-pure stably formally planar cooperad. A \textbf{conilpotent filtered} $\mathfrak{C}$-\textbf{coalgebra} is an object $(C_{top},\alpha_{n},a_{n})$ of $\overline{\mathpzc{Filt}}_{\textbf{PureMon}}(\mathpzc{M})$, together with a morphism of filtered objects
$$\Delta_{C}:C\rightarrow\mathfrak{C}\circ C$$
satisfying the obvious coassociativity and counitality conditions.
%such that the induced morphism
%$$(C)_{top}\rightarrow\mathfrak{C}_{top}\circ (C)_{top}$$
%endows $(C)_{top}$ with the structure of a $\mathfrak{C}_{top}$-coalgebra. A \textbf{morphism} of conilpotent $\mathfrak{C}$-coalgebras $$f:((C)_{top},\alpha_{n},a_{n})\rightarrow ((D)_{top},\beta_{n},b_{n})$$
%is a morphism of filtered objects such that the induced morphism 
%$$(f)_{top}:(C)_{top}\rightarrow(D)_{top}$$
%is a morphism of $\mathfrak{C}_{top}$-coalgebras.
\end{defn}

We denote by $\mathpzc{coAlg}_{\mathfrak{C}}^{conil}$ the category of conilpotent filtered $\mathfrak{C}$-coalgebras with morphisms of conilpotent filtered $\mathfrak{C}$-coalgebras.

Let $M$ be an object of $\overline{\mathpzc{Filt}}_{\textbf{PureMon}}(\mathpzc{M})$. Using Proposition \ref{prop:compproductpure} we get an induced filtration on $\mathfrak{C}_{top}\circ M_{top}$ by pure monomorphisms, which we denote by $\mathfrak{C}(M)$. This is then an object of $\mathpzc{coAlg}_{\mathfrak{C}}^{conil}$. In fact, this gives rise to an adjunction
$$\adj{|-|}{\mathpzc{coAlg}_{\mathfrak{C}}^{conil}}{\overline{\mathpzc{Filt}}_{\textbf{PureMon}}(\mathpzc{M})}{\mathfrak{C}(-)}.$$

In the other direction, filtrations on coalgebras can often be induced by the filtration on $\mathfrak{C}$ in the following precise sense. If $C$ is a not necessarily conilpotent coalgebra, then $C$ can be equipped with a canonical \textbf{induced filtration}, where $C_{n}$ is given by the following pullback

\begin{displaymath}
\xymatrix{
C_{n}\ar[d]\ar[r] & C\ar[d]^{\Delta}\\
\mathfrak{C}_{n}(C)\ar[r] & \hat{\mathfrak{C}}\circ C.
}
\end{displaymath}
Note that this does not always produce a filtration by pure monomorphisms, or even by monomorphisms.

The following result is clear.
\begin{prop}\label{prop:Cconilinduced}
If the induced filtration is exhaustive, then $C$ is conilpotent (recall Definition \ref{defn:conilpotent}).
\end{prop}

If $C$ is already conilpotent, then we can induce a filtration on $C$ by the following pullback diagram

\begin{displaymath}
\xymatrix{
C_{n}\ar[d]\ar[r] & C\ar[d]^{\Delta}\\
\mathfrak{C}_{n}(C)\ar[r] & \mathfrak{C}\circ C.
}
\end{displaymath}
Again, in general it may not be a filtration by pure monomorphisms, or even admissible monomorphisms.
%\begin{prop}
%
%\begin{enumerate}
%\item
%If the underlying filtered $\Sigma$-module of $\mathfrak{C}$ is a retract of a free $\Sigma$-module on an object of $\overline{\mathpzc{Filt}}_{\textbf{SplitMon}}(\mathpzc{Gr}_{\mathbb{N}_{0}}(\mathpzc{M}))$ then the first condition in Definition \ref{filtcoop} is superfluous.
%\item
%If the underlying $\Sigma$-module of $\mathfrak{C}$ is a retract of a free $\Sigma$-module on an object of $\overline{\mathpzc{Filt}}_{\textbf{SplitMon}}(\mathpzc{Gr}_{\mathbb{N}_{0}}(\mathpzc{M}))$, and the filtration on a $\mathfrak{C}_{top}$-coalgebra $A$ arises from a grading, then the first condition of Definition \ref{filtcoalg} is automatically satisfied. 
%\end{enumerate}
%\end{prop}

\begin{prop}\label{prop:inducedfiltration}

    Suppose that $\mathpzc{E}$ is a closed symmetric monoidal, complete, and cocomplete quasi-abelian category, in which all objects are flat. Let $\mathfrak{C}$ be a coinvariant-pure planar cooperad in $\mathpzc{E}$, and $C$ a coalgebra over $\mathfrak{C}$. Then the induced filtration on $C$ is by pure monomorphisms.
\end{prop}

\begin{proof}
    Under the assumptions on $\mathpzc{E}$, all monomorphisms are pure. Hence, we just need to verify that $C_{n}\rightarrow C$ is a monomorphism. Since  strict monomorphisms are stable under pullback, it suffices to observe that $\mathfrak{C}_{n}(C)\rightarrow\mathfrak{C}\circ C$ is a split, and hence strict, monomorphism. %Now, $\mathfrak{C}_{n}(C)\rightarrow\mathfrak{C}(C)$ is a monomorphism, as $\mathfrak{C}$ is stably formally planar. By the assumption on products and coproducts, the morphism $\mathfrak{C}\circ(C)\rightarrow\mathfrak{C}\hat{\circ}C$ is a monomorphism. This completes the proof. 
\end{proof}

\begin{rem}\label{rem:inducedfiltrationcor}

    Under the conditions of Proposition \ref{prop:inducedfiltration}, there is a functor
$$Corad:\mathpzc{coAlg}^{conil}_{\mathfrak{C}_{top}}\rightarrow\mathpzc{coAlg}^{conil}_{\mathfrak{C}}.$$
We have an equality $\Omega_{\alpha}^{filt}\circ Corad=\Omega_{\alpha}$.
\end{rem}

  Note that there are obvious functors $(-)_{top}:\mathpzc{coOp}^{pure}(\mathpzc{M})\rightarrow\mathpzc{coOp}^{stfp}(\mathpzc{M})$ and $(-)_{top}:\mathpzc{coAlg}^{conil}_{\mathfrak{C}}\rightarrow\mathpzc{coAlg}^{conil}_{\mathfrak{C}_{top}}(\mathpzc{M})$. The essential image of the latter functor is denoted $(\mathpzc{coAlg}_{\mathfrak{C}}^{conil})_{top}$. The point of all these technical assumptions on coinvariant-pure planar cooperads and filtered coalgebras is essentially that we can pass to associated graded objects without much hassle. The following is an immediate consequence of Corollary \ref{cor:grmn}.

\begin{prop}\label{assgradediso}
Let $\mathfrak{C}\in\mathpzc{coOp}^{pure}(\mathpzc{M})$ and $C\in\mathpzc{coAlg}^{conil}_{\mathfrak{C}}$. Then the natural morphisms $gr(\mathfrak{C}\circ\mathfrak{C})\rightarrow gr(\mathfrak{C})\circ gr(\mathfrak{C})$ and $gr(\mathfrak{C}\circ C)\rightarrow gr(\mathfrak{C})\circ gr(C)$ are isomorphisms. In particular, $gr(\mathfrak{C})$ is a (graded) stably formally planar cooperad and $gr(C)$ is a (graded) $gr(\mathfrak{C})$-coalgebra.
\end{prop}

\begin{rem}
    Note that $gr(\mathfrak{C})$ is automatically a coinvariant-pure, stably formally planar cooperad, with the split filtration induced by the grading. 
\end{rem}

%Note that the functor $(-)_{top}$ is strong monoidal. In particular, if $\mathfrak{C}\in\mathpzc{coOp}^{stfp}(\mathpzc{A})$ then $\mathfrak{C}_{top}\in\mathpzc{coOp}^{stfp}(\mathpzc{M})$.
%Let us fix an filtered Koszul category $\mathpzc{A}$. 

%In this section we shall introduce some notations and conventions for filtered formally planar cooperads and filtered co-algebras. By our assumptions on $\mathpzc{E}$ at the end of Section \ref{exactcats}, and by Proposition \ref{a} in \cite{homotopyexact}, the category $\mathpzc{Filt}_{\textbf{RegMon}}(\mathpzc{M})$ is a symmetric monoidal category.  

For a class of objects $\mathfrak{O}$ of $\mathpzc{M}$, we denote by $\mathpzc{coAlg}_{\mathfrak{C}_{top}}^{|\mathfrak{O}|}$ the full subcategory of $\mathpzc{coAlg}^{conil}_{\mathfrak{C}_{top}}$ consisting of those coalgebras $C$ such that the underlying object of $C$ is in $\mathfrak{O}$, and by $\mathpzc{coAlg}_{\mathfrak{C}}^{|\mathfrak{O}|}$ the full subcategory of $\mathpzc{coAlg}^{conil}_{\mathfrak{C}}$ consisting of coalgebras $A$ such that for each $n\in\mathbb{N}_{0}$, $gr_{n}(A)$ is in $\mathfrak{O}$.  Typically, $\mathfrak{O}$ will either be the class of $K$-flat objects or the class $c$ of cofibrant objects. In the former case, we will denote this category by $\mathpzc{coAlg}_{\mathfrak{C}}^{|K|}$ and, in the latter case, by $\mathpzc{coAlg}_{\mathfrak{C}}^{|c|}$. Similarly, for a class of objects $\mathfrak{O}$ of ${}\mathpzc{Mod}_{\Sigma}$ we define categories $\mathpzc{coOp}^{pure,|\mathfrak{O}|}(\mathpzc{M})$. By unwinding the definitions, the following is clear:

\begin{prop}\label{cofreefiltcoalg}

Let $\mathfrak{O}$ be a class of objects in $\mathpzc{M}$ that is closed under taking tensor products and countable sums. Let $\Sigma\otimes\mathfrak{O}$ denote the class of objects of ${}\mathpzc{Mod}_{\Sigma}$ consisting of $\Sigma$-modules that in each arity $n$ are free $\Sigma_{n}$-modules on objects of $\mathfrak{O}$. 
Let $(\mathfrak{C},\Delta)\in\mathpzc{coOp}^{pure,|\Sigma\otimes\mathfrak{O}|}(\mathpzc{M})$. Then for any object $V$ of $\mathfrak{O}$, $\mathfrak{C}(V)$ is in $\mathpzc{coAlg}_{\mathfrak{C}}^{|\mathfrak{O}|}$. 
\end{prop}

\subsection{The Bar-Cobar Adjunction}\label{subsec:barcobar}

We now develop the homotopy theory of $\mathfrak{C}$-coalgebras, and see how it behaves under the bar-cobar adjunction.

\subsubsection{The Filtered Bar-Cobar Adjunction}

%First we will need to restrict the categories and algebras, coalgebras and twisting morphisms that we consider.

Let $\mathfrak{P}$ be a filtered operad, $\mathfrak{C}$ a filtered cooperad, and $\alpha:\mathfrak{C}\rightarrow\mathfrak{P}$ a twisting morphism. 
For good homotopical properties later, we shall need that our twisting morphisms satisfy the following property.

\begin{ass}
 $\alpha|_{F_{0}\mathfrak{C}}=0$.
 \end{ass}

 We shall always assume this from now on, unless otherwise stated. In fact, the main consequence of this assumption is that the twisted differential in the cobar construction lowers the filtration. We spell this out in the result below.

%For the rest of this paper we shall assume that the following condition is satisfied.

%\begin{defn}
%
%We say that a twisting morphism $\alpha:\mathfrak{C}\rightarrow\mathfrak{P}$ is \textbf{admissible} 
%if for any finite sequence of integers $(p_{1},q_{1},\ldots,p_{l},q_{l},0,0,\ldots)$  when equipped with the tensor product filtration,
 %   the object $$(\mathfrak{C}^{p_{1}}\circ_{ns}\mathfrak{P}^{q_{1}}\circ_{ns}\ldots\circ_{ns}\mathfrak{C}^{p_{l}}\circ_{ns}\mathfrak{P}^{q_{l}})(n)$$ is in $\overline{\mathpzc{Filt}}_{\textbf{AdMon}}(\mathpzc{M})$  and has admissibly filtered coinvariants.
%\end{defn}

%\begin{ass}   \label{compatiblecompositefilt}
%
% From now on we shall assume that all twisting morphisms are admissible.
%\end{ass}
Recall the bar-cobar adjunction from Subsection \ref{subsubsec:barcobar}. Denote by $\Omega_{\alpha}^{filt}$ the composite functor
\begin{displaymath}
\xymatrix{
\mathpzc{coAlg}^{conil}_{\mathfrak{C}}\ar[r]^{(-)_{top}} & \mathpzc{coAlg}^{conil}_{\mathfrak{C}_{top}}\ar[r]^{\Omega_{\alpha}} & \mathpzc{Alg}_{\mathfrak{P}}.
}
\end{displaymath}

Denote by $B_{\alpha}^{filt}:\mathpzc{Alg}_{\mathfrak{P}}\rightarrow\mathpzc{coAlg}^{conil}_{\mathfrak{C}}$ the functor that sends a $\mathfrak{P}$-algebra $A$ to the coalgebra $B_{\alpha}A$ equipped with the filtration $\mathfrak{C}_{n}\circ A$. This functor is well-defined by Proposition \ref{prop:compproductpure}. 

\begin{prop}

    The functor $\Omega_{\alpha}^{filt}$ is left adjoint to $B_{\alpha}^{filt}$.
\end{prop}

\begin{proof}
    Let $D$ be a filtered coalgebra, and let $C$ be an arbitrary coalgebra. Let $f:D_{top}\rightarrow C$ be a morphism of coalgebras. Since the filtration on $D$ is a coalgebra filtration, the restriction of $f$ to $D_{n}$ necessarily factors through the $n$th part of the induced filtration on $C$ (recall the remark before Proposition \ref{prop:Cconilinduced}). For $V$ an object of $\mathpzc{M}$, the induced filtration on $\mathfrak{C}(V)$ is $\mathfrak{C}_{n}(V)$. Applying this with $C=B_{\alpha}(A)$ gives the claim.
\end{proof}

Throughout the following theorem and its proof, for $X$ a graded object, the filtration on $gr(X)$ will always be 
$$F_{n}(gr(X))=\bigoplus_{m\le n}gr_{m}(X).$$

\begin{thm}\label{thm:gradedcrap}
  \begin{enumerate}
      \item 
      If $A\in\mathpzc{Alg}_{\mathfrak{P}}$, then $B^{filt}_{\alpha}(A)\in\mathpzc{coAlg}^{conil}_{\mathfrak{C}}$. Moreover, the following hold.
      \begin{enumerate}
          \item 
          We have  $gr(B_{\alpha}^{filt}(A))\cong gr(\mathfrak{C})\circ A$ as complexes.
          \item 
          Consider the filtration on $\Omega_{\alpha}B_{\alpha}A$ given by 
          $$F_{n}\Omega_{\alpha}B_{\alpha}A=\sum_{k\ge1, m+n_{1}+\ldots+n_{k}\le n}F_{m}\mathfrak{P}(k)\otimes_{\Sigma_{k}}(F_{n_{1}}\mathfrak{C}(A)\otimes\ldots\otimes F_{n_{k}}\mathfrak{C}(A)).$$
          Then this is a filtration by pure monomorphisms, and
          $$gr(\Omega_{\alpha}B_{\alpha}A)\cong gr(\mathfrak{P}\circ_{\alpha}\mathfrak{C})\circ A$$
          as complexes.
      \end{enumerate}
   
      \item 
         Let $C\in\mathpzc{coAlg}_{\mathfrak{C}}^{conil}$. Then the following hold.
         \begin{enumerate}
             \item 
             Consider the filtration on $\Omega_{\alpha}C$ given by 
             $$F_{n}\Omega_{\alpha}C=\sum_{k\ge 1,\;m+n_{1}+\ldots+n_{k}\le n}F_{m}\mathfrak{P}(k)\otimes_{\Sigma_{k}}(F_{n_{1}}C\otimes\ldots\otimes F_{n_{k}}C).$$
             Then, as graded objects (i.e., forgetting the differential), we have 
             $$gr(\mathfrak{P}\circ C)\cong gr(\mathfrak{P})\circ gr(C).$$
             \item 
             Consider the filtration on $gr(\Omega_{\alpha}C)$ given by 
             $$F_{n}gr(\Omega_{\alpha}(C))=\sum_{k\ge1,n_{1}+\ldots+n_{k}\le n}gr(\mathfrak{P})(k)\otimes_{\Sigma_{k}}(F_{n_{1}}gr(C)\otimes\ldots\otimes F_{n_{k}}gr(C)).$$ 
             Then, \textit{as complexes}, 
             $$gr(gr(\Omega_{\alpha}C))\cong gr(\mathfrak{P})\circ gr(C).$$
             \item 
            Consider the filtration on $\Omega_{\alpha}C$ given by
$$F_{n}\Omega_{\alpha}C=\sum_{k\ge1,m+n_{1}+\ldots+n_{k}\le n}F_{m}\mathfrak{P}(k)\otimes_{\Sigma_{k}}(F_{n_{1}}C\otimes\ldots\otimes F_{n_{k}}C),$$
and the one on $\Omega_{\alpha}B_{\alpha}\Omega_{\alpha}C$ given by
$$F_{n}\Omega_{\alpha}B_{\alpha}\Omega_{\alpha}C=\sum_{k\ge1,p+q+m+n_{1}+\ldots+n_{k}\le n}(F_{p}\mathfrak{P}\circ F_{q}\mathfrak{C}\circ F_{m}\mathfrak{P})(k)\otimes_{\Sigma_{k}}(F_{n_{1}}C\otimes\ldots\otimes F_{n_{k}}C).$$
The morphism $\Omega_{\alpha}(\nu_{\alpha}(C))$ is then a filtered morphism, and, forgetting differentials, we have 
$$gr(\Omega_{\alpha}B_{\alpha}\Omega_{\alpha}C)\cong gr(\mathfrak{P}\circ\mathfrak{C}\circ\mathfrak{P})\circ gr(C).$$
\item 
Now consider the filtration on $gr(\Omega_{\alpha}C)$ given by
$$F_{n}gr(\Omega_{\alpha}C)=\sum_{k\ge1,n_{1}+\ldots+n_{k}\le n}gr(\mathfrak{P})(k)\otimes_{\Sigma_{k}} (gr_{n_{1}}(C)\otimes\ldots\otimes gr_{n_{k}}(C)),$$
and the filtration on $gr(\Omega_{\alpha}B_{\alpha}\Omega_{\alpha}C)$ given by 
$$F_{n}gr(\Omega_{\alpha}B_{\alpha}\Omega_{\alpha}C)=\sum_{k\ge1,n_{1}+\ldots+n_{k}\le n}gr(\mathfrak{P}\circ \mathfrak{C}\circ\mathfrak{P})(k)\otimes_{\Sigma_{k}}(gr_{n_{1}}(C)\otimes\ldots\otimes gr_{n_{k}}(C)).$$
Then the morphism $gr(\Omega_{\alpha}(\nu_{\alpha}(C)))$ is a filtered morphism, and, as complexes, we have 
$$gr(gr(\Omega_{\alpha}B_{\alpha}\Omega_{\alpha}C))\cong gr(\mathfrak{P}\circ_{\alpha}\mathfrak{C}\circ_{\alpha}\mathfrak{P})\circ gr(C).$$
         \end{enumerate}
  \end{enumerate}
\end{thm}

\begin{proof}
\begin{enumerate}
\item 
\begin{enumerate}
 %   \item 
%Forgetting differentials, the underlying filtered object is $\mathfrak{C}\circ M$. As explained above, Proposition \ref{prop:compproductpure} implies that the filtration on this object is by pure monomorphisms. Now we may just use Corollary \ref{cor:grmn}.
\item
  Let $A$ be an object of $\mathpzc{Alg}_{\mathfrak{P}}$, and let $\mathfrak{C}_{n}$ denote the filtration on $\mathfrak{C}$. 
The filtration on the underlying graded (in the homological sense) object of $B_{\alpha}A$, $\mathfrak{C}\circ A$, is given by
$$F_{n}(B^{filt}_{\alpha}A)=(\mathfrak{C}_{n})\circ A.$$
 Now, the differential on $B_{\alpha}A$ is given by the sum $d_{\mathfrak{C}}\circ Id_{A}+Id_{\mathfrak{C}}\circ' d_{A}+d_{2}$ obtained using the coderivation version of the construction of Subsection \ref{subsubsec:relative} applied to the morphism
\begin{displaymath}
\xymatrix{
\mathfrak{C}\circ A\ar[r]^{\alpha\circ Id_{A}} & \mathfrak{P}\circ A\ar[r]^{\gamma_{A}} & A.
}
\end{displaymath}
The formula defining this coderivation implies that $d_{2}$ lowers the filtration. Indeed, this unique coderivation is given by the formula 
\begin{displaymath}
\xymatrix{
    \mathfrak{C}\circ A\ar[rr]^{\Delta_{(1)}\circ Id_{A}} &&(\mathfrak{C}\circ_{(1)}\mathfrak{C})\circ A\ar[rr]^{(Id_{\mathfrak{C}}\circ_{(1)}\alpha)\circ Id_{A}} & &(\mathfrak{C}\circ_{(1)}\mathfrak{P})\circ A\ar[r]^{\cong} &\mathfrak{C}\circ (A;\mathfrak{P}\circ A)\ar[rr]^{Id_{\mathfrak{C}}\circ(Id_{A};\gamma_{A})}& & \mathfrak{C}\circ A.
    }
\end{displaymath}
Equip every term with the filtrations induced by the ones on $\mathfrak{C}$ and $\mathfrak{P}$:
$$F_{n}((\mathfrak{C}\circ_{(1)}\mathfrak{C})\circ A)=\sum_{i+j\le n}(F_{i}\mathfrak{C}\circ_{(1)}F_{j}\mathfrak{C})\circ A,$$
and
$$F_{n}((\mathfrak{C}\circ_{(1)}\mathfrak{P})\circ A)=\sum_{i+j\le n}(F_{i}\mathfrak{C}\circ_{(1)}F_{j}\mathfrak{P})\circ A$$
(these are in fact still filtrations by pure monomorphisms, similarly to Proposition \ref{prop:compproductpure}).
All morphisms are then morphisms of filtered objects. 
When $j=0$, the morphism $\alpha:F_{0}\mathfrak{C}\rightarrow F_{0}\mathfrak{P}$ is zero by assumption. Whenever $j>0$, the morphism $F_{j}\mathfrak{P}\circ A\rightarrow A$ strictly lowers the filtration. Thus, after passing to associated graded objects, we are only left with the differential $d_{gr(\mathfrak{C})}\circ Id_{A}+Id_{gr(\mathfrak{C})}\circ'd_{A}$. Hence, $gr(B_{\alpha}^{filt}A)\cong gr(\mathfrak{C})\circ A$.
\item 
After forgetting differentials, the underlying graded object (in the homological direction) of $\Omega_{\alpha}B_{\alpha}A$ is $\mathfrak{P}\circ\mathfrak{C}\circ A$.  We filter it by
$$F_{n}\Omega_{\alpha}B_{\alpha}A=\sum_{k\ge1, m+n_{1}+\ldots+n_{k}\le n}F_{m}\mathfrak{P}(k)\otimes_{\Sigma_{k}}(F_{n_{1}}\mathfrak{C}(A)\otimes\ldots\otimes F_{n_{k}}\mathfrak{C}(A)).$$
This is a filtration by pure monomorphisms by Proposition \ref{prop:compproductpure}, and induces a filtration of the chain complex $\Omega_{\alpha}B_{\alpha}A$. The differential on $\Omega_{\alpha}B_{\alpha}A$ is $d_{\mathfrak{P}\circ_{\alpha}\mathfrak{C}}\circ Id_{A}+Id_{\mathfrak{P}\circ\mathfrak{C}}\circ' d_{A}+\mathfrak{P}\circ' d_{2}$, where $d_{2}$ is the twisted part of the differential on $B_{\alpha}A$.
The first two differentials preserve the filtration, while the third lowers the filtration. 
We regard $A$ as a filtered algebra with the constant filtration, so that $gr(A)\cong A$ as a graded object concentrated in level $0$. The counit is a morphism of filtered complexes, so $gr(\epsilon_{\alpha}(A))$ is a morphism of graded objects. Now, $gr(\Omega_{\alpha}B_{\alpha}A)\cong gr(\mathfrak{P}\circ_{\alpha}\mathfrak{C})\circ A$ as complexes.
\end{enumerate}
    \item
    \begin{enumerate}
      %  \item 
    %The claim is not quite the definition of $\alpha$-admissibility of coalgebras, but it is close.
    %$$\sum_{r+s\le m}(F_{r}(\mathfrak{C}^{p_{1}}\circ_{ns}\mathfrak{P}^{q_{1}}\circ_{ns}\ldots\circ_{ns}\mathfrak{C}^{p_{l}}\circ_{ns}\mathfrak{P}^{q_{l}})(n)\otimes F_{s}C^{\otimes n})\rightarrow\sum_{a+b\le m+1}(F_{a}(\mathfrak{C}^{p_{1}}\circ_{ns}\mathfrak{P}^{q_{1}}\circ_{ns}\ldots\circ_{ns}\mathfrak{C}^{p_{l}}\circ_{ns}\mathfrak{P}^{q_{l}})(n)\otimes F_{b}C^{\otimes n})$$
    %is an admissible monomorphism with admissibly filtered coinvariants. Hence, for each fixed $r$ and $s$, the morphism
    %$$(F_{r}(\mathfrak{C}^{p_{1}}\circ_{ns}\mathfrak{P}^{q_{1}}\circ_{ns}\ldots\circ_{ns}\mathfrak{C}^{p_{l}}\circ_{ns}\mathfrak{P}^{q_{l}})(n)\otimes F_{s}C^{\otimes n})\rightarrow(F_{r}(\mathfrak{C}^{p_{1}}\circ_{ns}\mathfrak{P}^{q_{1}}\circ_{ns}\ldots\circ_{ns}\mathfrak{C}^{p_{l}}\circ_{ns}\mathfrak{P}^{q_{l}})(n)\otimes F_{s+1}C^{\otimes n})$$
    %is an admissible monomorphism.
\item
Now, let $C\in\mathpzc{coAlg}^{conil}_{\mathfrak{C}}$. We have the following filtration on $\Omega_{\alpha}(C)=(\mathfrak{P}(C),d_{1}-d_{2})$:
$$F_{n}\Omega_{\alpha}C=\sum_{k\ge 1,\;m+n_{1}+\ldots+n_{k}\le n}F_{m}\mathfrak{P}(k)\otimes_{\Sigma_{k}}(F_{n_{1}}C\otimes\ldots\otimes F_{n_{k}}C).$$
By Proposition \ref{prop:compproductpure}, this is a filtration by pure monomorphisms.
Recall that $d_{1}=d_{\mathfrak{P}}\circ Id_{C}+\mathfrak{P}\circ' d_{C}$, and $d_{2}$ is the twisted part of the derivation obtained from the construction of Subsection \ref{subsubsec:relative} applied to the morphism
\begin{displaymath}
\xymatrix{
    C\ar[r]^{\Delta_{C}}&\mathfrak{C}\circ C\ar[r]^{\alpha\circ Id_{C}} & \mathfrak{P}\circ C.
    }
\end{displaymath}
This morphism is given by the following composite
\begin{displaymath}
    \xymatrix{
    \mathfrak{P}\circ C\ar[rr]^{Id_{\mathfrak{P}}\circ'\Delta}& &\mathfrak{P}\circ(C;\mathfrak{C}\circ C)\ar[rrr]^{Id_{\mathfrak{P}}\circ(Id_{C};\alpha\circ Id_{C})} & &&\mathfrak{P}\circ(C;\mathfrak{P}\circ C)\ar[r]^{\cong} &(\mathfrak{P}\circ_{(1)}\mathfrak{P})(C)\ar[rr]^{\gamma_{(1)}\circ Id_{C}}& &\mathfrak{P}\circ C.
    }
\end{displaymath}
Now, all morphisms in the composition preserve the filtrations. Taking associated graded objects, we get a complex whose underlying algebra is $gr(\mathfrak{P})\circ gr(C)$. This has differential $d_{gr(\mathfrak{P})}\circ Id_{gr(C)}+gr(\mathfrak{P})\circ' d_{gr(C)}-gr(d_{2})$.
\item
Consider now the following filtration on $gr(\Omega_{\alpha}C)$: 
$$F_{n}gr(\Omega_{\alpha}(C))=\sum_{k\ge1,n_{1}+\ldots+n_{k}\le n}gr(\mathfrak{P})(k)\otimes_{\Sigma_{k}}(F_{n_{1}}gr(C)\otimes\ldots \otimes F_{n_{k}}gr(C)),$$
where the filtration on $gr(C)$ is the one coming from the grading. Note that this is a split, and hence admissible, filtration by pure monomorphisms. Because we are not taking any filtration on $gr(\mathfrak{P})$ into account, and $\alpha|_{F_{0}\mathfrak{C}}=0$, the morphism $gr(Id_{\mathfrak{P}}\circ(Id_{C};\alpha\circ Id_{C}))$ strictly lowers the filtration. Thus, $gr(d_{2})$ strictly lowers the filtration. Hence, after taking the associated graded of this filtration, the only differential we are left with is $d_{gr(\mathfrak{P})}\circ Id_{gr(C)}+gr(\mathfrak{P})\circ'd_{gr(C) }$, and our dg-algebra is then $gr(\mathfrak{P})\circ gr(C)$. 
\item 
The morphism $\Omega_{\alpha}(\nu_{\alpha}(C))$ clearly preserves the filtrations in the statement. 
Passing to associated graded objects gives a morphism
$$gr(\Omega_{\alpha}(\nu_{\alpha}(C))):gr(\Omega_{\alpha}C)\rightarrow gr(\Omega_{\alpha}B_{\alpha}\Omega_{\alpha}C).$$
The underlying object on the left-hand side is $gr(\mathfrak{P})\circ gr(C)$, with differential given by $d_{gr(\mathfrak{P})\circ gr(C)}-gr(d_{2})=d_{gr(\mathfrak{P})}\circ Id_{gr(C)}+Id_{gr(\mathfrak{P})}\circ'd_{gr(C)}-gr(d_{2})$, where $d_{2}$ is the twisted part of the differential. The underlying object on the right-hand side is $gr(\mathfrak{P}\circ\mathfrak{C}\circ\mathfrak{P})\circ gr(C)$. The differential on the right-hand side is given by the sum of terms
$$d_{gr(\mathfrak{P}\circ_{\alpha}\mathfrak{C}\circ_{\alpha}\mathfrak{P}\circ C)}-Id_{gr(\mathfrak{P}\circ\mathfrak{C})}\circ'gr(d_{2}).$$
\item 
The filtrations are split, and hence admissible. Moreover, $gr(\Omega_{\alpha}(\nu_{\alpha}(C)))$ preserves these filtrations. On the left-hand side, $d_{gr(\mathfrak{P})}\circ Id_{gr(C)}+Id_{gr(\mathfrak{P})}\circ'd_{gr(C)}$ preserves the filtration, while $gr(d_{2})$ lowers it. On the right-hand side, $d_{gr(\mathfrak{P}\circ_{\alpha}\mathfrak{C}\circ_{\alpha}\mathfrak{P}\circ C)}$ preserves the filtration while $Id_{gr(\mathfrak{P}\circ\mathfrak{C})}\circ'gr(d_{2})$ lowers it.
Thus, the associated graded of the filtration on $gr(\Omega_{\alpha}C)$ is $gr(\mathfrak{P})\circ gr(C)$, and the associated graded of the filtration on $gr(\Omega_{\alpha}B_{\alpha}\Omega_{\alpha}C)$ is $gr(\mathfrak{P}\circ_{\alpha}\mathfrak{C}\circ_{\alpha}\mathfrak{P})\circ gr(C)$. 
    \end{enumerate}
\end{enumerate}
\end{proof}

%Let us give some conditions under which the technical assumptions above simplify greatly.

%\begin{example}\label{example:alphadmin}
%
%    \begin{enumerate}
%    \item 
%    If $\mathpzc{M}={}_{R}\mathpzc{Mod}(Ch(\mathpzc{E}))$ where $\mathpzc{E}$ is abelian and enriched over $\mathbb{Q}$-vector spaces, then the condition in the above definition that the filtrations have admissibly filtered coinvariants is automatic.
%    \item

%\end{enumerate}
%\end{example}

\subsubsection{Ideals}

Our statement of Koszul duality later will consist of an abstract equivalence of $(\infty,1)$-categories using relative categories. In general, we will not be able to present the coalgebra side by a model category. However, in many cases, we will get a reasonable category of cofibrant objects. In particular, there need not be a subcategory of fibrant-cofibrant objects giving a helpful presentation of mapping spaces. Nonetheless, it is still useful to restrict to relative subcategories on the coalgebra side for identifying some weak equivalences on this side, as we shall see later. Such subcategories are furnished by $\alpha$-ideals.
 
\begin{defn}\label{defn:ideal}
Let $\alpha:\mathfrak{C}\rightarrow\mathfrak{P}$ be a twisting morphism in $\mathpzc{M}$, with $\mathfrak{P}$ an admissible operad. 
\begin{enumerate}
    \item 
An $\alpha$-\textbf{ideal of }$\mathpzc{M}$ is a pair $(\mathpzc{I},\mathpzc{J})$ where %$\mathpzc{I}$ is a full subcategory of $\mathpzc{coAlg}_{\mathfrak{C}_{top}}$ and $\mathpzc{J}$ a full subcategory of $\mathpzc{Alg}_{\mathfrak{P}}(\mathpzc{M})$ such that

\begin{enumerate}
%\item
%if $f:V_{1}\rightarrow V_{2}\in\mathpzc{J}$ is a filtered weak equivalence between objects in $\mathpzc{I}$ then $ gr(f)^{\otimes n}$ is a weak equivalence for each $n$;
\item 
$\mathpzc{I}$ is a full subcategory of $\mathpzc{coAlg}^{conil}_{\mathfrak{C}_{top}}$,
\item 
$\mathpzc{J}$ is a full subcategory of $\mathpzc{Alg}_{\mathfrak{P}}$, 
\item
if $V\in\mathpzc{J}$ then $B_{\alpha}(V)\in\mathpzc{I}$,
\item
if $C\in\mathpzc{I}$ then $\Omega_{\alpha}(C)\in\mathpzc{J}$,
\item
if $g:A\rightarrow B$ is a weak equivalence between objects in $\mathpzc{J}$ then $\Omega_{\alpha}B_{\alpha}g$ is a weak equivalence in $\mathpzc{J}$, and
\item
$\mathpzc{J}$ contains all cofibrant $\mathfrak{P}$-algebras.
%\item
%each $\mathfrak{C}(n)$ is $K$-transverse over $\Sigma_{n}$ to objects in $\mathpzc{M}$ of the form $V^{\otimes n}$ where $V$ is in $\mathpzc{J}$, that is if $f:V^{\otimes n}\rightarrow W^{\otimes n}$ is a weak equivalence, then so is $\mathfrak{C}(n)\otimes_{\Sigma_{n}}V^{\otimes n}\rightarrow\mathfrak{C}(n)\otimes_{\Sigma_{n}}W^{\otimes n}$;
%\item
%each $\mathfrak{P}(n)$ is $K$-transverse over $\Sigma_{n}$ to objects in $\mathpzc{M}$ of the form $(C)_{top}^{\otimes n}$, where $C\in\mathpzc{I}$.
\end{enumerate}
\item
A \textbf{filtered} $\alpha$\textbf{-ideal} of $\mathpzc{M}$ is a pair $(\mathpzc{I},\mathpzc{J})$ where
\begin{enumerate}
\item 
$\mathpzc{I}$ is a full subcategory of $\mathpzc{coAlg}_{\mathfrak{C}}^{conil}$,
\item 
$\mathpzc{J}$ is a full subcategory of $\mathpzc{Alg}_{\mathfrak{P}}$,
\item
if $V\in\mathpzc{J}$ then $B^{filt}_{\alpha}(V)\in\mathpzc{I}$,
\item
if $C\in\mathpzc{I}$ then $\Omega^{filt}_{\alpha}(C)\in\mathpzc{J}$,
\item
the object $gr(\mathfrak{C})$ is $K$-transverse over $\Sigma$ to objects in $\mathpzc{J}$ -- that is, if $f:V\rightarrow W$ is a weak equivalence between objects of $\mathpzc{J}$, then $Id_{gr(\mathfrak{C})}\circ f:gr(\mathfrak{C})\circ V\rightarrow gr(\mathfrak{C})\circ W$ is a weak equivalence, 
\item
the object $\mathfrak{P}$ is $K$-transverse over $\Sigma$ to objects in $\mathpzc{M}$ of the form $C_{top}$, where $C\in\mathpzc{I}$ -- that is, if $g:C_{top}\rightarrow D_{top}$ is a weak equivalence, where $C,D\in\mathpzc{I}$, then $\mathfrak{P}\circ C_{top}\rightarrow\mathfrak{P}\circ D_{top}$ is a weak equivalence, and
\item 
the class $\mathpzc{J}$ contains all cofibrant $\mathfrak{P}$-algebras.
\end{enumerate}
\end{enumerate}

\end{defn}

\begin{defn}\label{defn:homotopical}
A twisting morphism $\alpha:\mathfrak{C}\rightarrow\mathfrak{P}$ is said to be \textbf{homotopical} if an $\alpha$-ideal exists, and \textbf{filtered homotopical} if a filtered $\alpha$-ideal exists.
\end{defn}

For a filtered $\alpha$-ideal $(\mathpzc{I},\mathpzc{J})$, we let $\mathpzc{I}_{top}=\{C_{top}:C\in\mathpzc{I}\}$. Later we will see that if $(\mathpzc{I},\mathpzc{J})$ is a filtered $\alpha$-ideal, then $(\mathpzc{I}_{top},\mathpzc{J})$ is an $\alpha$-ideal.

A twisting morphism is (filtered) homotopical if and only if there exists at least one (filtered) $\alpha$-ideal. Homotopicality in fact implies that there is a minimal (filtered) $\alpha$-ideal, which we construct now.

\begin{example}

Let $\alpha:\mathfrak{C}\rightarrow\mathfrak{P}$ be a twisting morphism. Let $\alpha_{min}$ (resp. $\alpha_{min}^{filt}$) denote the pair $(\mathpzc{I}_{min},\mathpzc{J}_{min})$ (resp. $(\mathpzc{I}_{min}^{filt},\mathpzc{J}^{filt}_{min})$),  where 
$$\mathpzc{J}_{min}=\mathpzc{J}^{filt}_{min}=\bigcup_{n=0}^{\infty}(\Omega_{\alpha}\circ B_{\alpha})^{n}(\mathpzc{Alg}^{c}_{\mathfrak{P}})$$
and
$$\mathpzc{I}_{min}=B_{\alpha}(\mathpzc{J}_{min}) \;\;(\mathrm{resp. }\mathpzc{I}_{min}^{filt}=B_{\alpha}^{filt}(\mathpzc{J}_{min}))$$
(recalling that $\mathpzc{Alg}^{c}_{\mathfrak{P}}$ denotes the category of cofibrant $\mathfrak{P}$-algebras).
Any $\alpha$-ideal (resp. any filtered $\alpha$-ideal) must contain, in the obvious sense, $\alpha_{min}$ (resp. $\alpha_{min}^{filt}$).
Then $\alpha_{min}$ (resp. $\alpha_{min}^{filt}$) is an $\alpha$-ideal (resp. a filtered $\alpha$-ideal) precisely if $\alpha$ is homotopical (resp. filtered homotopical).
%This is an $\alpha$-ideal precisely if $\alpha$ is homotopical.
We call $(\mathpzc{I}_{min},\mathpzc{J}_{min})$ (resp. $(\mathpzc{I}_{min}^{filt},\mathpzc{J}_{min}^{filt})$) the \textbf{minimal }$\alpha$-\textbf{ideal} (resp. the \textbf{minimal filtered }$\alpha$-\textbf{ideal}).
%Indeed, conditions $(3)$, $(4)$, and $(5)$ of Definition \ref{defn:ideal} are clearly satisfied. Conditions $(1)$ and $(2)$ of  Definition \ref{defn:ideal} correspond to conditions $(1)$ and $(2)$ of Definition \ref{defn:homotopical} respectively.  Conditions $(6)$ and $(7)$ of Definition \ref{defn:ideal}  correspond to conditions $(3)$ and $(4)$ of Definition \ref{defn:homotopical} respectively. Condition $(8)$ of  Definition \ref{defn:ideal} corresponds to condition $(5)$ of  Definition \ref{defn:homotopical} .
\end{example}

%\begin{prop}

%    Suppose that $\mathpzc{M}$ is a Koszul category satisfying the $pp$-monoid axiom, and $\mathfrak{P}$ an admissible symmetric operad in $\mathpzc{M}$ such that, as a $\Sigma_{n}$-module, each $\mathfrak{P}(n)$ is free on a $K$-flat object $\tilde{\mathfrak{P}}(n)$ of $\mathpzc{M}$. Let $f:A\rightarrow B$ be a weak equivalence in $\mathpzc{Alg}_{\mathfrak{P}}^{c}$. Then, for any $n$, $f^{\otimes n}:A^{\otimes n}\rightarrow B^{\otimes n}$ is a weak equivalence for any $n$.
%\end{prop}

%\begin{proof}
 %   By Ken Brown's Lemma, we may assume that $f$ is an acyclic cofibration. First, assume that $f$ is of the form $\mathfrak{P}(g):\mathfrak{P}(X)\rightarrow\mathfrak{P}(Y)$ for some acyclic cofibration $g:X\rightarrow Y$. Then $$f^{\otimes n}=\bigoplus_{(i_{1},\ldots,i_{n})\in\mathbb{N}_{0}^{n}}(\mathfrak{P}(i_{1})\otimes_{\Sigma_{i_{1}}}g^{\otimes i_{1}})\otimes\ldots (\mathfrak{P}(i_{n})\otimes_{\Sigma_{i_{n}}}g^{\otimes i_{n}})\cong\bigoplus_{(i_{1},\ldots,i_{n})\in\mathbb{N}_{0}^{n}}(\tilde{\mathfrak{P}}(i_{1})\otimes\ldots \tilde{\mathfrak{P}}(i_{n})) \otimes g^{\otimes (i_{1}+\ldots+i_{n})}.$$
 %   By the monoid axiom, this is a weak equivalence.

 %   Now consider the case that $f$ is obtained as a pushout
  %  \begin{displaymath}
   %     \xymatrix{
   %     \mathfrak{P}(X)\ar[d]^{\mathfrak{P}(f)}\ar[r] & A\ar[d]\\
   %     \mathfrak{P}(Y)\ar[r] & B,
   %     }
   % \end{displaymath}
   % where $g$ is an acyclic cofibration, (and $A$ is cofibrant). 

%\end{proof}

At the moment this is all largely tautological -- let us give non-trivial examples of the existence of $\alpha$-ideals.

%Later we will see that, in any Koszul category, twisting morphisms arising, via the weak tensoring, from Koszul twisting morphisms from a $\Sigma$-cofibrant cooperad to a $\Sigma$-cofibrant operad in $Ch(\mathpzc{Ab})$, are necessarily homotopical. For most (but not all!) of our Koszul categories of interest there are more explicit examples.

\begin{example}\label{example:moreideals}
Let $\mathpzc{M}$ be a $C$-monoidal strong Koszul category and suppose that for each $n$, $\mathfrak{P}(n)$ and $gr(\mathfrak{C})(n)$ are cofibrant as $\Sigma_{n}$-modules. Note that $\mathpzc{M}$ is then a monoidal model category.
\begin{enumerate}
\item
 The pair $(\mathpzc{coAlg}^{|c|}_{\mathfrak{C}},\mathpzc{Alg}_{\mathfrak{P}}^{|c|})$ is a filtered $\alpha$-ideal. Indeed, by Proposition \ref{underlyingKflat} (3) and Example \ref{ex:appofabove} (2), we have $\mathpzc{Alg}_{\mathfrak{P}}^{c}\subseteq \mathpzc{Alg}_{\mathfrak{P}}^{|c|}$. Now, let $A$ be an object of $\mathpzc{Alg}_{\mathfrak{P}}^{|c|}$. We have $gr(B^{filt}_{\alpha}A)\cong gr(\mathfrak{C})\circ A$. By assumption this is cofibrant. In particular, $B^{filt}_{\alpha}A$ is filtered cofibrant. Moreover, weak equivalences between such objects are stable under tensor powers. Since $gr(\mathfrak{C})(n)$ is cofibrant as a $\Sigma_{n}$-module, $\mathfrak{C}(n)$ is $K$-transverse over $\Sigma_{n}$ to cofibrant objects. The argument that $\Omega^{filt}_{\alpha}$ sends $\mathpzc{coAlg}_{\mathfrak{C}}^{|c|}$ to $\mathpzc{Alg}_{\mathfrak{P}}^{|c|}$ is similar, as is the verification of the transversality requirements. 
\item
If, in addition, $\mathpzc{M}$ is $K$-monoidal, then $(\mathpzc{coAlg}^{|K|}_{\mathfrak{C}},\mathpzc{Alg}_{\mathfrak{P}}^{|K|})$ is a filtered $\alpha$-ideal. The proof of this is similar to the first case.
\end{enumerate}
\end{example}

\subsubsection{Homotopical Properties of $\alpha$-Ideals}
 Using our discussion in Section \ref{subsec:standard}, in this subsection we are essentially going to generalise Proposition 2.8 and Theorem 2.9 of \cite{vallette2014homotopy}. First we need to define some classes of morphisms.

\begin{defn}\label{coalgmod}
A morphism $f:C\rightarrow D$ of filtered $\mathfrak{C}$-coalgebras is said to be a \textbf{weak equivalence} (resp. \textbf{cofibration}) if the underlying morphism $|f|$ of filtered complexes is a weak equivalence (resp. cofibration). The morphism $f$ is said to be a \textbf{strict cofibration} if it is a cofibration and after forgetting differentials $|f|$ is a split monomorphism of filtered objects. The morphism $f$ is said to be a \textbf{fibration} if it has the right lifting property with respect to those morphisms that are both strict cofibrations and weak equivalences.
\end{defn}

Note that if $\mathpzc{M}$ is a projective Koszul category, then the condition that a cofibration be split after forgetting differentials is superfluous. 
 
\begin{prop}\label{factorfibcofib}
Let $\mathpzc{M}$ be a strong Koszul category, let $(C,\Delta_{C})$ and $(C',\Delta_{C'})$ be filtered $\mathfrak{C}$-coalgebras, and let $i:C'\rightarrow C$ be a strict cofibration of filtered coalgebras. Suppose that $gr(\Delta_{C})|_{\bigoplus_{n\ge 1}gr_{n}(C)}:\bigoplus_{n\ge 1}gr_{n}(C)\rightarrow gr(\mathfrak{C})\circ gr(C)$ factors through $gr(\mathfrak{C})\circ gr(C')$. Then $\Omega_{\alpha}(i)$ is a cofibration of $\mathfrak{P}$-algebras.
\end{prop}

\begin{proof}
After forgetting differentials, there is an isomorphism $C\cong C'\oplus E$ where $E$ is cofibrant. Thus,
$$\Omega_{\alpha}C\cong\mathfrak{P}(C')\coprod\mathfrak{P}(E)$$
(again after forgetting differentials).
Under the decomposition $C\cong C'\oplus E$, $d_{C}$ is the sum of three degree $-1$ morphisms
$$d_{C'}:C'\rightarrow C',\;d_{E}:E\rightarrow E,\gamma:E\rightarrow C'.$$
By the assumption on $gr(\Delta_{C})$, the composition
\begin{displaymath}
\xymatrix{
\beta:E\;\;\ar[r] & C\ar[r]^{\Delta_{C}}&\mathfrak{C}(C)\ar[r]^{\alpha(C)} & \mathfrak{P}(C)
}
\end{displaymath}
inducing the twisted differential on $\Omega_{\alpha}C$ factors through $\mathfrak{P}(C')$. Thus, $\Omega_{\alpha}C'\rightarrow\Omega_{\alpha}C$ is given by the standard cofibration $\Omega_{\alpha}C'\rightarrowtail\Omega_{\alpha}C'\coprod_{\gamma+\beta}\mathfrak{P}(E)$. This is a cofibration by Lemma \ref{standard}.
\end{proof}

Now we are able to give analogues of \cite{vallette2014homotopy} Theorem 2.9.
%\begin{prop}\label{preQuillen}
%Suppose $f:A\rightarrow A'$ is a quasi-isomorphism of $\mathfrak{P}$-algebras whose underlying chain complexes are $K$-flat. Then $B_{\alpha}f$ is a filtered quasi-isomorphism of $\mathfrak{C}$-coalgebras whose underlying filtered complexes are $K$-flat. 
%\end{prop}
%
%\begin{proof}
%
%\end{proof}

%\begin{prop}\label{filteredweak}
%Suppose that $\mathfrak{P}$ is an admissible operad. Then
%
%%\end{prop}
%
%\begin{proof}
%\begin{enumerate}
%\end{enumerate}
%\end{proof}
% is in $\mathpzc{Alg}^{|K|}_{\mathfrak{P}}$. 

%But as a complex this is just a transfinite extension of retracts of tensor products of $K$-flat objects, and hence is $K$-flat.  

\begin{thm}\label{preQuillen}

Let $(\mathpzc{I},\mathpzc{J})$ be a filtered $\alpha$-ideal.
\begin{enumerate}
\item
If $f:A\rightarrow B$ is a weak equivalence in $\mathpzc{J}$, then $B_{\alpha}^{filt}(f)$ is a weak equivalence in $\mathpzc{I}$.
\item
The functor $\Omega_{\alpha}^{filt}$ sends weak equivalences in $\mathpzc{I}$ to weak equivalences in $\mathpzc{Alg}_{\mathfrak{P}}$.
\item
If $\mathpzc{M}$ is $C$-monoidal strong Koszul, and the underlying graded $\Sigma$-module of $gr(\mathfrak{C})$ is a retract of a free $\Sigma$-module on a cofibrant object, then $B_{\alpha}^{filt}$ sends cofibrations between objects in $\mathpzc{Alg}_{\mathfrak{P}}^{|c|}$ to cofibrations in $\mathpzc{coAlg}_{\mathfrak{C}}^{|c|}$.
\item
If $\mathpzc{M}$ is $C$-monoidal strong Koszul, the cobar construction $\Omega_{\alpha}^{filt}$ sends strict cofibrations in $\mathpzc{I}$ to cofibrations in $\mathpzc{Alg}_{\mathfrak{P}}$.
\end{enumerate}
\end{thm}

\begin{proof}
\begin{enumerate}
\item
The filtration on the underlying graded object of $B_{\alpha}A$ is given by
$$(B_{\alpha}A)_{n}=\mathfrak{C}_{n}\circ A.$$
By Theorem \ref{thm:gradedcrap} (1) (a), we have
$$ gr(B_{\alpha}f)\cong gr(\mathfrak{C})(f):(gr(\mathfrak{C})(A),d_{\mathfrak{C}}\circ Id_{A}+Id_{gr(\mathfrak{C})}\circ'd_{A})\rightarrow(gr(\mathfrak{C})(B),d_{\mathfrak{C}}\circ Id_{B}+Id_{gr(\mathfrak{C})}\circ'd_{B}).$$
This is a graded weak equivalence by the transversality condition (Definition \ref{defn:ideal} (2) (e)). 
\item
This is similar to Part (1), using Theorem \ref{thm:gradedcrap} (2).
%Let $f:C\rightarrow D$ be a filtered quasi-isomorphism of objects in $\mathpzc{I}$. Consider again the following filtration on $\Omega_{\alpha}(C)=(\mathfrak{P}(C),d_{1}+d_{2})$:
%$$F_{n}\Omega_{\alpha}C=\sum_{k\ge 1,\; m+n_{1}+\ldots+n_{k}\le n}F_{m}\mathfrak{P}(k)\otimes_{\Sigma_{k}}(F_{n_{1}}C\otimes\ldots\otimes F_{n_{k}}C).$$
%Once again passing to associate gradeds we get the $\mathfrak{P}$-algebra whose underlying graded algebra is $gr(\mathfrak{P})\circ gr(C)$ with differential given by the sum $d_{gr(\mathfrak{P})}\circ Id_{gr(C)}+Id_{gr(\mathfrak{P})}\circ' gr(d_{C})+gr(d_{2})$. Again consider the filtration
%$$F_{n}gr(\Omega_{\alpha}C)=\sum_{k\ge 1,\; n_{1}+\ldots+n_{k}\le n}gr(\mathfrak{P}(k))\otimes_{\Sigma_{k}}(F_{n_{1}}gr(C)\otimes\ldots\otimes F_{n_{k}}gr(C)).$$
%The associated graded is, as before, isomorphic to $gr(\mathfrak{P})\circ gr(C)$, with differential given by $d_{gr(\mathfrak{P})}\circ Id_{gr(C)}+Id_{gr(\mathfrak{P})}\circ'd_{C}$. Considering the same filtrations for $D$, the induced morphism $gr(\mathfrak{P})\circ gr(C)\rightarrow gr(\mathfrak{P})\circ gr(D)$ is just $gr(\mathfrak{P})\circ gr(f)$. By assumption $ gr(f): gr(C)\rightarrow gr(D)$ is a weak equivalence of graded objects. By the fact that $\mathpzc{I}$ is a filtered $\alpha$-ideal, $ gr(f)^{\otimes n}$ is a weak equivalence, and each $\mathfrak{P}(n)\otimes_{\Sigma_{n}} gr(f)^{\otimes n}$ is a weak equivalence. Thus, $gr(\mathfrak{P})( gr(f))$ is a weak equivalence, and this completes the proof by Proposition \ref{prop:gradedfiltprop}.
\item
Let $f:A\rightarrow B$ be a cofibration between $\mathfrak{P}$-algebras whose underlying objects in $\mathpzc{M}$ are cofibrant. Again, by Proposition \ref{prop:gradedfiltprop}, we need to show that $gr(B_{\alpha}^{filt}(f))$ is a cofibration. But, as above, $gr(B_{\alpha}^{filt}(f))\cong gr(\mathfrak{C})(f)$. Since $\mathfrak{C}$ is filtered cofibrant, $gr(\mathfrak{C})$ is graded cofibrant. Moreover, $f$ has cofibrant domain and codomain. Hence, the underlying morphism of $gr(\mathfrak{C})(f)$ is a tensor product of cofibrations with cofibrant domain and codomain, so is a cofibration. 
\item
Let $f:C\rightarrow D$ be a strict cofibration of $\mathfrak{C}$-coalgebras with cokernel $E$. We shall denote the filtrations on $C$ and $D$ by $F_{n}C$ and $F_{n}D$, respectively. For any $n\in\mathbb{N}$, consider the sub-coalgebra of $D$ defined by 
$$D^{[n]}=f(C)+F_{n-1}D\defeq im(C\oplus F_{n-1}D\rightarrow D)$$
for $n\ge1$, and
$$D^{[0]}=C.$$
Let us first check that $D^{[n]}\rightarrow D^{[n+1]}$ is an admissible monomorphism with cofibrant cokernel. Let $g:D\rightarrow E$ be the cokernel of $f$. We claim that $\textrm{coker}(D^{[n]}\rightarrow D^{[n+1]})= gr_{n}(E)$. Since $E$ is a cofibrant object of $\overline{\mathpzc{Filt}}_{\textbf{AdMon}}(\mathpzc{M})$, this will prove the claim. Let us first show that the morphism $C\oplus F_{n}D\rightarrow D$ is admissible (i.e., it factors as $p\circ q$ where $q$ is an admissible epimorphism and $p$ an admissible monomorphism). It is sufficient to show that it is admissible after forgetting the differentials. In this case $D\cong C\oplus E$, as filtered objects, and we have the following commutative diagram
\begin{displaymath}
\xymatrix{
C\oplus F_{n}D\ar[r]\ar[d]^{\sim} & D\ar[d]^{\sim}\\
C\oplus F_{n}C\oplus F_{n}E\ar[r] & C\oplus E.
}
\end{displaymath}
To prove that the top morphism is admissible it suffices to prove that the bottom one is. Now, the morphism $C\oplus F_{n}C\rightarrow C$ is a split, and hence pure, epimorphism. The morphism $F_{n}E\rightarrow E$ is an admissible monomorphism. Thus, their sum is an admissible morphism. Indeed, we can factor $C\oplus F_{n}C\oplus F_{n}E\rightarrow C\oplus E$ as $C\oplus F_{n}C\oplus F_{n}E\rightarrow C\oplus F_{n}E\rightarrow C\oplus E$ where the first morphism is an admissible epimorphism, and the second is an admissible monomorphism.

Now, consider the admissible epimorphism 
$$(0,g_{n}):C\oplus F_{n}D\twoheadrightarrow  gr_{n}(E)$$
where $g_{n}$ is the projection to $F_{n}E$, and then to $ gr_{n}(E)$. The subobject $C\cap F_{n}D=F_{n}C$ is contained in the kernel of $(0,g_{n})$, so we get a well-defined morphism
$$C+F_{n}D\cong(C\oplus F_{n}D)\big\slash F_{n}C\rightarrow    gr_{n}(E)$$
that is an admissible epimorphism by the Obscure Axiom. We claim that its kernel is the morphism 
$$D^{[n]}\rightarrow D^{[n+1]}.$$
Again we may ignore differentials. Then, in the direct sum decomposition, this morphism corresponds to the inclusion
$$C\oplus F_{n-1}E\hookrightarrow C\oplus F_{n}E,$$
 that has cokernel equal to $ gr_{n}(E)$.
 We endow $D^{[n]}$ with the following filtration. 
\[
 F_{k}D^{[n]}=
  \begin{cases}
                                   D^{[k]} & \text{if $k\le n$} \\
                                   D^{[n]} & \text{if $k\ge n$.} 
  \end{cases}
\]
We have shown that, with this filtration, $D^{[n]}$ is an object of $\mathpzc{coAlg}_{\mathfrak{C}}^{|c|}$. Moreover, each of the morphisms $D^{[n]}\rightarrow D^{[n+1]}$ satisfies the conditions of Proposition \ref{factorfibcofib}. Therefore, each of the morphisms $\Omega^{filt}_{\alpha}D^{[n]}\rightarrow\Omega^{filt}_{\alpha}D^{[n+1]}$ is a cofibration. Now, $\Omega^{filt}_{\alpha}f:\Omega^{filt}_{\alpha}C=\Omega^{filt}_{\alpha}D^{[0]}\rightarrow\Omega^{filt}_{\alpha}D$ is a transfinite composition of cofibrations, and is therefore a cofibration. 
\end{enumerate}
\end{proof}

\begin{rem}
    
    In the proof of Theorem \ref{preQuillen} (4), it is not important that the filtration on $C$ is a coalgebra filtration. 
\end{rem}

\begin{cor}\label{cor:filteredalphaimpliesalpha}

If $(\mathpzc{I},\mathpzc{J})$ is a filtered $\alpha$-ideal, then $(\mathpzc{I}_{top},\mathpzc{J})$ is an $\alpha$-ideal.
\end{cor}

\begin{proof}
    This follows immediately from Theorem \ref{preQuillen} (1) and (2).
\end{proof}

%\begin{thm}
%Let $\mathpzc{M}$ be a $C$-monoidal Koszul category and suppose that each $\mathfrak{P}(n)$ is cofibrant as a $\Sigma_{n}$-module.
%\begin{enumerate}
%\item
%If $C\in\mathpzc{coAlg}_{|c|,\alpha-adm}$ then $\Omega_{\alpha}^{filt}(C)\in\mathpzc{Alg}^{c}_{\mathfrak{P}}$. 
%\item
%If $A\in\mathpzc{Alg}_{\mathfrak{P}}^{\mathfrak{O}}$ then $B_{\alpha}^{filt}(A)\in\mathpzc{coAlg}_{\mathfrak{C}}^{\mathfrak{O},\alpha-adm}$. \item
%
%
%\end{enumerate}
%\end{thm}
%
%\begin{proof}
%\begin{enumerate}
%\item
%Let $C\in \mathpzc{coAlg}_{\mathfrak{C}}^{|K|}$. Consider filtration on $\Omega_{\alpha}^{filt}(C)$ from the previous part. First note that $ gr(\Omega_{\alpha}(C))\cong\mathfrak{P}(C)$. Therefore, it suffices to show that the underlying complex of $\mathfrak{P}(C)$ is $K$-flat. But this is just a coproduct of tensor products of $K$-flat objects, so it is clearly $K$-flat. 
%\item
%a
%\item
%
%\item
%
%\end{enumerate}
%\end{proof}

Parts (3) and (4) of Theorem \ref{preQuillen} immediately give yet another example of an $\alpha$-ideal.

\begin{example}\label{example:standardexampleideal}
If $\mathpzc{M}$ is $C$-monoidal strong Koszul, and the underlying $\Sigma$-module of $gr(\mathfrak{C})$ is a retract of a free $\Sigma$-module on a cofibrant object, then $(\mathpzc{coAlg}^{|c|}_{\mathfrak{C}},\mathpzc{Alg}_{\mathfrak{P}}^{c})$ is a filtered $\alpha$-ideal.
\end{example}

\begin{example}

    Suppose that $\mathpzc{M}$ is $C$-monoidal strong Koszul, and that every object in $\mathpzc{M}$ is cofibrant. In particular, $\mathpzc{M}$ is also $K$-monoidal. Suppose that the objects $\mathfrak{P}$ and $gr_{n}(\mathfrak{C})$ for all $n\in\mathbb{N}_{0}$ are retracts of free $\Sigma$-modules on cofibrant objects. Then $(\mathpzc{coAlg}_{\mathfrak{C}}^{conil},\mathpzc{Alg}_{\mathfrak{P}})$ is a filtered $\alpha$-ideal. 
\end{example}

\subsubsection{Restricted Bar-Cobar Adjunctions}

%Our weak equivalences of coalgebras are those morphisms that are sent to weak equivalences of algebras by $\Omega_{\alpha}$. 

Here we develop the homotopy theory on the coalgebra side. We denote by $\mathpzc{coAlg}_{\mathfrak{C}_{top}}^{conil,pure}$ the full subcategory of $\mathpzc{coAlg}_{\mathfrak{C}_{top}}^{conil}$ consisting of those objects in the essential image of $(-)_{top}:\mathpzc{coAlg}_{\mathfrak{C}}^{conil}\rightarrow\mathpzc{coAlg}_{\mathfrak{C}_{top}}^{conil}$. Over a field, this functor is essentially surjective, as every coalgebra can be equipped with the coradical filtration. In general. the situation is more complicated.
\begin{defn}\label{alphamodel}

Let $\alpha:\mathfrak{C}\rightarrow\mathfrak{P}$ be a twisting morphism.
\begin{enumerate}
    \item 
    A morphism $f$ in $\mathpzc{coAlg}^{conil}_{\mathfrak{C}}$ (resp. in $\mathpzc{coAlg}_{\mathfrak{C}_{top}}^{conil,pure}$) is said to be an $\alpha$-\textbf{weak equivalence} if $\Omega^{filt}_{\alpha}f$ (resp. $\Omega_{\alpha}f$) is a weak equivalence in $\mathpzc{Alg}_{\mathfrak{P}}$.
    \item 
   A morphism $f$ in $\mathpzc{coAlg}^{conil}_{\mathfrak{C}}$ (resp. in $\mathpzc{coAlg}_{\mathfrak{C}_{top}}^{conil,pure}$) is said to be an
\begin{enumerate}
 \item 
 $\alpha$-\textbf{cofibration} if it is a strict cofibration (Definition \ref{coalgmod}) (resp. a finite composition of morphisms of the form $(f)_{top}$, with $f$ being a strict cofibration);
 %\item
 %\textbf{pre-cofibration} if there is a strict cofibration $\tilde{f}$ in $\mathpzc{I}$ such that $(\tilde{f})_{top}=f$.
 %\item
 %\textbf{cofibration} if it is a finite composition of pre-cofibrations
 \item
   an $\alpha$-\textbf{fibration} if it has the right lifting property with respect to those morphisms that are both strict cofibrations and $\alpha$-weak equivalences. 
  \end{enumerate}
  \end{enumerate}
\end{defn}

In order that the bar-cobar adjunction behaves well with respect to the homotopy theories on the algebra and coalgebra sides, we need to restrict our adjunction to $\alpha$-ideals.

%Let $\alpha:\mathfrak{C}\rightarrow\mathfrak{P}$ be a filtered twisting morphism and $(\mathpzc{I},\mathpzc{J})$ a filtered $\alpha$-ideal. We denote by $\mathpzc{I}_{top}$ the full subcategory of $\mathpzc{coAlg}_{\mathfrak{C}_{top}}^{conil}$ consisting of coalgebras that are in the essential image of the functor $(-)_{top}:\mathpzc{I}\rightarrow\mathpzc{coAlg}^{conil}_{\mathfrak{C}_{top}}$. 

%It is said to b
%\begin{enumerate}
%\item
%A morphism $f:C\rightarrow D$ of  $\mathfrak{C}$-coalgebras is said to be an $\alpha$-\textbf{weak equivalence} if $\Omega_{\alpha}f$ is a quasi-ismorphism of $\mathfrak{P}$-algebras.
%\item
%A morphism $f:C\rightarrow D$ between filtered coalgebras is said to be a \textbf{cofibration} if $|f|_{\mathfrak{C}}$ is a degree-wise split cofibration of filtered objects.
%\item
%A morphism $f:C\rightarrow D$ of filtered coalgebras is said to be a $\alpha$-\textbf{fibration} if it has the right lifting property with respect to those morphisms which are both cofibrations and $\alpha$-weak equivalences.
%\end{enumerate}
%\begin{prop}
%\end{prop}

%Consider the filtration on $\Omega_{\alpha}(f)$ from the proof of Proposition \ref{preQuillen}. The associated graded object is $\mathfrak{P}(f)$. Hence $\mathfrak{P}(f)$ is a weak equivalence. But $|f|$ is a retract of $|\mathfrak{P}(f)|$, and hence is a quasi-isomorphism.

%
%\begin{proof}
%a
%\end{proof}

%FIX: THIS NEEDS TO GO LATER 
The next result follows immediately from Theorem \ref{preQuillen}.

\begin{prop}\label{basicallyQuillen}

Let $(\mathpzc{I},\mathpzc{J})$ be an $\alpha$-ideal. The bar-cobar adjunction restricts to an adjunction of relative categories
$$\adj{\Omega_{\alpha}}{\mathpzc{I}}{\mathpzc{J}}{B_{\alpha}}.$$ Then the following hold.

\begin{enumerate}
    \item 
    If $(\mathpzc{I},\mathpzc{J})$ is a filtered $\alpha$-ideal, then the filtered bar-cobar adjunction restricts to an adjunction of relative categories
$$\adj{\Omega^{filt}_{\alpha}}{\mathpzc{I}}{\mathpzc{J}}{B^{filt}_{\alpha}}.$$
\item
In the $C$-monoidal strong Koszul case, $\Omega^{filt}_{\alpha}$ (resp. $\Omega_{\alpha
}$) sends $\alpha$-cofibrations between objects of $\mathpzc{I}$ (resp. $\mathpzc{I}_{top}$) to cofibrations, and $B^{filt}_{\alpha}$ (resp. $B_{\alpha}$) sends fibrations between objects of $\mathpzc{J}$ to $\alpha$-fibrations.
\end{enumerate}

%$$\adj{\Omega_{\alpha}}{\mathpzc{I}}{\mathpzc{J}}{B_{\alpha}}$$
%If in addition $\mathpzc{M}$ is cofibrant monoidal then it also restricts to an adjunction
%$$\adj{\Omega_{\alpha}}{\mathpzc{coAlg}_{\mathfrak{C}}^{|c|,\alpha-adm}}{\mathpzc{Alg}_{\mathfrak{P}}^{|c|}}{B_{\alpha}}$$
%Finally if $\mathpzc{M}$ is hereditary Koszul, it restricts to an adjunction of relative categories
%$$\adj{\Omega_{\alpha}}{\mathpzc{coAlg}_{\mathfrak{C}}^{|c|,\alpha-adm}}{\mathpzc{Alg}_{\mathfrak{P}}^{c}}{B_{\alpha}}$$
\end{prop}

%\begin{defn}\label{coalgmod}
%Let $\alpha:\mathfrak{C}\rightarrow\mathfrak{P}$ be a twisting morphism. A morphism $f:C\rightarrow D$ in $\mathpzc{coAlg}_{\mathfrak{C}}^{K}$ is said to be an $\alpha$-\textbf{weak equivalence} (resp. $\alpha$-\textbf{cofibration} if there is a morphism $\tilde{f}$ in $\mathpzc{coAlg}_{\mathfrak{C}}(\overline{\mathpzc{Filt}}^{K}_{\textbf{PureMon}}(\mathpzc{M}))$ such that $(\tilde{f})_{top}=f$.
%\end{defn}

We have the following generalisation of \cite{vallette2014homotopy} Theorem 2.6 (2). Some of the techniques used in the following proposition are similar to those in \cite{hirsh2012curved} Proposition 5.1.5.

\begin{prop}\label{prop:alphaequivfilteredequiv}
    Let $f:C\rightarrow D$ be a morphism in $\mathpzc{coAlg}_{\mathfrak{C}}^{conil}$. Equip $\Omega_{\alpha}C$ with the filtration
    $$F_{n}\Omega_{\alpha}C=\sum_{k\ge1,m+n_{1}+\ldots+n_{k}\le n}F_{m}\mathfrak{P}(k)\otimes_{\Sigma_{k}}(F_{n_{1}}C\otimes\ldots\otimes F_{n_{k}}C),$$
    and similarly for $F_{n}\Omega_{\alpha}D$.
    Suppose that $\Omega_{\alpha}(f)$ is a filtered quasi-isomorphism. Then $f$ is a filtered quasi-isomorphism.
\end{prop}

\begin{proof}
This follows immediately from Theorem \ref{thm:gradedcrap} (2)(b) and the fact that $gr(f)$ is a retract of $gr(\mathfrak{P})\circ gr(f)$.

 %First, we consider the associated graded morphism
 %$$gr(\Omega_{\alpha}C)\rightarrow gr(\Omega_{\alpha}D).$$
 %This is a graded quasi-isomorphism by assumption. Now, consider the filtration 
 %$$F_{n}gr(\Omega_{\alpha}C)=\sum_{k\ge1,n_{1}+n_{2}+\ldots+n_{k}\le n}gr(\mathfrak{P})(k)\otimes_{\Sigma_{k}}(gr_{n_{1}}C\otimes\ldots\otimes gr_{n_{k}}C),$$
 %and similarly for $gr(\Omega_{\alpha}D)$. This is a split, and hence admissible, filtration. The morphism $gr(\Omega_{\alpha}(f))$ is a morphism of filtered objects. Moreover, it is a filtered quasi-isomorphism. Indeed, each associated graded piece is a sum  
 %each
 %$$\bigoplus_{k=1}^{\infty}\bigoplus_{m+n_{1}+\ldots+n_{k}=n}gr_{m}(\mathfrak{P}(k))\otimes_{\Sigma_{k}}(gr_{n_{1}}(f)\otimes\ldots\otimes gr_{n_{k}}(f))$$
 %is a quasi-isomorphism, as it is the $n$th associated graded piece of the filtered quasi-isomorphism $\Omega_{\alpha}(f)$. Then 

%$$F_{n}gr(\Omega_{\alpha}f)=\bigoplus_{m=0}^{\infty}\bigoplus_{k=1}^{\infty}\bigoplus_{m+n_{1}+n_{2}+\ldots+n_{k}\le n}gr_{m}(\mathfrak{P}(k))\otimes_{\Sigma_{k}}(gr_{n_{1}}(f)\otimes\ldots\otimes gr_{n_{k}}(f)),$$
%and each $m$-indexed summand is a weak equivalence. Consequently, 
 
%The associated graded of $gr(\Omega_{\alpha}(f))$ for these new filtrations is
% $$gr(gr(\Omega_{\alpha}(f)))\cong gr(\mathfrak{P})(gr(f)).$$
% This is a quasi-isomorphism. The morphism $gr(f)$ is a retract of this morphism, and hence is also a quasi-isomorphism. This completes the proof.
\end{proof}

\subsection{Koszul Morphisms}

In this section we are going to discuss twisting morphisms $\alpha$, called Koszul morphisms, for which the relative adjunctions of Proposition \ref{basicallyQuillen} are relative equivalences. Essentially, we shall set up the necessary technical machinery so that we can generalise the proof in \cite{vallette2014homotopy} of Theorem 2.1 parts (1) and (3) (and also some of the proofs in \cite{hirsh2012curved}) to exact categories. As will become clear, some parts of the proof go through with only minor modifications, while some of them require significant effort to generalise. As we stated before, we assume that $\mathfrak{C}$ is coaugmented. Moreover, we shall assume that there is an augmented operad 
$\tilde{\mathfrak{P}}\in\mathpzc{Op}^{pure}(\mathpzc{M})$ such that $\mathfrak{P}=(\tilde{\mathfrak{P}})_{top}$. The filtration is a technical condition, and we will rarely need to consider filtrations on algebras over $\mathfrak{P}$. From now on, we shall identify $\mathfrak{P}$ with $\tilde{\mathfrak{P}}$.

Now let $\alpha:\mathfrak{C}\rightarrow\mathfrak{P}$ be a degree $-1$ morphism of filtered objects such that $(\alpha)_{top}$ is a twisting morphism. Then we may regard $\mathfrak{C}\circ_{\alpha}\mathfrak{P}$ and $\mathfrak{P}\circ_{\alpha}\mathfrak{C}$ as filtered complexes. Let $\overline{\mathfrak{C}}$ be the kernel of $\mathfrak{C}\rightarrow I$, and $\overline{\mathfrak{P}}$ the cokernel of $I\rightarrow\mathfrak{P}$. Then, as filtered objects, we have $\mathfrak{C}\cong I\oplus\overline{\mathfrak{C}}$ and $\mathfrak{P}\cong I\oplus\overline{\mathfrak{P}}$. Then,
$$\mathfrak{C}\circ\mathfrak{P}\cong I\oplus\overline{\mathfrak{C}}\oplus \overline{\mathfrak{P}}\oplus\overline{\mathfrak{C}}\circ\overline{\mathfrak{P}},$$
and 
$$\mathfrak{P}\circ\mathfrak{C}\cong I\oplus\overline{\mathfrak{P}}\oplus \overline{\mathfrak{C}}\oplus\overline{\mathfrak{P}}\circ\overline{\mathfrak{C}}.$$
In particular, there are filtered morphisms 
$$\mathfrak{C}\circ\mathfrak{P}\rightarrow I,\;\;\;\; I\rightarrow \mathfrak{P}\circ\mathfrak{C}$$
given by projection to, and inclusion of, the factor $I$, respectively.
By multiplying by $\mathfrak{P}$ on the left, we get a filtered morphism
$$\mathfrak{P}\circ\mathfrak{C}\circ\mathfrak{P}\rightarrow\mathfrak{P}.$$

\begin{ass}\label{ass:twistingextra}

    From now on, we shall assume that any twisting morphism we consider from a filtered cooperad to a filtered operad preserves the filtration, and satisfies 
    $$\alpha|_{F_{0}\mathfrak{C}}=0, d_{\mathfrak{C}}|_{F_{0}\mathfrak{C}}=0,d_{\mathfrak{P}}|_{F_{0}\mathfrak{P}}=0.$$
\end{ass}

\begin{rem}\label{rem:filteredbar-cobar}
     
    Consider the monoidal additive category $\mathpzc{Fun}(\mathbb{N}_{0}^{\le},\mathpzc{M})$, equipped with the Day convolution monoidal structure. We may then regard $\mathfrak{C}$ and $\mathfrak{P}$ as a cooperad and operad, respectively, in $\mathpzc{Fun}(\mathbb{N}_{0}^{\le},\mathpzc{M})$. Assumption \ref{ass:twistingextra} implies that $\alpha:\mathfrak{C}\rightarrow\mathfrak{P}$ is then a twisting morphism in the category $\mathpzc{Fun}(\mathbb{N}_{0}^{\le},\mathpzc{M})$. In particular, we get the `filtered' bar-cobar adjunction
    $$\adj{\overline{\Omega}_{\alpha}}{\mathpzc{coAlg}^{conil}_{\mathfrak{C}}(\mathpzc{Fun}(\mathbb{N}_{0}^{\le},\mathpzc{M}))}{{\mathpzc{Alg}_{\mathfrak{P}}(\mathpzc{Fun}(\mathbb{N}_{0}^{\le},\mathpzc{M}))}}{\overline{B}_{\alpha}}.$$
    This restricts to an adjunction
        $$\adj{\overline{\Omega}_{\alpha}}{\mathpzc{coAlg}^{conil}_{\mathfrak{C}}(\mathpzc{Fun}_{\textbf{PureMon}}(\mathbb{N}_{0}^{\le},\mathpzc{M}))}{{\mathpzc{Alg}_{\mathfrak{P}}(\mathpzc{Fun}_{\textbf{PureMon}}(\mathbb{N}_{0}^{\le},\mathpzc{M}))}}{\overline{B}_{\alpha}}$$
        by Proposition \ref{prop:compproductpure}.
\end{rem}

\begin{prop}
Let $\alpha:\mathfrak{C}\rightarrow\mathfrak{P}$ be a twisting morphism. Then the morphisms
$$\mathfrak{C}\circ\mathfrak{P}\rightarrow I, I\rightarrow \mathfrak{P}\circ\mathfrak{C},\textrm{  and  }\mathfrak{P}\circ\mathfrak{C}\circ\mathfrak{P}\rightarrow\mathfrak{P}$$
induce morphisms of complexes
$$\mathfrak{C}\circ_{\alpha}\mathfrak{P}\rightarrow I,\;\;\;\; I\rightarrow \mathfrak{P}\circ_{\alpha}\mathfrak{C},\textrm{  and  }\mathfrak{P}\circ_{\alpha}\mathfrak{C}\circ_{\alpha}\mathfrak{P}\rightarrow\mathfrak{P}.$$
\end{prop}

\begin{proof}
We prove it for the first morphism, the second and third being similar. It suffices to observe that $d_{\alpha}$ preserves the kernel of the augmentation. Since $\alpha$ preserves the filtrations, it suffices to show that $d_{\mathfrak{C}\circ_{\alpha}\mathfrak{P}}|_{F_{0}(\mathfrak{C}\circ_{\alpha}\mathfrak{P})}=0$. This is clear from the assumptions.
\end{proof}

%Recall the two-sided composite product $\mathfrak{P}\circ_{\alpha}\mathfrak{C}\circ_{\alpha}\mathfrak{P}$ from Section \ref{appcatalg}. There is a canonical morphism  
% defined as follows. 
 
\begin{prop}\label{prop:koszequiv}

Let $\alpha:\mathfrak{C}\rightarrow\mathfrak{P}$ be a twisting morphism as above. Suppose that $\alpha|_{F_{0}\mathfrak{C}}=0$, $d_{\mathfrak{C}}|_{F_{0}\mathfrak{C}}=0$, and $d_{\mathfrak{P}}|_{F_{0}\mathfrak{P}}=0$. 
\begin{enumerate}
\item 
Suppose that $\mathfrak{P}\circ_{\alpha}\mathfrak{C}\circ_{\alpha}\mathfrak{P}\rightarrow\mathfrak{P}$ is a filtered weak equivalence. Then both $\mathfrak{C}\circ_{\alpha}\mathfrak{P}\rightarrow I$ and $I\rightarrow \mathfrak{P}\circ_{\alpha}\mathfrak{C}$ are filtered weak equivalences.
    \item 
   Suppose that $gr_{m}(\mathfrak{C})(n)$ and $gr_{m}(\mathfrak{P})(n)$ are all $K$-flat as $\Sigma_{n}$-modules. The following are equivalent.
\begin{enumerate}
\item
The morphism $\mathfrak{P}\circ_{\alpha}\mathfrak{C}\circ_{\alpha}\mathfrak{P}\rightarrow\mathfrak{P}$ is a filtered weak equivalence.
\item
The morphism $\mathfrak{C}\circ_{\alpha}\mathfrak{P}\rightarrow I$ is a filtered weak equivalence.
\item
The morphism $I\rightarrow \mathfrak{P}\circ_{\alpha}\mathfrak{C}$ is a filtered weak equivalence.
\end{enumerate}
\item  
If $\mathpzc{M}$ is strong Koszul and $C$-monoidal and $gr(\mathfrak{P})$ and $gr(\mathfrak{C})$ are (graded) cofibrant, then $gr(\mathfrak{C}\circ_{\alpha}\mathfrak{P})$ is graded cofibrant.
\item 
If $\mathpzc{M}$ is strong Koszul and $C$-monoidal then the following are equivalent. 
\begin{enumerate}
\item
The $\Sigma$-module $gr(\mathfrak{P}\circ_{\alpha}\mathfrak{C}\circ_{\alpha}\mathfrak{P})$ is graded cofibrant.
\item
The $\Sigma$-module $gr(\mathfrak{C}\circ_{\alpha}\mathfrak{P})$ is graded cofibrant.
\item
The $\Sigma$-module $gr(\mathfrak{P}\circ_{\alpha}\mathfrak{C})$ is graded cofibrant.
\end{enumerate}
\end{enumerate}

\end{prop}

\begin{proof}
We will make extensive use of Proposition \ref{prop:gradedfiltprop} throughout. 
\begin{enumerate}
\item 
First, we pass to the associated graded objects of all objects with the induced filtrations, so that we may assume that we are working with split filtrations.
 Consider the filtration on $gr(\mathfrak{P}\circ_{\alpha}\mathfrak{C}\circ_{\alpha}\mathfrak{P})$ given by 
$$F_{n}gr(\mathfrak{P}\circ_{\alpha}\mathfrak{C}\circ_{\alpha}\mathfrak{P})=\sum_{k+l\le n}gr(\mathfrak{P})\circ gr_{k}(\mathfrak{C})\circ gr_{l}(\mathfrak{P})$$
The associated graded of this filtration is $gr(\mathfrak{P})\circ gr(\mathfrak{C}\circ_{\alpha}\mathfrak{P})$, similarly to Theorem \ref{thm:gradedcrap} (2) (b). Now, the morphism $gr(\mathfrak{C}\circ_{\alpha}\mathfrak{P})\rightarrow gr(I)$  is a retract of $gr(\mathfrak{P})\circ gr(\mathfrak{C}\circ_{\alpha}\mathfrak{P})\rightarrow gr(\mathfrak{P})$. Thus, if  $gr(\mathfrak{P}\circ_{\alpha}\mathfrak{C}\circ_{\alpha}\mathfrak{P})\rightarrow gr(\mathfrak{P})$ is a weak equivalence, then so is $gr(\mathfrak{P})\circ gr(\mathfrak{C}\circ_{\alpha}\mathfrak{P})\rightarrow gr(\mathfrak{P})$, and hence $gr(\mathfrak{C}\circ_{\alpha}\mathfrak{P})\rightarrow gr(I)$ is as well. 

Now consider the filtration on $gr(\mathfrak{P}\circ_{\alpha}\mathfrak{C}\circ_{\alpha}\mathfrak{P})$ given by
$$F_{n}gr(\mathfrak{P}\circ_{\alpha}\mathfrak{C}\circ_{\alpha}\mathfrak{P})=\sum_{k+l\le n}gr_{k}(\mathfrak{P})\circ gr_{l}(\mathfrak{C})\circ gr(\mathfrak{P}).$$
Again, similarly to Theorem \ref{thm:gradedcrap} (2) (b), we find that the associated graded of this filtration is $gr(\mathfrak{P}\circ_{\alpha}\mathfrak{C})\circ gr(\mathfrak{P})$. Again, we find that $gr(\mathfrak{P}\circ_{\alpha}\mathfrak{C})\rightarrow I$ is a retract of the weak equivalence $gr(\mathfrak{P}\circ_{\alpha}\mathfrak{C})\circ gr(\mathfrak{P})\rightarrow gr(\mathfrak{P})$, and hence is itself a weak equivalence.
    \item 
    We have shown that $(a)\Rightarrow (b)$ and $(a)\Rightarrow (c)$. 
 The converse implications $(b)\Rightarrow (a)$ and $(c)\Rightarrow (a)$ follow by using the same filtrations and applying Propositions \ref{prop:Kflatcirc1}, \ref{prop:Kflatcirc2}, and \ref{prop:Kflatcirc3}. 
   %First, we claim that each $gr(\mathfrak{C}\circ_{\alpha}\mathfrak{P})(n)$ is $K$-flat as a $\Sigma_{n}$-module. Indeed, considering the filtration 

   % $$F_{m}gr(\mathfrak{C}\circ_{\alpha}\mathfrak{P})=gr_{m}(\mathfrak{C})\circ gr(\mathfrak{P}),$$
   % we find that the associated graded is $gr(\mathfrak{C})\circ gr(\mathfrak{P})$. Since each $gr(\mathfrak{C})(n)$ and each $gr(\mathfrak{P})(n)$ is $K$-flat as a $\Sigma_{n}$-module, each $(gr(\mathfrak{C})\circ gr(\mathfrak{P}))(n)$ is as well - this means that it is also $K$-flat as an object of $\mathpzc{M}$. T

   % Now suppose $(b)$ holds. Since everything is $K$-flat, we have that $gr(\mathfrak{P})\circ(\mathfrak{C}\circ_{\alpha}\mathfrak{P})\rightarrow gr(\mathfrak{P})$ is a (graded) weak equivalence. Using the same filtration as in $(a)$, this implies that $gr(\mathfrak{P}\circ_{\alpha}\mathfrak{C}\circ_{\alpha}\mathfrak{P})$. 

  %  That $(c)\Rightarrow(a)$ is similar. 

 % {prop:Kflatcirc3}
%Conversely if $gr(\mathfrak{C}\circ_{\alpha}\mathfrak{P})\rightarrow gr(I)$ is a weak equivalence then, by , taking the composite product on the left with $gr(\mathfrak{P})$ gives that $gr(\mathfrak{P})\circ(\mathfrak{C}\circ_{\alpha}\mathfrak{P})\rightarrow gr(\mathfrak{P})$ and hence $gr(\mathfrak{P}\circ_{\alpha}\mathfrak{C}\circ_{\alpha}\mathfrak{P})\rightarrow gr(\mathfrak{P})$ are weak equivalences.

\item
Suppose that $gr(\mathfrak{P})$ and $gr(\mathfrak{C})$ are graded cofibrant. We need to show that $gr(\mathfrak{C}\circ_{\alpha}\mathfrak{P})$ is filtered cofibrant, with the filtration arising from the grading. Consider the filtration given by
 $$F_{m}gr(\mathfrak{C}\circ_{\alpha}\mathfrak{P})=gr_{m}(\mathfrak{C})\circ gr(\mathfrak{P}).$$
 Each filtered piece of $gr(\mathfrak{C}\circ_{\alpha}\mathfrak{P})$ is a summand of some $F_{m}gr(\mathfrak{C}\circ_{\alpha}\mathfrak{P})$. Hence, it suffices to prove that the filtered object $(F_{m}gr(\mathfrak{C}\circ_{\alpha}\mathfrak{P}))$ is filtered cofibrant. Its
 associated graded is $gr(\mathfrak{C})\circ gr(\mathfrak{P})$. Since $\mathpzc{M}$ is $C$-monoidal, this object is cofibrant as a $\Sigma$-module. 
 \item 
This is a game of filtrations identical to parts (1) and (2), using Proposition \ref{prop:gradedfiltprop} (2). What is needed, and not completely trivial, is that if either $gr(\mathfrak{C}\circ_{\alpha}\mathfrak{P})$ or $gr(\mathfrak{P}\circ_{\alpha}\mathfrak{C})$ is graded cofibrant, then $gr(\mathfrak{P})$ is graded cofibrant. However, in part (3) we showed that $gr_{m}(\mathfrak{C})\circ gr(\mathfrak{P})$ is cofibrant whenever $gr(\mathfrak{C}\circ_{\alpha}\mathfrak{P})$ is graded cofibrant. Hence, if $gr(\mathfrak{C}\circ_{\alpha}\mathfrak{P})$ is graded cofibrant then $gr(\mathfrak{C})\circ gr(\mathfrak{P})$ is cofibrant and, as a retract thereof, $gr(\mathfrak{P})$ is cofibrant. One can also consider the filtration $F_{n}gr(\mathfrak{P}\circ_{\alpha}\mathfrak{C})=gr(\mathfrak{P})\circ gr_{n}(\mathfrak{C})$. Similarly to part $(1)$, if $gr(\mathfrak{P}\circ_{\alpha}\mathfrak{C})$ is graded cofibrant, then $(F_{n}gr(\mathfrak{P}\circ_{\alpha}\mathfrak{C}))$ is filtered cofibrant, with the total object of the resulting associated graded being $gr(\mathfrak{P})\circ gr(\mathfrak{C})$. As a retract of $gr(\mathfrak{P})\circ gr(\mathfrak{C})$, $gr(\mathfrak{P})$ is cofibrant.
\end{enumerate}

\end{proof}

\subsubsection{Pre-Koszul morphisms}

In this subsection, we introduce so-called pre-Koszul morphisms and discuss some of their homotopical properties.

\begin{defn}
 A twisting morphism $\alpha:\mathfrak{C}\rightarrow\mathfrak{P}$ is said to be \textbf{pre-Koszul} if 
  \begin{enumerate}
\item
 $\mathfrak{C}$ is a coinvariant-pure stably formally planar cooperad;
\item
$\mathfrak{P}$ is admissible;
 \item
Assumption \ref{ass:twistingextra} is satisfied, and
 \item
$\mathfrak{P}\circ_{\alpha}\mathfrak{C}\circ_{\alpha}\mathfrak{P}\rightarrow \mathfrak{P}$ is a filtered weak equivalence.
\end{enumerate} 
\end{defn}

\begin{rem}
    
    Suppose that $\alpha$ is pre-Koszul. Since both $\mathfrak{C}\circ_{\alpha}\mathfrak{P}\rightarrow I$ and $I\rightarrow\mathfrak{P}\circ_{\alpha}\mathfrak{C}$ are equivalences, we have equivalences $\mathfrak{C}(0)\cong 0$ and $\mathfrak{P}(0)\cong 0$. Thus, $\mathfrak{P}$ and $\mathfrak{C}$ are equivalent to their reductions. We will always assume from now on that both $\mathfrak{C}$ and $\mathfrak{P}$ are reduced.
\end{rem}

\begin{prop}\label{unitacyclic}

    Let $\alpha$ be a pre-Koszul twisting morphism, and let $C$ be an object of $\mathpzc{coAlg}_{\mathfrak{C}}^{conil}$ such that 
    $$gr(\mathfrak{P}\circ_{\alpha}\mathfrak{C}\circ_{\alpha}\mathfrak{P})\circ gr(C)\rightarrow gr(\mathfrak{P})\circ gr(C)$$
    is a weak equivalence. When $B_{\alpha}\Omega_{\alpha}C$ is equipped with the filtration
$$F_{n}B_{\alpha}\Omega_{\alpha}C=\sum_{k\ge 1,m+l+n_{1}+\ldots+n_{k}\le n}(\mathfrak{C}_{l}\circ\mathfrak{P}_{m})(k)\otimes_{\Sigma_{k}}(F_{n_{1}}C\otimes\ldots\otimes F_{n_{k}}C),$$
the unit morphism $\nu_{\alpha}(C):C\rightarrow B_{\alpha}\Omega_{\alpha}C$ is a filtered quasi-isomorphism and a filtered admissible monomorphism.
\end{prop}

\begin{proof}
After forgetting differentials, the morphism $\nu_{\alpha}(C)$ is even split in each filtered degree by Proposition \ref{prop:(co)unitadmissible}.

Using Theorem \ref{thm:gradedcrap} (2) (d), we reduce to showing that
\begin{displaymath}
    \xymatrix{
gr(\mathfrak{P})\circ gr(C)\ar[rrr]^{gr(\Omega_{\alpha}(\nu_{\alpha}(C)))\;\;\;\;}& & &gr(\mathfrak{P}\circ_{\alpha}\mathfrak{C}\circ_{\alpha}\mathfrak{P})\circ gr(C)\ar[r] & gr(\mathfrak{P})\circ gr(C)}
\end{displaymath}
is a weak equivalence.
%Let us now show that $\nu_{\alpha}(C)$ is an $\alpha$-weak equivalence.  
The composite
\begin{displaymath}
\xymatrix{
gr(\mathfrak{P})\circ gr(C)\ar[rrr]^{gr(\Omega_{\alpha}(\nu_{\alpha}(C)))\;\;\;\;}& & &gr(\mathfrak{P}\circ_{\alpha}\mathfrak{C}\circ_{\alpha}\mathfrak{P})\circ gr(C)\ar[r] & gr(\mathfrak{P})\circ gr(C)
}
\end{displaymath}
is the identity. The morphism $gr(\mathfrak{P}\circ_{\alpha}\mathfrak{C}\circ_{\alpha}\mathfrak{P})\circ gr(C)\rightarrow gr(\mathfrak{P})\circ gr(C)$ is a weak equivalence by assumption. By the $2$-out-of-$3$ property, $gr(\Omega_{\alpha}(\nu_{\alpha}(C)))$ is an equivalence. Therefore, $\Omega_{\alpha}(\nu_{\alpha}(C))$ is a filtered weak equivalence.

%Now with the stipulated filtrations the morphism $\nu_{\alpha}(C)$ is a filtered retract of the morphism $\Omega_{\alpha}(\nu_{\alpha}(C))$ which is a filtered weak equivalence.
\end{proof}

\begin{rem}
    Note that the filtration on $B_{\alpha}\Omega_{\alpha}(C)$ in the proposition \textit{is not the same as the filtration on} $B_{\alpha}^{filt}\Omega_{\alpha}^{filt}(C)$. The latter filtration is given by 
    $$F_{n}(B_{\alpha}^{filt}\Omega_{\alpha}^{filt}(C))=\sum_{k\ge1,m+l\le n}(\mathfrak{C}_{l}\circ\mathfrak{P}_{m})(k)\otimes_{\Sigma_{k}}C^{\otimes k}.$$ In other words, it is the filtration on $\overline{\Omega}_{\alpha}\circ\overline{B}_{\alpha}(C)$ from Remark \ref{rem:filteredbar-cobar}.
    In general, the unit $\nu_{\alpha}^{filt}$ of the adjunction between $\Omega_{\alpha}^{filt}$ and $B_{\alpha}^{filt}$ is \textit{not} a filtered weak equivalence.
\end{rem}

%FIX: REDEFINE KOSZUL
%The proof of Proposition \ref{unitacyclic} actually implies the following.
%
%\begin{cor}\label{filteredunit}
%Let $C\in\mathpzc{coAlg}_{\mathfrak{C}}^{|K|}$. With the induced filtration on $B_{\alpha}\Omega_{\alpha}C$ the morphism $\nu_{\alpha}(C)$ is a filtered weak equivalence.
%\end{cor}
%
%\begin{proof}
% 

This also allows us to prove that many $\alpha$-weak equivalences are contained in the class of quasi-isomorphisms as in \cite{vallette2014homotopy} Proposition 2.5.
\begin{prop}\label{prop:alphaequivweak}

 Let $\alpha$ be a pre-Koszul twisting morphism.
    Let $C,D\in\mathpzc{coAlg}^{conil}_{\mathfrak{C}}$ be such that 
    $$gr(\mathfrak{P}\circ_{\alpha}\mathfrak{C}\circ_{\alpha}\mathfrak{P})\circ gr(C)\rightarrow gr(\mathfrak{P})\circ gr(C),$$
    and
    $$gr(\mathfrak{P}\circ_{\alpha}\mathfrak{C}\circ_{\alpha}\mathfrak{P})\circ gr(D)\rightarrow gr(\mathfrak{P})\circ gr(D)$$
    are both weak equivalences. 
If $f:C\rightarrow D$ is an $\alpha$-weak equivalence, then $(f)_{top}$ is a weak equivalence of the underlying complexes. 
\end{prop}

\begin{proof}
Consider the commutative diagram 
\begin{displaymath}
\xymatrix{
C\ar[d]^{f}\ar[r]^{\nu_{\alpha}(C)} & B_{\alpha}\Omega_{\alpha}C\ar[d]^{B_{\alpha}\Omega_{\alpha}f}\\
D\ar[r]^{\nu_{\alpha}(D)} & B_{\alpha}\Omega_{\alpha}D.
}
\end{displaymath}
The top, bottom, and right-hand morphisms are weak equivalences by Theorem \ref{preQuillen} (1), (2) and Proposition \ref{unitacyclic}. Therefore, $f$ is a weak equivalence. 
\end{proof}

\subsubsection{Koszul morphisms}
We are now ready to define Koszul morphisms.

\begin{defn}\label{defn:koszmor}

\begin{enumerate}
\item
Let $\alpha$ be a pre-Koszul twisting morphism.
\begin{enumerate}
    \item 
        An $\alpha$-ideal $(\mathpzc{I},\mathpzc{J})$ is said to be \textbf{Koszul} if for any $A$ in $\mathpzc{J}$, we have that the morphism
    $$A\cong I\circ A\rightarrow gr(\mathfrak{P}\circ_{\alpha}\mathfrak{C})\circ A$$
    is a weak equivalence. 
    \item 
    A filtered $\alpha$-ideal $(\mathpzc{I},\mathpzc{J})$ is said to be \textbf{filtered Koszul} if $(\mathpzc{I}_{top},\mathpzc{J})$ is a Koszul $\alpha$-ideal.
    \end{enumerate}
    \item 
    A pre-Koszul twisting morphism $\alpha$ is said to be \textbf{Koszul} (resp. \textbf{filtered Koszul}) if it admits a Koszul $\alpha$-ideal (resp. a filtered Koszul $\alpha$-ideal).
    \item 
       A pre-Koszul twisting morphism $\alpha$ is said to be \textbf{pure Koszul (resp. filtered pure Koszul)} if $(\mathpzc{coAlg}_{\mathfrak{C}_{top}}^{conil},\mathpzc{Alg}_{\mathfrak{P}})$ (resp. $(\mathpzc{coAlg}_{\mathfrak{C}}^{conil},\mathpzc{Alg}_{\mathfrak{P}})$) is a Koszul $\alpha$-ideal (resp. a filtered Koszul $\alpha$-ideal).%$$\alpha:\mathfrak{P}\circ_{\alpha}\mathfrak{C}\circ_{\alpha}\mathfrak{P}\rightarrow \mathfrak{P}$ is a filtered weak equivalence that is graded pure - that is, for any object $M$ of $\mathpzc{M}$, the morphism $gr(\mathfrak{P}\circ_{\alpha}\mathfrak{C}\circ_{\alpha}\mathfrak{P})\circ M\rightarrow gr(\mathfrak{P})\circ M$ is an equivalence. 

\end{enumerate}

\end{defn}

%\begin{defn}
%
%A twisting morphism $\alpha$ is said to be 
%\begin{enumerate}
%    \item 
%    \textbf{Koszul} if 
%\begin{enumerate}
%\item
% $\mathfrak{C}$ is a filtered formally planar cooperad;
%\item
%it is homotopical;
%\item
%$\mathfrak{P}$ is admissible;
% \item
% $\alpha$ induces a morphism of filtered objects, and $\alpha|_{F_{0}\mathfrak{C}}=0$;
% \item
%$\mathfrak{P}\circ_{\alpha}\mathfrak{C}\circ_{\alpha}\mathfrak{P}\rightarrow \mathfrak{P}$ is a filtered weak equivalence;
%\end{enumerate} 
%\item 
%\textbf{filtered Koszul} if it is Koszul and filtered homotopical;
%\item 
%\textbf{pure Koszul} if it is Koszul, and $\alpha:\mathfrak{P}\circ_{\alpha}\mathfrak{C}\circ_{\alpha}\mathfrak{P}\rightarrow \mathfrak{P}$ is a filtered weak equivalence that is graded pure - that is, for any object $M$ of $\mathpzc{M}$, the morphism $gr(\mathfrak{P}\circ_{\alpha}\mathfrak{C}\circ_{\alpha}\mathfrak{P})\circ M\rightarrow gr(\mathfrak{P})\circ M$ is an equivalence. 
%\end{enumerate}
%\end{defn}

\begin{prop}
    Let $\alpha:\mathfrak{C}\rightarrow\mathfrak{P}$ be a pre-Koszul twisting morphism. Suppose that the morphism $\phi:\mathfrak{P}\circ_{\alpha}\mathfrak{C}\circ_{\alpha}\mathfrak{P}\rightarrow \mathfrak{P}$ is a filtered weak equivalence that is graded pure -- that is, for any object $M$ of $\mathpzc{M}$, the morphism $gr(\mathfrak{P}\circ_{\alpha}\mathfrak{C}\circ_{\alpha}\mathfrak{P})\circ M\rightarrow gr(\mathfrak{P})\circ M$ is an equivalence. Then $\alpha$ is a pure Koszul morphism.
    \end{prop}

    \begin{proof}
    As in the proof of Proposition \ref{prop:koszequiv} (2), the morphism $I\rightarrow gr(\mathfrak{P}\circ_{\alpha}\mathfrak{C})$ is also graded pure. Thus, for any $\mathfrak{P}$-algebra $A$, the morphism $I\circ A\rightarrow gr(\mathfrak{P}\circ_{\alpha}\mathfrak{C})\circ A$ is a weak equivalence.
    %We claim that the counit morphism $\Omega_{\alpha}B_{\alpha}M\rightarrow A$ is an equivalence for all objects $A$. Exactly as in the proof of Proposition \ref{prop:counitequiv}, there is a filtration on $\Omega_{\alpha}B_{\alpha}A$, such that for the trivial filtration on $A$, the counit $\Omega_{\alpha}B_{\alpha}A\rightarrow A$ is a filtered morphism, with $(gr(\Omega_{\alpha}B_{\alpha}A)\rightarrow gr(A))\cong(gr(\mathfrak{P}\circ_{\alpha}\mathfrak{C})\rightarrow I)\circ A$. Since $\alpha$ is graded pure, this associated graded morphism is an equivalence. Hence the counit is an equivalence, as required.
\end{proof}

On its face, Definition \ref{defn:koszmor} looks somewhat different from the usual definition from \cite{loday2012algebraic}/\cite{vallette2014homotopy}. However, when $\mathpzc{E}$ is the category of vector spaces over some field $k$ of characteristic zero, the definitions agree. Indeed, in the context of \cite{loday2012algebraic}/\cite{vallette2014homotopy}, for a twisting morphism $\alpha:\mathfrak{C}\rightarrow\mathfrak{P}$, both $\mathfrak{C}$ and $\mathfrak{P}$ are equipped with weight gradings, giving rise to filtrations, and $\alpha$ is a morphism of filtered objects. The assumptions on the gradings imply that $F_{0}\mathfrak{C}=I=F_{0}\mathfrak{P}$. Since $\alpha$ is a degree $-1$ morphism, the morphism $\alpha|_{F_{0}\mathfrak{C}}$ must be zero.  
With these filtrations, over a field of characteristic zero, $\mathfrak{C}\in\mathpzc{coOp}^{pure}(\mathpzc{M})$ and $\mathfrak{P}\in\mathpzc{Op}^{pure}(\mathpzc{M})$ -- indeed, all $\Sigma$-modules are stably formally planar and stably coinvariant-purifying. All operads over a field of characteristic zero are also admissible. By Theorem \ref{preQuillen} (1) and (2), $\alpha$ is filtered homotopical -- in fact, because all exact sequences are pure exact over a field of characteristic $0$, the induced filtration on any $\mathfrak{C}$-coalgebra $C$ makes it an object of $\mathpzc{coAlg}_{\mathfrak{C}}^{conil}$. If $\mathpzc{coAlg}_{\mathfrak{C}}^{corad}$ denotes the category of filtered coalgebras consisting of $\mathfrak{C}_{top}$-coalgebras equipped with the induced filtration, then $(\mathpzc{coAlg}_{\mathfrak{C}}^{corad},\mathpzc{Alg}_{\mathfrak{P}})$ is a filtered $\alpha$-ideal. Moreover, $(\mathpzc{coAlg}_{\mathfrak{C}}^{corad})_{top}=\mathpzc{coAlg}^{conil}_{\mathfrak{C}_{top}}$.

That $(\mathfrak{P}\circ_{\alpha}\mathfrak{C}\circ_{\alpha}\mathfrak{P})\rightarrow\mathfrak{P}$ is a filtered weak equivalence is equivalent to the usual Koszul condition over a field of characteristic $0$. Finally, all exact sequences of $\Sigma_{n}$-modules over a field of characteristic zero are split exact, from which it follows that all weak equivalences of filtered $\Sigma$-modules are graded pure exact. Hence, a Koszul morphism in the usual sense, over a field of characteristic zero, is filtered pure Koszul in the sense of the present work.

\begin{prop}
    
    Let $L:\mathpzc{M}\rightarrow\mathpzc{N}$ be a strongly monoidal, colimit-preserving, additive functor between Koszul categories where $\mathpzc{M}$ is $C$-monoidal and strong. Let 
    $$\overline{\otimes}:\mathpzc{M}\otimes\mathpzc{N}\rightarrow\mathpzc{N}$$
    denote the functor $L(-)\otimes_{\mathpzc{N}}(-)$, and suppose that $\overline{\otimes}$ realises $\mathpzc{N}$ as a category weakly tensored over $\mathpzc{M}$. Let $\alpha:\mathfrak{C}\rightarrow\mathfrak{P}$ be a pure Koszul twisting morphism in $\mathpzc{M}$. Suppose that
    \begin{enumerate}
        \item 
        $L(\mathfrak{P})$ is admissible,
        \item 
     for each $m,n\in\mathbb{N}_{0}$, the objects $F_{m}(\mathfrak{C}(n))$, $gr_{m}(\mathfrak{C}(n))$ and $gr_{m}(\mathfrak{P}(n))$ are retracts of cofibrant $\Sigma_{n}$-modules,
      \item 
      and, for each $m,n\in\mathbb{N}_{0}$, the morphisms $F_{m}(\mathfrak{C}(n))\rightarrow F_{m+1}(\mathfrak{C}(n))$ and $F_{m}(\mathfrak{P}(n))\rightarrow F_{m+1}(\mathfrak{P}(n))$ are split as morphisms of $\Sigma_{n}$-modules.
    \end{enumerate}
    Then $L(\alpha)$ is a pure Koszul twisting morphism.
\end{prop}

\begin{proof}
    First, we verify that $L(\alpha)$ is pre-Koszul. By assumption $(1)$, $L(\mathfrak{P})$ is admissible. By assumption $(2)$, each $L(F_{m}\mathfrak{C}(n))$ is a retract of a free $\Sigma_{n}$-module, and hence is stably formally planar. Assumptions $(2)$ and $(3)$ imply that, with filtrations $L(F_{m}(\mathfrak{C}))$ and $L(F_{m}(\mathfrak{P}))$, respectively, $L(\mathfrak{C})$ and $L(\mathfrak{P})$ are stably coinvariant-purifyingly filtered. Clearly, $L(\alpha)$ is a morphism of filtered objects, and $L(\alpha)|_{L(F_{0}\mathfrak{C})}=0$. Now, by Proposition \ref{prop:koszequiv} (4), $gr(\mathfrak{P})$ and $gr(\mathfrak{P}\circ_{\alpha}\mathfrak{C}\circ_{\alpha}\mathfrak{P})$ are cofibrant as $\Sigma$-modules. By Proposition \ref{prop:koszequiv} (3), $gr(\mathfrak{C}\circ_{\alpha}\mathfrak{P})$ is also graded cofibrant as a $\Sigma$-module. The pure tensoring assumption implies that the functor $L$ preserves weak equivalences between cofibrant objects by Proposition \ref{prop:puretnesorLequiv}. In particular, $L(\alpha)$ is pre-Koszul. Similarly, Proposition \ref{prop:tensorevedoverringRinduced} implies that $L(\alpha)$ is even pure Koszul. 
\end{proof}

\begin{cor}\label{cor:twistingenriched}

    Let $\mathpzc{M}\cong{}_{R}\mathpzc{Mod}(Ch(\mathpzc{E}))$ where $\mathpzc{E}$ is enriched over some unital commutative ring $S$. Let $\alpha:\mathfrak{C}\rightarrow\mathfrak{P}$ be a Koszul twisting morphism in $Ch({}_{S}\mathpzc{Mod}(\mathpzc{Ab}))$, where $\mathfrak{C}$ and $\mathfrak{P}$ are weight graded, and $\Sigma$-cofibrant for the projective model structure. Then, under the canonical symmetric monoidal functor $L:Ch({}_{S}\mathpzc{Mod}(\mathpzc{Ab}))\rightarrow\mathpzc{M}$, $L(\alpha)$ is a pure Koszul twisting morphism, provided that $L(\mathfrak{P})$ is admissible. 
\end{cor}

In particular, if $S=k$ is a field of characteristic zero, then we get the following.

\begin{cor}\label{inducedkoszul}

Let $k$ be a field of characteristic $0$, and let $\alpha:\mathfrak{C}\rightarrow\mathfrak{P}$ be a Koszul morphism in $Ch({}_{k}\mathpzc{Mod})$. Let $\mathpzc{M}$ be a Koszul category with unit $R$ that is enriched over $k$. Then
$$L_{R}(\alpha):L_{R}(\mathfrak{C})\rightarrow L_{R}(\mathfrak{P})$$
is a Koszul morphism in $\mathpzc{M}$.
\end{cor}

This corollary gives rise to Lie-cocommutative duality, which we return to later.

 %Then $\overline{B}(\mathfrak{P})$ is also concentrated in non-negative degrees, and each $\overline{B}(\mathfrak{P})(n)$ is flat as a $\Sigma_{n}$-module. The Koszul complex $\overline{B}(\mathfrak{P})\circ_{\theta_{R}}\mathfrak{P}$ is then also a bounded-below complex such that each $\overline{B}(\mathfrak{P})\circ_{\theta_{R}}\mathfrak{P}(n)$ is flat. In particular, 

\subsection{An Equivalence of $(\infty,1)$-Categories}
We are ready to prove our first version of Koszul duality.

As in \cite{vallette2014homotopy} Theorem 2.6 (1) we have the following.
\begin{prop}\label{prop:counitequiv}
Let $\alpha$ be a pre-Koszul morphism in a Koszul category $\mathpzc{M}$. Let $(\mathpzc{I},\mathpzc{J})$ be a Koszul $\alpha$-ideal. Then, for any $A\in\mathpzc{J}$, the counit $\epsilon_{\alpha}(A):\Omega_{\alpha}B_{\alpha}A\rightarrow A$ is a weak equivalence.
\end{prop}

\begin{proof}
After forgetting differentials, the underlying graded object of $\Omega_{\alpha}B_{\alpha}A$ is $\mathfrak{P}\circ\mathfrak{C}\circ A$.  We first filter it by
$$F_{n}\Omega_{\alpha}B_{\alpha}A=\sum_{k\ge1, m+n_{1}+\ldots+n_{k}\le n}F_{m}\mathfrak{P}(k)\otimes_{\Sigma_{k}}(F_{n_{1}}\mathfrak{C}(A)\otimes\ldots\otimes F_{n_{k}}\mathfrak{C}(A)).$$
%This induces a filtration of the chain complex $\Omega_{\alpha}B_{\alpha}A$. The differential on $\Omega_{\alpha}B_{\alpha}A$ is $d_{\mathfrak{P}\circ_{\alpha}\mathfrak{C}}\circ Id_{A}+Id_{\mathfrak{P}\circ\mathfrak{C}}\circ' d_{A}+\mathfrak{P}\circ' d_{2}$, where $d_{2}$ is the twisted part of the differential on $B_{\alpha}A$.
%The first two differentials preserve the filtration, while the third lowers the filtration. 
We regard $A$ as a filtered algebra with the constant filtration, so that $gr(A)\cong A$ as a graded object concentrated in level $0$. The counit is a morphism of filtered complexes, so $gr(\epsilon_{\alpha}(A))$ is a morphism of graded objects. Now $gr(\Omega_{\alpha}B_{\alpha}A)\cong gr(\mathfrak{P}\circ_{\alpha}\mathfrak{C})\circ A$ as complexes by Theorem \ref{thm:gradedcrap}.

Consider the composite associated graded morphism
$$A\rightarrow gr(\mathfrak{P}\circ_{\alpha}\mathfrak{C})\circ A\rightarrow A.$$
The composite is the identity. The morphism $A\rightarrow gr(\mathfrak{P}\circ_{\alpha}\mathfrak{C})\circ A$ is an equivalence by assumption. By the $2$-out-of-$3$ property, the morphism  $gr(\mathfrak{P}\circ_{\alpha}\mathfrak{C})\circ A\rightarrow A$ is an equivalence. This completes the proof.
\end{proof}

\begin{thm}[Koszul Duality]\label{coopKoszuldual}

Let $\mathpzc{M}$ be a Koszul category, and let $\alpha:\mathfrak{C}\rightarrow\mathfrak{P}$ be a Koszul morphism in $\mathpzc{M}$. Let $(\mathpzc{I},\mathpzc{J})$ be a Koszul $\alpha$-ideal. Regard $\mathpzc{I}$ and $\mathpzc{J}$ as relative categories, with weak equivalences inherited from the $\alpha$-weak equivalences on $\mathpzc{coAlg}_{\mathfrak{C}_{top}}^{conil}$, and usual weak equivalences on $\mathpzc{Alg}_{\mathfrak{P}}$, respectively. Then the following hold.

\begin{enumerate}
\item 
The adjunction of categories
$$\adj{\Omega_{\alpha}}{\mathpzc{I}}{\mathpzc{J}}{B_{\alpha}}$$
is a relative equivalence.
In particular, the bar-cobar adjunction induces an adjoint equivalence of $(\infty,1)$-categories
$$\adj{\Omega_{\alpha}}{\mathrm{L^{H}}(\mathpzc{I})}{\textbf{Alg}_{\mathfrak{P}}}{\textbf{B}_{\alpha}}.$$
\item 
If $(\mathpzc{I},\mathpzc{J})$ is a filtered $\alpha$-ideal, then the adjunction of categories
$$\adj{\Omega^{filt}_{\alpha}}{\mathpzc{I}}{\mathpzc{J}}{B^{filt}_{\alpha}}$$
is a relative equivalence.
In particular, the bar-cobar adjunction induces an adjoint equivalence of $(\infty,1)$-categories
$$\adj{\Omega_{\alpha}^{filt}}{\mathrm{L^{H}}(\mathpzc{I})}{\textbf{Alg}_{\mathfrak{P}}}{\textbf{B}^{filt}_{\alpha}}.$$
Here, we regard $\mathpzc{I}$ as a relative category whose weak equivalences are the $\alpha$-weak equivalences, and $\mathpzc{J}$ as a relative category whose weak equivalences are the usual weak equivalences of algebras over an operad.
\end{enumerate}
\end{thm}

\begin{proof}
This follows immediately from Theorem \ref{prerelequivuc}, Corollary \ref{relequivuc}, Proposition \ref{infinitysame}, and Proposition \ref{prop:counitequiv}.
\end{proof}

\begin{rem}
The theorem implies that $\mathrm{L^{H}}(\mathpzc{I})$ is independent of the choice of $\alpha$-ideal (resp. filtered $\alpha$-ideal).
\end{rem}

%\begin{thm}[Koszul Duality]
%The bar-cobar adjunction induces an adjoint equivalence of $(\infty,1)$-categories
%$$\adj{\Omega_{\alpha}}{\textbf{coAlg}^{conil}_{\mathfrak{C}}^{|K|,\alpha-adm}}{\textbf{Alg}_{\mathfrak{P}}}{\textbf{B}_{\alpha}}$$
%\end{thm}

In particular, in the presence of an $\alpha$-ideal, we get explicit models for the essential image of $\textbf{B}_{\alpha}$.

%By Theorem \ref{relequivuc} it suffices to show that the unit and counit of the adjunction are weak equivalences.

%In the case of the twisting morphism $\kappa:\mathfrak{S}^{c}\otimes_{H}\mathfrak{coComm}\rightarrow\mathfrak{Lie}$ we have the following.
%
%\begin{prop}
%Let $C\in\mathpzc{coAlg}^{|K|,\alpha-adm}_{\mathfrak{S}^{c}\otimes_{H}\mathfrak{coComm}}$ be concentrated in non-negative degrees. Then $\nu_{\kappa}(C)$ is an $\alpha$-acyclic cofibration.
%\end{prop}
%
%\begin{proof}
%FIX: a
%\end{proof}

%\begin{prop}\label{prop:counequiv}
%
 %    Let $\alpha$ be a Koszul twisting morphism with Koszul $\alpha$-ideal $(\mathpzc{I},\mathpzc{J})$.
  %  Let $C\in\mathpzc{I}$. The unit $\nu_{\alpha}:C\rightarrow B_{\alpha}\Omega_{\alpha} C$ is a weak equivalence.
%\end{prop}

%\begin{proof}
 %     We need to show that $\Omega_{\alpha}(\nu_{\alpha})$ is an equivalence of algebras. Now, we have $Id_{\Omega_{\alpha}(C)}=\epsilon_{\Omega_{\alpha}(C)}\circ \Omega_{\alpha}(\nu_{\alpha})$. By the two-out-of-three property, it suffices to prove that $\epsilon_{\Omega_{\alpha}(C)}$ is an equivalence. This follows from Proposition \ref{prop:counitequiv}.
%\end{proof}

\subsubsection{Pure Koszul Duality}

Here we explain why Koszul duality takes a very satisfying form when the twisting morphism $\alpha:\mathfrak{C}\rightarrow\mathfrak{P}$ is pure Koszul.

\begin{prop}
    Suppose that $\alpha:\mathfrak{C}\rightarrow\mathfrak{P}$ is a pure Koszul twisting morphism. Then there is an equivalence of $(\infty,1)$-categories
    $$\Omega_{\alpha}:\mathrm{L^{H}}(\mathpzc{coAlg}_{\mathfrak{C}_{top}}^{conil})\cong\textbf{Alg}_{\mathfrak{P}}.$$
    If $\alpha$ is filtered pure Koszul, then there is an equivalence
     $$\Omega^{filt}_{\alpha}:\mathrm{L^{H}}(\mathpzc{coAlg}_{\mathfrak{C}}^{conil})\cong\textbf{Alg}_{\mathfrak{P}}.$$
\end{prop}

%consisting of those $\mathfrak{C}_{top}$ coalgebras that are in the essential image of the functor $(-)_{top}:\mathpzc{coAlg}_{\mathfrak{C}}^{conil}\rightarrow \mathpzc{coAlg}_{\mathfrak{C}_{top}}^{conil}$. 
Still working in the context of a pure Koszul morphism, recall the category $\mathpzc{coAlg}_{\mathfrak{C}_{top}}^{conil,pure}\subseteq\mathpzc{coAlg}_{\mathfrak{C}_{top}}^{conil}$. 
For $A$ an object of $\mathpzc{Alg}_{\mathfrak{P}}$, the coalgebra $B_{\alpha}A$ is an object of $\mathpzc{coAlg}_{\mathfrak{C}_{top}}^{conil,pure}$, since $B_{\alpha}^{filt}(A)$ is in $\mathpzc{coAlg}^{conil}_{\mathfrak{C}}$. Moreover, for $C\in\mathpzc{coAlg}_{\mathfrak{C}_{top}}^{conil}$, the unit of the bar-cobar adjunction furnishes a natural $\alpha$-weak equivalence 
$$C\rightarrow B_{\alpha}\Omega_{\alpha}C.$$
By the dual of Corollary \ref{relequivuc2} (1), we get an equivalence of $(\infty,1)$-categories
$$\mathrm{L^{H}}(\mathpzc{coAlg}_{\mathfrak{C}_{top}}^{conil,pure})\cong \mathrm{L^{H}}(\mathpzc{coAlg}_{\mathfrak{C}_{top}}^{conil})$$
where we localise both sides at $\alpha$-equivalences. In particular, we get the following result.

\begin{cor}
    There is an equivalence of $(\infty,1)$-categories
    $$\adj{\Omega_{\alpha}}{\mathrm{L^{H}}(\mathpzc{coAlg}_{\mathfrak{C}_{top}}^{conil,pure})}{\textbf{Alg}_{\mathfrak{P}}}{\textbf{B}_{\alpha}}.$$
\end{cor}

The advantage of this presentation is that, by Proposition \ref{prop:alphaequivweak}, the $\alpha$-weak equivalences in $\mathpzc{coAlg}_{\mathfrak{C}_{top}}^{conil,pure}$ are slightly less mysterious -- we at least know that they are weak equivalences of the underlying objects. 

\begin{cor}
    Let $\mathpzc{M}$ be a Koszul category.
    Let $\alpha:\mathfrak{C}\rightarrow\mathfrak{P}$ be a Koszul twisting morphism such that each $gr(\mathfrak{C}\circ_{\alpha}\mathfrak{P})(n)$ is $K$-flat as a $\Sigma_{n}$-module. Then $\alpha$ is a pure Koszul twisting morphism.
       In particular, this holds if $\mathpzc{M}$ is $KC$-monoidal and $gr(\mathfrak{C})$ and $gr(\mathfrak{P})$ are retracts of free $\Sigma$-modules on graded cofibrant objects. 
\end{cor}

Essentially all of our explicit examples of Koszul morphisms are pure. However, the advantage of having filtered $\alpha$-ideals is as follows. Suppose that $\alpha$ is pure Koszul and we have a filtered Koszul ideal $(\mathpzc{I},\mathpzc{J})$. In particular, 
$$\mathrm{L^{H}}(\mathpzc{I}_{top})\cong\textbf{Alg}_{\mathfrak{P}}.$$
On the one hand, we know that the $\alpha$-weak equivalences on the left-hand side are in particular weak equivalences of underlying objects. On the other hand, we can identify some of the $\alpha$-weak equivalences in $\mathpzc{I}_{top}$ -- indeed, the class of $\alpha$-equivalences includes all morphisms of the form $(f)_{top}$, with $f$ being a filtered weak equivalence in $\mathpzc{I}$ by Theorem \ref{preQuillen}.

\begin{example}
    Suppose that we are in the special case where $\mathpzc{M}\cong Ch(\mathpzc{E})$, where $\mathpzc{E}$ is a locally presentable, closed symmetric monoidal quasi-abelian category, in which filtered colimits  commute with finite limits and all objects are flat. In particular, we may equip $Ch(\mathpzc{E})$ with the flat model structure as in Example \ref{example:flatmodel}. With our assumptions, $Ch(\mathpzc{E})$ is a monoidal model category satisfying the monoid axiom, and is a pure Koszul category. 
    %Suppose further that, for any countable family $\{A_{n}\}_{n\in\mathbb{N}_{0}}$ of objects of $\mathpzc{A}$, the canonical morphism $\bigoplus_{n\in\mathbb{N}_{0}}A_{n}\rightarrow\prod_{n\in\mathbb{N}_{0}}A_{n}$ is a monomorphism. 
    Let $\alpha:\mathfrak{C}\rightarrow\mathfrak{P}$ be a pre-Koszul twisting morphism with each $F_{m}(\mathfrak{C}(n))$, each $gr_{m}(\mathfrak{C}(n))$ and each $gr(\mathfrak{P})(n)$ being a retract of a free $\Sigma_{n}$-module on a cofibrant object. Then $(\mathpzc{coAlg}_{\mathfrak{C}}^{conil},\mathpzc{Alg}_{\mathfrak{P}})$ is a filtered Koszul ideal. Moreover, since the canonical filtration on a conilpotent $\mathfrak{C}_{top}$-coalgebra $C$ is automatically a filtration by strict monomorphisms, and since all objects are flat, it is even a filtration by pure monomorphisms. Thus, the functor $(-)_{top}:\mathpzc{coAlg}_{\mathfrak{C}}^{conil}\rightarrow\mathpzc{coAlg}_{\mathfrak{C}_{top}}^{conil}$ is surjective. Consequently, $\alpha$ is a pure Koszul twisting morphism, and we have equivalences of $(\infty,1)$-categories
    $$\mathrm{L^{H}}(\mathpzc{coAlg}_{\mathfrak{C}}^{conil})\cong\mathrm{L^{H}}(\mathpzc{coAlg}_{\mathfrak{C}_{top}}^{conil})\cong\textbf{Alg}_{\mathfrak{P}}.$$
  Furthermore, the class of weak equivalences on the left-hand side is contained in the class of quasi-isomorphisms, and contains filtered quasi-isomorphisms. This generalises the characteristic zero field case.
\end{example}

%As a corollary of Proposition \ref{infinitysame}, we get the following version of Koszul duality.

\subsection{Examples of Koszul Morphisms}

Let us now give examples of Koszul twisting morphisms in myriad contexts.

%Let $\mathpzc{M}$ be a Koszul category, and let $C$ be a commutative monoid in $\mathpzc{M}$. The morphism $\alpha:\mathfrak{G}^{c}\otimes\mathpzc{coAss}\rightarrow\mathpzc{Ass}$ is a twisting morphism (not necessarily homotopical!). Let $A$ be an associative algebra that is $K$-flat as an object of $\mathpzc{M}$. W\input{name}e get a (shifted) cocommutative coalgebra $B_{\alpha}(A)$. The adjunction between $B_{\alpha}$ and $\Omega_{\alpha}$ gives rise to a counit morphism 
%$$\Omega_{\alpha}B_{\alpha}(A)\rightarrow A.$$
%As in \cite{loday2012algebraic}  Theorem 2.2.9 \ we get a twisting morphism $B_{\alpha}A\rightarrow A$. 
\subsubsection{Bootstrapping examples}

When $\mathpzc{E}$ is the category of modules over some ring $R$,  there are many known examples. It turns out we can often bootstrap these.

\begin{example}\label{example:induced}

\begin{enumerate}
    \item 
     Let $\mathpzc{M}$ be a Koszul category enriched over a ring $S$. In particular, $\mathpzc{M}$ is weakly tensored over the monoidal model category $Ch({}_{S}\mathpzc{Mod}(\mathpzc{Ab}))$, equipped with the projective model structure.
    Let $\alpha:\mathfrak{C}\rightarrow\mathfrak{P}$ be a Koszul twisting morphism of operads in $Ch({}_{S}\mathpzc{Mod}(\mathpzc{Ab}))$, where $\mathfrak{C}$ and $\mathfrak{P}$ are $\Sigma$-cofibrant. In particular, $\mathfrak{C}$ and $\mathfrak{P}$ are both fibrant and cofibrant, and $\mathfrak{C}\circ_{\alpha}\mathfrak{P}$ is also both fibrant and cofibrant by Proposition \ref{prop:koszequiv} (4). Suppose further that $L(\mathfrak{P})$ is admissible as an operad in $\mathpzc{M}$. Then $L(\alpha)$ is a Koszul twisting morphism. The equivalence $\mathfrak{C}\circ_{\alpha}\mathfrak{P}\rightarrow I$ is an equivalence of fibrant-cofibrant objects, and hence is a homotopy equivalence. Consequently, $L(\alpha)$ is even a pure Koszul twisting morphism. For example, if $\mathfrak{P}$ is a retract of a free $\Sigma$-operad on a non-symmetric operad, then $L(\alpha)$ is Koszul. If $\mathbb{Q}\subset S$, then $\mathfrak{P}$ can be any such $\Sigma$-operad.
    \item 
       By (1), for any Koszul category, we get a Koszul twisting morphism  $\mathfrak{G}^{c}\otimes\mathfrak{coAss}\rightarrow\mathfrak{Ass}$. In particular, in any Koszul category, we get duality between (shifted) coassociative coalgebras and associative algebras.
\end{enumerate}
%By Proposition \ref{prop:weaktensKosz}, $L(\alpha)$ is a Koszul morphism in $\mathpzc{M}$. If $\alpha$ is pure Koszul, then so is $L(\alpha)$.  

\end{example}

As mentioned in the above example, if $\mathbb{Q}\subset S$ then the situation is even simpler. Indeed, in this case, all operads are admissible. 

\begin{example}\label{ex:liecommutative}

Recall that in the category ${}_{\Q}\mathpzc{Vect}$ of vector spaces over $\Q$ there is a Koszul morphism 
$$\kappa:\mathfrak{S}^{c}\otimes_{H}\mathfrak{coComm}^{ncu}\rightarrow\mathfrak{Lie}.$$
This follows from the fact, which we will not explore in detail here, that $\mathfrak{S}^{c}\otimes_{H}\mathfrak{coComm}^{ncu}$ is the quadratic dual cooperad of $\mathfrak{Lie}$. See \cite{dehling2017weak} Section 2.1 for details. Let $\mathpzc{M}$ be a Koszul category enriched over $\mathbb{Q}$ with unit $R$. Then, by Corollary \ref{inducedkoszul} we get a pure Koszul morphism in $\mathpzc{M}$
%$$\kappa:\mathfrak{S}^{c}\otimes_{H}\mathfrak{coComm}^{ncu}\rightarrow\mathfrak{Lie}$$ is the canonical Koszul morphism in ${}_{\Q}\mathpzc{Mod}$ then we get a pure Koszul morphism 
\begin{displaymath}
\xymatrix{
\kappa:\mathfrak{S}^{c}\otimes_{H}\mathfrak{coComm}^{ncu}\ar[r]^{\cong} & \mathfrak{S}^{c}\otimes_{H}\mathfrak{coComm}^{ncu}\ar[r]^{\;\;\;\;\;\;\;\;\;\;\;\;\;R\otimes\alpha} & \mathfrak{Lie}
}
\end{displaymath}
in $\mathpzc{M}$. This generalises classical Koszul duality between Lie algebras and cocommutative coalgebras to arbitrary $\Q$-Koszul categories.

There is also the (shifted) dual pure Koszul twisting morphism
$$\kappa^{\vee}:\mathfrak{S}^{c}\otimes_{H}\mathfrak{Lie}^{\vee}\rightarrow\mathfrak{Comm}^{nu}.$$
\end{example}

\subsubsection{Symmetric-alternating duality}

\begin{prop}\label{Koszulfree}
Suppose that $\mathpzc{M}\cong{}_{R}\mathpzc{Mod}(Ch(\mathpzc{E}))$ is an exact Koszul category enriched over $\Q$. Let $V$ be an object in $\mathpzc{M}$ and consider the twisting morphism
$$\alpha:S^{c}(V[-1])\rightarrow V[-1]\rightarrow V\rightarrow S(V)$$
from Example \ref{cofreetwist}. Let us suppose that $V=R\otimes V_{0}$ where $V_{0}$ is an object of $Ch(\mathpzc{E})$ concentrated in homological degree $0$. 
\begin{enumerate}
\item 
If $V$ is $K$-flat then $S(V)$ has $K$-flat entries.
\item
The morphism $\overline{\phi}:S^{c}(V[-1])\circ_{\alpha}S(V)\rightarrow 0$ is a graded quasi-isomorphism, where $S^{c}(V[-1])$ and $S(V)$ are equipped with the gradings by symmetric powers, and $S^{c}(V[-1])\circ_{\alpha}S(V)$ is equipped with the tensor power grading.
\item
When $S(V)$ and $S^{c}(V[-1])$ are equipped with the tensor product gradings, the respective differentials and $\alpha$ preserve the filtrations. Moreover, $S^{c}(V[-1])\in\mathpzc{coOp}^{pure}$, and $S(V)\in\mathpzc{Op}^{pure}$. 
\end{enumerate}
In particular, if $V$ is $K$-flat, then $\alpha$ is a filtered pure Koszul morphism. Finally, $(\mathpzc{coAlg}_{S^{c}(V[-1])}^{conil},\mathpzc{Alg}_{S(V)})$ is a filtered Koszul ideal.
\end{prop}

\begin{proof}

For (1), the object $S(V)$ is a direct sum of retracts of tensor products of $K$-flat objects, and hence is $K$-flat. Note that, from this, it follows immediately that $(\mathpzc{coAlg}_{S^{c}(V[-1])}^{conil},\mathpzc{Alg}_{S(V)})$ is a filtered $\alpha$-ideal. Indeed, $K$-flatness of $S(V)$ and $S^{c}(V)$ imply the transversality conditions (e) and (f) for a filtered $\alpha$-ideal. 

 We now prove (2). The morphism $\alpha$ is obtained by base-change from the morphism $\alpha_{0}:S^{c}(V_{0}[-1])\rightarrow S(V_{0})$ in $Ch(\mathpzc{E})$. By Corollary \ref{inducedkoszul} it suffices to prove the claim for $R=k$, where $k$ is the unit of $\mathpzc{E}$. The proof of Proposition 3.4.8 in \cite{loday2012algebraic} works mutatis mutandis in the setting of $Ch(\mathpzc{E})$.

 Claim (3) is clear.

 If $V$ is $K$-flat, then the complex $S^{c}(V[-1])\circ_{\alpha}S(V)$ is $K$-flat, since it admits a filtration whose associated graded is the tensor product of $K$-flat objects. Consequently, for any object $W$ of $\mathpzc{M}$, we have $(S^{c}(V[-1])\circ_{\alpha}S(V))\circ W\cong (S^{c}(V[-1])\circ_{\alpha}S(V))\otimes W\cong  (S^{c}(V[-1])\circ_{\alpha}S(V))\otimes^{\mathbb{L}} W$. It follows that $\alpha$ is pure Koszul. 
\end{proof}

One can prove a similar claim, in an essentially identical way, for the tensor algebra.

\begin{rem}\label{rem:Koszulfree}
By Proposition \ref{Koszulfree}, we get a
 graded quasi-isomorphism
    $$\phi:S_{R}^{c}(V[-1])\circ_{\alpha}S_{R}(V)\rightarrow R,$$
    where $R$ is the monoidal unit of $\mathpzc{M}$.
   We call this resolution the \textbf{Koszul resolution of }$R$ \textbf{associated to }$\alpha$. 
\end{rem}

%\begin{rem}
%We expect that without much effort, the condition that $\mathpzc{M}$ is enriched over $\Q$ and $V$ is concentrated in arity $1$ can be removed. However, for our purposes we do not need such a general result. 
%\end{rem}

This result has the following extremely useful consequence.

\begin{prop}\label{conditionfordeceny}
%:LURIE 
Let $\mathpzc{M}={}_{R}\mathpzc{Mod}(Ch(\mathpzc{E}))$ be a KC-monoidal Koszul category with exact countable products and exact countable coproducts. Let $V_{0}$ be a flat object in $\mathpzc{E}$, regarded as a complex concentrated in degree $0$. Put $V=R\otimes V_{0}$. Suppose that $V$ is formally $\aleph_{1}$-filtered relative to the class $R\otimes V_{0}^{\otimes}\otimes Sym(V_{0})$. Then
$\hat{S}_{R}(V)\otimes^{\mathbb{L}}_{S_{R}( V)}R\cong \hat{S}_{R}(V)\otimes_{S_{R}( V)}R\cong R$. Hence the diagram below
\begin{displaymath}
\xymatrix{
S_{R}(V)\ar[d]^{i}\ar[r]^{\pi} & R\ar[d]\\
\hat{S}_{R}(V)\ar[r]^{\hat{\pi}} & R
}
\end{displaymath}
is a homotopy pushout, where $i$ is the canonical morphism, $\pi$ and $\hat{\pi}$ are the augmentations, and the right-hand vertical morphism is the identity. In particular, $V$ is decent. 
\end{prop}
\begin{proof}
Consider the Koszul resolution $\phi$ of $R$ from Proposition \ref{Koszulfree} and Remark \ref{rem:Koszulfree}, 
$$\phi:S_{R}^{c}(V[-1])\otimes_{\alpha} S_{R}( V)\rightarrow R.$$
This is a resolution of $R$ by a $K$-flat object in the category of $S_{R}( V)$-modules. Moreover, regarding $R$ as a graded object concentrated in degree $0$, and with the tensor product grading induced from the gradings on $S_{R}^{c}( V[-1])$ and $S_{R}( V)$, this is a graded equivalence. In particular, for $n=0$, we have an equivalence
$$\phi_{0}:(S_{R}^{c}(V[-1])\otimes_{\alpha} S_{R}(V))_{0}\rightarrow R,$$
and for $n>0$, we have an equivalence
$$\phi_{n}:(S_{R}^{c}(V[-1])\otimes_{\alpha} S_{R}(V))_{n}\rightarrow 0.$$

Tensoring with $\hat{S}_{R}(V)$ over $S_{R}( V)$ gives the complex
$$S_{R}^{c}( V[-1])\otimes_{\alpha} \hat{S}_{R}(V),$$
which computes $\hat{S}_{R}(V)\otimes^{\mathbb{L}}_{S_{R}( V)}R$.
We want to show that the induced morphism $\hat{\phi}:S_{R}^{c}(V[-1])\otimes_{\alpha} \hat{S}_{R}(V)\rightarrow R$ is an equivalence. By the $\aleph_{1}$-filtered assumptions, each graded piece of $S_{R}^{c}(V)$ is $\aleph_{1}$-filtered relative to $R\otimes V_{0}^{\otimes}\otimes Sym(V_{0})$. Therefore,
$$S_{R}^{c}(V[-1])\otimes_{\alpha} \hat{S}_{R}( V)\cong\prod_{n=0}^{\infty}(S_{R}^{c}(V[-1])\otimes_{\alpha} S_{R}( V))_{n}.$$
The morphism $\hat{\phi}:\prod_{n=0}^{\infty}(S_{R}^{c}( V[-1])\otimes_{\alpha} S_{R}( V))_{n}\rightarrow R$ is the product of the morphisms
$$\hat{\phi}_{0}:(S_{R}^{c}( V[-1])\otimes_{\alpha} S_{R}(V))_{0}\rightarrow R,$$
and
$$\hat{\phi}_{n}:(S_{R}^{c}(V[-1])\otimes_{\alpha} S_{R}(V))_{n}\rightarrow 0$$
as $n$ ranges over positive integers.
Since the product functor is exact, the induced morphism on products is an equivalence, and we are done.
%$$S^{c}(V[1])\otimes\hat{S}(V))\cong\bigoplus_{m=0}^{\infty}\prod_{n=0}^{\infty}Sym^{m}(V[1])\otimes Sym^{n}(V)$$
\end{proof}

%Before we proceed to Koszul duality for Koszul morphisms, let us make sure we have some examples. 

%\subsubsection{The operadic bar construction - a sketch}\label{subsubsecoperadbar}
\begin{rem}
    
We expect that the previous example should generalise as in \cite{fresse2004koszul} Sections 4.4.5 and 4.6. This would include Koszul duality for quadratic operads, as in \cite{fresse2004koszul}, and \cite{loday2012algebraic} Section 7.2. In particular, by regarding (co)algebras as (co)operads concentrated in arity $1$, Savage's work \cite{savage2023koszul}, specifically Section 7, could then be used to produce novel examples of Koszul morphisms not induced by Koszul morphisms in categories of modules. 
\end{rem}

\subsubsection{The Cotangent Complex}\label{subsubsec:cotangent}

Let $\mathpzc{M}$ be a Koszul category. Let $\alpha:\mathfrak{C}\rightarrow\mathfrak{P}$ be a Koszul twisting morphism, and $(\mathpzc{I},\mathpzc{J})$ a Koszul $\alpha$-ideal. Suppose that $\mathfrak{P}$ is a weakly rectifiable, admissible operad. Under these hypotheses, we have seen that the category $\textbf{Alg}_{\mathfrak{P}}(\mathpzc{M})$ is generated under homotopy sifted colimits by free algebras on cofibrant objects. Being an equivalence of categories, the functor $\textbf{B}_{\alpha}$ preserves homotopy sifted colimits. Thus, computing $\textbf{B}_{\alpha}$ essentially reduces to computing what it does to free algebras. Now let $\mathpzc{M}$ be a $C$-monoidal and strong Koszul category, and let $V$ be an object of $\mathpzc{M}$. Consider the trivial $\mathfrak{C}$-coalgebra structure on $V$ given by $V\cong I\circ V\rightarrow(\overline{\mathfrak{C}}\circ V)\oplus (I\circ V)\cong\mathfrak{C}\circ V$, where the middle morphism is inclusion. This is a filtered $\mathfrak{C}$-coalgebra when we equip $V$ with the trivial filtration, i.e., for each $n$, $F_{n}V=V$. This gives the functorial construction $tr_{\mathfrak{C}}:\mathpzc{M}\rightarrow\mathpzc{coAlg}^{conil}_{\mathfrak{C}}$.

\begin{prop}\label{prop:Balphafree}
Let $V\in\mathpzc{M}$ be such that $\mathrm{Free}_{\mathfrak{P}}(V)\in\mathpzc{J}$. There is a natural transformation of functors that is an object-wise $\alpha$-weak equivalence, 
$$t:tr_{\mathfrak{C}}\rightarrow B_{\alpha}\circ \mathrm{Free}_{\mathfrak{P}}(-).$$
\end{prop}

\begin{proof}
By adjunction, writing a morphism $tr_{\mathfrak{C}}\rightarrow B_{\alpha}\circ \mathrm{Free}_{\mathfrak{P}}(-)$ is equivalent to writing a morphism $\Omega_{\alpha}\circ tr_{\mathfrak{C}}\rightarrow \mathrm{Free}_{\mathfrak{P}}(-)$. But $\Omega_{\alpha}\circ tr_{\mathfrak{C}}$ coincides with $\mathrm{Free}_{\mathfrak{P}}(-)$. We need to prove that $\Omega_{\alpha}(t_{V}):\mathrm{Free}_{\mathfrak{P}}(V)\cong \Omega_{\alpha}\circ tr_{\mathfrak{C}}(V)\rightarrow \Omega_{\alpha}\circ B_{\alpha}\circ \mathrm{Free}_{\mathfrak{P}}(V)$ is an equivalence. The Koszul condition gives that the counit $ \epsilon_{\alpha}(\mathrm{Free}_{\mathfrak{P}}(V)):\Omega_{\alpha}\circ B_{\alpha}\circ \mathrm{Free}_{\mathfrak{P}}(V)\rightarrow \mathrm{Free}_{\mathfrak{P}}(V)$ is an equivalence, and the composition $\epsilon_{\alpha}(\mathrm{Free}_{\mathfrak{P}}(V))\circ\Omega_{\alpha}(t_{V})$ is the identity.  Hence, $\Omega_{\alpha}(t_{V})$ is also an equivalence.
\end{proof}

Let $V$ be an object of $\mathpzc{M}$. Denote by $tr_{\mathfrak{P}}(V)$ the $\mathfrak{P}$-algebra that has $V$ as its underlying object of $\mathpzc{M}$, with the trivial $\mathfrak{P}$-algebra structure 
$$\mathfrak{P}\circ V\cong(\overline{\mathfrak{P}}\circ V)\oplus (I\circ V)\cong(\overline{\mathfrak{P}}\circ V)\oplus V\rightarrow V,$$
where the final morphism is projection. This gives a functorial construction $$tr_{\mathfrak{P}}:\mathpzc{M}\rightarrow\mathpzc{Alg}_{\mathfrak{P}}.$$

Similarly to Proposition \ref{prop:Balphafree}, we have the following.

\begin{prop}
   Let $\alpha$ be a Koszul morphism. Let $V\in\mathpzc{M}$ be such that $\mathfrak{C}(V)\in\mathpzc{I}$.
   Then we have a natural transformation of functors that is an object-wise weak equivalence, $\Omega_{\alpha}\circ\mathfrak{C}(-)\rightarrow tr_{\mathfrak{P}}:\mathpzc{M}\rightarrow\mathpzc{Alg}_{\mathfrak{P}}$.
\end{prop}

\begin{rem}\label{rem:cotforgetnat}
    Note that, for a $\mathfrak{P}$-algebra $A$, we have a natural morphism $tr_{\mathfrak{C}}(|A|)\rightarrow B_{\alpha}(A)$. Indeed, by adjunction, this is equivalent to writing a morphism $\Omega_{\alpha}(tr_{\mathfrak{C}}(|A|))\rightarrow A$. But $\Omega_{\alpha}(tr_{\mathfrak{C}}(|A|))\cong\mathfrak{P}(|A|)$. The morphism $\mathfrak{P}(|A|)\rightarrow A$ is then just given by multiplication.
\end{rem}

The functor $tr_{\mathfrak{P}}$ is right Quillen, and has left adjoint $D$ given by the cokernel of the multiplication morphism $\mu_{A}:\mathfrak{P}\circ(A)\rightarrow A$. We write $\mathbb{L}_{0}^{\mathfrak{P}}(-)=\mathbb{L}D(-)$ for the (derived) left adjoint.

\begin{rem}
Let $\mathfrak{B}$ be an operad that is not necessarily reduced such that $\mathfrak{P}=\mathfrak{B}^{red}$ is the kernel of the morphism of operads $\mathfrak{B}\rightarrow\mathfrak{B}(0)$. Note that $\mathfrak{B}(0)$ has a canonical $\mathfrak{B}$-algebra structure, making it the initial $\mathfrak{B}$-algebra (see, e.g., \cite{berger2008derived}).  The category ${}_{\mathfrak{B}(0)}\mathpzc{Mod}^{\mathfrak{B}}$ of $\mathfrak{B}(0)$-modules in the sense of \cite{loday2012algebraic} 12.3.1 is equivalent to $\mathpzc{M}$ by \cite{berger2008derived} Lemma 1.4. 
 There is a functor $\mathfrak{B}(0)\ltimes(-):{}_{\mathfrak{B}(0)}\mathpzc{Mod}^{\mathfrak{B}}\rightarrow(\mathpzc{Alg}_{\mathfrak{B}})_{\big\slash \mathfrak{B}(0)}$ called the \textbf{square-zero extension} functor, generalising the one for commutative monoids (see \cite{loday2012algebraic} 12.3.3). There is an equivalence of categories $\mathpzc{Alg}_{\mathfrak{P}}\cong(\mathpzc{Alg}_{\mathfrak{B}})_{\slash\mathfrak{B}(0)}$ given on objects by $A\mapsto \mathfrak{B}(0)\oplus A$. Here, $\mathfrak{B}(0)\oplus A$ is equipped with the obvious $\mathfrak{B}$-algebra structure. The inverse functor sends $(B\rightarrow\mathfrak{B}(0))\mapsto ker(B\rightarrow\mathfrak{B}(0))$. Under the equivalence 
$$U:\mathpzc{Alg}_{\mathfrak{P}}\cong(\mathpzc{Alg}_{\mathfrak{B}})_{\slash\mathfrak{B}(0)},$$
we have $U(tr_{\mathfrak{P}}(-))\cong\mathfrak{B}(0)\ltimes(-)$.  The functor $\mathbb{L}D$ is in fact the cotangent complex of $U(A)$ `at the point $U(A)\rightarrow\mathfrak{B}(0)$', in the sense of \cite{loday2012algebraic} Chapter 12.3.10. 
\end{rem}

Let $\alpha:\mathfrak{C}\rightarrow\mathfrak{P}$ be a Koszul twisting morphism. Let $A$ be a cofibrant $\mathfrak{P}$-algebra of the form $A\cong\Omega_{\alpha}C$, and let $V$ be a fibrant object of $\mathpzc{M}$. Letting $|-|:\mathpzc{coAlg}_{\mathfrak{C}_{top}}^{conil}\rightarrow\mathpzc{M}$ denote the forgetful functor, we have 

$$\mathrm{Hom}(\Omega_{\alpha}C,tr_{\mathfrak{P}}(V))\cong\mathrm{Hom}(C,\mathfrak{C}(V))\cong\mathrm{Hom}(|C|,V).$$
Hence, $D(A)\cong |C|$. Since $A$ is cofibrant, we have $\mathbb{L}D(A)\cong D(A)\cong |C|$.
Note that there is an $\alpha$-equivalence $C\cong B_{\alpha}(A)$.
Thus, we may interpret $|B_{\alpha}(A)|$ as the cotangent complex of $A$. Put another way, the cotangent complex of a $\mathfrak{P}$-algebra has a $\mathfrak{C}$-coalgebra structure, determined up to $\alpha$-weak equivalences. In the case where $\mathpzc{M}$ is $C$-monoidal strong Koszul, we can make this precise.

\begin{prop}\label{prop:compBalphacot}
    Suppose that $\mathpzc{M}$ is a $C$-monoidal, strong Koszul category, and that $gr(\mathfrak{C})$ is a retract of a free $\Sigma$-module on a graded cofibrant object. Let $A$ be a cofibrant $\mathfrak{P}$-algebra. Suppose further that $\mathfrak{C}$ is filtered cofibrant as a $\Sigma$-module. Then there is a natural equivalence $\mathbb{L}_{0}^{\mathfrak{P}}(A)\cong |B_{\alpha}(A)|$.
\end{prop}

\begin{proof}
    In this case, both $A$ and $\Omega_{\alpha}B_{\alpha}(A)$ are cofibrant by Theorem \ref{preQuillen} (3) and (4). Moreover, the morphism $\Omega_{\alpha}B_{\alpha}(A)\rightarrow A$ is a weak equivalence between cofibrant objects. Thus,
    $$\mathbb{L}_{0}^{\mathfrak{P}}(A)\cong\mathbb{L}_{0}^{\mathfrak{P}}(\Omega_{\alpha}B_{\alpha}(A))\cong |B_{\alpha}(A)|.$$
\end{proof}

%begin{rem}
%We have the unit $A\rightarrow tr_{\mathfrak{P}}(\mathbb{L}_{0}^{\mathfrak{P}}(A))$ of the adjunction 
 %   $$\adj{\mathbb{L}^{\mathfrak{P}}_{0}}{\textbf{Alg}_{\mathfrak{P}}}{\mathrm{L^{H}}(\mathpzc{M})}{tr_{\mathfrak{P}}}.$$
%Under the forgetful functor and the isomorphism $\mathbb{L}_{0}^{\mathfrak{P}}(A)\cong |B_{\alpha}(A)|$, these morphisms coincide
%\end{rem}

\begin{example}\label{ex:cotangentusual}

    Consider the pure Koszul twisting morphism
    $$\kappa^{\vee}:\mathfrak{S}^{c}\otimes_{H}(\mathfrak{Lie}^{\vee})\rightarrow\mathfrak{Comm}^{nu}.$$
    Let $A\in\mathpzc{Alg}_{\mathfrak{Comm}^{nu}}$. We regard this as an augmented algebra, under the equivalence $\mathpzc{Alg}_{ \mathfrak{Comm}^{nu}}\cong\mathpzc{Alg}_{\mathfrak{Comm}}^{aug}$. Then, as underlying objects of $\mathrm{L^{H}}(\mathpzc{M})$, we have $\mathbb{L}_{0}^{\mathfrak{Comm}^{nu}}(A)\cong\mathbb{L}_{0}(A)$, where $\mathbb{L}_{0}(A)$ is as defined in Subsection \ref{subsec:commalg}. Therefore, 
     under the equivalence
    $$[1]:\mathpzc{coAlg}^{conil}_{\mathfrak{S}^{c}\otimes_{H} \mathfrak{Lie}^{\vee}}\cong \mathpzc{coAlg}^{conil}_{\mathfrak{Lie}^{\vee}},$$
    the `shifted cotangent complex', $\mathbb{L}_{0}(A)[1]$, may be equipped with a coLie coalgebra structure.
    %Under the equivalence
    %$$[1]:\mathpzc{coAlg}^{conil}_{\mathfrak{S}^{c}\otimes_{H}(R\otimes \mathfrak{Lie}^{\vee})}\cong \mathpzc{coAlg}^{conil}_{R\otimes\mathfrak{Lie}^{\vee}},$$
    %and using Proposition \ref{prop:L0shift}, we get that $\mathbb{L}_{0}^{R\otimes \mathfrak{Comm}^{nu}}(A)[1]\cong\mathbb{L}_{0}(A)$, 
\end{example}

Finally, the Quillen adjunction
$$\adj{D}{\mathpzc{Alg}_{\mathfrak{P}}}{\mathpzc{M}}{tr_{\mathfrak{P}}}$$
gives rise to an adjunction of $(\infty,1)$-categories
$$\adj{\mathbf{D}}{\textbf{Alg}_{\mathfrak{P}}}{\mathrm{L^{H}}(\mathpzc{M})}{\textbf{tr}_{\mathfrak{P}}}.$$
Let us now sketch the relationship between Koszul duality in the sense of the present work and derived Koszul duality in the sense of \cite{francis2012chiral} and \cite{heuts2024koszul}.
In particular, we get a comonad $\mathbf{D}\circ\textbf{tr}_{\mathfrak{P}}$ on the category $\mathrm{L^{H}}(\mathpzc{M})$. For simplicity, assume for the remainder of this paragraph that $\mathpzc{M}$ is $C$-monoidal, and that for each $n\in\mathbb{N}_{0}$, $\mathfrak{C}(n)$ and $\mathfrak{P}(n)$ are retracts of free right $\Sigma_{n}$-modules on cofibrant objects. Let $V$ be an object of $\mathpzc{M}$ that is both fibrant and cofibrant. We have 
$$\mathbf{D}\circ\textbf{tr}_{\mathfrak{P}}(V)\cong\mathbf{D}(\Omega_{\alpha}(\mathfrak{C}(V)))\cong D(\Omega_{\alpha}(\mathfrak{C}(V)))\cong\mathfrak{C}(V).$$
Suppose further that, in the model structure on $\mathpzc{M}$, every object is fibrant (for example, if $\mathpzc{M}\cong{}_{R}\mathpzc{Mod}(Ch(\mathpzc{E}))$, and the model structure is transferred from a projective one on $Ch(\mathpzc{E})$). Then $\mathfrak{C}(V)$ is still fibrant and cofibrant. We then have $\mathbf{D}\circ\textbf{tr}_{\mathfrak{P}}\circ \mathbf{D}\circ\textbf{tr}_{\mathfrak{P}}(V)\cong(\mathfrak{C}\circ\mathfrak{C})(V)$. For a general $\mathfrak{P}$-algebra $A$, the unit $\eta$ of the $D$-$tr_{\mathfrak{P}}$ adjunction is given by the projection $\pi_{A}:A\rightarrow tr_{\mathfrak{P}}(coker(\mu_{A}))$. When $A\cong\Omega_{\alpha}(C)$, this is the morphism adjoint to the comultiplication of $C$, $C\rightarrow \mathfrak{C}(|C|)\cong B_{\alpha}(tr_{\mathfrak{P}}|C|)$. Suppose that $A\cong tr_{\mathfrak{P}}(V)$, with $V$ fibrant and cofibrant. Then $A\cong\Omega_{\alpha}(\mathfrak{C}(V))$. The morphism $\mathfrak{C}(V)\rightarrow\mathfrak{C}\circ\mathfrak{C}(V)$ adjoint to the unit $\Omega_{\alpha}(\mathfrak{C}(V))\rightarrow\mathfrak{C}(V)$, that furnishes the comonad structure, is just $\Delta_{\mathfrak{C}}\circ Id_{V}$. Therefore, in general, we may think of the comonad $\mathbf{D}\circ\textbf{tr}_{\mathfrak{P}}$ as an $(\infty,1)$-categorical enhancement of the (comonad associated to the) cooperad $\mathfrak{C}$. In derived Koszul duality (\cite{francis2012chiral}, \cite{heuts2024koszul}), the comonad $\mathbf{D}\circ\textbf{tr}_{\mathfrak{P}}$ is, essentially by definition, the Koszul dual cooperad of $\mathfrak{P}$. The $(\infty,1)$-category of coalgebras over this comonad is in general not presented by a model category of coalgebras over a cooperad.

%is a weak equivalence since $(\mathfrak{P}\circ_{\alpha}\mathfrak{C})\cong R$ and the underlying object of $A$ is $K$-flat.

%The counit then gives a morphism of filtered complexes. We consider the spectral sequence associated to these filtered complexes. For a filtered complex $C_{\bullet}$ we denote by $E^{i}(C_{\bullet})$  the direct sum of the complexes on the $i$th page of the spectral sequence associated to $C_{\bullet}$. Then $E^{0}(\Omega_{\alpha}B_{\alpha}A)\cong(\mathfrak{P}\circ_{\alpha}\mathfrak{C})\circ A$, while $E^{0}(F_{0}(A))\cong A\cong I\circ A$. Since $\alpha$ is a Koszul morphism, the morphism $j:I\rightarrow\mathfrak{P}\circ_{\alpha}\mathfrak{C}$ is an acyclic cofibration. Therefore, the morphism $j_{A}:A\rightarrow(\mathfrak{P}\circ_{\alpha}\mathfrak{C})\circ A$ is a weak equivalence by Proposition \ref{coalgmod}. Moreover, $E^{0}(\epsilon_{\alpha}(A))\circ j_{A}=Id_{A} $. Therefore, $E^{0}(\epsilon_{\alpha}(A))$ is also a quasisomorphism. We now apply the convergence theorem for spectral sequences of filtered complexes which can be bootstrapped from the abelian version by passing to an abelianization.

\subsection{Koszul Duality in Strong Koszul Categories}

In this subsection we are going to analyse Koszul duality in the context of $C$-monoidal strong Koszul categories. By Theorem \ref{coopKoszuldual} we have an abstract equivalence of $(\infty,1)$-categories. However, in order to be able to do computations it would be convenient to have some sort of factorisation property for morphisms, or even a model structure. The proof of the existence of the model structure on coalgebras below follows \cite{grignou2023homotopicaloperadiccalculuspositive} Theorem 6 closely, where a model structure is constructed on curved coalgebras over a quasi-planar conilpotent curved cooperad over a field of arbitrary characteristic. In particular, the authors of that paper use \cite{hess2017necessary} Theorem 2.2.1, which we will also use. The crucial ingredient is constructing a factorisation of the fold map 
$$\nabla:C\oplus C\rightarrow C$$
of a coalgebra $C$.
%The additional complication in  \cite{grignou2023homotopicaloperadiccalculuspositive} arises from the fact that, over a non-reduced cooperad, th 
Let $\mathpzc{M}=Ch(\mathpzc{E})$ be $C$-monoidal strong Koszul, and let $\alpha:\mathfrak{C}\rightarrow\mathfrak{P}$ be a pure Koszul morphism. Let $C$ be a filtered cofibrant $\mathfrak{C}$-coalgebra. Consider the filtered algebra $\overline{\Omega}_{\alpha}C$ and the fold map $\nabla:C\oplus C\rightarrow C$. By Proposition \ref{prop:hurewiczmodel}, we have a factorisation in $\mathpzc{Alg}_{\mathfrak{P}}(Ch(\mathpzc{Fun}(\mathbb{N}_{\ge0},\mathpzc{E})))$, 
\begin{displaymath}
    \xymatrix{
   \overline{\Omega}_{\alpha}(C\oplus C)\ar[r]^{i} & A\ar[r]^{p} &\overline{\Omega}_{\alpha}(C),
    }
\end{displaymath}
where $i$ is a morphism that is split as an underlying morphism in $Ch(\mathpzc{Fun}(\mathbb{N}_{\ge0},\mathpzc{M}))$, and $p$ is a morphism that is a degree-wise split surjective homotopy equivalence as a morphism in $\mathpzc{Fun}(\mathbb{N}_{\ge0},\mathpzc{M})$. Let $K$ denote the kernel of $A\rightarrow\overline{\Omega}_{\alpha}(C)$. Since $A\rightarrow\overline{\Omega}_{\alpha}(C)$ is a surjective homotopy equivalence, so is $K\rightarrow 0$. After forgetting the differentials, we may write $A\cong\overline{\Omega}_{\alpha}(C)\oplus K$. Consider the following commutative diagram, where the right-hand square is defined as a pullback in $\mathpzc{coAlg}^{conil}_{\mathfrak{C}}(\mathpzc{Fun}(\mathbb{N}_{\ge0},\mathpzc{M}))$:
\begin{displaymath}
\xymatrix{
C\oplus C\ar[d]^{\nu_{\alpha}(C\oplus C)}\ar[r]^{i^{\dagger}\times\nabla} &\mathrm{Cyl}(C)\ar[d]^{j}\ar[r]^{\mathrm{proj}_{C}} & C\ar[d]^{\nu_{\alpha}(C)}\\
\overline{B}_{\alpha}\overline{\Omega}_{\alpha}(C\oplus C)\ar[r]^{\overline{B}_{\alpha}(i)} & \overline{B}_{\alpha}(A)\ar[r]^{\overline{B}_{\alpha}(p)} & \overline{B}_{\alpha}\overline{\Omega}_{\alpha}C.
}
\end{displaymath}
Note that, after forgetting differentials, we have $\overline{B}_{\alpha}(A)\cong\mathfrak{C}(\overline{\Omega}_{\alpha}(C))\times\mathfrak{C}(K)$. Then, still forgetting the differentials, $\mathrm{Cyl}(C)\cong C\times\mathfrak{C}(K)$. The morphism $\nu_{\alpha}(C\oplus C)$ is a split monomorphism of filtered objects by Proposition \ref{prop:(co)unitadmissible}. Since $i$ is a split monomorphism, the underlying morphism of $\overline{B}_{\alpha}(i)$ is split. By the Obscure Axiom, $i^{\dagger}\times\nabla$ is a split monomorphism of underlying objects.
We will need the following assumption. 
\begin{ass}\label{times12}
For any coalgebra $C$ whose underlying object is filtered cofibrant, the object $\mathrm{Cyl}(C)$ lives in $\overline{\mathpzc{Filt}}_{\textbf{PureMon}}(\mathpzc{M})$, and the morphism $\mathrm{proj}_{C}$ is an admissible epimorphism that is an $\alpha$-weak equivalence.
\end{ass}
\begin{rem}\label{rem:qacsimp12}
    In a quasi-abelian category in which every object is flat, the condition that $\mathrm{Cyl}(C)$ lives in $\overline{\mathpzc{Filt}}_{\textbf{PureMon}}(\mathpzc{M})$ is automatic. Indeed, as in the proof of \cite{grignou2023homotopicaloperadiccalculuspositive} Lemma 37, the underlying object of $\mathrm{Cyl}(C)$ is computed as an equaliser in $\mathpzc{Fun}(\mathbb{N}_{\ge0},\mathpzc{M})$ of two morphisms $\mathfrak{C}(X)\rightarrow\mathfrak{C}(Y)$, where $X$ is defined by the following pullback diagram in $\mathpzc{Fun}(\mathbb{N}_{\ge0},\mathpzc{M})$
    \begin{displaymath}
        \xymatrix{
        X\ar[d]\ar[r] & C\ar[d]\\
        \overline{B}_{\alpha}(A)\ar[r] & \overline{B}_{\alpha}\overline{\Omega}_{\alpha}C,
        }
    \end{displaymath}
    and $Y$ is defined by the following pullback diagram in $\mathpzc{Fun}(\mathbb{N}_{\ge0},\mathpzc{M})$
    \begin{displaymath}
        \xymatrix{
        Y\ar[d]\ar[r] & \mathfrak{C}\circ\overline{B}_{\alpha}(A)\ar[d]\\
        \mathfrak{C}\circ\overline{B}_{\alpha}(A)\ar[r] & \mathfrak{C}\circ\overline{B}_{\alpha}\overline{\Omega}_{\alpha}C.
        }
    \end{displaymath}
    However, in a quasi-abelian category, \cite{kelly2016homotopy} Proposition 5.38 implies that $\overline{\mathpzc{Filt}}_{\textbf{RegMon}}(\mathpzc{M})$ is a reflective subcategory of $\mathpzc{Fun}(\mathbb{N}_{\ge0},\mathpzc{M})$. In particular, it is closed under limits. Finally, observe that, since every object is assumed flat, $\overline{\mathpzc{Filt}}_{\textbf{RegMon}}(\mathpzc{M})=\overline{\mathpzc{Filt}}_{\textbf{PureMon}}(\mathpzc{M})$.
\end{rem}
\begin{example}
\begin{enumerate}
    \item 
    Suppose that $\mathfrak{P}$ and $\mathfrak{C}$ are both concentrated in arity $1$. Then $\mathpzc{Alg}_{\mathfrak{P}}(\mathpzc{M})\cong{}_{\mathfrak{P}(1)}\mathpzc{Mod}(\mathpzc{M})$. The product of two $\mathfrak{C}(1)$-comodules is just their direct sum. After forgetting differentials, we therefore have $\mathrm{Cyl}(C)\cong C\oplus (\mathfrak{C}(1)\otimes K)$. Since $\mathfrak{C}(1)$, $C$, and $K$ are all filtered by pure monomorphisms, so is $\mathrm{Cyl}(C)$. The morphism $\mathrm{Cyl}(C)\rightarrow C$ is the projection. The kernel is $\mathfrak{C}(1)\otimes K$. As a filtered object, this has associated graded (still after forgetting differentials) $gr(\mathfrak{C}(1))\otimes gr(K)$. Now, consider the constant filtration on $gr(K)$, the grading filtration on $gr(\mathfrak{C}(1))$, and the total filtration on $gr(\mathfrak{C}(1))\otimes gr(K)$. The associated graded of this, as a complex, is $gr(\mathfrak{C}(1))\otimes gr(K)$. Since $gr(K)$ is acyclic, so is $gr(\mathfrak{C}(1))\otimes gr(K)$. Therefore, the projection $\mathrm{proj}_{C}$ has kernel whose associated graded is acyclic, and this morphism is therefore an $\alpha$-weak equivalence.
    \item 
    Suppose that $\mathpzc{E}$ is enriched over $\mathbb{Q}$, and consider the twisting morphism $\kappa:\mathfrak{S}^{c}\otimes_{H}\mathfrak{coComm}^{ncu}\rightarrow\mathfrak{Lie}$. Observe that the product of any two (shifted) cocommutative coalgebras $C$ and $D$ is $(C\otimes D)[-1]$. From now on let us suppress the shift. The product of two non-counital cocommutative coalgebras is $(C\oplus D\oplus (C\otimes D))$. The filtration on $C\oplus S^{c}(K)\oplus (C\otimes S^{c}(K))$ is induced by retracts, sums, and tensor products of pure filtrations, and hence is itself a pure filtration. Again, the morphism $C\oplus S^{c}(K)\oplus (C\otimes S^{c}(K))\rightarrow C$ is just given by projection to the first factor $C$. We need to show that there is some filtration on  $C\oplus S^{c}(K)\oplus (C\otimes S^{c}(K))$ and on $C$ such that this morphism is a filtered quasi-isomorphism. By passing to associated gradeds, we may assume that $C$ has filtration arising from a grading. Endow $K$ with the constant filtration, and keep the original filtration on $C$. This gives a filtration on $\mathrm{Cyl}(C)$ whose associated graded is $C\oplus S^{c}(K)\oplus (C\otimes S^{c}(K))$ as a complex. The projection $C\oplus S^{c}(K)\oplus (C\otimes S^{c}(K))\rightarrow C$ has kernel given by the complex $S^{c}(K)\oplus (C\otimes S^{c}(K))$. Since $K$ is pure acyclic, this object is acyclic (as a graded object).
    \end{enumerate}
    \end{example}
We will now give a larger class of examples in quasi-abelian categories.
\begin{defn}[\cite{grignou2023homotopicaloperadiccalculuspositive}]
    A filtered cooperad $\mathfrak{C}$ is said to be \textit{quasi-planar} if
    \begin{enumerate}
        \item 
        the underlying homologically graded cooperad (i.e., after forgetting differentials) of each $F_{i}\mathfrak{C}$ is of the form $\Sigma\otimes\mathfrak{C}_{pl}^{(i)}$ for some non-symmetric cooperad $\mathfrak{C}_{pl}^{(i)}$, 
        \item 
        as a morphism of homologically graded cooperads, each morphism $F_{i}\mathfrak{C}\rightarrow F_{i+1}\mathfrak{C}$ is of the form $Id_{\Sigma}\otimes f_{i}$ for a morphism of cooperads $f_{i}:\mathfrak{C}_{pl}^{(i)}\rightarrow\mathfrak{C}_{pl}^{(i+1)}$ that is degree-wise split injective, and
        \item
        the differential on $gr_{i+1}\mathfrak{C}\cong\Sigma\otimes(\mathfrak{C}^{(i+1)}_{pl}\big\slash\mathfrak{C}_{pl}^{(i)})$ has the form $Id_{\Sigma}\otimes d_{pl}$. 
      %  \item 
      %  for any filtered $\mathfrak{C}_{pl}$-coalgebra $C$ and any filtered object $M$ of $\mathpzc{M}$, the natural map 
      %  $$\mathfrak{C}_{pl}(M)\times C\rightarrow\mathfrak{C}_{pl}(M\times C)$$
      %  is an admissible monomorphism of filtered objects.
    \end{enumerate}
\end{defn}
%For example, any filtered cooperad concentrated in arity $1$ is quasi-planar. The morphism $\mathfrak{C}_{pl}(M)\times C\rightarrow\mathfrak{C}_{pl}(M\times C)$ is always an equaliser, so in a quasi-abelian category it is automatically an admissible monomorphism.

\begin{thm}
    Suppose that $\mathpzc{E}$ is a quasi-abelian category in which every object is flat, that filtered colimits commute with finite limits, and that $\mathfrak{C}$ is quasi-planar. Then Assumption \ref{times12} holds.
\end{thm}
\begin{proof}
By Remark \ref{rem:qacsimp12}, what remains to prove is that $\mathrm{proj}_{C}:\mathrm{Cyl}(C)\rightarrow C$ is an $\alpha$-weak equivalence.
We explain why the proof of \cite{grignou2023homotopicaloperadiccalculuspositive} Theorem 6 works in this generality. 

Following their proof, first observe that $\overline{B}_{\alpha}(p)$ is an $\alpha$-weak equivalence, since $\alpha$ is pure Koszul. The extension of the contracting homotopy on $K$ to $\Omega_{\alpha}C$ and then the construction of $H$ on $B_{\alpha}(A)$ work as in the proof of \cite{grignou2023homotopicaloperadiccalculuspositive} Theorem 6. 

Lemma 37 of \cite{grignou2023homotopicaloperadiccalculuspositive} also generalises. Since we are not considering curvatures, the proof is in fact simpler. Indeed, consider the filtration on $\mathrm{Cyl}(C)$ given by the pullback diagram in $Ch(\mathpzc{E})$
\begin{displaymath}
    \xymatrix{
    F_{i}\mathrm{Cyl}(C)\ar[d]\ar[r] &\mathrm{Cyl}(C)\ar[d]^{q}\\
    F_{i}B^{filt}_{\alpha}A\ar[r] & B^{filt}_{\alpha}A.
    }
\end{displaymath}
Note that, since $\mathpzc{E}$ is quasi-abelian and all objects are flat, this does indeed give a filtration by pure monomorphisms. Then Lemma 38 works identically, as do Lemma 39 and Proposition 38. 

As in Proposition 39 of \cite{grignou2023homotopicaloperadiccalculuspositive}, we find that $\mathrm{proj}_{C}$ is a filtered quasi-isomorphism for the given filtration on $C$, and the filtration on $\mathrm{Cyl}(C)$ that we have induced above.
 % Now, suppose that $\mathfrak{C}$ is virtually quasi-planar, so that we have a filtered acyclic cofibration $\psi:\mathfrak{C}\rightarrow\mathfrak{C}'$, for $\mathfrak{C}'$ quasi-planar. Let $C$ be a filtered $\mathfrak{C}$-coalgebra that we may regard as a $\mathfrak{C}'$-coalgebra. Then we can factor the fold map of $C$ as 
 % $$C\oplus C\rightarrow\mathrm{Cyl}'(C)\rightarrow C$$
 % in the category of filtered $\mathfrak{C}'$-coalgebras. Now, since $\mathfrak{C}\rightarrow\mathfrak{C}'$ is an acyclic cofibration, it is in fact a monomorphism. Therefore 
%$\psi_{*}(C)\cong C$ and $\psi_{*}(C\oplus C)\cong C\oplus C$. 
\end{proof}
\begin{thm}\label{thm:quillenequiv}
Throughout, let $\alpha:\mathfrak{C}\rightarrow\mathfrak{P}$ be a pure Koszul twisting morphism with $\mathfrak{C}$ satisfying Assumption \ref{times12}.
\begin{enumerate}
\item
Let $\mathpzc{M}$ be a strong $C$-monoidal Koszul category in which every object is cofibrant. Then, with the cofibrations being the $\alpha$-cofibrations (Definition \ref{alphamodel}), and the weak equivalences being the $\alpha$-weak equivalences, there exists the structure of a category of cofibrant objects on $\mathpzc{coAlg}^{conil}_{\mathfrak{C}}$. The bar-cobar adjunction restricts to a relative equivalence between the categories of cofibrant objects
$$\adj{\Omega_{\alpha}^{filt}}{\mathpzc{coAlg}_{\mathfrak{C}}^{conil}(\mathpzc{M})}{\mathpzc{Alg}_{\mathfrak{P}}^{c}(\mathpzc{M})}{B^{filt}_{\alpha}},$$
and the functor $\Omega_{\alpha}^{filt}$ sends $\alpha$-cofibrations to cofibrations.
    \item 
        Suppose that $\mathpzc{E}$ is a locally presentable quasi-abelian category in which filtered colimits commute with finite limits, and all objects are flat. Let $\mathpzc{M}=Ch(\mathpzc{E})$ be the Koszul category obtained by equipping $Ch(\mathpzc{E})$ with the flat model structure. Then the left-transferred model structure along the adjunction
    $$\adj{\Omega^{filt}_{\alpha}}{\mathpzc{coAlg}_{\mathfrak{C}}^{conil}}{\mathpzc{Alg}_{\mathfrak{P}}}{B^{filt}_{\alpha}}$$
    exists on $\mathpzc{coAlg}_{\mathfrak{C}}^{conil}$.  With respect to this model structure, the adjunction is a Quillen equivalence. 
    \item 
        Suppose that $\mathpzc{E}$ is a locally presentable quasi-abelian category with exact filtered colimits in which filtered colimits commute with finite limits, and all objects are flat. Let $\mathpzc{M}=Ch(\mathpzc{E})$ be the Koszul category obtained by equipping $Ch(\mathpzc{E})$ with the flat model structure. Then the left-transferred model structure along the adjunction
    $$\adj{\Omega_{\alpha}}{\mathpzc{coAlg}_{\mathfrak{C}_{top}}^{conil}}{\mathpzc{Alg}_{\mathfrak{P}}}{B_{\alpha}}$$
    exists on $\mathpzc{coAlg}_{\mathfrak{C}_{top}}^{conil,pure}=\mathpzc{coAlg}_{\mathfrak{C}_{top}}^{conil}$. With respect to this model structure, the adjunction is a Quillen equivalence. If every object of $\mathpzc{M}$ is fibrant, then the cofibrations in the transferred model structure are precisely those morphisms that are degree-wise split as morphisms of complexes.
\end{enumerate}
\end{thm}
\begin{proof}
\begin{enumerate}
\item 
First, the category $\mathpzc{coAlg}_{\mathfrak{C}}^{conil}$ admits finite coproducts. Indeed, a direct sum of pure monomorphisms is pure, so that $\overline{\mathpzc{Filt}}_{\textbf{PureMon}}(\mathpzc{M})$ admits finite coproducts. The direct sum in $\overline{\mathpzc{Filt}}_{\textbf{PureMon}}(\mathpzc{M})$ of two objects of $\mathpzc{coAlg}_{\mathfrak{C}}^{conil}$ may be equipped with the structure of an object of $\mathpzc{coAlg}_{\mathfrak{C}}^{conil}$, rendering it a coproduct there.
 Note that, under our assumptions, every object of $\mathpzc{M}$ is cofibrant. Thus, an $\alpha$-cofibration is just a morphism that is a split monomorphism after forgetting the $\mathfrak{C}$-coalgebra structures and the differentials. Pushouts exist in the category of functors $\mathpzc{Fun}(\mathbb{N}_{0}^{\le},\mathpzc{M})$, and the pushout in this category of a morphism that is split after forgetting the differential is still split (after forgetting the differential). Since the forgetful functor commutes with pushouts in $\mathpzc{coAlg}_{\mathfrak{C}}^{conil}$, it is then clear that the class of $\alpha$-cofibrations is stable under pushout. The factorisation requirement is satisfied by assumption. By Theorem \ref{preQuillen} (4), the functor $\Omega_{\alpha}^{filt}$ factors as a functor $\mathpzc{coAlg}_{\mathfrak{C}}^{conil}(\mathpzc{M})\rightarrow\mathpzc{Alg}^{c}_{\mathfrak{P}}(\mathpzc{M})$ through the inclusion of cofibrant algebras, and this functor preserves cofibrations. 
    \item 
        For claims (2) and (3), we use \cite{hess2017necessary} Theorem 2.2.1. 
First, let us explain why the category $\mathpzc{coAlg}_{\mathfrak{C}}^{conil}$ is locally presentable. Now, as a functor category with locally presentable target, $\mathpzc{Fun}(\mathbb{N}_{0}^{\le},\mathpzc{M})$ is locally presentable. Since every object of $\mathpzc{E}$ is flat and $\mathpzc{E}$ is quasi-abelian, it follows from \cite{kelly2016homotopy} Proposition 5.15 that $\overline{\mathpzc{Filt}}_{\textbf{PureMon}}(\mathpzc{M})$ is a reflective subcategory of $\mathpzc{Fun}(\mathbb{N}_{0}^{\le},\mathpzc{M})$. Since filtered colimits commute with finite limits, and hence are exact in $\mathpzc{E}$, the inclusion $\overline{\mathpzc{Filt}}_{\textbf{PureMon}}(\mathpzc{M})\rightarrow\mathpzc{Fun}(\mathbb{N}_{0}^{\le},\mathpzc{M})$ commutes with filtered colimits. Therefore, as an accessible, reflective subcategory of a locally presentable category, $\overline{\mathpzc{Filt}}_{\textbf{PureMon}}(\mathpzc{M})$ is locally presentable. Moreover, the functor $\mathfrak{C}\circ(-)$ commutes with filtered colimits. Consequently, by \cite{ching2014coalgebraic} Proposition A.1, the category $\mathpzc{coAlg}_{\mathfrak{C}}^{conil}$ of coalgebras over the comonad $\mathfrak{C}\circ (-)$ on $\overline{\mathpzc{Filt}}_{\textbf{PureMon}}(\mathpzc{M})$ is locally presentable.
Now, by Theorem \ref{preQuillen} (4), we may take the functor $Q$ in \cite{hess2017necessary} Theorem 2.2.1 (1), and the natural transformation $\epsilon$ in \cite{hess2017necessary} Theorem 2.2.1 (2) to both be the identity. The factorisation axiom \cite{hess2017necessary} Theorem 2.2.1 (3) is satisfied by assumption.
Hence, the model structure exists and the adjunction is tautologically Quillen; it is a Quillen equivalence by Theorem \ref{coopKoszuldual}.
\item 
This is similar to Part (2). Local presentability of $\mathpzc{coAlg}_{\mathfrak{C}_{top}}^{conil}$ also follows from \cite{ching2014coalgebraic} Proposition A.1, applied to the comonad $\mathfrak{C}_{top}\circ(-)$ on $\mathpzc{M}$.
Consider the functor $Corad:\mathpzc{coAlg}^{conil}_{\mathfrak{C}_{top}}\rightarrow\mathpzc{coAlg}_{\mathfrak{C}}$ from Remark \ref{rem:inducedfiltrationcor}. Since we have $\Omega_{\alpha}^{filt}\circ Corad=\Omega_{\alpha}$, it then follows that $\Omega_{\alpha}(C)$ is cofibrant as an algebra for any coalgebra $C$. Again we may take the functor $Q$ in \cite{hess2017necessary} Theorem 2.2.1 (1), and the natural transformation $\epsilon$ in \cite{hess2017necessary} Theorem 2.2.1 (2) to both be the identity. Applying the fold factorisation in Part (2) to $Corad(C)$, and then forgetting the filtrations, also gives the fold factorisation for Part (3). 
Hence, the model structure exists and the adjunction is tautologically Quillen; it is a Quillen equivalence by Theorem \ref{coopKoszuldual}.
Suppose now that every object of $\mathpzc{M}$ is fibrant. This implies that every object in $\mathpzc{E}$ is both projective and injective, and hence every strict monomorphism splits. Let $f:C\rightarrow D$ be a degree-wise split monomorphism of coalgebras. Consider $Corad(f):Corad(C)\rightarrow Corad(D)$. Since every strict monomorphism splits, the filtrations on $Corad(C)$ and on $Corad(D)$ are by degree-wise split monomorphisms. Moreover, the Obscure Axiom implies that each $F_{n}Corad(f)$ is degree-wise split. The proof of \cite{kelly2016homotopy} Lemma 3.57 then implies that $Corad(f)$ is degree-wise split. It tautologically has cofibrant cokernel, since every object is cofibrant. By Theorem \ref{preQuillen} (4), $\Omega^{filt}_{\alpha}(Corad(f))=\Omega_{\alpha}(f)$ is a cofibration. On the other hand, let $f:C\rightarrow D$ be such that $\Omega_{\alpha}(f)$ is a cofibration of algebras. It is then a cofibration on underlying objects, and hence is degree-wise split. Forgetting differentials and coalgebra structures, $f$ is a retract of $\Omega_{\alpha}(f)$. The morphism $f$ is therefore also degree-wise split.
%We know that $\alpha$-cofibrations are cofibrations in the transferred model structure by Theorem \ref{preQuillen} (4). On the other hand, let $f:C\rightarrow D$ be such that $\Omega_{\alpha}(f)$ is a strict monomorphism. Forgetting differentials and coalgebra structures, $f$ is a retract of $\Omega_{\alpha}(f)$. Therefore, $f$ is a strict monomorphism. If every object of $\mathpzc{M}$ is fibrant and cofibrant, then every monomorphism splits.  
%By the proof of \cite{kelly2016homotopy} Lemma 3.57, every admissible monomorphism of filtered objects also splits. Therefore, $Corad(f)$ splits (after forgetting the $\mathfrak{C}$-coalgebra structure and differentials).
\end{enumerate}
\end{proof}
%Now, let $\mathpzc{E}$ be a locally presentable, closed symmetric monoidal exact category with exact filtered colimits. Then $Ch(\mathpzc{E})$ may be equipped with the flat model structure by \cite{kelly2024flat} Theorem 1.1. Consider the exact category $\mathpzc{E}_{\otimes}$, that has the same underlying category as $\mathpzc{E}$, but is equipped with the $\otimes$-pure exact structure. The category $Ch(\mathpzc{E}_{\otimes})$ may be equipped with the flat model structure by  \cite{kelly2024flat} Theorem 1.1 - since all objects are flat in $\mathpzc{E}_{\otimes}$, this coincides with the flat model structure.
% Let $R\in Ch(\mathpzc{E})$. We regard $\mathpzc{M}={}_{R}\mathpzc{Mod}(Ch(\mathpzc{E}))$ as a model category with the transferred model structure from the flat model structure on $Ch(\mathpzc{E})$, and $\mathpzc{M}_{\otimes}={}_{R}\mathpzc{Mod}(Ch(\mathpzc{E}_{\otimes}))$ also as a model category with the transferred model structure from the flat/ injective model structure on $Ch(\mathpzc{E}_{\otimes})$. We again get a Quillen coreflection
% $$\adj{Id}{\mathpzc{M}}{\mathpzc{M}_{\otimes}}{Id}.$$
\subsubsection{The $\otimes$-pure exact structure and mapping spaces}
Even when not all objects of a Koszul category are cofibrant, so that we do not get a model structure or the structure of a category of cofibrant objects on the coalgebra side of a Koszul morphism, we can often still describe the mapping spaces.
Suppose that the flat model structure exists on $\mathpzc{M}=Ch(\mathpzc{E})$, rendering it a $C$-monoidal strong Koszul category. Let $\alpha:\mathfrak{C}\rightarrow\mathfrak{P}$ be a pre-Koszul twisting morphism, where $\mathfrak{C}$ is a retract of a free $\Sigma$-module on a cofibrant object, and $\mathfrak{P}$ is cofibrant as a $\Sigma$-module. By Example \ref{example:moreideals},
$$(\mathpzc{coAlg}_{\mathfrak{C}}^{|c|},\mathpzc{Alg}_{\mathfrak{P}}^{|c|})$$
is a filtered $\alpha$-ideal. Recall the Quillen coreflection
$$\adj{Id}{Ch(\mathpzc{E})}{Ch(\mathpzc{E}_{\otimes})}{Id}$$
from Example \ref{ex:pureflatmodel}, where both sides are equipped with the flat model structure. Suppose now that $\mathfrak{P}$ is also admissible for the flat model structure on $\mathpzc{M}_{\otimes}=Ch(\mathpzc{E}_{\otimes})$. We then also get a Quillen coreflection
 $$\adj{Id}{\mathpzc{Alg}_{\mathfrak{P}}(\mathpzc{M})}{\mathpzc{Alg}_{\mathfrak{P}}(\mathpzc{M}_{\otimes})}{Id}.$$
In particular, we get a fully faithful functor of $(\infty,1)$-categories
$$\mathrm{L^{H}}(\mathpzc{Alg}_{\mathfrak{P}}(\mathpzc{M}))\rightarrow\mathrm{L^{H}}(\mathpzc{Alg}_{\mathfrak{P}}(\mathpzc{M}_{\otimes})).$$
Suppose now that $\alpha$ is filtered pure Koszul with respect to the flat model structure on $\mathpzc{M}_{\otimes}$. Then it is also filtered pure Koszul for the flat model structure on $\mathpzc{M}$.
We then have a commutative diagram of relative functors between relative categories
\begin{displaymath}
    \xymatrix{
    \mathpzc{coAlg}_{\mathfrak{C}}^{|c|}(\mathpzc{M})\ar[d]\ar[r]^{\Omega^{filt}_{\alpha}} & \mathpzc{Alg}_{\mathfrak{P}}(\mathpzc{M})\ar[d]\\
    \mathpzc{coAlg}^{conil}_{\mathfrak{C}}(\mathpzc{M}_{\otimes})\ar[r]^{\Omega_{\alpha}^{filt}} & \mathpzc{Alg}_{\mathfrak{P}}(\mathpzc{M}_{\otimes}).
    }
\end{displaymath}
The top and bottom functors are relative equivalences, and the right-hand vertical functor induces a fully faithful functor of $(\infty,1)$-categories upon hammock localisation. Thus, the left-hand functor is fully faithful on mapping spaces, and this allows us to describe the mapping spaces in $\mathrm{L^{H}}(\mathpzc{coAlg}_{\mathfrak{C}}^{|c|}(\mathpzc{M}))$. Using the description of mapping spaces in categories of cofibrant objects, this brings us to the following result. 
\begin{thm}
    Suppose that $\mathfrak{C}$ satisfies Assumption \ref{times12} in $\mathpzc{M}_{\otimes}$. 
    Under the conditions set out in this subsection, for $C,D\in\mathpzc{coAlg}_{\mathfrak{C}}^{|c|}(\mathpzc{M})$, the mapping space
    $$\textbf{Map}_{\mathrm{L^{H}}}(C,D)$$
    is equivalent to the nerve of the category whose objects are cospans
    \begin{displaymath}
        \xymatrix{
        C\ar[r]^{\simeq} & X & D\ar[l]
        }
    \end{displaymath}
    with $X$ being an object of $\mathpzc{coAlg}^{conil}_{\mathfrak{C}}(\mathpzc{M}_{\otimes})$, and the left-hand morphism being a weak equivalence in $\mathpzc{coAlg}_{\mathfrak{C}}^{conil}(\mathpzc{M}_{\otimes})$. The morphisms are commutative diagrams of cospans.  
\end{thm}
%\begin{cor}
 %    
  %  Let $\alpha:\mathfrak{C}\rightarrow\mathfrak{P}$ be Koszul. Suppose that Assumption \ref{times12} holds. Then, the category $(\mathpzc{coAlg}_{\mathfrak{C}}^{|c|})_{top}$ may be given the structure of a category of cofibrant objects 
%\end{cor}
\subsection{Connective Koszul Duality}
Classically, Koszul duality was formulated for bounded complexes of $\Q$-vector spaces in \cite{quillen1969rational}. In this section, we consider bounded Koszul duality in the more general setting of connective Koszul categories. Let $\mathpzc{M}$ be a Koszul category equipped with a complete strict $t$-structure $(\mathpzc{M}_{\ge n})_{n\in\mathbb{Z}}$ that is compatible with the monoidal structure. Let $\mathfrak{P}$ be an operad, and $\mathfrak{C}$ a filtered cooperad such that, for all $n$, $\mathfrak{P}(n)$ and $\mathfrak{C}(n)$ are in $\mathpzc{M}_{\ge0}$. 
%\subsection{Koszul Duality in Connective Koszul Categories}
\begin{defn}
\begin{enumerate}
\item
Let $\mathpzc{M}$ be a stable combinatorial model category equipped with a symmetric monoidal structure and a strict $t$-structure $(\mathpzc{M}_{\ge n})_{n\in\mathbb{Z}}$ that is compatible with the monoidal structure, and let $\mathfrak{P}$ be a symmetric sequence in $\mathpzc{M}$. A \textbf{connectivity constraint for }$\mathfrak{P}$ is an integer $m\ge0$ such that $\mathfrak{P}\circ(-):\mathpzc{M}\rightarrow\mathpzc{M}$ restricts to a functor 
$$\mathfrak{P}\circ(-):\mathpzc{M}_{\ge m}\rightarrow\mathpzc{M}_{\ge m}.$$
\item
Let $\alpha:\mathfrak{C}\rightarrow\mathfrak{P}$ be a twisting morphism. A \textbf{connectivity constraint for }$\alpha$ is an integer $m\ge0$ that is a connectivity constraint for both $\mathfrak{C}$ and $\mathfrak{P}$. 
\end{enumerate}
\end{defn}
Let $\alpha:\mathfrak{C}\rightarrow\mathfrak{P}$ be a Koszul morphism with connectivity constraint $m\ge0$. 
\begin{defn}

Let $m\ge 0$.
\begin{enumerate}
\item
An object $A$ of $\mathpzc{Alg}_{\mathfrak{P}}$ is said to be $m$-\textbf{connected} if the underlying object of $A$ is in $\mathpzc{M}_{\ge m}$. The full subcategory of $\mathpzc{Alg}_{\mathfrak{P}}$ consisting of $m$-connected  $\mathfrak{P}$-algebras is denoted $\mathpzc{Alg}_{\mathfrak{P},\ge m}$.
\item
An object $C$ of $\mathpzc{coAlg}^{conil}_{\mathfrak{C}_{top}}$ is said to be $m$-\textbf{connected} if the underlying object of $C$ is in $\mathpzc{M}_{\ge m}$. The full subcategory of $\mathpzc{coAlg}^{conil}_{\mathfrak{C}_{top}}$ consisting of $m$-connected conilpotent $\mathfrak{C}_{top}$-coalgebras is denoted $\mathpzc{coAlg}^{conil}_{\mathfrak{C}_{top},\ge m}$.
\item
An object $C$ of $\mathpzc{coAlg}^{conil}_{\mathfrak{C}}$ is said to be $m$-\textbf{connected} if the underlying filtered object of $C$ 
is such that each $F_{n}C$ is in $\mathpzc{M}_{\ge m}$. The full subcategory of $\mathpzc{coAlg}^{conil}_{\mathfrak{C}}$ consisting of $m$-connected conilpotent $\mathfrak{C}$-coalgebras is denoted $\mathpzc{coAlg}^{conil}_{\mathfrak{C},\ge m}$.
\end{enumerate}
\end{defn}

There is a model structure for $m$-connected $\mathfrak{P}$-algebras. Indeed, if $m$ is a connectivity constraint for $\mathfrak{P}$, then we have an adjunction
$$\adj{\mathfrak{P}(-)}{\mathpzc{M}_{\ge m}}{\mathpzc{Alg}_{\mathfrak{P},\ge m}}{|-|}.$$
%The following is then immediate from \cite{kelly2016homotopy} Proposition 6.3.19.

\begin{lem}

Suppose that $\mathfrak{P}$ is admissible as an operad in $\mathpzc{M}$. Let $m$ be a connectivity constraint for $\mathfrak{P}$. Then the transferred model structure exists on $\mathpzc{Alg}_{\mathfrak{P},\ge m}$. Moreover, if $A\in\mathpzc{Alg}_{\mathfrak{P},\ge m}$ is cofibrant, then it is also cofibrant as an object of $\mathpzc{Alg}_{\mathfrak{P}}$.
\end{lem}

\begin{proof}
    Since $\mathpzc{M}_{\ge m}$ is a coreflective subcategory of $\mathpzc{M}$, it is closed under colimits. Thus, the small object argument procedure for factorising a morphism between objects in $\mathpzc{Alg}_{\mathfrak{P}}(\mathpzc{M}_{\ge m})$ remains in $\mathpzc{Alg}_{\mathfrak{P}}(\mathpzc{M}_{\ge m})$. The result now follows immediately.
\end{proof}

If $C$ is a connective $\mathfrak{C}$-coalgebra, i.e., $C$ is concentrated in degrees $\ge0$, we define the \textbf{connective filtration} on $\Omega_{\alpha}(C)$ as follows. We set 
$$F_{p}\Omega_{\alpha}C=\bigoplus_{n\ge -p}\mathfrak{P}(n)\otimes_{\Sigma_{n}}|C|^{\otimes n}_{filt}.$$

\begin{prop}\label{qisconn}
Let $m\ge0$ be a connectivity constraint for $\alpha$, and $(\mathpzc{I},\mathpzc{J})$ a filtered $\alpha$-ideal. Suppose that $f:C\rightarrow D$ is a morphism of $m$-connected $\mathfrak{C}$-coalgebras in $\mathpzc{I}$, that $\alpha|_{\mathfrak{C}(1)}$ vanishes, and that the underlying morphism of $f_{top}$ in $\mathpzc{M}$ is a weak equivalence. Then $\Omega^{filt}_{\alpha}(f)$ is a weak equivalence.
\end{prop}

\begin{proof}

We adapt the proof of \cite{loday2012algebraic} Lemma 11.2.6. Consider the connective filtrations on $\Omega^{filt}_{\alpha}(C)$ and $\Omega_{\alpha}^{filt}D$. As in \cite{loday2012algebraic} Lemma 11.2.6, these are bounded-below filtrations. Hence, we can test if $f$ is a weak equivalence by passing to associated graded. Again as in \cite{loday2012algebraic} Lemma 11.2.6, the associated graded is 
$$\bigoplus_{q}(Id_{\mathfrak{P}(-p)}\otimes_{\Sigma_{-p}}f^{-p})_{p+q}:\bigoplus_{q}(\mathfrak{P}(-p)\otimes_{\Sigma_{-p}}C^{\otimes -p})_{p+q}\rightarrow\bigoplus_{q}(\mathfrak{P}(-p)\otimes_{\Sigma_{-p}}D^{\otimes -p})_{p+q}.$$
The transversality condition for a filtered $\alpha$-ideal implies that this morphism is a weak equivalence.

%This is a bounded-below filtration. There are spectral sequences whose $E_{1}$-pages consist of homology groups of the associated gradeds, and which converge to the homology groups of $\Omega_{\alpha}(C)$ and $\Omega_{\alpha}(D)$ respectively. Since $ gr(\Omega_{\alpha}f)=\mathfrak{P}(f)$ is a weak equivalence, the $E_{1}$-pages of the spectral sequences are isomorphic. Thus, $f$ induces an isomorphism between the homology groups of $\Omega_{\alpha}(C)$ and $\Omega_{\alpha}(D)$. Since the $t$-structure is complete, this implies that $f$ is an equivalence. 
\end{proof}

\begin{defn}\label{defn:mideal}
Let $\alpha:\mathfrak{C}\rightarrow\mathfrak{P}$ be a twisting morphism in $\mathpzc{M}$ with connectivity constraint $m$, and with $\mathfrak{P}$ an admissible operad.

\begin{enumerate}
    \item 
An $m$-\textbf{connected }$\alpha$-\textbf{ideal of }$\mathpzc{M}$ is a pair $(\mathpzc{I}_{\ge m},\mathpzc{J}_{\ge m})$ where 

\begin{enumerate}
%\item
%if $f:V_{1}\rightarrow V_{2}\in\mathpzc{J}$ is a filtered weak equivalence between objects in $\mathpzc{I}$ then $ gr(f)^{\otimes n}$ is a weak equivalence for each $n$;
\item 
$\mathpzc{I}_{\ge m}$ is a full subcategory of $\mathpzc{coAlg}^{conil}_{\mathfrak{C}_{top},\ge m}$,
\item 
$\mathpzc{J}_{\ge m}$ is a full subcategory of $\mathpzc{Alg}_{\mathfrak{P},\ge m}(\mathpzc{M})$, 
\item
if $V\in\mathpzc{J}_{\ge m}$ then $B_{\alpha}(V)\in\mathpzc{I}_{\ge m}$,
\item
if $C\in\mathpzc{I}_{\ge m}$ then $\Omega_{\alpha}(C)\in\mathpzc{J}_{\ge m}$,
\item
if $g:A\rightarrow B$ is a weak equivalence between objects in $\mathpzc{J}_{\ge m}$ then $\Omega_{\alpha}B_{\alpha}g$ is a weak equivalence in $\mathpzc{J}_{\ge m}$, and
\item
$\mathpzc{J}_{\ge m}$ contains all cofibrant $\mathfrak{P}$-algebras in $\mathpzc{M}_{\ge m}$.
%\item
%each $\mathfrak{C}(n)$ is $K$-transverse over $\Sigma_{n}$ to objects in $\mathpzc{M}$ of the form $V^{\otimes n}$ where $V$ is in $\mathpzc{J}$, that is if $f:V^{\otimes n}\rightarrow W^{\otimes n}$ is a weak equivalence, then so is $\mathfrak{C}(n)\otimes_{\Sigma_{n}}V^{\otimes n}\rightarrow\mathfrak{C}(n)\otimes_{\Sigma_{n}}W^{\otimes n}$;
%\item
%each $\mathfrak{P}(n)$ is $K$-transverse over $\Sigma_{n}$ to objects in $\mathpzc{M}$ of the form $(C)_{top}^{\otimes n}$, where $C\in\mathpzc{I}$.
\end{enumerate}
\item
A \textbf{filtered} $m$-\textbf{connected }$\alpha$-\textbf{ideal of }$\mathpzc{M}$ is a pair $(\mathpzc{I}_{\ge m},\mathpzc{J}_{\ge m})$ where
\begin{enumerate}
\item 
$\mathpzc{I}_{\ge m}$ is a full subcategory of $\mathpzc{coAlg}_{\mathfrak{C},\ge m}^{conil}$,
\item 
$\mathpzc{J}_{\ge m}$ is a full subcategory of $\mathpzc{Alg}_{\mathfrak{P},\ge m}$,
\item
if $V\in\mathpzc{J}_{\ge m}$ then $B^{filt}_{\alpha}(V)\in\mathpzc{I}_{\ge m}$,
\item
if $C\in\mathpzc{I}_{\ge m}$ then $\Omega^{filt}_{\alpha}(C)\in\mathpzc{J}_{\ge m}$,
\item
$gr(\mathfrak{C})$ is $K$-transverse over $\Sigma$ to objects in $\mathpzc{J}_{\ge m}$ -- that is, if $f:V\rightarrow W$ is a weak equivalence between objects of $\mathpzc{J}_{\ge m}$, then $Id_{gr(\mathfrak{C})}\circ f:gr(\mathfrak{C})\circ V\rightarrow gr(\mathfrak{C})\circ W$ is a weak equivalence, 
\item
$\mathfrak{P}$ is $K$-transverse over $\Sigma$ to objects in $\mathpzc{M}$ of the form $C_{top}$, where $C\in\mathpzc{I}_{\ge m}$ -- that is, if $g:C_{top}\rightarrow D_{top}$ is a weak equivalence, where $C,D\in\mathpzc{I}_{\ge m}$, then $\mathfrak{P}\circ C_{top}\rightarrow \mathfrak{P}\circ D_{top}$ is a weak equivalence, and 
\item 
$\mathpzc{J}_{\ge m}$ contains all cofibrant objects of $\mathpzc{Alg}_{\mathfrak{P},\ge m}(\mathpzc{M})$.
\end{enumerate}
\end{enumerate}

\end{defn}

\begin{example}
\begin{enumerate}
\item
If $\alpha$ is a homotopical twisting morphism, then there is an $m$-connected minimal $\alpha$-ideal given by
$$\mathpzc{J}=\bigcup_{n=0}^{\infty}(\Omega_{\alpha}\circ B_{\alpha})^{n}(\mathpzc{Alg}^{c}_{\mathfrak{P},\ge m}),$$
and
$$\mathpzc{I}=B_{\alpha}(\mathpzc{J}).$$
Similarly, if $\alpha$ is a filtered homotopical twisting morphism, then there is a filtered $m$-connected minimal $\alpha$-ideal given by
$$\mathpzc{J}=\bigcup_{n=0}^{\infty}(\Omega^{filt}_{\alpha}\circ B^{filt}_{\alpha})^{n}(\mathpzc{Alg}^{c}_{\mathfrak{P},\ge m}),$$
and
$$\mathpzc{I}=B^{filt}_{\alpha}(\mathpzc{J}).$$
\item
Let $\mathpzc{M}$ be a $C$-monoidal strong Koszul category and suppose that, for all $n\in\mathbb{N}_{0}$, $\mathfrak{P}(n)$ and $gr(\mathfrak{C})(n)$ are retracts of free $\Sigma_{n}$-modules on cofibrant objects. Then $(\mathpzc{coAlg}^{|c|}_{\mathfrak{C},\ge m},\mathpzc{Alg}_{\mathfrak{P},\ge m}^{|c|})$ and $(\mathpzc{coAlg}^{|c|}_{\mathfrak{C},\ge m},\mathpzc{Alg}_{\mathfrak{P},\ge m}^{c})$ are $m$-connected filtered $\alpha$-ideals. If, in addition, $\mathpzc{M}$ is $K$-monoidal, then $(\mathpzc{coAlg}^{|K|}_{\mathfrak{C},\ge m},\mathpzc{Alg}_{\mathfrak{P},\ge m}^{|K|})$ is an $m$-connected filtered $\alpha$-ideal.
\end{enumerate}
\end{example}

% and $(\mathpzc{I},\mathpzc{J})$ an $\alpha$-ideal, then we define
%$$(\mathpzc{I}_{\ge m},\mathpzc{J}_{\ge m})\defeq(\mathpzc{I}\cap\mathpzc{M}_{\ge m},\mathpzc{J}\cap\mathpzc{M}_{\ge m})$$
%Note that this satisfies all of the assumptions of being an $\alpha$-ideal, except that $\mathpzc{J}_{\ge n}$ does not contain all cofibrant algebras. Both $\mathpzc{I}_{\ge m}$ and $\mathpzc{J}_{\ge m}$ are relative categories in the obvious way. 

%For $n\in\mathbb{Z}$ we denote by $(\mathpzc{coAlg}_{\mathfrak{C}})^{|K|,\alpha-adm}_{\ge n}$ the full subcategory of  $\mathpzc{coAlg}_{\mathfrak{C}}^{|K|,\alpha-adm}$ whose underlying complexes are concentrated in degrees $\ge n$. We also denote by $(\mathpzc{Alg}_{\mathfrak{P}})_{\ge n}$ denote the category of $\mathfrak{P}$-algebras concentrated in degrees $\ge n$. 
%These are both relative categories in the obvious way

%We also denote by $\mathpzc{coAlg}_{\mathfrak{C}}^{|K|,\alpha-adm,H^{\ge n}}$ the full subcategory of $\mathpzc{coAlg}_{\mathfrak{C}}^{|K|,\alpha-adm}$ consisting of those coalgebras $C$ such that its underlying complex is homologically concentrated in degrees $\ge n$. 

%, and let $\mathpzc{coAlg}_{\mathfrak{C}}^{\ge n+m}$ denote the category of $\mathfrak{P}$-algebras  concentrated in degrees $\ge m+n$. 
Given an $m$-connected $\alpha$-ideal $(\mathpzc{I}_{\ge m},\mathpzc{J}_{\ge m})$, the bar-cobar adjunction restricts to the adjunction.

$$\adj{\Omega_{\alpha}}{\mathpzc{I}_{\ge m}}{\mathpzc{J}_{\ge m}}{B_{\alpha}}.$$

The next result follows from Proposition \ref{qisconn} together with the observation that, under a connectivity constraint $m$, the factorisation in Proposition \ref{prop:hurewiczmodel} and Assumption \ref{times12} restricts to $\mathpzc{coAlg}^{conil}_{\mathfrak{C},\ge m}$.
%Let $\mathpzc{coAlg}_{\mathfrak{C}_{top},\ge m}^{|c|}=\mathpzc{coAlg}_{\mathfrak{C}_{top}}^{|c|}\cap \mathpzc{coAlg}_{\mathfrak{C}_{top},\ge m}$.

\begin{thm}
Let $\mathpzc{M}$ be a Koszul category with a strict $t$-structure, and let $\alpha:\mathfrak{C}\rightarrow\mathfrak{P}$ be a pure Koszul twisting morphism with connectivity constraint $m$. Suppose that $\mathfrak{C}$ satisfies Assumption \ref{times12}.
\begin{enumerate}
\item
Suppose that all objects of $\mathpzc{M}$ are cofibrant. Then, with the cofibrations being the $\alpha$-cofibrations (Definition \ref{alphamodel}), and the weak equivalences being the $\alpha$-weak equivalences, there exists the structure of a category of cofibrant objects on $\mathpzc{coAlg}^{conil}_{\mathfrak{C},\ge m}$. The bar-cobar adjunction restricts to a relative equivalence between the categories of cofibrant objects
$$\adj{\Omega_{\alpha}^{filt}}{\mathpzc{coAlg}_{\mathfrak{C},\ge m}^{conil}(\mathpzc{M})}{\mathpzc{Alg}_{\mathfrak{P},\ge m}^{c}(\mathpzc{M})}{B^{filt}_{\alpha}},$$
and the functor $\Omega_{\alpha}^{filt}$ sends $\alpha$-cofibrations to cofibrations.
    \item 
        Suppose that $\mathpzc{E}$ is a locally presentable quasi-abelian category with exact filtered colimits in which filtered colimits commute with finite limits, and all objects are flat. Let $\mathpzc{M}=Ch(\mathpzc{E})$ be the Koszul category obtained by equipping $Ch(\mathpzc{E})$ with the flat model structure. Then the left-transferred model structure along the adjunction
    $$\adj{\Omega^{filt}_{\alpha}}{\mathpzc{coAlg}_{\mathfrak{C},\ge m}^{conil}}{\mathpzc{Alg}_{\mathfrak{P},\ge m}}{B^{filt}_{\alpha}}$$
    exists on $\mathpzc{coAlg}_{\mathfrak{C},\ge m}^{conil}$. With respect to this model structure, the adjunction is a Quillen equivalence. If $\alpha|_{\mathfrak{C}(1)}=0$, then the weak equivalences on the left-hand side are those morphisms $f$ such that $f_{top}$ is a weak equivalence in $\mathpzc{M}$. 
    \item 
        Suppose that $\mathpzc{E}$ is a locally presentable quasi-abelian category with exact filtered colimits in which filtered colimits commute with finite limits, and all objects are flat. Let $\mathpzc{M}=Ch(\mathpzc{E})$ be the Koszul category obtained by equipping $Ch(\mathpzc{E})$ with the flat model structure. Then the left-transferred model structure along the adjunction
    $$\adj{\Omega_{\alpha}}{\mathpzc{coAlg}_{\mathfrak{C}_{top},\ge m}^{conil}}{\mathpzc{Alg}_{\mathfrak{P},\ge m}}{B_{\alpha}}$$
    exists on $\mathpzc{coAlg}_{\mathfrak{C}_{top},\ge m}^{conil,pure}=\mathpzc{coAlg}_{\mathfrak{C}_{top},\ge m}^{conil}$. With respect to this model structure, the adjunction is a Quillen equivalence.  If $\alpha|_{\mathfrak{C}(1)}=0$, then a morphism on the left-hand side is an $\alpha$-weak equivalence precisely if its underlying morphism in $\mathpzc{M}$ is a weak equivalence. Finally, if every object of $\mathpzc{M}$ is also fibrant, then the cofibrations in the transferred model structure coincide with the $\alpha$-cofibrations.
\end{enumerate}
\end{thm}

Applying this result to the twisting morphism $\kappa:\mathfrak{S}^{c}\otimes_{H}\mathfrak{coComm}^{ncu}\rightarrow\mathfrak{Lie}$ and the category ${}_{k}\mathpzc{Vect}$ for $k$ a field of characteristic $0$ recovers the famous result of Quillen \cite{quillen1969rational}. Indeed, suppose that $Ch(\mathpzc{E})$ is $C$-monoidal strong Koszul where $\mathpzc{E}$ is quasi-abelian, and every object of $\mathpzc{E}$ is flat. For each $r\ge 0$, there is a model category structure on $\mathpzc{coAlg}^{conil}_{(\mathfrak{S}^{c}\otimes_{H}\mathfrak{coComm}^{ncu})_{top},\ge r}$ such that the bar-cobar adjunction induces a Quillen equivalence
$$\adj{\Omega_{\kappa}}{\mathpzc{coAlg}^{conil}_{(\mathfrak{S}^{c}\otimes_{H}\mathfrak{coComm}^{ncu})_{top},\ge r}}{\mathpzc{Alg}_{\mathfrak{Lie},\ge r}}{B_{\kappa}}.$$
Consider the model category structure on $\mathpzc{coAlg}^{conil}_{\mathfrak{coComm}^{ncu}_{top},\ge r+1}$ induced by the equivalence
$$[-1]:\mathpzc{coAlg}^{conil}_{\mathfrak{S}^{c}\otimes_{H}\mathfrak{coComm}^{ncu}_{top},\ge r}\rightarrow \mathpzc{coAlg}^{conil}_{\mathfrak{coComm}^{ncu}_{top},\ge r+1}.$$

There is then a Quillen equivalence
$$\adj{\Omega_{\kappa}\circ[1]}{\mathpzc{coAlg}^{conil}_{\mathfrak{coComm}^{ncu}_{top},\ge r+1}}{\mathpzc{Alg}_{\mathfrak{Lie},\ge r}}{[-1]\circ B_{\kappa}}.$$

\section{Operadic Koszul Duality}\label{secopkosz}

In this section, we fix a Koszul twisting morphism $\alpha:\mathfrak{C}\rightarrow\mathfrak{P}$ in a Koszul category $\mathpzc{M}$.  We shall assume that $\mathpzc{M}$ is \textit{closed Koszul} and $C$-monoidal, i.e., that $\mathpzc{M}$ is a closed monoidal model category. We are going to consider the functor 
$$\hat{C}_{\alpha}:\mathpzc{Alg}_{\mathfrak{P}}(\mathpzc{M})\rightarrow\mathpzc{Alg}_{\mathfrak{C}^{\vee}}(\mathpzc{M})^{op}$$
defined in Section \ref{appcatalg}, which is given by the composition $(-)^{\vee}\circ B_{\alpha}$.
%\subsection{Dualisation of Coalgebras}
%Let $\alpha:\mathfrak{C}\rightarrow\mathfrak{P}$ be a Koszul twisting morphism. The dualizing functor $(-)^{\vee}\defeq\underline{\mathrm{Hom}}(-,R):\mathpzc{E}\rightarrow\mathpzc{E}^{op}$ is lax monoidal. Let $\mathfrak{C}$ be a formally planar cooperad. so it induces a functor. 

%\begin{defn}
%An object $X$ of $\mathpzc{M}$ is said to be \textbf{finitely} $K$-\textbf{projective}. if the morphism $\mathrm{Hom}(X,R)\rightarrow\mathbb{R}\underline{\mathrm{Hom}}(X,R)$ is an equivalence.
%\end{defn}
%%
%FIX: Put this earlier. Note that $\underline{\mathrm{Hom}}(-,R)$ preserves weak equivalences between finitely $K$-cotorsion objects. If $R$ is fibrant in $\mathpzc{M}$ it is clear that any cofibrant object is finitely $K$-cotorsion. This for example happens in the case that $\mathpzc{M}={}_{R}\mathpzc{Mod}(Ch(\mathpzc{E}))$, with $\mathpzc{E}$ a monoidal elementary exact category, and $Ch(\mathpzc{E})$ is equipped with the projective model structure. 

%\begin{prop}
%Let $\mathpzc{M}=Ch(\mathpzc{E})$ for $\mathpzc{E}$ a basic Koszul category. Suppose that $X_{\bullet}$ is a complex of objects such that each $X_{n}$ satisfies 
%\end{prop}
%
%\begin{prop}
%Let $\mathpzc{M}=Ch(\mathpzc{E})$ for $\mathpzc{E}$ a basic Koszul category. Suppose that $\mathrm{Hom}(k,k)$ is a field. If
%$$0\rightarrow k\rightarrow Y\rightarrow X\rightarrow 0$$
%Moreover, an object $X$ of $\mathpzc{E}$ is finitely projective if and only if any morphism $X\rightarrow k$ splits. 
%\end{prop}
%
%\begin{proof}
%Let us first show that 
%\end{proof}
\subsection{Operadic Koszul Categories and Morphisms}
For the functor $\hat{C}_{\alpha}$ to have nice properties, we need some assumptions on both $\mathpzc{M}$ and $\alpha$.

\begin{defn}\label{defn:alphaoperadic}

Let $\mathpzc{M}$ be a closed Koszul category and $\alpha:\mathfrak{C}\rightarrow\mathfrak{P}$ a filtered Koszul twisting morphism in $\mathpzc{M}$. The Koszul category $\mathpzc{M}$ is said to be $\alpha$-\textbf{operadic} if
\begin{enumerate}
 \item
the underlying object of $B_{\alpha}A$ is finitely $K$-cotorsion whenever $A\in\mathpzc{Alg}_{\mathfrak{P}}^{c}$,
 \item 
for all $n$, $\mathfrak{C}_{top}(n)$ and $gr(\mathfrak{C})(n)$ are $\Sigma_{n}$-cofibrant,
\item 
$\mathfrak{C}^{\vee}$ is a weakly rectifiable, admissible operad,
\item
  the countable product functor $\prod_{n\in\mathbb{N}_{0}}$ preserves weak equivalences, and
  \item 
  for any countable collection of objects $\{M_{n}\}_{n\in\mathbb{N}_{0}}$, the natural morphism
  $$\bigoplus_{n\in\mathbb{N}_{0}}M_{n}\rightarrow\prod_{n\in\mathbb{N}_{0}}M_{n}$$
  is a monomorphism.
  \end{enumerate}
\end{defn}
%\begin{prop}
%FIX: MAYBE UNTRUE The functor $(-)^{\vee}:\mathpzc{M}\rightarrow\mathpzc{M}^{op}$ preserves weak equivalences between $K$-flat objects.
%\end{prop}

%\begin{prop}
%Let $\mathpzc{M}$ be a closed Koszul category. If $C$ is cofibrant and $F$ is fibrant then the morphism $\underline{\mathrm{Hom}}(C,F)\rightarrow\mathbb{R}\underline{\mathrm{Hom}}(C,F)$ is an equivalence.
%\end{prop}
\begin{example}
Let $\mathpzc{M}={}_{R}\mathpzc{Mod}(Ch(\mathpzc{E}))$ be a projective Koszul category that is a closed monoidal model category. Suppose further that $\mathpzc{E}$ has a set of small projective generators $\mathcal{G}$ such that, for each $G\in\mathcal{G}$, $\mathrm{Hom}(G,-):\mathpzc{E}\rightarrow\mathpzc{Ab}$ commutes with countable direct sums. Let $\alpha:\mathfrak{C}\rightarrow\mathfrak{P}$ be a filtered Koszul twisting morphism where $\mathfrak{C}^{\vee}$ is a weakly rectifiable, admissible operad such that $gr(\mathfrak{C})(n)$ and $\mathfrak{C}(n)_{top}$ are $\Sigma_{n}$-cofibrant for each $n$. Then $\mathpzc{M}$ is $\alpha$-operadic. Indeed, in an exact category with enough projectives, arbitrary products are exact. Since, for $A\in\mathpzc{Alg}_{\mathfrak{P}}^{c}$, we have that $B_{\alpha}A$ is cofibrant as an object of $\mathpzc{M}$, it remains to check that cofibrant objects are finitely $K$-cotorsion. However, in a closed monoidal model category, the morphism $\underline{\mathrm{Hom}}(C,F)\rightarrow\mathbb{R}\underline{\mathrm{Hom}}(C,F)$ is an equivalence whenever $C$ is cofibrant and $F$ fibrant. In a projective Koszul category, every object is fibrant. Since, for $G\in\mathcal{G}$, $\mathrm{Hom}(G,-)$ commutes with countable coproducts and limits, and the functors $\mathrm{Hom}(G,-)$ as $G$ ranges over $\mathcal{G}$ are jointly conservative, condition (5) is satisfied. 
\end{example}

%\begin{prop}
%Let $\mathpzc{M}$ be an elementary Koszul category. Then $\mathpzc{M}$ is operadic. 
%\end{prop}
%
%\begin{proof}
%Indeed in an exact category with enough projectives filtered projective limits are exact by Proposition \ref{projexact}. It remains to check that cofibrant objects are finitely $K$-cotorsion. Let $C$ be cofibrant. First suppose that it is free on on a cofibrant object $\tilde{C}$ of $Ch(\mathpzc{E})$. Then we have equivalences
%$$\mathbb{R}Hom_{\mathpzc{M}}(C,R)\cong\mathbb{R}Hom_{Ch(\mathpzc{E})}(\tilde{C},R)\cong \mathrm{Hom}(\tilde{C},R)$$
%
%\end{proof}

%\begin{defn}
%A Koszul morphism $\alpha:\mathfrak{C}\rightarrow\mathfrak{P}$ is said to be \textbf{operadic} if
%\begin{enumerate}
%\item 
%each $\mathfrak{C}(n)$ is a retract of a free $\Sigma_{n}$-module on a $K$-flat object of $\mathpzc{M}$, and
 %   \item 
  %  $\mathfrak{S}\otimes_{H}\mathfrak{C}^{\vee}$ is a rectifiably semiadmissible operad.
%\end{enumerate}
%\end{defn}

%\begin{rem}
 %   In principle this is difficult to check. However, it always holds when $\mathfrak{P}$ is $\Sigma$-cofibrant (and admissible). In particular, it holds in $\mathbb{Q}$-enriched settings. .
%\end{rem}

For the remainder of this section, whenever discussing a fixed Koszul twisting morphism $\alpha$, the relevant Koszul category $\mathpzc{M}$ will be an $\alpha$-operadic Koszul category.
%and $\alpha:\mathfrak{C}\rightarrow\mathfrak{P}$ will be an operadic Koszul morphism. 
%
%\begin{defn}
%An object $C$ of $\mathpzc{coAlg}^{conil}_{\mathfrak{C}}$ is said to be \textbf{finitely }$K$-\textbf{cotorsion} if its underlying object is finitely $K$-cotorsion as an object of $\mathpzc{M}$. The full subcategory of conilpotent coalgebras consisting of those which are finitely $K$-cotorsion is denoted $\mathpzc{coAlg}^{fKcot}_{\mathfrak{C}}$.
%\end{defn}
Let $W$ denote the class of morphisms in $\mathpzc{coAlg}_{\mathfrak{C}_{top}}$ consisting of those morphisms whose underlying morphism in $\mathpzc{M}$ is a weak equivalence. Let $W_{\alpha}$ denote the class of morphisms in $\mathpzc{coAlg}_{\mathfrak{C}_{top}}$ consisting of $\alpha$-weak equivalences. Consider the full subcategory $\{B_{\alpha}(A):A\in\mathpzc{Alg}_{\mathfrak{P}}^{c}\}$ of $\mathpzc{coAlg}_{\mathfrak{C}_{top}}$. Let $\textbf{U}$ denote the $(\infty,1)$-categorical localisation of $\{B_{\alpha}(A):A\in\mathpzc{Alg}_{\mathfrak{P}}^{c}\}$ at $W$, and let $\textbf{U}_{\alpha}$ denote the $(\infty,1)$-categorical localisation of $\{B_{\alpha}(A):A\in\mathpzc{Alg}_{\mathfrak{P}}^{c}\}$ at $W_{\alpha}$. By Proposition \ref{prop:alphaequivweak}, the identity gives a relative functor 
$$(\{B_{\alpha}(A):A\in\mathpzc{Alg}_{\mathfrak{P}}^{c}\},W_{\alpha})\rightarrow (\{B_{\alpha}(A):A\in\mathpzc{Alg}_{\mathfrak{P}}^{c}\},W).$$
We therefore get a functor of $(\infty,1)$-categories, $\textbf{U}_{\alpha}\rightarrow\textbf{U}$. Note that for any $\alpha$-ideal $(\mathpzc{I},\mathpzc{J})$, and any $C\in\mathpzc{I}$, there is an equivalence $C\cong B_{\alpha}(A)$ for $A$ a cofibrant $\mathfrak{P}$-algebra. In particular, we have an equivalence of $(\infty,1)$-categories $\mathrm{L^{H}}(\mathpzc{I})\cong\textbf{U}_{\alpha}$.

\begin{prop}

Let $\mathpzc{M}$ be a Koszul category and $\alpha:\mathfrak{C}\rightarrow\mathfrak{P}$ a filtered Koszul morphism with a filtered Koszul ideal $(\mathpzc{I},\mathpzc{J})$, such that $\mathpzc{M}$ is $\alpha$-operadic. Then the functor $(-)^{\vee}:\mathpzc{coAlg}^{conil}_{\mathfrak{C}_{top}}\rightarrow(\mathpzc{Alg}_{\mathfrak{C}_{top}^{\vee}})^{op}$ induces a functor of $(\infty,1)$-categories
$$(-)^{\vee}_{\alpha}:\mathrm{L^{H}}(\mathpzc{I}_{top})\cong\textbf{U}_{\alpha}\rightarrow\textbf{Alg}_{\mathfrak{C}^{\vee}}^{op}.$$
\end{prop}

\begin{proof}
    This claim simply follows from the fact that the underlying object of any object of the form $B_{\alpha}(A)$, with $A$ a cofibrant $\mathfrak{P}$-algebra, is assumed to be finitely $K$-cotorsion.
\end{proof}

%We let 
%$$\hat{C}_{\alpha}:\mathpzc{Alg}_{\mathfrak{P}}\rightarrow\mathpzc{Alg}_{\mathfrak{S}^{c}\otimes_{H}\mathfrak{C}^{\vee}}$$
%denote the composition $[-1]\circ(-)^{\vee}\circ B_{\alpha}$. Here $\otimes_{H}$ is the Hadamard tensor product defined in Section \ref{Appopp}. 

%By Theorem \ref{preQuillen} and Section \ref{Kcotorproj} 

There is an induced functor of $(\infty,1)$-categories
$$\hat{\mathbf{C}}_{\alpha}:\textbf{Alg}_{\mathfrak{P}}\rightarrow\textbf{Alg}_{\mathfrak{C}^{\vee}}^{op},$$
given by the composition $\hat{\mathbf{C}}_{\alpha}=(-)^{\vee}_{\alpha}\circ\textbf{B}_{\alpha}$.
We are going to generalise \cite{lurie2011derived} Proposition 2.2.12. The proof will also be a generalisation of the one in \cite{lurie2011derived}.

\begin{rem}
     Recall the functor $\hat{C}_{\alpha}$ from Subsection \ref{subsec:dualitytwist}. Then, in an $\alpha$-operadic Koszul category $\mathpzc{M}$, for a cofibrant $\mathfrak{P}$-algebra $A$ we have $\hat{\mathbf{C}}_{\alpha}(A)\cong\hat{C}_{\alpha}(A)$.
\end{rem}

\begin{thm}\label{alphadjoint}

The functor $\hat{\mathbf{C}}_{\alpha}:\textbf{Alg}_{\mathfrak{P}}\rightarrow(\textbf{Alg}_{\mathfrak{C}^{\vee}})^{op}$ admits a right adjoint $\mathbf{D}_{\alpha}$.

\end{thm}

%
%\begin{defn}
%Let $\mathpzc{E}$ be a closed symmetric monoidal category. A morphism $f:V\rightarrow W$ in $\mathpzc{E}$ is said to be \textbf{nuclear} if it lies in the image of $\mathrm{Hom}(k,V^{\vee}\otimes W)\rightarrow \mathrm{Hom}(V,W)$. 
%%An object of $\mathpzc{E}$ is said to be \textbf{nuclear} if i
%\end{defn}
%
%%\begin{defn}
%%Let $\mathpzc{E}$ be a monoidal $\textbf{AdMon}$-elementary quasi-abelian category. 
%%\end{defn}
%
%\begin{prop}\label{polykoszul}
%Suppose that each $\mathfrak{P}(n)$ is nuclear? $\mathfrak{g}$. Then the differential on  $\hat{C}_{\alpha}$ restricts to a differential on $\bigoplus_{n}((\mathfrak{S}^{c}(n)\otimes\mathfrak{C}(n))\otimes_{\Sigma_{n}}\mathfrak{g}^{\otimes n}[1])^{\vee}$. 
%\end{prop}
%
%\begin{proof}
%
%\end{proof}

%One can check that it is closed under the differential on $\hat{C}_{\alpha}(\mathfrak{g})$ whenever $\mathfrak{P}$ and $A$ are FIX: GENERALISE nuclear objects. We denote this subalgebra by $C_{\alpha}(\mathfrak{g})$. 
%
%\begin{defn}
%FIX: Put earlier? Let $\mathpzc{E}$ be a closed monoidal additive category such that the monoidal unit $k$ is compact. An object $V$ of $\mathpzc{E}$ is said to be \textbf{nuclear} if for any compact object $K$ the natural morphism $V^{\vee}\otimes K\rightarrow\underline{\mathrm{Hom}}(V,K)$ is an isomorphism. 
%\end{defn}

%In particular, free objects on $R$ are nuclear.

%\begin{rem}\label{polykoszul}
%
%\end{rem} 

%\begin{prop}\label{polykoszul}
%Let 
%\end{prop}
%
%\begin{proof}
%
%\end{proof}

%Before doing so we note the following technical result

\begin{proof}
Fix a filtered Koszul $\alpha$-ideal $(\mathpzc{I},\mathpzc{J})$. Write $\textbf{coAlg}^{conil}_{\mathfrak{C}}=\mathrm{L^{H}}(\{B^{filt}_{\alpha}(A):A\in\mathpzc{Alg}_{\mathfrak{P}}^{c}\})$, and $\textbf{coAlg}^{conil}_{\mathfrak{C}_{top}}=\mathrm{L^{H}}(\{B_{\alpha}(A):A\in\mathpzc{Alg}_{\mathfrak{P}}^{c}\})$.

Using Lurie's $(\infty,1)$-adjoint functor theorem, and noting that $\textbf{Alg}_{\mathfrak{P}}$ is locally presentable, we need to show that $\hat{\mathbf{C}}_{\alpha}$ preserves homotopy colimits. Since the functor $|-|:\textbf{Alg}_{\mathfrak{C}^{\vee}}\rightarrow\mathrm{L^{H}}(\mathpzc{M})$ is conservative, preserves and reflects homotopy limits, and preserves homotopy sifted colimits by Section \ref{siftedgeneration}, it suffices to show that $|-|\circ\hat{\mathbf{C}}_{\alpha}:\textbf{Alg}_{\mathfrak{P}}\rightarrow(\mathrm{L^{H}}(\mathpzc{M}))^{op}$ preserves homotopy colimits. Let us first show that it preserves homotopy sifted colimits. We have a commutative diagram, in which $|-|$ means the forgetful functor in both instances,
\begin{displaymath}
\xymatrix{
\textbf{Alg}_{\mathfrak{P}}\ar[d]^{\textbf{B}_{\alpha}}\ar[r]^{\textbf{B}_{\alpha}} &\textbf{coAlg}^{conil}_{\mathfrak{C}_{top}}\ar[r]^{(-)^{\vee}} &(\textbf{Alg}_{\mathfrak{C}^{\vee}})^{op}\ar[d]^{|-|}\\
\textbf{coAlg}^{conil}_{\mathfrak{C}_{top}} \ar[r]^{|-|}& \mathrm{L^{H}}(\mathpzc{M}) \ar[r]^{(-)^{\vee}} & (\mathrm{L^{H}}(\mathpzc{M}))^{op},
}
\end{displaymath}
and both compositions are equal to $|-|\circ\hat{\mathbf{C}}_{\alpha}$. The functor $(-)^{\vee}:\mathrm{L^{H}}(\mathpzc{M})\rightarrow (\mathrm{L^{H}}(\mathpzc{M}))^{op}$ is a left adjoint and therefore it commutes with all homotopy colimits. Therefore, we reduce to showing that $|-|\circ \textbf{B}_{\alpha}$ commutes with homotopy sifted colimits. There is a factorisation of $|-|\circ \textbf{B}_{\alpha}$:
\begin{displaymath}
\xymatrix{
\textbf{Alg}_{\mathfrak{P}}\ar[r]^{\textbf{B}^{filt}_{\alpha}} & \textbf{coAlg}^{conil}_{\mathfrak{C}}\ar[rr]^{|-|} &&  \overline{\textbf{Filt}}(\mathrm{L^{H}}(\mathpzc{M}))\ar[rr]^{(-)_{top}} & & \mathrm{L^{H}}(\mathpzc{M}).
}
\end{displaymath}
 The functor $\textbf{B}^{filt}_{\alpha}$ is the one induced on hammock localisation by the restriction of $B_{\alpha}^{filt}$ to cofibrant objects.

The functor $(-)_{top}$ is homotopy colimit-preserving. To show that $|-|\circ\mathbf{B}_{\alpha}$ preserves homotopy sifted colimits, it remains to show that $ |-|\circ \textbf{B}^{filt}_{\alpha}$ preserves homotopy sifted colimits. Now, $(-)_{0}\circ|-|\circ\textbf{B}^{filt}_{\alpha}=|-|$ preserves homotopy sifted colimits. By Proposition \ref{inductgraded} we finally reduce to showing that the composition $\textbf{gr}\circ |-|\circ  \textbf{B}^{filt}_{\alpha}$ is homotopy sifted colimit-preserving. But this functor is equivalent to the composition 
\begin{displaymath}
\xymatrix{
\textbf{Alg}_{\mathfrak{P}}\ar[r]^{|-|} & \mathrm{L^{H}}(\mathpzc{M})\ar[rr]^{|-|\circ gr(\mathfrak{C})(-)} && \mathrm{L^{H}}(\mathpzc{M}).
}
\end{displaymath}
 Each $gr(\mathfrak{C})(n)\otimes_{\Sigma_{n}}(-)^{\otimes n}$ commutes with homotopy sifted colimits, since each $gr(\mathfrak{C})(n)$ is a cofibrant $\Sigma_{n}$-module. Thus, all of the functors in this composition preserve homotopy sifted colimits. Finally, we need to show that $\hat{\mathbf{C}}_{\alpha}$ preserves finite homotopy products. Now, by Section \ref{siftedgeneration}, the category $\textbf{Alg}_{\mathfrak{P}}$ is generated under homotopy sifted colimits by free objects $\mathfrak{P}(V)$ on cofibrant objects $V$. Thus, it is enough to show that $\hat{\mathbf{C}}_{\alpha}$ sends finite coproducts of the form $\mathfrak{P}(V)\coprod\mathfrak{P}(W)\cong\mathfrak{P}(V\oplus W)$ to products. But 
$$\hat{\mathbf{C}}_{\alpha}(\mathfrak{P}(V))\cong(\mathfrak{C}\circ_{\alpha}\mathfrak{P}(V))^{\vee}\cong V^{\vee}.$$
If
\begin{displaymath}
\xymatrix{
\mathfrak{P}(0)\ar[r]\ar[d] & \mathfrak{P}(V)\ar[d]\\
\mathfrak{P}(W)\ar[r] & \mathfrak{P}(V\oplus W)
}
\end{displaymath}
is a homotopy coproduct diagram in $\textbf{Alg}_{\mathfrak{P}}$, applying $\hat{\mathbf{C}}_{\alpha}$ gives the diagram
\begin{displaymath}
\xymatrix{
V^{\vee}\oplus W^{\vee}\ar[r]\ar[d] & V^{\vee}\ar[d]\\
W^{\vee}\ar[r] & 0.
}
\end{displaymath}
This is a homotopy product diagram in $\textbf{Alg}_{\mathfrak{C}^{\vee}}$.

%Finally, we have that $\mathbf{D}_{\alpha}\circ\mathfrak{C}^{\vee}(-)$ is right adjoint to $|-|\circ\hat{\mathbf{C}}_{\alpha}$, where $|-|:\textbf{Alg}_{\mathfrak{C}^{\vee}}\rightarrow\mathrm{L^{H}}(\mathpzc{M})$ is the forgetful functor. Now, we have $|-|\circ\hat{\mathbf{C}}_{\alpha}\cong  |-|\circ(-)^{\vee}\circ \textbf{B}_{\alpha}$. We claim that this functor is left adjoint to $tr_{\mathfrak{P}}\circ (-)^{\vee}$. We have 
%\begin{align*}
%\textbf{Map}_{\textbf{Alg}_{\mathfrak{P}}}(A,tr_{\mathcal{P}}(V^{\vee}))&\cong\textbf{Map}(\mathbb{L}_{0}^{\mathfrak{P}}(A),V^{\vee})\\
%&\cong
%\end{align*}
\end{proof}

\begin{prop}

If $\mathpzc{M}$ is strong Koszul and the underlying $\Sigma$-module of $\mathfrak{P}$ is a retract of a free $\Sigma$-module on a cofibrant object, then we have a natural equivalence of functors
$$\mathbf{D}_{\alpha}\circ\mathfrak{C}^{\vee}(-)\cong tr_{\mathfrak{P}}\circ(-)^{\vee}:\mathrm{L^{H}}(\mathpzc{M})^{op}\rightarrow\textbf{Alg}_{\mathfrak{P}}.$$
\end{prop}

\begin{proof}
    The functor $\mathbf{D}_{\alpha}\circ\mathfrak{C}^{\vee}(-)$ is right adjoint to $|-|\circ\hat{\mathbf{C}}_{\alpha}$, and the functor $tr_{\mathfrak{P}}\circ(-)^{\vee}$ is right adjoint to $(-)^{\vee}\circ\mathbb{L}_{0}^{\mathfrak{P}}$. We have 

$$|-|\circ\hat{\mathbf{C}}_{\alpha}\cong |-|\circ(-)^{\vee}_{\alpha}\circ \textbf{B}_{\alpha}\cong (-)^{\vee}\circ |-|\circ \textbf{B}_{\alpha}.$$ 

By Proposition \ref{prop:compBalphacot}, we have a natural equivalence $(-)^{\vee}\circ |-|\circ \textbf{B}_{\alpha}\cong (-)^{\vee}\circ\mathbb{L}_{0}^{\mathfrak{P}}$. This completes the proof.
\end{proof}

Since we used the adjoint functor theorem, the proof of the existence of $\mathbf{D}_{\alpha}$ is not constructive. However, we can give an interpretation of $\mathbf{D}_{\alpha}$ in terms of the shifted tangent complex, similarly to Subsection \ref{subsubsec:cotangent}.

\subsubsection{The Shifted Tangent Complex}

Recall, for an admissible operad $\mathfrak{P}$, the cotangent complex functor $\mathbb{L}_{0}^{\mathfrak{P}}:\textbf{Alg}_{\mathfrak{P}}(\mathpzc{M})\rightarrow\mathrm{L^{H}}(\mathpzc{M})$ that we defined in Subsection \ref{subsubsec:cotangent}. The \textbf{tangent complex functor} is the composition $\mathbb{T}^{\mathfrak{P}}_{0}\defeq (-)^{\vee}\circ\mathbb{L}^{\mathfrak{P}}_{0}$. 

%We shall abuse notation and write $\mathbf{D}_{\alpha}$ for the composite $\mathbf{D}_{\alpha}\circ I:\textbf{Alg}_{\overline{\mathfrak{P}}}^{aug}(\mathpzc{M})\rightarrow \textbf{Alg}_{\mathfrak{P}}(\mathpzc{M})\rightarrow(\textbf{Alg}_{\mathfrak{C}^{\vee}}(\mathpzc{M}))^{op}$, where $I$ is the inverse to the functor $U$ from Subsection \ref{subsubsec:cotangent}.

%In classical Koszul duality for Lie algebras and commutative algebras the functor
%$$ |-|\circ\mathbf{D}_{\kappa}:\textbf{Alg}^{aug}_{\mathfrak{Comm}}(Ch(\mathpzc{Vec}_{k}))\rightarrow\textbf{Alg}_{\mathfrak{Lie}}(\textbf{Ch}(\mathpzc{Vec}_{k}))^{op}\rightarrow\textbf{Ch}(\mathpzc{Vec}_{k})$$
%is equivalent to the shifted tangent complex. We will now see that this works when $\mathpzc{M}$ is enriched over $\Q$. 

%
%\begin{prop}\label{quasifreecotangent}
%Let $A\in\mathpzc{Alg}_{\mathfrak{Comm}}^{aug}(\mathpzc{M})$ be a quasi-free object on a bounded-below cofibrant complex $V$. Then $\mathbb{L}_{0}(A)\cong V$ and $\mathbb{L}_{A}\cong A\otimes V$ 
%\end{prop}
%\begin{proof}
%Since $A$ by is cofibrant by Corollary \ref{quasifreecofib} we may use Proposition \ref{maximalideal}. This gives $\mathbb{L}_{0}\cong V$. Now $V$ is cofibrant so again by Proposition \ref{maximalideal} $\mathbb{L}_{A}\cong A\otimes V$.
%\end{proof}

\begin{prop}\label{shiftedtanget}
The functor $|-|\circ\mathbf{D}_{\alpha}$ is naturally equivalent to the tangent complex functor $\mathbb{T}^{\mathfrak{C}^{\vee}}_{0}$.
\end{prop}

\begin{proof}

We show that $|-|\circ \mathbf{D}_{\alpha}$ and $\mathbb{T}^{\mathfrak{C}^{\vee}}_{0}$ are both right adjoint to the same functor. Now, $|-|\circ\mathbf{D}_{\alpha}$ is right adjoint to the functor $\hat{\mathbf{C}}_{\alpha}\circ(\mathfrak{P}(-))$. This is equivalent to the functor $tr_{\mathfrak{C}^{\vee}}\circ (-)^{\vee}$. But this functor is left adjoint to $\mathbb{T}_{0}^{\mathfrak{C}^{\vee}}$. 
\end{proof}

\begin{rem}\label{rem:u}
  Let $\mathfrak{g}$ be a cofibrant $\mathfrak{P}$-algebra. Recall the morphism $tr_{\mathfrak{C}}(|\mathfrak{g}|)\rightarrow B_{\alpha}(\mathfrak{g})$ from Remark \ref{rem:cotforgetnat}. By taking duals, we get a morphism $i:\hat{\mathbf{C}}_{\alpha}(\mathfrak{g})\rightarrow tr_{\mathfrak{C}^{\vee}}(|\mathfrak{g}|^{\vee})$, and therefore, by adjunction, a morphism $u:\mathbb{L}_{0}^{\mathfrak{C}^{\vee}}(\hat{\mathbf{C}}_{\alpha}(\mathfrak{g}))\rightarrow |\mathfrak{g}|^{\vee}$. 
\end{rem}

\begin{prop}\label{doubledualunit}
  
 After forgetting to $\mathrm{L^{H}}(\mathpzc{M})$, the unit $|\eta_{\mathfrak{g}}|:|\mathfrak{g}|\rightarrow|\mathbf{D}_{\alpha}\circ\hat{\mathbf{C}}_{\alpha}(\mathfrak{g})|$ factors as the composite of the natural morphism $J_{|\mathfrak{g}|}:|\mathfrak{g}|\rightarrow\mathbb{R}(|\mathfrak{g}|)^{\vee\vee}$, and the dual of the morphism $u:\mathbb{L}_{0}^{\mathfrak{C}^{\vee}}(\hat{\mathbf{C}}_{\alpha}(\mathfrak{g}))\rightarrow |\mathfrak{g}|^{\vee}$ furnished by Remark \ref{rem:u}.
\end{prop}

\begin{proof}
By a sequence of adjunctions, we have
\begin{align*}
    \textbf{Map}_{\mathrm{L^{H}}(\mathpzc{M})}(|\mathfrak{g}|,|\mathbf{D}_{\alpha}\hat{\mathbf{C}}_{\alpha}(\mathfrak{g})|)&\cong\textbf{Map}_{\textbf{Alg}_{\mathfrak{P}}}(\mathfrak{P}(|\mathfrak{g}|),\mathbf{D}_{\alpha}\hat{\mathbf{C}}_{\alpha}(\mathfrak{g}))\\
    &\cong\textbf{Map}_{\textbf{Alg}_{\mathfrak{C}^{\vee}}}(\hat{\mathbf{C}}_{\alpha}(\mathfrak{g}),tr_{\mathfrak{C}^{\vee}}(|\mathfrak{g}|^{\vee}))\\
    &\cong\textbf{Map}_{\mathrm{L^{H}}(\mathpzc{M})}(\mathbb{L}_{0}^{\mathfrak{C}^{\vee}}(\hat{\mathbf{C}}_{\alpha}(\mathfrak{g})),|\mathfrak{g}|^{\vee})\\
    &\cong\textbf{Map}_{\mathrm{L^{H}}(\mathpzc{M})}(|\mathfrak{g}|,\mathbb{T}_{0}^{\mathfrak{C}^{\vee}}(\hat{\mathbf{C}}_{\alpha}(\mathfrak{g}))).
\end{align*}
In the second equivalence we used $\hat{\mathbf{C}}_{\alpha}\mathfrak{P}(|\mathfrak{g}|)\cong tr_{\mathfrak{C}^{\vee}}(|\mathfrak{g}|^{\vee})$. Chasing through these equivalences gives the result.

\end{proof}

\subsubsection{Main Example: The Lie Operad}

%\subsection{Lie Algebras and Commutative Algebras}

%If $\mathpzc{E}$ is a monoidal exact category then we can bootstrap the Koszul morphism in ${}_{\mathbb{Q}}\mathpzc{Vect}$ to one in $\mathpzc{E}$.

Let us now study the example of commutative-Lie duality
\begin{displaymath}
\xymatrix{
\kappa:\mathfrak{S}^{c}\otimes_{H}\mathfrak{coComm}^{ncu}\ar[r] &\mathfrak{Lie}
}
\end{displaymath}
from Example \ref{ex:liecommutative} in greater detail. In particular, we are interested in its relevance to (affine) derived geometry. Fix a KC-monoidal Koszul category of the form ${}_{R}\mathpzc{Mod}(Ch(\mathpzc{E}))$ that is enriched over $\mathbb{Q}$ and is $\kappa$-operadic.

\begin{defn}
A Lie algebra $\mathfrak{g}$ is said to be \textbf{very passable} if it satisfies the following properties:
\begin{enumerate}
\item
the algebra $\mathfrak{g}$ is separable (Definition \ref{defn:separable});
\item
the object $|\mathfrak{g}|$ is finitely $K$-cotorsion (Definition \ref{defn:Kcotors}) and the underlying graded object of $|\mathfrak{g}|^{\vee}$ is isomorphic to $R\otimes\tilde{\mathfrak{g}}^{\vee}_{\bullet}$, where $\tilde{\mathfrak{g}}^{\vee}_{\bullet}$ is concentrated in strictly positive degrees, and each $\tilde{\mathfrak{g}}^{\vee}_{i}$ is $K$-flat;
\item
for each $0\le i<\infty$, $\tilde{\mathfrak{g}}^{\vee}_{i}$ is formally $\aleph_{1}$-filtered relative to the class $R\otimes(\tilde{\mathfrak{g}}^{\vee})^{\otimes}\otimes (\tilde{\mathfrak{g}}^{\vee})^{sym}$ (Definition \ref{defn:alephfiltered2}).
%\item
%$|\mathfrak{g}|^{\vee}$ is $K$-flat.  
\end{enumerate}
The algebra $\mathfrak{g}$ is said to be \textbf{very good} if, in addition, $|\mathfrak{g}|$ is homotopically reflexive (Definition \ref{defn:htpyref}). $\mathfrak{g}$ is said to be \textbf{passable} (resp. \textbf{good}) if it is equivalent to a very passable (resp. very good) algebra.
\end{defn}

%\begin{defn}
%A Lie algebra $\mathfrak{g}$ is said to be \textbf{very good} if $|\mathfrak{g}|$ is finitely $K$-cotorsion, concentrated in negative degrees and homotopically reflexive, $|\mathfrak{g}_{-1}|$ is formally $\aleph_{1}$-filtered, and $|\mathfrak{g}|^{\vee}$ is $K$-flat. $\mathfrak{g}$ is said to be \textbf{good} if it is equivalent to a very good algebra. 
%\end{defn}
Let us populate the class of very good algebras. Following \cite{hennion2015tangent}, we define cellular finite Lie algebras.

\begin{example}\label{goodthings}
 Let $R$ be non-negatively graded.
Say that a Lie algebra $\mathfrak{g}$ is cellular finite if there is a filtration 
 $$0=L_{0}\rightarrow L_{1}\rightarrow\ldots\rightarrow L_{m}=\mathfrak{g}$$
 where, for each $0\le n<m$, there is a pushout diagram
\begin{displaymath}
\xymatrix{
L(R\otimes S^{m_{n}-1}(k^{p_{m_{n}}}))\ar[d]\ar[r] & L_{n}\ar[d]\\
L(R\otimes D^{m_{n}}(k^{p_{m_{n}}}))\ar[r] & L_{n+1}
}
\end{displaymath}
where $m_{n}\le -1$. (This is the definition in \cite{hennion2015tangent} Definition 1.3.7 of a very good Lie algebra.) Then $\mathfrak{g}$ is very passable. Indeed, the presentation of $\mathfrak{g}$ immediately implies that it is cofibrant, and hence finitely $K$-cotorsion. Moreover, its underlying graded object (i.e., after forgetting differentials) is a finite sum of objects of the form $\bigoplus_{m=1}^{\infty}R^{k_{m}}[m]$. Such an object is evidently finitely $K$-cotorsion.

The $R$-linear dual of $\bigoplus_{m=1}^{\infty}R^{k_{m}}[m]$ is of the form $\prod_{m=1}^{\infty}R^{k_{m}}[-m]\cong \bigoplus_{m=1}^{\infty}R^{k_{m}}[-m]$ (where we use that $R$ is concentrated in non-negative degrees). This clearly gives conditions (2) and (3). If $R$ is homologically bounded as a complex, then $\mathfrak{g}$ is very good. Indeed, in this case, we have
$$\mathbb{R}\underline{\mathrm{Hom}}(\bigoplus_{m=1}^{\infty}R^{k_{m}}[-m],R)\cong\prod_{m=1}^{\infty}R^{k_{m}}[m]\cong\bigoplus_{m=1}^{\infty}R^{k_{m}}[m],$$
where in the last isomorphism we use that $R$ is homologically bounded.

Separability is an easy consequence of Definition \ref{defn:alphaoperadic} (5).
%he underlying $R$-module of $\mathfrak{g}$ is cofibrant by Proposition \ref{underlyingKflat}. The algebra $\mathfrak{g}$ is also $\mathfrak{G}$-non-negatively graded, where $\mathfrak{G}$ consists of flat objects - namely, objects of the form $R\otimes S^{m}(k^{p_{m}})$. Therefore, its underlying $R$-module is $K$-flat and (finitely $K$-cotorsion). If $R$ is cohomologically bounded as a complex, then $\mathfrak{g}$ is very good. 
\end{example}

%\begin{prop}\label{koszulqfree}
%Let $V$ be flat. Then the morphism $S^{c}(V[1])\rightarrow S(V)$ is a Koszul morphism. FIX: GENERALISE TOQUASIFREE
%\end{prop}
%
%\begin{proof}
%%https://www.maths.ed.ac.uk/~mbooth/hodgeproject.pdf
%%Loday-Vallette
%%https://books.google.co.uk/books?Id=hrv1CwAAQBAJ&pg=PA210&lpg=PA210&dq=free+operad+is+koszul&source=bl&ots=p4rwtO6jOG&sig=ACfU3U0C3rJSL3Bqy36G8WzlRQVqGifbmQ&hl=en&sa=X&ved=2ahUKEwi75N3mrJjhAhXuQhUIHVW3AP84ChDoATAQegQIBxAB#v=onepage&q=free%20operad%20is%20koszul&f=false
%Consider the morphism 
%\end{proof}

%
Recall the sub-$\mathfrak{S}\otimes_{H}\mathfrak{Comm}^{nu}$-algebra $C_{\kappa}(\mathfrak{g})\rightarrow\hat{C}_{\kappa}(\mathfrak{g})$ considered in Proposition \ref{polykoszul} for $\mathfrak{g}$ separable. Note that $C_{\kappa}(\mathfrak{g})[1]$ is a quasi-free commutative monoid on $\mathfrak{g}^{\vee}[1]$.

By Proposition \ref{conditionfordeceny}, Proposition \ref{completiontwoways}, and Proposition \ref{completionhpush} we get the following.

\begin{cor}\label{koszulpushout}
Let $\mathfrak{g}$ be very passable. The morphism $\hat{S}(\mathfrak{g}_{-1}^{\vee})\otimes^{\mathbb{L}}_{S(\mathfrak{g}_{-1}^{\vee})}C_{\kappa}(\mathfrak{g})[1]\rightarrow\hat{C}_{\kappa}(\mathfrak{g})[1]$ is an equivalence.
\end{cor} 
%\begin{proof}
%The morphism $S(\mathfrak{g}_{-1}^{\vee})\rightarrow C_{\kappa}(\mathfrak{g})$ is $K$-flat. Therefore, the tensor product doesn't need to be derived. Thus, it remains to check that the morphism $\hat{S}(\mathfrak{g}_{-1}^{\vee})\otimes_{S(\mathfrak{g}_{-1}^{\vee})}C_{\kappa}(\mathfrak{g})\rightarrow\hat{C}_{\kappa}(\mathfrak{g})$ is a degree-wise isomorphism. This is clear.
%\end{proof}
The abstract machinery we have set up allows us to generalise Lurie's \cite{lurie2011derived} Lemma 2.3.5 and its proof.

\begin{thm}\label{operadickoszul}

Let $\mathpzc{M}\cong{}_{R}\mathpzc{Mod}(Ch(\mathpzc{E}))$ be a Koszul category enriched over $\mathbb{Q}$ that is $\kappa$-operadic, with $R$ non-negatively graded. If $\mathfrak{g}$ is passable, then the morphism $\tilde{J}_{|\mathfrak{g}|}:\mathbb{R}(|\mathfrak{g}|)^{\vee\vee}\rightarrow|\mathbf{D}_{\kappa}\circ\hat{\mathbf{C}}_{\kappa}(\mathfrak{g})|$ furnished by Proposition \ref{doubledualunit} is an equivalence. In particular, if $\mathfrak{g}$ is good, then the unit $\mathfrak{g}\rightarrow\mathbf{D}_{\kappa}\circ\hat{\mathbf{C}}_{\kappa}(\mathfrak{g})$ is an equivalence. 
\end{thm}

\begin{proof}
Suppose that $\mathfrak{g}$ is passable. Without loss of generality, we may assume that $\mathfrak{g}$ is very passable. It remains to show that the morphism of complexes $\gamma:\mathbb{R}(|\mathfrak{g}|)^{\vee\vee}\rightarrow\mathbb{T}^{\mathfrak{S}^{-1}\otimes_{H}\mathfrak{Comm}^{nu}}_{0}(\hat{\mathbf{C}}_{\kappa}(\mathfrak{g}))$ from Proposition \ref{doubledualunit} is an equivalence. Proposition \ref{doubledualunit} shows that this morphism is in fact the derived dual of a natural morphism
$$\hat{f}:\mathbb{L}^{\mathfrak{S}^{-1}\otimes_{H}\mathfrak{Comm}^{nu}}_{0}(\hat{\mathbf{C}}_{\kappa}(\mathfrak{g}))\rightarrow|\mathfrak{g}|^{\vee},$$ so it suffices to show that this morphism is an equivalence. By shifting, and using Example \ref{ex:cotangentusual}, this corresponds to a morphism
$$\mathbb{L}_{0}(\hat{\mathbf{C}}_{\kappa}(\mathfrak{g})[1])\rightarrow|\mathfrak{g}|^{\vee}[1].$$

By Corollary \ref{koszulpushout}, we have an equivalence $0\cong\mathbb{L}_{\hat{S}(\mathfrak{g}_{-1}^{\vee})\big\slash S(\mathfrak{g}_{-1}^{\vee})}\otimes^{\mathbb{L}}_{\hat{S}(\mathfrak{g}_{-1}^{\vee})}\hat{\mathbf{C}}_{\kappa}(\mathfrak{g})[1]\cong\mathbb{L}_{\hat{\mathbf{C}}_{\kappa}(\mathfrak{g})[1]\big\slash \mathbf{C}_{\kappa}(\mathfrak{g})[1]}$. In particular, it suffices to show that the composition
$$f:\mathbb{L}_{0}(\mathbf{C}_{\kappa}(\mathfrak{g})[1])\rightarrow|\mathfrak{g}|^{\vee}[1]$$
is an equivalence. However, $\mathbf{C}_{\kappa}(\mathfrak{g})[1]$ is quasi-free on the bounded-below complex of $K$-flat modules $R\otimes\tilde{\mathfrak{g}}_{\bullet}^{\vee}[1]$. This morphism is then precisely the equivalence from Proposition \ref{cotangentqfree}. 
 \end{proof}
%We have $|\mathbf{D}_{\kappa}\circ\hat{\mathbf{C}}_{\kappa}(\mathfrak{g})|\cong\mathbb{T}_{0}(\hat{\mathbf{C}}_{\kappa}(\mathfrak{g}))[-1]\cong\mathbb{R}(\mathbb{L}_{0}(\hat{\mathbf{C}}_{\kappa}(\mathfrak{g})))^{\vee}$. There is clearly a morphism $\mathbb{L}_{0}(\hat{C}_{\kappa}(\mathfrak{g}))\rightarrow|\mathfrak{g}|^{\vee}\rightarrow\mathbb{R}\mathfrak{g}^{\vee}$, and therefore a morphism $\mathbb{R}(|\mathfrak{g}|)^{\vee}$

%Tensoring with $R$ over $\hat{C}_{\kappa}(\mathfrak{g})$  Proposition \ref{a} gives that $\mathbb{L}_{\hat{C}_{\kappa}(\mathfrak{g})\big\slash C_{\kappa}(\mathfrak{g})}\otimes^{\mathbb{L}}_{\hat{C}_{\kappa}(\mathfrak{g})}R\cong 0$. Finally, considering the homotopy cofibre sequence
%$$\mathbb{L}_{\hat{C}_{\kappa}(\mathfrak{g})\big\slash C_{\kappa}(\mathfrak{g})}\otimes^{\mathbb{L}}_{\hat{C}_{\kappa}(\mathfrak{g})}k\rightarrow\mathbb{L}_{0}(C_{\kappa}(\mathfrak{g}))[1]\rightarrow\mathbb{L}_{0}(\hat{C}_{\kappa}(\mathfrak{g}))[1]$$
%gives that $\mathbb{L}_{0}(C_{\kappa}(\mathfrak{g}))\rightarrow\mathbb{L}_{0}(\hat{C}_{\kappa}(\mathfrak{g}))$ is an equivalence. 
%Since $\mathfrak{g}$ is degree-wise free of finite rank, $C(\mathfrak{g})$ is quasi-free on a cofibrant complex.

%\appendix

\section{Examples, Applications, and Further Directions}\label{secexamples}

%\begin{enumerate}
%\item
%Consider the categories ${}_{k}\mathpzc{Vect}$ for $k$ a field and $Ind(Ban_{k})$ for $k$ a Banach field. Then if $\mathfrak{g}$ is concentrated in negative degrees and each $\mathfrak{g}_{n}$ is free of finite rank, $\mathfrak{g}$ is very good. Indeed in these cases any complex of free objects of finite rank is split. In particular, it is a coproduct of objects of the form $S^{n}(k^{m})$ and $D^{r}(k^{s})$, so it is cofibrant. In particular, this recovers Lurie's conditions in Lemma 2.3.5 of \cite{lurie2011derived}.
%\item
%Let $k$ be any spherically complete Banach field. 
%\item
%Let $k$ be a non-Archimedean spherically complete Banach field. FIX: Then any negatively-graded Banach Lie algebra $\mathfrak{g}$ such that $\mathfrak{g}_{-1}$ is finite-dimensional is passable. By the Hahn-Banach theorem for spherically complete fields any complex of Banach spaces is finitely $K$-cotorsion. Moreover, all Banach spaces over a non-Archimedean field are flat by \cite{steindomains}. 
%\end{enumerate}
.
Throughout, $\alpha:\mathfrak{C}\rightarrow\mathfrak{P}$ will be a pure Koszul twisting morphism.
For simplicity, in the examples below we assume that $\mathfrak{P}(n)$ and $gr(\mathfrak{C})(n)$ are retracts of free $\Sigma_{n}$-modules on cofibrant objects. We write
$$\textbf{coAlg}_{\mathfrak{C}_{top}}^{conil}(\mathpzc{M})=\mathrm{L^{H}}(\mathpzc{coAlg}_{\mathfrak{C}_{top}}^{conil}(\mathpzc{M})),$$
where we localise at $\alpha$-equivalences.

\subsection{Algebraic Koszul Duality}

Let $R$ be any differentially graded ring in $Ch(\mathpzc{Ab})$, and let $\alpha:\mathfrak{C}\rightarrow\mathfrak{P}$ be a Koszul morphism. There is an equivalence of $(\infty,1)$-categories

$$\textbf{coAlg}_{\mathfrak{C}_{top}}^{conil}({}_{R}\mathpzc{Mod})\cong\textbf{Alg}_{\mathfrak{P}}({}_{R}\mathpzc{Mod}).$$

Moreover, we always have the filtered $\alpha$-ideal $(\mathpzc{coAlg}_{\mathfrak{C}}^{|K|},\mathpzc{Alg}_{\mathfrak{P}}^{K})$. Suppose now that $R=R_{0}$ is concentrated in degree $0$ and is a von Neumann regular ring (for example a field or, more generally, a semisimple ring). Then every object of $Ch({}_{R}\mathpzc{Mod})$ is flat, and  the equivalence of $(\infty,1)$-categories arises from a Quillen equivalence
$$\adj{\Omega_{\alpha}}{\mathpzc{coAlg}^{conil}_{\mathfrak{C}_{top}}({}_{R}\mathpzc{Mod})}{\mathpzc{Alg}_{\mathfrak{P}}({}_{R}\mathpzc{Mod})}{B_{\alpha}}.$$

%Moreover, since ${}_{R}\mathpzc{Mod}$ is also abelian, the conditions of Assumption \ref{compatiblecompositefilt} are always satisfied when each $\mathfrak{C}(n)$ is a cofibrant $\Sigma_{n}$-module.

%In particular, the categories $\mathpzc{coAlg}^{conil}_{\mathfrak{C}_{top}}$ and $\mathpzc{coAlg}_{\mathfrak{C}}^{|K|,\alpha-adm}$ coincide. There is therefore an equivalence of categories
%$$\textbf{coAlg}_{\mathfrak{C}_{top}}^{conil}({}_{R}\mathpzc{Mod})\cong\textbf{Alg}_{\mathfrak{P}}({}_{R}\mathpzc{Mod}).$$

%Equip $Ch({}_{R}\mathpzc{Mod})$ with projective model structure, so that it is an elementary Koszul category. Suppose that $\alpha$ is a Quillen Koszul morphism (e.g. for duality between Lie algebras and cocommutative coalgebras). Then the equivalence of $(\infty,1)$-categories above arises from a Quillen equivalence
%$$\adj{\Omega_{\alpha}}{\mathpzc{coAlg}^{|c|}_{\mathfrak{C}}({}_{R}\mathpzc{Mod})}{\mathpzc{Alg}^{{|c|}}_{\mathfrak{P}}({}_{R}\mathpzc{Mod})}{B_{\alpha}}$$
 Equip $Ch({}_{R}\mathpzc{Mod})$ with the projective model structure, so that it is an elementary Koszul category. Recall that for a morphism to be Koszul it is required that $\mathfrak{P}$ be admissible. When $R=R_{0}$ is a finite product of fields, every object is flat, fibrant, and cofibrant. In this case, the model structure on $\mathpzc{coAlg}^{conil}_{\mathfrak{C}_{top}}$ coincides precisely with that of \cite{grignou2023homotopicaloperadiccalculuspositive} Theorem 6 (or rather, its non-curved version) in arbitrary characteristic, and \cite{vallette2014homotopy} Theorem 2.1 in the characteristic zero case.

%and according to Lemma B.1 in \cite{vallette2014homotopy} every Koszul morphism satisfies Assumption \ref{times12}.

%Thus, in this case, the equivalence of $(\infty,1)$-categories arises from a Quillen equivalence
%$$\adj{\Omega_{\alpha}}{\mathpzc{coAlg}^{conil}_{\mathfrak{C}_{top}}({}_{R}\mathpzc{Mod})}{\mathpzc{Alg}_{\mathfrak{P}}({}_{R}\mathpzc{Mod})}{B_{\alpha}}.$$

%Consider the flat model structure on $Ch({}_{R}\mathpzc{Mod})$ \cite{Gillespie2}. With this model structure every object of $Ch({}_{R}\mathpzc{Mod})$ is cofibrant. 

%Therefore, there is a model structure on $\mathpzc{coAlg}_{\mathfrak{C}}({}_{R}\mathpzc{Mod})$ and, whenever Assumption \ref{times12} holds (e.g. for duality between Lie algebras and cocommutative coalgebras), 

Consider specifically the Koszul morphism $\kappa:\mathfrak{S}^{c}\otimes_{H}\mathfrak{coComm}^{ncu}\rightarrow\mathfrak{Lie}$. Suppose that $R$ is homologically bounded. As we have mentioned, the cellular finite Lie algebras from Example \ref{goodthings} in this case include the very good algebras of \cite{hennion2015tangent}. Since $R_{0}$ is not required to be Noetherian, Theorem \ref{operadickoszul} in fact generalises \cite{hennion2015tangent} Lemma 1.4.12.

Suppose that $R=R_{0}=k$ is a field. Then, if $\mathfrak{g}$ is concentrated in negative degrees and each $\mathfrak{g}_{n}$ is free of finite rank, $\mathfrak{g}$ is very good. Indeed, in this case, any complex of free objects of finite rank is split. In particular, such a complex is a coproduct of objects of the form $S^{n}(k^{m})$ and $D^{r}(k^{s})$, so it is cofibrant. This recovers Lurie's conditions in Lemma 2.3.5 of \cite{lurie2011derived}.

\subsubsection{Filtered Objects}

Let $\mathpzc{E}$ be an elementary quasi-abelian category. Then the category $\reallywidehat{\overline{\mathpzc{Filt}}}_{\textbf{AdMon}}(\mathpzc{E})$ of complete, exhaustively filtered objects in $\mathpzc{E}$ (see \cite{kelly2016homotopy} Chapter 5) is a hereditary projective Koszul category. We get the following. 

\begin{thm}
The bar-cobar construction induces an equivalence of $(\infty,1)$-categories
$$\adj{\Omega_{\alpha}}{\textbf{coAlg}_{\mathfrak{C}_{top}}^{conil}(Ch(\reallywidehat{\overline{\mathpzc{Filt}}}_{\textbf{AdMon}}(\mathpzc{E})))}{\textbf{Alg}_{\mathfrak{P}}(Ch(\reallywidehat{\overline{\mathpzc{Filt}}}_{\textbf{AdMon}}(\mathpzc{E})))}{\textbf{B}_{\alpha}}.$$
\end{thm}

In particular, consider the category $\mathpzc{Vect}_{k}$ of vector spaces over a field $k$. In this case, all objects of $\reallywidehat{\overline{\mathpzc{Filt}}}_{\textbf{AdMon}}(\mathpzc{Vect}_{k})$ are flat, so that the Koszul duality equivalence of $(\infty,1)$-categories arises from a Quillen equivalence. %$(\mathpzc{coAlg}_{\mathfrak{C}}^{conil},\mathpzc{Alg}_{\mathfrak{P}})$ is a filtered $\alpha$-ideal. 

\subsection{Geometric Examples}

\subsubsection{Sheaves on Spaces}

Let $\mathcal{X}$ be a topological space, and let $\mathcal{O}_{\mathcal{X}}$ be a sheaf of rings on $\mathcal{X}$. As a consequence of the work of \cite{gillespie2006flat}, when equipped with the flat model structure, $_{\mathcal{O}_{\mathcal{X}}}\mathpzc{Mod}$ is a basic Koszul category. Thus, we get an equivalence of $(\infty,1)$-categories
$$\textbf{coAlg}_{\mathfrak{C}_{top}}^{conil}(Ch(_{\mathcal{O}_{\mathcal{X}}}\mathpzc{Mod}))\cong\textbf{Alg}_{\mathfrak{P}}(Ch(_{\mathcal{O}_{\mathcal{X}}}\mathpzc{Mod})).$$

Flatness is a local condition. Therefore, if $\mathcal{O}_{\mathcal{X}}$ is locally von Neumann regular (for example if $\mathcal{O}_{\mathcal{X}}$ is the constant sheaf associated to a field), every object in $Ch({}_{\mathcal{O}_{\mathcal{X}}}\mathpzc{Mod})$ is cofibrant. In this case we once again get that the equivalence of $(\infty,1)$-categories arises from a Quillen equivalence.

%a filtered $\alpha$-ideal $(\mathpzc{coAlg}_{\mathfrak{C}}^{conil},\mathpzc{Alg}_{\mathfrak{P}})$. 

 %whenever each $\mathfrak{C}(n)$ is $\Sigma_{n}$-cofibrant we get an equivalence of $(\infty,1)$-categories
%$$\textbf{coAlg}_{\mathfrak{C}_{top}}^{conil}({}_{R}\mathpzc{Mod})\cong\textbf{Alg}_{\mathfrak{P}}({}_{R}\mathpzc{Mod}),$$
%where $\textbf{coAlg}_{\mathfrak{C}_{top}}^{conil}({}_{R}\mathpzc{Mod})$ is the localisation of the entire category of conilpotent $\mathfrak{C}_{top}$-coalgebras.

%whenever $\alpha$ is Quillen Koszul holds the equivalence of $(\infty,1)$-categories arises from a Quillen equivalence 
%$$\adj{\Omega_{\alpha}}{\mathpzc{coAlg}^{conil}_{\mathfrak{C}}({}_{\mathcal{O}_{\mathcal{X}}}\mathpzc{Mod})}{\mathpzc{Alg}_{\mathfrak{P}}({}_{\mathcal{O}_{\mathcal{X}}}\mathpzc{Mod})}{B_{\alpha}}$$

\subsubsection{Quasi-coherent Sheaves on Stacks}\label{sec:qcsheaves}

Let $\mathcal{X}$ be an Artin stack with enough flat objects (for example, if $\mathcal{X}$ is geometric), and let $\mathcal{O}_{\mathcal{X}}$ be its structure sheaf. By Theorem 8.1 in \cite{estrada2014derived}, $QCoh(\mathcal{X})$ is a basic Koszul category. Therefore, there is an equivalence of $(\infty,1)$-categories
$$\textbf{coAlg}_{\mathfrak{C}_{top}}^{conil}(QCoh(\mathcal{X}))\cong\textbf{Alg}_{\mathfrak{P}}(QCoh(\mathcal{X})).$$

For example, if $\mathcal{X}=BG$ for $G$ an algebraic group, we get an equivalence of $(\infty,1)$-categories between $\mathfrak{P}$-algebras in $G$-representations and $\mathfrak{C}$-coalgebras in $G$-representations.
%https://arxiv.org/pdf/1303.6542.pdf

\subsection{Analytic Koszul Duality}

Let $k$ be a Banach ring, and let $\mathpzc{E}$ denote either $Ind(Ban_{k})$, the formal cocompletion of $Ban_{k}$ by inductive limits, or its full subcategory $CBorn_{k}$ of complete bornological spaces over $k$. If the norm on $k$ is non-Archimedean, we further allow $\mathpzc{E}$ to be $Ind(Ban^{nA}_{k})$ or $CBorn_{k}^{nA}$. For details about these quasi-abelian categories, see \cite{koren}, \cite{orenbambozzi}, \cite{bambozzi}, and \cite{kelly2016homotopy}. Note that $Ind(Ban_{k})$ is elementary, but $CBorn_{k}$ is only $\textbf{AdMon}$-elementary. Nonetheless, they are monoidally Quillen equivalent. Consequently, in principle, for homotopical algebra we can freely pass between them.

Let $k$ be a non-trivially valued Banach field -- that is, $k$ is a field equipped with a multiplicative norm, and there exists a non-zero element $\pi$ with norm strictly less than $1$. Let $\mathcal{T}_{c,k}$ denote the category of locally convex topological vector spaces over $k$. An object of this category is a $k$-vector space equipped with a topology determined by a family of seminorms. The category $\mathcal{T}_{c,k}$ is quasi-abelian (see \cite{prosmans1999derived} for details), and contains, for example, the category of Fr\'{e}chet (metrisable and complete) $k$-vector spaces. The category $\mathcal{T}_{c,k}$ also contains all seminormed $k$-vector spaces, and any object can be written as a cofiltered limit $X\cong\lim_{\rho\in\mathcal{P}}X_{\rho}$, where $\mathcal{P}$ is a family of seminorms determining the topology on $X$, and $X_{\rho}$ denotes the underlying vector space of $X$ equipped with the seminorm $\rho$. There is a functor 
$$(-)^{b}:\mathcal{T}_{c,k}\rightarrow Ind(Ban_{k})$$
that sends $\lim_{\rho\in\mathcal{P}}X_{\rho}$ to $\lim_{\rho\in\mathcal{P}}\hat{X}_{\rho}$, where, in the latter expression, $\hat{X}_{\rho}$ denotes the separated completion, and the limit is computed in $Ind(Ban_{k})$. We call this the \textbf{(von Neumann) bornologification} of $X$. We are particularly interested in a class of spaces called \textbf{dual nuclear Fr\'{e}chet spaces}.

Recall that a map $f:X\rightarrow Y$ between Banach spaces is \textit{trace-class (or nuclear)} if there is a bounded sequence $(f_{n})_{n\in\mathbb{N}_{0}}\subset X^{\vee}$ of continuous functionals, a bounded sequence $(y_{n})_{n\in\mathbb{N}_{0}}\subset Y$, and a summable sequence $\lambda_{n}\in k$ such that 
    $$f(x)=\sum_{n=0}^{\infty}\lambda_{n}f_{n}(x)y_{n}.$$

A Fr\'{e}chet space $F\in\mathcal{T}_{c,k}$ is called \textbf{nuclear} if it can be written as 
$$F\cong\lim_{n\in\mathbb{N}^{op}}X_{n}$$
where
\begin{enumerate}
\item 
each $X_{n}$ is a Banach space,
    \item 
    each $X_{n+1}\rightarrow X_{n}$ has dense image,
    \item
    and each $X_{n+1}\rightarrow X_{n}$ is trace class. 
\end{enumerate}
An object $G$ of $\mathcal{T}_{c,k}$ is said to be a \textbf{dual nuclear Fr\'{e}chet space} if it is the strong dual (see \cite{reconstruction} Definition 1.10) of a nuclear Fr\'{e}chet space. An object $E$ of $Ind(Ban_{k})$ is said to be a \textbf{bornological (dual) nuclear Fr\'{e}chet space} if it is of the form $F^{b}$ for $F$ a (dual) nuclear Fr\'{e}chet space in $\mathcal{T}_{c,k}$. For details about the claims regarding nuclear Fr\'{e}chet and dual nuclear Fr\'{e}chet spaces made above, see \cite{reconstruction} Section 1.

Since $\mathpzc{E}=Ind(Ban_{k})$ is monoidal elementary quasi-abelian, it is basic Koszul. For $R\in\mathpzc{Alg}_{\mathfrak{Comm}}(Ch(\mathpzc{E}))$ let $\mathpzc{M}={}_{R}\mathpzc{Mod}(Ch(\mathpzc{E}))$. Once more we get an equivalence of $(\infty,1)$-categories
$$\textbf{coAlg}_{\mathfrak{C}_{top}}^{conil}(\mathpzc{M})\cong\textbf{Alg}_{\mathfrak{P}}(\mathpzc{M}).$$
Now suppose that $R=k$ is a Banach field containing $\Q$, and consider the Koszul morphism $\kappa:\mathfrak{S}^{c}\otimes_{H}\mathfrak{coComm}^{ncu}\rightarrow\mathfrak{Lie}$. Let $\mathfrak{g}\in\mathpzc{Alg}_{\mathfrak{Lie}}(\mathpzc{M})$ be concentrated in negative degrees, and suppose that $|\mathfrak{g}|_{n}$ is a bornological nuclear Fr\'{e}chet space for each $n\in\mathbb{Z}$. We will show that $\mathfrak{g}$ is very passable, and give conditions under which it is very good.

%FIX: PUBLISHED REDEFINE formally $\aleph_{1}$-filtered unpublished result of Ben-Bassat and Kremnitzer \cite{bambozzi2015stein} Fr\'{e}chet spaces are formally $\aleph_{1}$-filtered. Moreover, binuclear Fr\'{e}chet spaces are reflexve. 

%\begin{prop}
%
%\end{prop}

%\begin{prop}
%Let $Y$ be a topological bornological space and let 
%$$0\rightarrow k\rightarrow X\rightarrow Y\rightarrow 0$$ 
%be an exact sequence in $CBorn_{k}$ where $k$ is a spherically complete Banach field. Then $X\cong k\oplus Y$. In particular, $X$ is also a topological bornological space. Moreover, any topological bornological space is finitely projective.
%\end{prop}
%
%\begin{proof}
%The functor $(-)^{t}:CBorn_{k}\rightarrow\mathcal{T}_{k}$ preserves cokernels, so $coker(k\rightarrow X^{t})\cong Y^{t}$. Now any finite-dimensional subspace of a locally convex space is closed, and by the Hahn-Banach Theorem we have $X^{t}\cong k\oplus Y^{t}$. Thus, $(X^{t})^{b}\cong k\oplus Y$. Consider the commutative diagram
%\begin{displaymath}
%\xymatrix{
%0\ar[r] & k\ar[r] & (X^{t})^{b}\ar[d]\ar[r] & Y\ar[r]\ar[d] & 0\\
%0\ar[r] & k\ar[r] & X\ar[r] & Y\ar[r] & 0
%}
%\end{displaymath}
%Both sequence are exact. Therefore, the middle vertical morphism is an isomorphism, and we are done. 
%
%The last claim follows immediately from the first.
%\end{proof}
%
%\begin{cor}
%Any complex of dual nuclear Fr\'{e}chet spaces is finitely $K$-cotorsion.
%\end{cor}
%
%We also have the following:
%Now let $k$ be a spherically complete Banach field, and consider $\mathpzc{E}=CBorn_{k}$. Then $k$ is injective?

\begin{prop}\label{dnf}
Let $F$ be a bornological dual nuclear Fr\'{e}chet space (see \cite{reconstruction}). Then the following hold.
%Let $E$ be a binuclear bornological Fr\'{e}chet space or a binuclear bornological dual nuclear Fr\'{e}chet space. Then 
\begin{enumerate}
\item
The object $F$ is finitely $K$-cotorsion in $Ind(Ban_{k})$.
\item
The object $F$ is reflexive in $Ind(Ban_{k})$.
\item
The object $F^{\vee}$ is formally $\aleph_{1}$-filtered.
\item
The object $F^{\vee}$ is nuclear Fr\'{e}chet.
\end{enumerate}
\end{prop}

\begin{proof}
The first claim follows from Proposition 2.12 in \cite{reconstruction}. Let $E$ be a nuclear Fr\'{e}chet space such that $F=(E^{\vee})^{b}$. By Proposition 1.13 in \cite{reconstruction}, we have
$$F^{\vee}=\underline{\mathrm{Hom}}_{CBorn_{k}}((E^{\vee})^{b},k)=(\underline{\mathrm{Hom}}_{\mathcal{T}_{c}}(E^{\vee},k))^{b}\cong E^{b}.$$
Since $E$ is a nuclear Fr\'{e}chet space, $E^{b}$ is a nuclear bornological Fr\'{e}chet space. The third claim is a consequence of Corollary 3.65 in \cite{bambozzi2015stein}. Moreover, again using Proposition 1.13 in \cite{reconstruction}, we have $F^{\vee\vee}=(E^{\vee})^{b}=F$, using the fact that nuclear Fr\'{e}chet spaces are reflexive. 
\end{proof}

Together with Proposition \ref{prop:elexsuff} this implies the following.

\begin{cor}
Let $F_{\bullet}$ be a complex of dual nuclear Fr\'{e}chet spaces. Then $F_{\bullet}$ is finitely $K$-cotorsion and $K$-flat, and $F_{\bullet}^{\vee}$ is $K$-flat. 
\end{cor}
%
%\begin{proof}
%This follows from the previous proposition, Proposition \ref{a}, and Proposition \ref{a}. 
%\end{proof}
%$F^{\vee}=E^{b}$ is a bornological Fr\'{e}chet space, so is formally $\aleph_{1}$-filtered by \cite{Stein}. 
%https://www.encyclopediaofmath.org/index.php/Nuclear_space 
% Therefore, $F=(E)^{b}$ for some locally convex space $E$. Since by $F$ and $k$ are topological bornological spaces, $F^{\vee}=(E^{\vee})^{b}$. 

%Now $E^{\vee}$ is a (toplogical) nuclear Fr\'{e}chet space. Therefore, $(E^{\vee})^{b}$ is a nuclear bornological space. 

%First suppose that $E$ is a dual nuclear Fr\'{e}chet space. We may write $E=lim_{rightarrow}P_{n}$ where $P_{n}$ is a projective Banach space and $P_{n}\rightarrow P_{n+1}$ is nuclear. Then $E^{\vee}=(F^{\vee})^{b}$ where $F^{\vee}$ 
%
%$E$ is a countable projective limit
\begin{prop}\label{BanachKcotors}
Let $k$ be a spherically complete Banach field and let $E$ be a Banach space. Then $E$ is finitely $K$-cotorsion.
\end{prop}

\begin{proof}
When $k$ is spherically complete, it is injective as an object in the category of Banach spaces over $k$. Let $P_{\bullet}\rightarrow E$ be a projective resolution of $E$ in $Ind(Ban_{k})$. We may assume that each $P_{i}$ is a projective \textit{Banach space}. Since $k$ is an injective Banach space, we have equivalences
$$\underline{\mathrm{Hom}}(E,k)\cong \underline{\mathrm{Hom}}(P_{\bullet},k)\cong \mathbb{R}\underline{\mathrm{Hom}}(E,k)$$
as required. 
\end{proof}

\begin{defn}\label{defn:bfs}
A \textbf{bornological Fredholm complex} over a spherically complete Banach field $k$ is a complex $X_{\bullet}$ in $CBorn_{k}$ such that each $X_{n}$ is a bornological space, and each non-zero morphism $d_{n}:X_{n}\rightarrow X_{n-1}$ has finite-dimensional kernel and cokernel.
\end{defn}

\begin{prop}
Let $k$ be spherically complete. A complex $X_{\bullet}$ is homotopically reflexive in each of the following cases.
\begin{enumerate}
\item
The complex $X_{\bullet}$ is a Fredholm complex that is cohomologically bounded in one direction. 
\item
The complex $X_{\bullet}$ is equivalent to a bounded complex of reflexive Banach spaces.
\end{enumerate}
\end{prop}

\begin{proof}
\begin{enumerate}
\item
Since the kernel and cokernel of the morphisms in the complex are finite-dimensional, and $k$ is spherically complete, $X_{\bullet}$ is homotopy equivalent to its cohomology. This is a complex of the form $\bigoplus S^{n}(F_{n})$ for finite-dimensional Banach spaces $F_{n}$. Clearly such complexes are cofibrant, have cofibrant dual, and are reflexive. 
\item
This follows from Proposition \ref{BanachKcotors} and Proposition \ref{boundedbelowKcotors}.
\end{enumerate}
\end{proof}

As a consequence, we get the following.

\begin{thm}\label{analyticopkoszul}
Let $\mathfrak{g}$ be a Lie algebra in $Ch(CBorn_{k})$ concentrated in negative degrees such that each $|\mathfrak{g}|_{n}$ is a dual nuclear Fr\'{e}chet space. Then the morphism $|\mathfrak{g}|^{\vee\vee}\rightarrow |\mathbf{D}_{\kappa}\circ\hat{\mathbf{C}}_{\kappa}(\mathfrak{g})|$ is an equivalence. The unit $\eta_{\mathfrak{g}}:\mathfrak{g}\rightarrow\mathbf{D}_{\kappa}\circ\hat{\mathbf{C}}_{\kappa}(\mathfrak{g})$ is an equivalence in the following cases:
\begin{enumerate}
\item
the underlying module $|\mathfrak{g}|$ is a Fredholm complex; 
\item
the underlying module $|\mathfrak{g}|$ is equivalent to a bounded complex of reflexive Banach spaces;
 \end{enumerate}
\end{thm}

In future work we expect to be able to use this result to study deformations of analytic moduli spaces of instantons.

\subsubsection{A Remark on the Nuclear Subcategory}

Let us here restrict to the case where $k=\mathbb{C}$. Following page 86 of \cite{kelly2025localising}, define $Nuc_{\omega_{0}}(\mathbb{C})$ to be the full subcategory of $Ind(Ban_{\mathbb{C}})$ consisting of objects $F$ that can be written as 
$$F\cong\colim_{\mathbb{N}}X_{n}$$
where each $X_{n}$ is a projective Banach space, and each $X_{n}\rightarrow X_{n+1}$ is trace class. 

\begin{rem}
    Note that in \cite{kelly2025localising} we work with the more restrictive \textit{infinitely nuclear subobjects} $Nuc^{\infty}_{\omega_{0}}(\mathbb{C})$. The proofs for infinitely nuclear objects work mutatis mutandis for nuclear objects. 
\end{rem}

As countable colimits of compact objects, objects of $Nuc_{\omega_{0}}(\mathbb{C})$ are $\aleph_{1}$-compact. We define $Nuc_{\mathbb{C}}=Ind_{\aleph_{1}}(Nuc_{\omega_{0}}(\mathbb{C}))$. It is a non-trivial result (\cite{stempfhuber2024bornological} Proposition 4.2.43) that $Nuc_{\mathbb{C}}$ is closed under finite limits, all colimits, and tensor products in $Ind(Ban_{\mathbb{C}})$. Moreover, it contains dual nuclear Fr\'{e}chet spaces and nuclear Fr\'{e}chet spaces (by \cite{reconstruction} Lemma 2.11 and \cite{bambozzi2015stein} Lemma 3.60, respectively). In particular, $Nuc_{\mathbb{C}}$ inherits the quasi-abelian property from $Ind(Ban_{\mathbb{C}})$, as well as the property of having exact filtered colimits that commute with kernels. Moreover, as filtered colimits of projective, and hence flat, objects, all objects of $Nuc_{\mathbb{C}}$ are flat. By construction, the category $Nuc_{\mathbb{C}}$ is manifestly locally presentable. We may therefore equip $Ch(Nuc_{\mathbb{C}})$ with the flat model structure (that coincides with the injective model structure). Then $Ch(Nuc_{\mathbb{C}})$ can be given the structure of a basic Koszul category.
For a Koszul twisting morphism $\alpha:\mathfrak{C}\rightarrow\mathfrak{P}$ in $Ch(Nuc_{\mathbb{C}})$ we get a Quillen equivalence
$$\adj{\Omega^{filt}_{\alpha}}{\mathpzc{coAlg}_{\mathfrak{C}}^{conil}(Ch(Nuc_{\mathbb{C}}))}{\mathpzc{Alg}_{\mathfrak{P}}(Ch(Nuc_{\mathbb{C}}))}{B_{\alpha}^{filt}},$$
where the left-hand side is equipped with the left-transferred model structure. In fact, one can also replace $Nuc_{\mathbb{C}}$ by $Nuc^{\infty}_{\mathbb{C}}\defeq Ind_{\aleph_{1}}Nuc^{\infty}_{\omega_{0},\mathbb{C}})$.

\subsubsection{Non-Archimedean Banach Spaces}\label{subsubsec:nABan}
Suppose now that $k$ is a non-Archimedean Banach field. By Lemma 3.49 of \cite{bambozzi2015stein}, every object in $Ind(Ban^{nA}_{k})$ is flat. %Therefore, by Proposition \ref{boundedKflat} every complex in $Ch(Ind(Ban^{nA}_{k}))$ is $K$-flat, and every admissible monomorphism is pure. We get the following.

\begin{thm}
The bar-cobar construction induces an equivalence of $(\infty,1)$-categories
$$\adj{\Omega^{filt}_{\alpha}}{\textbf{coAlg}_{\mathfrak{C}}^{conil}(Ch(Ind(Ban^{nA}_{k})))}{\textbf{Alg}_{\mathfrak{P}}(Ch(Ind(Ban^{nA}_{k})))}{\textbf{B}^{filt}_{\alpha}}.$$
Moreover, this equivalence arises from a Quillen equivalence
$$\adj{\Omega^{filt}_{\alpha}}{\mathpzc{coAlg}_{\mathfrak{C}}^{conil}(Ch(Ind(Ban^{nA}_{k})))}{\mathpzc{Alg}_{\mathfrak{P}}(Ch(Ind(Ban^{nA}_{k})))}{B^{filt}_{\alpha}}.$$
%$(\mathpzc{coAlg}_{\mathfrak{C}}^{conil},\mathpzc{Alg}_{\mathfrak{P}})$ is a filtered $\alpha$-ideal.
\end{thm}
%Moreover, in Theorem \ref{analyticopkoszul} we can relax the assumptions on $\mathfrak{g}$.  

Finally, we consider a quasi-abelian category, but an exact structure thereon that is not the quasi-abelian one: the wide subcategory $Ban^{nA,\le 1}_{k}\subset Ban^{nA}_{k}$ of non-Archimedean Banach spaces spanned by morphisms with norm at most $1$. This is a closed symmetric monoidal quasi-abelian category. However, there is another exact structure on this category, introduced in \cite{kelly2016homotopy} and called the \textbf{strong exact structure}, which makes $Ban^{nA,\le 1}$ into a weakly $\textbf{AdMon}$-elementary exact category with enough projectives.  Yet again we get the familiar equivalence of $(\infty,1)$-categories:

$$\textbf{coAlg}_{\mathfrak{C}_{top}}^{conil}(Ch(Ban^{nA,\le 1}_{k}))\cong\textbf{Alg}_{\mathfrak{P}}(Ch(Ban^{nA,\le 1}_{k})).$$
Suppose that $k$ is a spherically complete non-Archimedean Banach field containing $\Q$, and once more consider the Koszul morphism $\kappa:\mathfrak{S}^{c}\otimes_{H}\mathfrak{coComm}^{ncu}\rightarrow\mathfrak{Lie}$. If $\mathfrak{g}\in\mathpzc{Alg}_{\mathfrak{Lie}}(Ch(Ban^{nA,\le1}_))$ is concentrated in negative degrees, is such that each $|\mathfrak{g}|_{n}$ is free of finite rank, and is such that each $d_{n}$ is an admissible morphism, then $\mathfrak{g}$ is very good. Therefore, the unit $\eta_{\mathfrak{g}}:\mathfrak{g}\rightarrow\mathbf{D}_{\kappa}\circ\hat{\mathbf{C}}_{\kappa}(\mathfrak{g})$ is an equivalence. Suppose that $\mathfrak{g}_{-1}$ is $n$-dimensional. The underlying space of $(B_{\kappa}(\mathfrak{g})[1])_{0}$ is the subspace of the space of formal power series $k[[t_{1},\ldots,t_{n}]]$ consisting of power series $\sum_{I\in\mathbb{N}_{0}^{n}}a_{I}t^{I}$ with the condition $|a_{I}|\rightarrow 0$ as $|I|\rightarrow\infty$. This is a Banach space with norm $||\sum_{I\in\mathbb{N}_{0}^{n}}a_{I}t^{I}||=\sup_{I}|a_{I}|$. The coproduct is uniquely determined by $t_{i}\mapsto 1\otimes t_{i}+t_{i}\otimes1$. We think of this object as the Tate coalgebra in $n$ variables.

\subsubsection{Formal Koszul Duality}
 Let $K$ be a non-trivially valued, discretely valued non-Archimedean Banach field, with ring of integers $R=\mathcal{O}_{K}$, and maximal ideal $\mathfrak{m}=(\varpi)$. The ring $R$ is a non-Archimedean Banach ring when equipped with the $\varpi$-adic norm. Examples to keep in mind are the $p$-adic integers $\mathbb{Z}_{p}$ equipped with the $p$-adic norm, and the field of Laurent series $F((t))$ over some field $F$ with the $t$-adic norm. We consider the exact category $\mathrm{Ban}^{nA,\le1}_{R}$, equipped with the strong exact structure.
Denote by $k$ the residue field of $K$. This has the trivial norm, for which all non-zero elements have norm $1$. Let $B$ be an $E_{\infty}$-algebra in $Ch(\mathrm{Ban}^{nA,\le1}_{R})$. Write $k_{n}=\mathcal{O}_{K}\big\slash\mathfrak{m}^{n}$, and $B_{n}=k_{n}\coprod^{\mathbb{L}}_{R}B$, with the coproduct taken in the category of $E_{\infty}$-algebras (note that, on underlying objects, this is the relative tensor product).  We have $R\cong\lim_{n\in\mathbb{N}}k_{n}\cong\mathbb{R}\lim_{n\in\mathbb{N}}k_{n}$. The second equivalence follows from the fact that $k_{n+1}\rightarrow k_{n}$ is an admissible epimorphism (as we establish below), the fact that $Ban_{R}^{nA,\le1}$ has enough projectives, and the Mittag-Leffler theorem. Equip the category $Ch(Ban_{R}^{nA,\le1})$ with the projective model structure. Then it becomes a strong, $KC$-monoidal Koszul category. The morphisms $R\rightarrow K$ and $R\rightarrow k_{n}$ give rise to respective base-change functors
$$K\otimes^{\mathbb{L}}_{R}(-):\textbf{Ch}(Ban_{R}^{nA,\le1})\rightarrow\textbf{Ch}(Ban_{K}^{nA,\le1})$$
$$k_{n}\otimes^{\mathbb{L}}_{R}(-):\textbf{Ch}(Ban_{R}^{nA,\le1})\rightarrow\textbf{Ch}(Ban_{k_{n}}^{nA,\le1}).$$
Geometrically, we interpret the first functor as taking the generic fibre, and the functor $k\otimes^{\mathbb{L}}_{R}(-)$ as taking the special fibre. 

%Note that for any Banach ring, there is a symmetric (strongly) monoidal, quasi-isomorphism-preserving functor 
%$$i_{R}:Ban_{R}^{\le1,nA}\rightarrow Ch(Ind(Ban_{R}^{nA})).$$ 
%For any morphism $R\rightarrow S$ of unital commutative Banach rings, we get a commutative diagram of functors
%\begin{displaymath}
 %   \xymatrix{
  %  \textbf{Ch}(Ban_{R}^{nA,\le1})\ar[r]^{i_{R}}\ar[d]^{S\hat{\otimes}^{\mathbb{L}}(-)} & {Ch}(Ind(Ban_{R}^{nA}))\ar[d]^{S\hat{\otimes}^{\mathbb{L}}(-)} \\
  %   \textbf{Ch}(Ban_{S}^{nA,\le1})\ar[r]^{i_{S}} & {Ch}(Ind(Ban_{S}^{nA})).
   % }
%\end{displaymath}  

Let $\phi:A\rightarrow k$ be an object of ${}_{R\backslash}(\textbf{Alg}_{E_{\infty}})_{\slash k}$. Geometrically, one thinks of $A$ as a formal scheme, and of $A\rightarrow k$ as a $k$-point of $A$. We want to know when we can lift such a point to an `analytic' point $A\rightarrow R$. As we explain below, each $k_{n+1}\rightarrow k_{n}$ is an admissible epimorphism. Now, $R\cong\mathbb{R}\lim_{n} k_{n}$ by Mittag-Leffler. Therefore, to find a lift $A\rightarrow R$, we need to find an element of 
$$\pi_{0}\textbf{Map}_{{}_{R\backslash}(\textbf{Alg}_{E_{\infty}})_{\slash k}}(A\rightarrow k,R\rightarrow k)\cong\pi_{0}(\mathbb{R}\lim_{n}\textbf{Map}_{{}_{R\backslash}(\textbf{Alg}_{E_{\infty}})_{\slash k}}(k_{n}\coprod^{\mathbb{L}}_{R}A\rightarrow k,k_{n}\rightarrow k)).$$
Consequently, we need to iteratively construct lifts of morphisms $k_{n+1}\coprod^{\mathbb{L}}_{R}A\rightarrow k_{n}$ to morphisms $k_{n+1}\coprod^{\mathbb{L}}_{R}A\rightarrow k_{n+1}$.

%In particular, if we can show that each $\textbf{Map}_{{}_{k_{n+1}\backslash}(\textbf{Alg}_{E_{\infty}})_{\slash k_{n}}}(k_{n+1}\coprod^{\mathbb{L}}_{R}A\rightarrow k,k_{n}\rightarrow k))$ is contractible, we can always find a lift. By (the $E_{\infty}$-version of) \cite{toen2004homotopical} Proposition 1.4.2.5, this space is contractible provided that each 

Now, the morphism $k_{n+1}\rightarrow k_{n}$ is a strong epimorphism. Indeed, it suffices to establish that each $R\rightarrow k_{n}$ is a strong epimorphism. Let $[x]\in k_{n}$ be non-zero, and let $m$ be maximal with the property that $x=\varpi^{m}x_{m}$ for some $x_{m}$.  Note that $m<n$. Then $|x|=|\varpi|^{m}$. Now $|[x]|=\inf_{y\in R}\{|x+\varpi^{n}y|\}$. We have $|\varpi^{n}y|\le |\varpi|^{n}<|\varpi|^{m}=|x|$. The non-Archimedean property gives $|x+\varpi^{n}y|=|x|$. Hence $|[x]|=|x|$. In particular, $\pi_{n+1}:k_{n+1}\rightarrow k_{n}$ is a fibration. The kernel then coincides with the homotopy fibre, and is $ker(\pi_{n+1})=\{[x]|\exists\tilde{x}:\varpi^{n}\tilde{x}=x\}$.

%Note that, as $k_{n+1}$-modules, we have an isomorphism $k_{n+1}\cong ker(\pi_{n+1})$ given by multiplication by $\varpi$. In particular, $ker(\pi_{n+1})$ is a cofibrant $k_{n+1}$-module. We then have 
Consider the morphism
$$ker(\pi_{n+1})\hat{\otimes}^{\mathbb{L}}_{k_{n+1}}ker(\pi_{n+1})\rightarrow ker(\pi_{n+1}).$$
Note that, since $ker(\pi_{n+1})$ lives in homological degree $0$, and $ker(\pi_{n+1})\hat{\otimes}^{\mathbb{L}}_{k_{n+1}}ker(\pi_{n+1})$ is concentrated in non-negative degrees, we have
$$\textbf{Map}(ker(\pi_{n+1})\hat{\otimes}^{\mathbb{L}}_{k_{n+1}}ker(\pi_{n+1}),ker(\pi_{n+1}))\cong\textbf{Map}(ker(\pi_{n+1})\hat{\otimes}_{k_{n+1}}ker(\pi_{n+1}),ker(\pi_{n+1})).$$
The restriction of the multiplication morphism on $k_{n+1}$ to
$$ker(\pi_{n+1})\hat{\otimes}_{k_{n+1}}ker(\pi_{n+1})\rightarrow ker(\pi_{n+1})$$
is clearly zero. Thus, $k_{n+1}\rightarrow k_{n}$ is a $0$-small extension in the sense of \cite{higheralgebra} Definition 7.4.1.18. Suppose that we have constructed the requisite morphism $A\rightarrow k_{n}$ over $A\rightarrow k$. Then, by (the $E_{\infty}$-version of) \cite{toen2004homotopical} Proposition 1.4.2.5, the space of lifts 
$$\textbf{Map}_{{}_{R\backslash}(\textbf{Alg}_{E_{\infty}})_{\slash k_{n}}}(A\rightarrow k_{n},k_{n+1}\rightarrow k_{n})$$
is controlled by the cotangent complex $k_{n}\hat{\otimes}_{R}^{\mathbb{L}}\mathbb{L}_{A\big\slash R}$.
%\cong\mathbb{L}_{k_{n}\big\slash k_{n}\coprod_{R}^{\mathbb{L}}A}$. 
Up to shifting, we think of this object as a co-$L_{\infty}$-coalgebra. Here, by co-$L_{\infty}$-coalgebra, we mean a coalgebra over the Koszul dual cooperad to a cofibrant $E_{\infty}$-operad, furnished by Example \ref{example:induced} and the classical operadic bar construction in the category of abelian groups.

\subsubsection{Topological Koszul Duality}

Let $k$ be a spherically complete non-trivially valued Banach field. Consider the example $Pro(Ban_{k})$ of Example \ref{example:topalg} (2), which we think of as some sort of category of complete locally convex topological vector spaces over $k$. Then $Ch((Pro(Ban_{k}))^{op})$ is a strong Koszul category, when equipped with the projective model structure. For simplicity, let $\alpha:\mathfrak{C}\rightarrow\mathfrak{P}$ be a pure Koszul twisting morphism in $Ch((Pro(Ban_{k}))^{op})$. We get our usual equivalence of $(\infty,1)$-categories
$$\adj{\Omega_{\alpha}}{\textbf{coAlg}_{\mathfrak{C}_{top}}^{conil}(Ch((Pro(Ban_{k}))^{op}))}{\textbf{Alg}_{\mathfrak{P}}(Ch((Pro(Ban_{k}))^{op}))}{\textbf{B}_{\alpha}}.$$
However, we are really interested in $Pro(Ban_{k})$, so we pass to opposite categories, and get an equivalence
$$\adj{\textbf{B}_{\alpha}^{op}}{\textbf{coAlg}_{\mathfrak{P}^{op}}(Ch(Pro(Ban_{k})))}{\textbf{Alg}^{nil}_{(\mathfrak{C}_{top})^{op}}(Ch(Pro(Ban_{k})))}{\Omega_{\alpha}^{op}},$$
where 
\begin{enumerate}
    \item 
    $\mathfrak{P}^{op}$ is the cooperad opposite to $\mathfrak{P}$,
    \item 
     $(\mathfrak{C}_{top})^{op}$ is the operad opposite to $(\mathfrak{C}_{top})$,
     \item 
     $\textbf{coAlg}_{\mathfrak{P}^{op}}(Ch(Pro(Ban_{k})))$ is the localisation of all coalgebras (not just conilpotent ones!) at the class of quasi-isomorphisms (not $\alpha$-equivalences!), and
     \item  $\textbf{Alg}^{nil}_{(\mathfrak{C}_{top})^{op}}(Ch(Pro(Ban_{k})))$ is the localisation of the category of nilpotent algebras at those morphisms that become weak equivalences upon application of $\Omega_{\alpha}^{op}$.
\end{enumerate}

For example, if $k$ has characteristic $0$, we can take the pure Koszul twisting morphism $\alpha:\mathfrak{S}^{c}\otimes_{H}\mathfrak{coComm}^{ncu}\rightarrow\mathfrak{Lie}$ in  $Ch(Pro(Ban_{k})^{op})$, and we get an equivalence
$$\adj{\textbf{B}_{\alpha}^{op}}{\textbf{coAlg}_{\mathfrak{coLie}}(Ch(Pro(Ban_{k})))}{\textbf{Alg}^{nil}_{\mathfrak{Comm}^{aug}}(Ch(Pro(Ban_{k})))}{\Omega_{\alpha}^{op}}.$$
The functor $\Omega_{\alpha}^{op}$ outputs the cotangent Lie coalgebra.

%Because $Ch(Pro(Ban_{k}))$ is not $C$-monoidal (and not $K$-monoidal), filtered $\alpha$-ideals are harder to come by, unless both $\mathfrak{P}$ and $\mathfrak{C}$ are concentrated in arity $1$. 

\subsubsection{Condensed/Solid/Liquid Koszul Duality}

Let $\kappa$ be a regular cardinal, and let $\mathrm{Cond}_{\kappa}(\mathrm{Ab})$ denote the category of $\kappa$-condensed abelian groups (\cite{clausenscholze1} Definition 2.1). As explained in \cite{kelly2021analytic} Theorem 10.1, it is easily deduced from \cite{clausenscholze1} that $\mathrm{Cond}_{\kappa}(\mathrm{Ab})$ is an elementary abelian category with enough projectives. This is also true of the category of solid abelian groups (\cite{clausenscholze1} Lecture V), and the category of $p$-liquid $\mathbb{C}$-vector spaces for any $0<p\le 1$ (\cite{clausenscholze2} Lecture VI). In fact, the solid category is even monoidal elementary. Thus, our Koszul duality result also applies in these settings.

\subsubsection{Extensions to Modules Over Quasi-Bialgebras}\label{subsubsecquasi-bi}

Here, we show how to enhance some of the previous examples.

Let $\mathpzc{E}$ be a complete and cocomplete monoidal exact category, and let $H$ be an associative algebra in $\mathpzc{E}$. By \cite{kelly2016homotopy} Proposition 2.73, we may consider the category of left $H$-modules, ${}_{H}\mathpzc{Mod}(\mathpzc{E})$. This category is naturally equipped with the exact structure wherein a sequence
$$0\rightarrow X\rightarrow Y\rightarrow Z\rightarrow 0$$
of $H$-modules is exact precisely if 
$$0\rightarrow |X|\rightarrow |Y|\rightarrow |Z|\rightarrow 0$$
is exact in $\mathpzc{E}$, where $|-|:{}_{H}\mathpzc{Mod}(\mathpzc{E})\rightarrow\mathpzc{E}$ denotes the forgetful functor. 

\begin{defn}
    We say that $H$ is a \textbf{quasi-bialgebra} if there exists a symmetric monoidal structure on ${}_{H}\mathpzc{Mod}(\mathpzc{E})$, whereby the forgetful functor $|-|:{}_{H}\mathpzc{Mod}(\mathpzc{E})\rightarrow \mathpzc{E}$ is strong symmetric monoidal.
\end{defn}

For example, $H$ could be a bialgebra.

\begin{defn}
    Let $\mathpzc{E}$ be a monoidal exact category. We say that an object $X$ of $\mathpzc{E}$ is \textbf{faithfully flat} if $X$ is flat and the functor $X\otimes(-):\mathpzc{E}\rightarrow\mathpzc{E}$ reflects admissible monomorphisms.
\end{defn}

\begin{prop}\label{prop:pureHMod}
\begin{enumerate}
    \item 
        A morphism $f:M\rightarrow N$ in ${}_{H}\mathpzc{Mod}(\mathpzc{E})$ is a pure monomorphism if $|f|$ is a pure monomorphism in $\mathpzc{E}$. The converse is true provided that $H$ is faithfully flat as an object of $\mathpzc{E}$.
        \item 
        Let $g:X\rightarrow Y$ be a pure monomorphism in $\mathpzc{E}$. Then $Id_{H}\otimes g$ is a pure monomorphism in ${}_{H}\mathpzc{Mod}(\mathpzc{E})$.
\end{enumerate}
\end{prop}

\begin{proof}
\begin{enumerate}
    \item 
        Suppose that $|f|$ is a pure monomorphism. Let $Z$ be a left $H$-module. We have $|Id_{Z}\otimes f|\cong |Id_{Z}|\otimes |f|$. Since $|f|$ is pure, $|Id_{Z}|\otimes |f|$ is an admissible monomorphism. Since the forgetful functor reflects exactness, $Id_{Z}\otimes f$ is also an admissible monomorphism.

    Conversely, suppose that $f$ is a pure monomorphism, and that $H$ is faithfully flat. Let $Z\in\mathpzc{E}$. Consider the free left $H$-module, $H\otimes Z$. Then $Id_{H\otimes Z}\otimes f$ is an admissible monomorphism. Since the forgetful functor is exact, $|Id_{H\otimes Z}\otimes f|\cong |Id_{H}|\otimes Id_{Z}\otimes |f|$ is an admissible monomorphism in $\mathpzc{E}$. By faithful flatness, $Id_{Z}\otimes |f|$ is an admissible monomorphism.
    \item 
    We have $|Id_{H}\otimes g|\cong |Id_{H}|\otimes |g|$. By (1) and the fact that pure monomorphisms are stable under tensor products with objects, $Id_{H}\otimes g$ is a pure monomorphism.
\end{enumerate}

\end{proof}

Suppose now that $\mathpzc{E}$ is a basic Koszul category, with $\{G_{i}\}_{i\in\mathcal{I}}$ and $\{f_{l}\}_{l\in\mathcal{L}}
$ as in Definition \ref{basic Koszul} (4), and $\{g_{k}\}_{k\in\mathcal{K}}$ as in Definition \ref{basic Koszul} (5). Let $H$ be a quasi-bialgebra that is faithfully flat as an object of $\mathpzc{E}$. Consider the transferred model structure on 
$${}_{H}\mathpzc{Mod}(Ch(\mathpzc{E}))\cong Ch({}_{H}\mathpzc{Mod}(\mathpzc{E})).$$
With this model structure, $Ch({}_{H}\mathpzc{Mod}(\mathpzc{E}))$ is also a combinatorial model category, and the weak equivalences are precisely the quasi-isomorphisms. As a consequence of Proposition \ref{prop:pureHMod}, $Ch({}_{H}\mathpzc{Mod}(\mathpzc{E}))$ is also weakly $\textbf{PureMon}$-elementary, and transfinite compositions of pure monomorphisms are pure monomorphisms. The collection
$$\mathfrak{G}_{H}=\{S^{n}(H\otimes G_{i})\rightarrow D^{n+1}(H\otimes G_{i})\}_{n\in\Z,i\in\mathcal{I}}\cup\{0\rightarrow D^{n}(H\otimes G_{i})\}_{n\in\Z,i\in\mathcal{I}}\cup\{S^{n}(Id_{H}\otimes f_{l})\}_{n\in\Z,l\in\mathcal{L}}$$
is a generating set of cofibrations for the model structure on $Ch({}_{H}\mathpzc{Mod}(\mathpzc{E}))$, and the collection $\{D^{n}(Id_{H}\otimes g_{k})\}_{n\in\mathbb{Z},k\in\mathcal{K}}$ is a generating set of acyclic cofibrations. Also, by Proposition \ref{prop:pureHMod}, the morphisms $Id_{H}\otimes f_{l}$ and $Id_{H}\otimes g_{k}$ are all pure monomorphisms. We have proved the following.

\begin{prop}
    Let $\mathpzc{E}$ be a basic Koszul category. Let $H$ be a quasi-bialgebra in $\mathpzc{E}$ that is faithfully flat as an object of $\mathpzc{E}$. Then ${}_{H}\mathpzc{Mod}(\mathpzc{E})$ has a canonical basic Koszul category structure. Moreover, if the basic Koszul structure on $\mathpzc{E}$ is projective, then so is the induced one on ${}_{H}\mathpzc{Mod}(\mathpzc{E})$.
\end{prop}

\begin{example}
\begin{enumerate}
    \item 
    Let $k$ be a field of characteristic zero, and $\mathfrak{g}$ a Lie algebra defined over $k$. Consider the enveloping Hopf algebra $U(\mathfrak{g})$. As a bialgebra, it is a quasi-bialgebra. Over a field of characteristic zero, all quasi-bialgebras are faithfully flat. Then the category ${}_{U(\mathfrak{g})}\mathpzc{Mod}({}_{k}\mathpzc{Vect})$ of $\mathfrak{g}$-modules may be endowed with a projective basic Koszul category structure when $Ch({}_{k}\mathpzc{Vect})$ is equipped with the projective model structure. This can be generalised to any basic Koszul category enriched over $\mathbb{Q}$.
    \item 
    Let $X$ be a smooth algebraic variety defined over a field $k$ of characteristic zero. Consider the basic Koszul category $QCoh(X)$ from Subsection \ref{sec:qcsheaves}. Let $\mathcal{D}_{X}$ denote the sheaf of differential operators on $X$. Then $\mathcal{D}_{X}$ is a quasi-bialgebra in $QCoh(X)$. For a definition of this sheaf, see \cite{hotta2007d} Section 1.1, and for a proof of the quasi-bialgebra structure, see \cite{hotta2007d} Proposition 1.2.9 (i). The object $\mathcal{D}_{X}$ is faithfully flat via the local symbol isomorphism. Indeed, when $X=\mathrm{Spec}(k[x_{1},\ldots,x_{n}])$, as a $k[x_{1},\ldots,x_{n}]$-module, $\mathcal{D}_{\mathrm{Spec}(k[x_{1},\ldots,x_{n}])}(\mathrm{Spec}(k[x_{1},\ldots,x_{n}]))\cong k[x_{1},\ldots,x_{n},y_{1},\ldots,y_{n}]$. Consequently, $\mathcal{D}_{X}$ is locally faithfully flat, and therefore faithfully flat, as an $\mathcal{O_X}$-module.
\end{enumerate}
\end{example}

\subsubsection{Future work: a sketch of deformation theory}

Let $\mathpzc{E}$ be a locally presentable, weakly $\textbf{PureMon}$-elementary exact category, enriched over $\mathbb{Q}$. Consider the HA context $Ch_{\ge0}(\mathpzc{E})$. Write $\textbf{Comm}(\mathpzc{E})$ for the $(\infty,1)$-category presented by the transferred model structure on $\mathpzc{Alg}_{\mathfrak{Comm}}(Ch_{\ge0}(\mathpzc{E}))$, and $\textbf{Aff}(\mathpzc{E})=\textbf{Comm}(\mathpzc{E})^{op}$. Suppose that there exists a Grothendieck pretopology (\cite{toen2004homotopical} Definition 1.3.1.1) on $\textbf{Aff}(\mathpzc{E})$ in which every cover has a finite refinement. Suppose further that, whenever $\{\mathrm{Spec}(A_{i})\rightarrow\mathrm{Spec}(A)\}_{i=1}^{n}$ is a cover in $\tau$, the morphism $A\rightarrow A_{i}$ is a homotopy epimorphism -- that is, the multiplication morphism $A_{i}\otimes^{\mathbb{L}}_{A}A_{i}\rightarrow A_{i}$ is an equivalence. In particular, by \cite{dang} Lemma 2.1.42, such morphisms are formally \'{e}tale -- that is $\mathbb{L}_{A_{i}\big\slash A}\cong 0$ for each $i$. We further suppose that the topology $\tau$ satisfies \v{C}ech descent (\cite{toen2004homotopical} Section 1.2.12). By \cite{dang} and \cite{kelly2025localising}, all the assumptions above are satisfied by the Zariski topology in algebraic geometry, the rational topology in rigid analytic geometry over a non-Archimedean, non-trivially valued Banach field $K$, and the analytic topology over $\mathbb{C}$ (on compact Steins). Let $\textbf{P}$ denote the class of formally \'{e}tale morphisms in $\textbf{Comm}(\mathpzc{E})$. The tuple
$$(\textbf{Aff}(\mathpzc{E}),\textbf{Aff}(\mathpzc{E}),\tau,\textbf{P})$$
is very close to being a homotopical algebraic geometry context in the sense of \cite{toen2004homotopical} Definition 1.3.2.13. The only axiom we weaken is the requirement that modules satisfy \textit{hyperdescent}: we only ask for \v{C}ech descent. Nonetheless, essentially all of the theory of \cite{toen2004homotopical} still applies, as long as we understand the category $\textbf{Stk}(\mathpzc{E},\tau)$ of stacks to mean the localisation of prestacks $\textbf{PreStk}(\mathpzc{E})=\textbf{Fun}(\textbf{Comm}(\mathpzc{E}),\textbf{Spaces})$ at \v{C}ech covers rather than hypercovers. 

Fix a derived scheme $\mathcal{X}\rightarrow\mathrm{Spec}(B)$ (\cite{kelly2021analytic} Definition 6.16) for this context, defined over $\mathrm{Spec}(B)$ (in the dagger affinoid setting over $\mathbb{C}$, the category of schemes includes complex analytic manifolds, and in the affinoid case over a non-Archimedean field, the category of schemes includes rigid analytic varieties). As proven in \cite{kelly2021analytic} Proposition 8.12, $\mathcal{X}$ admits a global cotangent complex.

%Let $\mathrm{Spec}(B)\rightarrow\mathcal{X}$ be a morphism. There exists some cover $\{\mathrm{Spec}(B_{j})\rightarrow\mathrm{Spec}(B)\}_{j=1}^{m}$ such that each composite morphism $\mathrm{Spec}(B_{j})\rightarrow\mathrm{Spec}(B)\rightarrow\mathcal{X}$ factors through some $\mathrm{Spec}(A_{i(j)})\rightarrow\mathcal{X}$. 

Now, let $x:\mathrm{Spec}(B)\rightarrow\mathcal{X}$ be a morphism. We regard this as a $B$-point of $\mathcal{X}$. Consider the cotangent complex at $x$
$$\mathbb{L}_{\mathcal{X},x}\cong x^{*}\mathbb{L}_{\mathcal{X}\big\slash\mathrm{Spec}(B)}.$$
Fix an atlas $\{\mathrm{Spec}(A_{i})\rightarrow\mathcal{X}\}_{i\in\mathcal{I}}$ for $\mathcal{X}$. By the definition of an atlas, there exists a cover in $\tau$, $\{\phi_{j}:\mathrm{Spec}(B_{j})\rightarrow\mathrm{Spec}(B)\}_{j=1}^{n}$, such that each $\mathrm{Spec}(B_{j})\rightarrow\mathrm{Spec}(B)\rightarrow\mathcal{X}$ factors as $\mathrm{Spec}(B_{j})\rightarrow\mathrm{Spec}(A_{i(j)})\rightarrow\mathcal{X}$. Finally, by the homotopy epimorphism condition, in the following diagram both squares in the top row are pullbacks
\begin{displaymath}
    \xymatrix{
    \mathrm{Spec}(B_{j})\ar[d]\ar[r]^{\psi'_{j}} & \mathrm{Spec}(A_{i(j)}\otimes^{\mathbb{L}}_{B}B_{j})\ar[d]^{g_{j}}\ar[r] & \mathrm{Spec}(B_{j})\ar[d]^{\phi_{j}}\\
    \mathrm{Spec}(B_{j})\ar[r]^{\psi_{j}}\ar[d]_{\phi_{j}} & \mathrm{Spec}(A_{i(j)})\ar[d]^{f_{i(j)}}\ar[r] & \mathrm{Spec}(B)\ar[d]\\
    \mathrm{Spec}(B)\ar[r]^{x} & \mathcal{X}\ar[r] & \mathrm{Spec}(B).
    }
\end{displaymath}

Write $x_{j}=x\circ\phi_{j}$.
We have 
$$\mathbb{L}_{\mathcal{X},x_{j}}=\phi_{j}^{*}x^{*}\mathbb{L}_{\mathcal{X}\big\slash\mathrm{Spec}(B)}\cong \psi_{j}^{*}f_{i(j)}^{*}\mathbb{L}_{\mathcal{X}\big\slash\mathrm{Spec}(B)}\cong\mathbb{L}_{0}(A_{i(j)}\otimes^{\mathbb{L}}_{B}B_{j}).$$
In particular, $\mathbb{L}_{\mathcal{X},x_{j}}[1]\cong\mathbb{L}_{0}(A_{i(j)}\otimes^{\mathbb{L}}_{B}B_{j})[1]$, and this gives a coLie coalgebra structure on $\mathbb{L}_{\mathcal{X},x_{j}}[1]$.

A priori, this construction is not manifestly functorial. In particular, the coLie coalgebra structure depends on $A_{i(j)}$ and $B_{j}$. However, in forthcoming work with Rhiannon Savage, we will globalise this structure, analogously to \cite{hennion2015tangent}. In fact, this will work for $\mathcal{X}$ a general Artin stack. We will show that $(\mathbb{L}_{\mathcal{X},x}[1])^{\vee}$ is the Lie algebra governing the formal moduli problem of (formal) deformations of $\mathcal{X}$ at the point $x$, in a similar sense to \cite{lurie2011derived}.

\bibliographystyle{amsalpha}
\bibliography{Koszularxiv2prediff.bib}

\end{document}